\documentclass[12pt]{article}
\usepackage{amsmath,amssymb,mathtools}
\usepackage{verbatim}
\usepackage{dsfont,bm,mathrsfs}
\usepackage{color}
\usepackage{graphicx}
\usepackage{amsthm}
\usepackage[shortlabels]{enumitem} 
\usepackage{esint} 
\usepackage[title]{appendix}

\usepackage{xcolor}
\definecolor{darkgreen}{rgb}{0,0.75,0}
\definecolor{darkred}{rgb}{0.75,0,0}
\definecolor{darkmagenta}{rgb}{0.5,0,0.5}
\usepackage[colorlinks,citecolor=darkgreen,linkcolor=darkred,urlcolor=darkmagenta]{hyperref}
\newtheorem{theorem}{Theorem}[section]
\newtheorem{thm}[theorem]{Theorem}

\newtheorem{cor}[theorem]{Corollary}
\newtheorem{lemma}[theorem]{Lemma}
\newtheorem{lem}[theorem]{Lemma}

\newtheorem{prop}[theorem]{Proposition}

\theoremstyle{definition}
\newtheorem{definition}[theorem]{Definition}
\newtheorem{defn}[theorem]{Definition}
\newtheorem{assumption}[theorem]{Assumption}
\newtheorem{assum}[theorem]{Assumption}
\newtheorem{remark}[theorem]{Remark}
\newtheorem{rmk}[theorem]{Remark}
\newtheorem*{rmk*}{Remark} 
\newtheorem{con}[theorem]{Conjecture}

\newtheorem{problem}[theorem]{Problem}

\numberwithin{equation}{section}
\numberwithin{figure}{section}

\newcommand{\trinorm}[1]{{\left\vert\kern-0.25ex\left\vert\kern-0.25ex\left\vert #1
    \right\vert\kern-0.25ex\right\vert\kern-0.25ex\right\vert}}
\newcommand{\indicator}[1]{\mathds{1}_{#1}} 
\newcommand{\bord}[1]{\partial#1}
\newcommand{\closure}[1]{\overline{#1}}

\DeclareMathOperator*{\osc}{osc}

\newcommand{\len}{\mathsf{len}}

\newcommand{\CAP}{\mathrm{cap}}

\newcommand{\PATH}{\mathsf{Path}}
\newcommand{\MOD}{\mathrm{Mod}}
\newcommand{\ADM}{\mathsf{Adm}}

\newcommand{\BCL}{\hyperref[BCL]{\textup{BCL}}}
\newcommand{\UBCL}{\hyperref[cond.UBCL]{\textup{BCL}}}

\def\be{\begin{equation}}
\def\ee{\end{equation}}
\def\bes{\begin{equation*}}
\def\ees{\end{equation*}}

\newcommand{\set}[1]{\left\{ #1 \right\}}

\newcommand{\abs}[1]{{\left\vert\kern-0.25ex #1
		\kern-0.25ex\right\vert}}
\newcommand\norm[1]{\left\lVert#1\right\rVert} 
\newcommand{\one}{\mathds{1}} 

 \def\sE {{\mathcal E}} \def\sF {{\mathcal F}}
  
\def\sJ {{\mathcal J}}

\def\bG {{\mathbb G}}  
  
 \def\bN {{\mathbb N}} 
  \def\bR {{\mathbb R}}

 \def\bZ {{\mathbb Z}}

\def\ignore#1{}

\def\ol{\overline}

\def\to {\rightarrow}

\def\q{\quad} \def\qq{\qquad}
\def\dint{\int\kern-.6em\int}

\newcommand\restr[2]{{
		\left.\kern-\nulldelimiterspace 
		#1 
		\vphantom{\big|} 
		\right|_{#2} 
	}} 

\newcommand{\contfunc}{\mathcal{C}}
\newcommand{\measure}{m}
\newcommand{\metric}{d}
\newcommand{\Lmeasure}{\mu}
\newcommand{\Lmetric}{\theta}
\newcommand{\pwalk}{d_{\mathrm{w}}(p)}
\newcommand{\hdim}{d_{\mathrm{f}}}

\newcommand{\diam}{\mathop{{\rm diam}}\nolimits}
\newcommand{\dist}{\mathop{{\rm dist}}\nolimits}

\newcommand{\supp}{\mathop{{\rm supp}}\nolimits}

\def\Cap{\operatorname{Cap}}
\def\Mod{{\mathop{{\rm Mod}}}}

\newcommand{\on}[1]{\operatorname{ #1}}

	\def\wt{\widetilde}
	\def\wh{\widehat}

	\def\be{\begin{equation}}
	\def\ee{\end{equation}}
	\def\bes{\begin{equation*}}
	\def\ees{\end{equation*}}
	\def\ba{\begin{align}}
	\def\ea{\end{align}}
	\def\xxea{\end{align}}
\def\bas{\begin{align*}}
\def\eas{\end{align*}}

\definecolor{dgreen}{rgb}{0, 0.6, 0.1}
\definecolor{dblue}{rgb}{0, 0.0, 0.6}
\definecolor{vdblue}{rgb}{0,.08, 0.45}
\definecolor{dred}{rgb}{0.7, 0.0, 0.0}
\definecolor{vdblue}{rgb}{0,.08, 0.45}
\definecolor{lred}{rgb}{0.79,0.17,0.57}

\definecolor{purple}{rgb}{0.6, 0.0, 0.6}
\definecolor{mytext}{rgb}{0.1, 0.1, 0.1}

\definecolor{darkmagenta}{rgb}{0.5,0,0.5}
\definecolor{R}{rgb}{0.79,0.17,0.57} 

\newcommand{\mr}[1]{{\tt \href{http://mathscinet.ams.org/mathscinet-getitem?mr=#1}{MR#1}}}
\newcommand{\arxiv}[1]{{\tt \href{http://arxiv.org/abs/#1}{arXiv:#1}}}
\newcommand{\doi}[1]{{\tt \href{https://doi.org/#1}{DOI:#1}}}

\begin{document}

	\font\titlefont=cmbx14 scaled\magstep1
	\title{\titlefont First-order Sobolev spaces, self-similar energies and energy measures on the Sierpi\'{n}ski carpet}

	\author{
	Mathav Murugan\thanks{Research partially supported by NSERC (Canada) and the Canada Research Chairs program.}
	\qquad
	Ryosuke Shimizu\thanks{Research partially supported by JSPS KAKENHI Grant Number JP20J20207 and JP23KJ2011.}
	}
	\date{February 24, 2025}
	\maketitle
	\vspace{-0.5cm}
	\begin{abstract}

		We construct and investigate $(1, p)$-Sobolev space, $p$-energy, and the corresponding $p$-energy measures on the planar Sierpi\'{n}ski carpet for all $p \in (1, \infty)$.
		Our method is based on the idea of Kusuoka and Zhou [\emph{Probab. Theory Related Fields} \textbf{93} (1992), no. 2, 169--196], where Brownian motion (the case $p = 2$) on self-similar sets including the planar Sierpi\'{n}ski carpet were constructed.
		Similar to this earlier work, we use a sequence of discrete graph approximations and the corresponding discrete $p$-energies to define the Sobolev space and $p$-energies.
		However, we need a new approach to ensure that our  $(1, p)$-Sobolev space has a dense set of continuous functions when $p$ is less than the Ahlfors regular conformal dimension.  The new ingredients are the use of   Loewner type estimates on combinatorial modulus to obtain Poincar\'e inequality   and   elliptic Harnack inequality on  a sequence of approximating graphs.
		An important feature of our Sobolev space is the  self-similarity of our $p$-energy, which allows us to define corresponding $p$-energy measures on the planar Sierpi\'{n}ski carpet.  We show that our Sobolev space can also be viewed as a Korevaar--Schoen type space.

		We apply our results to the attainment problem for Ahlfors regular conformal dimension of the Sierpi\'{n}ski carpet. In particular, we show that if the Ahlfors regular conformal dimension, say $\dim_{\mathrm{ARC}}$, is attained, then any optimal   measure  which attains $\dim_{\mathrm{ARC}}$ should be comparable with the $\dim_{\mathrm{ARC}}$-energy measure of some function in our $(1, \dim_{\mathrm{ARC}})$-Sobolev space up to a multiplicative constant.
		In this case, we also prove that the Newton--Sobolev space corresponding to any optimal measure  and metric  can be identified as our self-similar $(1, \dim_{\mathrm{ARC}})$-Sobolev space.
		\vskip.2cm
		\noindent {\it Keywords: Sierpi\'{n}ski carpet, Sobolev space, Ahlfors regular conformal dimension, Loewner space}
		\
	\end{abstract}

\newpage
\tableofcontents

\section{Introduction and main results}
The goal of this work is to construct and investigate properties of $(1,p)$-Sobolev space, $p$-energy and $p$-energy measures on the Sierpi\'{n}ski carpet. Our $(1,p)$-Sobolev space can be considered to be an analogue of   $W^{1,p}(\mathbb{R}^n)$  on Euclidean space, the $p$-energy of a function $f$ is an analogue of $\int_{\mathbb{R}^n} \abs{\nabla f}^p(x) \,dx$, and the $p$-energy measure of a function $f$ is an analogue of the measure $A \mapsto \int_{A} \abs{\nabla f}^p(x) \,dx$. Similar $(1,p)$-Sobolev spaces were constructed in recent works of Kigami and the second-named author but much of the results there only apply to the case  $p >   \dim_{\on{ARC}}$, where $ \dim_{\on{ARC}}$ is the Ahlfors regular conformal dimension \cite{Shi24,Kig23}.

Our approach and that of \cite{Shi24,Kig23}   goes back to the analytic construction of Brownian motion on the Sierpi\'{n}ski carpet by Kusuoka and Zhou \cite{KZ92}. (The first construction of Brownian motion on the Sierpi\'{n}ski carpet was done by Barlow and Bass \cite{BB89} in a purely probabilistic way.) 
The Dirichlet form corresponding to the Brownian motion on the Sierpi\'{n}ski carpet is a special case of $p$-energy when $p=2$. The idea behind defining a $p$-energy of a function $f$ on a metric space $(X,d)$ is to approximate a metric space by a sequence of graphs $\{\mathbb{G}_n=(V_n,E_n): n \in \bN\}$ on a sequence of increasingly finer scales and to consider a sequence of discrete approximations $M_n f: V_n \to \bR$ of the function $f: X \to \bR$.
Consider the \emph{discrete $p$-energies},
\[
\mathcal{E}^{\mathbb{G}_n}_p(M_nf)= \sum_{ \{x,y\} \in E_n} \abs{(M_n f)(x)-(M_nf)(y)}^p.
\]
We then choose a sequence $\{r_n: n \in \bN\}$ of re-scaling factors $r_n \in (0,\infty)$   so that the quantities $\limsup_{n \to \infty}  r_n\mathcal{E}^{\mathbb{G}_n}_p(M_nf),  \liminf_{n \to \infty} r_n\mathcal{E}^{\mathbb{G}_n}_p(M_nf)$, and $\sup_{n \in \mathbb{N}}  r_n\mathcal{E}^{\mathbb{G}_n}_p(M_nf)$ are comparable uniformly
for all integrable functions $f$.
The existence of such a sequence $r_n$  is guaranteed by analytic properties on the sequence of graphs $\mathbb{G}_n$ such as bounds on capacity and Poincar\'e inequality. The Sobolev space is then defined as
\[
\sF_p:= \Bigl\{f \in L^p:  \sup_{n \in \mathbb{N}}  r_n\mathcal{E}^{\mathbb{G}_n}_p(M_nf) < \infty \Bigr\}.
\]
To describe our results, we recall a definition of the Sierpi\'{n}ski carpet.  Let	$q_1=(-1,-1)=-q_5, q_3=(1,-1)=-q_7$ denote the corners of a square in $\mathbb{R}^2$ and let $q_2= (0,-1)=-q_6, \q q_4=(1,0)=-q_8$ denote the midpoints of the sides of the corresponding square. The Sierpi\'{n}ski carpet $K$ is the unique non-empty compact subset of $\mathbb{R}^2$ such that
\[
K = \bigcup_{i=1}^8 f_i(K), \quad \mbox{ where $f_i: \mathbb{R}^2 \to \mathbb{R}^2$ is  the map $f_i(x) := \frac{1}{3}(x-q_i)+q_i, i \in \{ 1,\ldots,8 \}$}.
\]

Next, we describe a sequence of graphs that approximate $K$. Let $V_n=S^n$ denote the set of words of length $n$ over the alphabet $S=\{1,2,\ldots,8\}$. Let $F_i:= \restr{f_i}{K}$ for $i \in S$ and for $w=w_1\cdots w_n \in V_n$, we set 	$F_{w} := F_{w_{1}} \circ F_{w_{2}} \circ \cdots \circ F_{w_{n}}$. Let $\mathbb{G}_n=(V_n, E_n)$ be the graph whose vertex set is the set of words $V_n$ with $n$-alphabets and the edge set is defined by
\[
E_n =\{ \{u,v\}: u,v \in V_n, u \neq v, F_u(K) \cap F_v(K) \neq \emptyset \}.
\]
The sequence of graphs $\mathbb{G}_n, n \in \bN$ approximate the Sierpi\'{n}ski carpet $K$ (see Figure \ref{fig.app}).

\begin{figure}\centering
	\includegraphics[height=120pt]{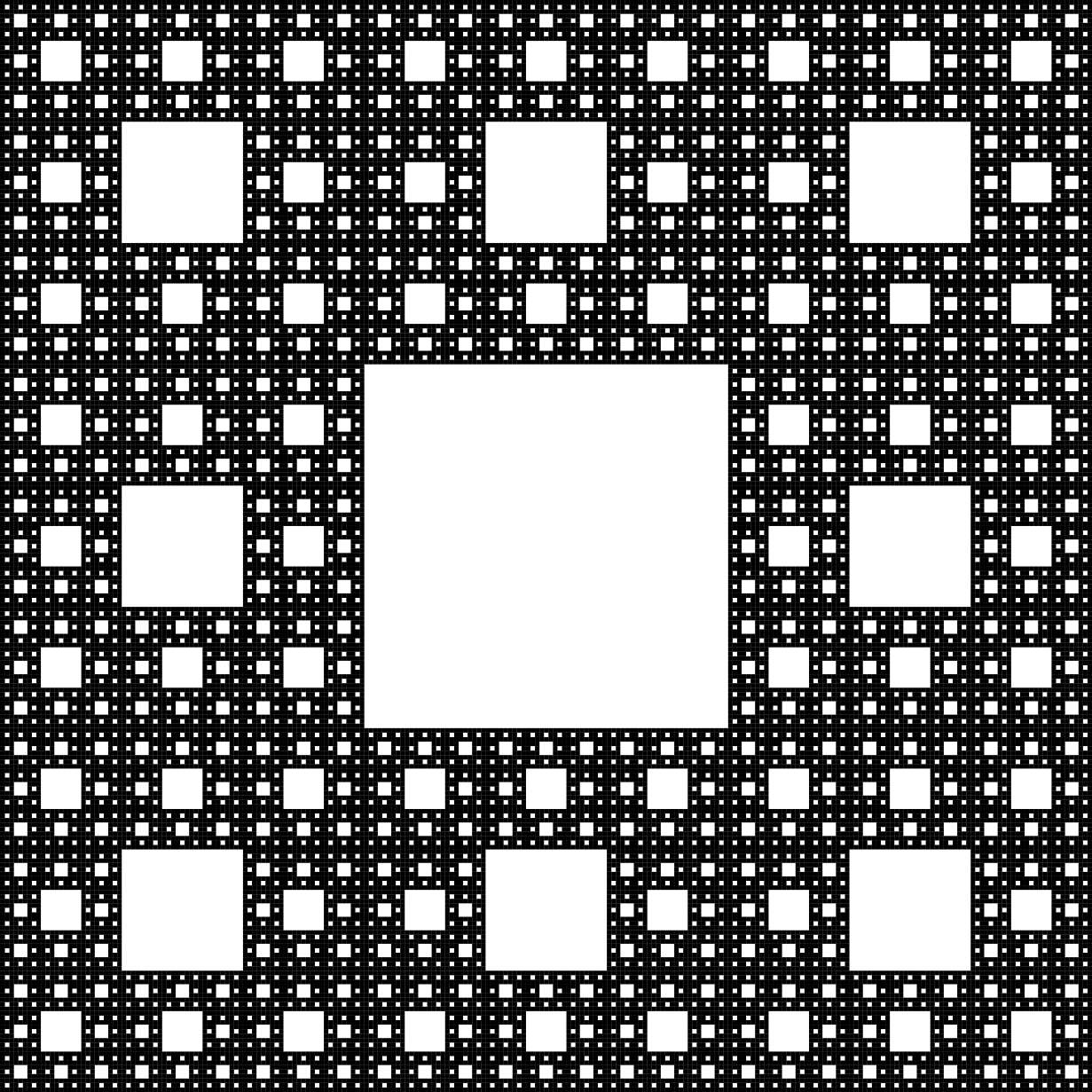}\hspace*{35pt}
	\includegraphics[height=120pt]{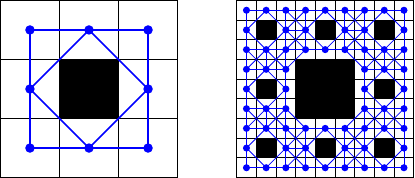}
	\caption{The planar Sierpi\'{n}ski carpet and its approximation graphs $\{ \mathbb{G}_{n} \}$. ($\mathbb{G}_{1}$ and $\mathbb{G}_{2}$ are drawn in blue.)}\label{fig.app}
\end{figure}

We now describe how to approximate a function on $K$ by a function on $\mathbb{G}_n$. To this end, we equip $K$  with the Euclidean metric $d$ and
 the self-similar Borel probability measure $m$ on $K$ such that $m(F_w(K))=8^{-n}$ for all $w \in V_n, n\in \bN$.  For $n \in \bN$, we define the discrete approximation operators $M_n: L^p(K,m) \to \mathbb{R}^{V_n}$ as
 \[
 (M_n f)(u):= \frac{1}{m(F_u(K))} \int_{F_u(K)} f \,dm, \quad \mbox{for all $u \in V_n$}.
 \]
For any $p\in(1,\infty)$, we show the existence of an exponent $\rho(p) \in (0,\infty)$ and some constant $C \in (1,\infty)$ such that
\[
\sup_{n \in \mathbb{N}}\rho(p)^n\mathcal{E}_{p}^{\mathbb{G}_{n}}(M_{n}f) \le C\limsup_{n \to \infty} \rho(p)^n\mathcal{E}_{p}^{\mathbb{G}_{n}}(M_{n}f) \le C^2 \liminf_{n \to \infty} \rho(p)^n\mathcal{E}_{p}^{\mathbb{G}_{n}}(M_{n}f)
\]
for all $f \in L^p(K,m)$. This implies that each of the three expressions in the above display are uniformly comparable up to multiplicative constants. One of them, say  $\sup_{n \in \mathbb{N}}\rho(p)^n\mathcal{E}_{p}^{\mathbb{G}_{n}}(M_{n}f)$ could be considered as a candidate $p$-energy. However, we would like to construct an \emph{improved} $p$-energy $\sE_p: \sF_p \to [0,\infty)$ that is comparable to the above candidate $p$-energy but satisfies desirable properties such as self-similarity,   Lipschitz contractivity, and strong locality that the above candidate need not satisfy.  The definitions of these properties are included in the statement of Theorem \ref{t:main-sob-psc}. For $f \in L^p(K,m)$, by $\supp_m[f]$ we denote the support of the measure $f\,dm$.
The following theorem describes the definition and basic properties of our Sobolev spaces.
\begin{thm}[Construction of $(1,p)$-Sobolev space and $p$-energy]\label{t:main-sob-psc}
	Let $p \in (1, \infty)$ and let $(K, \metric, \measure)$ be the Sierpi\'{n}ski carpet equipped with the Euclidean metric and the self-similar measure described above.   Then there exists $\rho(p) \in (0,\infty)$ such that the normed linear space $(\sF_p, \norm{\,\cdot\,}_{\sF_p})$ defined by
	\[
	\mathcal{F}_{p} \coloneqq \biggl\{ f \in L^{p}(K, \measure) \biggm| \int_{K}\abs{f}^{p}\,d\measure + \sup_{n \in \mathbb{N}}\rho(p)^{n}\mathcal{E}_{p}^{\mathbb{G}_{n}}(M_{n}f) < \infty \biggr\},
	\]
	and
	\[
 \abs{f}_{\sF_p} \coloneqq \left(\sup_{n \in \mathbb{N}}\rho(p)^n\mathcal{E}_{p}^{\mathbb{G}_{n}}(M_{n}f)\right)^{1/p}, \quad	\norm{f}_{\mathcal{F}_{p}} \coloneqq \norm{f}_{L^p} + \abs{f}_{\sF_p},
	\]
	satisfies the following properties.
	\begin{enumerate}[\rm(i)]
		\item \label{sob-banach} $(\mathcal{F}_{p}, \norm{\,\cdot\,}_{\mathcal{F}_{p}})$ is a reflexive separable Banach space.
		\item \label{sob-reg} \textup{(Regularity)} $\mathcal{F}_{p} \cap \contfunc(K)$ is a dense subspace in   the Banach spaces $(\mathcal{F}_{p}, \norm{\,\cdot\,}_{\mathcal{F}_{p}})$ and   $(\contfunc(K), \norm{\,\cdot\,}_\infty)$.
	\end{enumerate}
	Furthermore, there exist $C \ge 1$ and     $\mathcal{E}_p \colon \mathcal{F}_{p} \to [0, \infty)$ satisfying the following:
	\begin{enumerate}[\rm(i)]\setcounter{enumi}{2}
		\item \label{sob-sn} $\mathcal{E}_p(\,\cdot\,)^{1/p}$ is a semi-norm satisfying $C^{-1}\abs{f}_{\mathcal{F}_{p}} \le \mathcal{E}_{p}(f)^{1/p} \le C\abs{f}_{\mathcal{F}_{p}}$ for all $f \in \mathcal{F}_{p}$.
		\item \label{sob-uc} \textup{(Uniform convexity)} $\mathcal{E}_p(\,\cdot\,)^{1/p}$ is uniformly convex.
		\item \label{sob-lip} \textup{(Lipschitz contractivity)} For every $f \in \mathcal{F}_{p}$ and $1$-Lipschitz function $\varphi\colon\mathbb{R} \to \mathbb{R}$, we have $\varphi \circ f \in \sF_p$ and $\mathcal{E}_{p}(\varphi \circ f) \le \mathcal{E}_{p}(f)$.
		\item \label{sob-sg} \textup{(Spectral gap)} It holds that
		\begin{equation*}
			\norm{f - f_{K}}_{L^{p}(\measure)}^{p} \le C  \sE_{p}(f) \quad \text{for all $f \in \sF_{p}$,}
		\end{equation*}
		where $f_K:= \int_{K} f\,dm$ is the $m$-average of $f$. In particular,
		\begin{equation}\label{irred}
		\{f \in \sF_p: \sE_p(f)=0\}= \{f \in L^p(K,m): \mbox{$f$ is constant $m$-a.e.}\}.
		\end{equation}

		\item \label{sob-sl} \textup{(Strong locality)} If $f, g, h \in \mathcal{F}_{p}$ satisfy $\supp_{\measure}[f] \cap \supp_{\measure}[g - a\indicator{K}] = \emptyset$ for some $a \in \mathbb{R}$, then
		\[
		\mathcal{E}_{p}(f + g + h) + \mathcal{E}_{p}(h) = \mathcal{E}_{p}(f + h) + \mathcal{E}_{p}(g + h).
		\]
		\item \label{sob-ss} \textup{(Self-similarity)} For every $f \in \mathcal{F}_{p}$, we have $f \circ F_{i} \in \mathcal{F}_{p}$ for all $i \in S$ and
		\begin{equation*}
			\mathcal{E}_{p}(f) = \rho(p)\sum_{i \in S}\mathcal{E}_{p}(f \circ F_{i}).
		\end{equation*}
		Furthermore, $\sF_p \cap \contfunc(K)= \{f \in \contfunc(K): f \circ F_i \in \sF_p \mbox{ for all $i \in S$}\}$.
		\item \label{sob-sym} \textup{(Symmetry)} Let $D_4$ denote the dihedral group of isometries of $K$. For every $f \in \mathcal{F}_{p}$ and $\Phi \in D_{4}$, we have $f \circ \Phi \in \mathcal{F}_{p}$ and $\mathcal{E}_{p}(f \circ \Phi) = \mathcal{E}_{p}(f)$.
	\end{enumerate}
\end{thm}
\begin{rmk*}
	One can show the generalized $p$-contraction property, which was introduced in \cite{KS24+}, for $(\mathcal{E}_{p},\mathcal{F}_{p})$. 
	See also \cite[Remark 8.20]{KS24+}. 
	This property gives improvements of Theorem \ref{t:main-sob-psc}\ref{sob-uc}, \ref{sob-lip}. 
\end{rmk*}

We compare the above result with earlier results in \cite{Shi24,Kig23}. Theorem \ref{t:main-sob-psc} was previously known only in the case   $p > \dim_{ \on{ARC}}(K,d)$, where $\dim_{ \on{ARC}}(K,d) \in (1,\infty)$ is the Ahlfors regular conformal dimension \cite{Shi24} (we recall the definition of Ahlfors regular conformal dimension in Definition \ref{d:cgauge}). Similar to this work, Kigami  uses an approach based on discrete energies and introduces a \emph{conductive homogeneity} condition under which the Sobolev space was constructed \cite{Kig23}. However much of the results apply only to the case $p > \dim_{ \on{ARC}}(K,d)$ as the author points out ``Regrettably, we do not have much for the case $p \le \dim_{ \on{ARC}}(K,d)$" in \cite[p. 8]{Kig23}. In particular, Theorem \ref{t:main-sob-psc} answers a question of Kigami \cite[\textsection 6.3, Problem 1]{Kig23} for the Sierpi\'nski carpet which asks for the property \ref{sob-reg} above. This property is known as \emph{regularity} in the theory of Dirichlet form \cite{FOT}.

The difficulty in the case  $p \le \dim_{ \on{ARC}}(K,d)$ is due to the fact that the Sobolev space contains discontinuous functions. 
Indeed, by a recent result by Cao, Chen and Kumagai \cite[Theorem~1.1]{CCK24}, under the conductive homogeneity condition, the Sobolev space constructed by Kigami \cite{Kig23} contains discontinuous functions if and only if $p \le \dim_{\mathrm{ARC}}(K,d)$. 
If $p >\dim_{ \on{ARC}}(K,d)$, there is a version of Morrey's embedding theorem which makes the analysis easier. 

Another difficulty is that  the conductive homogeneity condition of \cite{Kig23} (or its analogue `knight move condition' in \cite{Shi24}) was not obtained on the Sierpi\'nski carpet if $p \le \dim_{ \on{ARC}}(K,d)$. The Poincar\'e inequality for graphs $\mathbb{G}_n$ shown in our work (Theorem \ref{thm.PI-discrete}) implies  these conditions when  $p \le \dim_{ \on{ARC}}(K,d)$; see Theorem \ref{t:pch}.  
However, we do not use the conductive homogeneity condition to obtain Theorem \ref{t:main-sob-psc} as our approach only relies on Poincar\'e inequality and certain upper bounds on capacity across annulus on the sequence of graphs $\mathbb{G}_n$.  

As we will see in Theorem \ref{thm.LB.psc} and Proposition \ref{p.axiom}, the value of $\rho(p)$ in Theorem \ref{t:main-sob-psc} is uniquely determined by the above properties. If $\rho(p)$ were larger, the Sobolev space $\sF_p$ would only consist of constant functions violating property \ref{sob-reg}. 
If $\rho(p)$ were smaller, then the resulting $p$-energy would be too small to satisfy  property \ref{sob-sg}. 

Our next result is the existence of energy measures. To motivate energy measure, let us consider the following question: \emph{what  information does the energy measure contain about a function?} In the primary example on $\mathbb{R}^n$, the $p$-energy measure of a function $f \in W^{1,p}(\mathbb{R}^n)$ is the measure $A \mapsto \int_{A} \abs{\nabla f(x)}^p \,dx$. By considering the Radon-Nikodym derivative of the energy measure with respect to Lebesgue measure, we see that the energy measure contains the same information as $\abs{\nabla f}$ up to sets of Lebesgue measure zero, where $\nabla f$ is the distributional gradient of $f$. A generalization of $\abs{\nabla f}$ is given by the \emph{minimal $p$-weak upper gradient} in the theory of Newton-Sobolev space \cite{HKST}. In these settings, the energy measure is always absolutely continuous with respect to the reference measure. In the setting of diffusion on fractals, the energy measure (for $p=2$) is typically singular with respect to the reference measure \cite{Hin05,KM20}. As we will see in Theorem \ref{t:attain}, not requiring the $p$-energy measure to be absolutely continuous with respect to the  reference measure is useful as the reference measure might not be suited to express energies and also because the energy measure might satisfy better properties such as the Loewner property. Based on the above analogy, we think of our energy measures as containing similar information about the function as the minimal $p$-weak upper gradient in the setting of Newton-Sobolev spaces.

Let us describe the construction of energy measure. Following an idea in \cite{Hin05,Kus89}, we use the self-similarity property of the $p$-energy to construct our $p$-energy measure. To describe it, we let $\Sigma= S^\bN$ be the set of all infinite words in the alphabet $S$ equipped with the product topology. Recall that the \emph{canonical projection} (or \emph{coding map}) $\chi: \Sigma \to K$ is defined to satisfy
\[
\{\chi(\omega) \}= \bigcap_{n\in\bN}( F_{w_1}\circ\cdots\circ F_{w_n})(K), \quad \mbox{where $\omega= (w_1,w_2,\cdots) \in \Sigma$.}
\]
For $w \in S^n$, let $\Sigma_w \subset \Sigma$ be the set of infinite words whose beginning $n$ alphabets coincide  with $w$.
For any function $f \in  \sF_p$, self-similarity of the $p$-energy $\sE_p(\cdot)$ and Kolmogorov's extension theorem guarantee the existence of a measure $\mathfrak{m}_{p}\langle f  \rangle$ on $\Sigma$ such that $\mathfrak{m}_{p}\langle f  \rangle (\Sigma_w)= \rho(p)^n \sE_p(f \circ F_w)$ for all $w \in S^n, n \in \bN$. The energy measure is then defined to be the pushforward measure $\Gamma_{p}\langle f \rangle  :=\chi_*( \mathfrak{m}_{p}\langle f  \rangle)$.
Our next theorem shows the existence of  energy measure corresponding to self-similar energy and describes some of its basic properties.
\begin{thm}[Existence of $p$-energy measure] \label{t:main-em}
	Let $p \in (1, \infty)$ and let $(K, \metric, \measure)$ be the Sierpi\'{n}ski carpet.
	Let $(\mathcal{E}_{p}, \mathcal{F}_{p})$ be the $p$-energy in Theorem \ref{t:main-sob-psc}.
	There exists a family of finite Borel measures $\{ \Gamma_{p}\langle f \rangle \}_{f \in \mathcal{F}_{p}}$ on $K$ satisfying the following:
	\begin{enumerate}[\rm(i)]
		\item \label{em-cell} For any $f \in \mathcal{F}_{p}$, we have $\Gamma_{p}\langle f \rangle(K) = \mathcal{E}_{p}(f)$ and
		\begin{equation}\label{e:em.exact-cell}
			\Gamma_{p}\langle f \rangle(F_w(K))=\rho(p)^{n}\mathcal{E}_{p}(f\circ F_w)  \quad \text{for all $w \in S^{n}, n \in \mathbb{N}$.}
		\end{equation}
		\item \label{em-tri} \textup{(Triangle inequality)} For any $f_1, f_2 \in \mathcal{F}_{p}$ and  Borel function $g \colon K \to [0, \infty]$,
		\[
		\left(\int_{K}g\,d\Gamma_{p}\langle f_1 + f_2 \rangle\right)^{1/p}
		\le \left(\int_{K}g\,d\Gamma_{p}\langle f_1 \rangle\right)^{1/p} + \left(\int_{K}g\,d\Gamma_{p}\langle f_2 \rangle\right)^{1/p}.
		\]
		\item \label{em-lip} \textup{(Lipschitz contractivity)} For any $f \in \mathcal{F}_{p}$, Borel function $g \colon K \to [0, \infty]$ and $1$-Lipschitz function $\varphi \colon \mathbb{R} \to \mathbb{R}$,
		\[
		\int_{K}g\,d\Gamma_{p}\langle \varphi \circ f \rangle \le \int_{K}g\,d\Gamma_{p}\langle f \rangle.
		\]
		\item \label{em-ss} \textup{(Self-similarity)} For any $n \in \mathbb{N}$ and $f \in \mathcal{F}_{p}$,
		\[
		\Gamma_{p}\langle f \rangle = \rho(p)^{n}\sum_{w \in S^n}(F_{w})_{\ast}\bigl(\Gamma_{p}\langle f \circ F_{w} \rangle\bigr).
		\]
		\item \label{em-sym} \textup{(Symmetry)} For any $f \in \mathcal{F}_{p}$ and $\Phi \in D_{4}$, we have $\Phi_{\ast}\bigl(\Gamma_{p}\langle f \rangle\bigr) = \Gamma_{p}\langle f \circ \Phi \rangle$.
		\item \label{em-chain} \textup{(Chain rule and strong locality)} For any $\Psi \in C^1(\mathbb{R})$ and $f \in \mathcal{F}_{p} \cap \contfunc(K)$,
		\[
		\Gamma_{p}\langle \Psi \circ f \rangle(dx) = \abs{\Psi'(f(x))}^{p}\Gamma_{p}\langle f \rangle(dx).
		\]
		If $f, g \in \mathcal{F}_{p} \cap C(K)$ and $A \in \mathcal{B}(K)$ satisfy $\restr{(f - g)}{A} = a \cdot \indicator{A}$ for some $a \in \mathbb{R}$, then $\Gamma_{p}\langle f \rangle(A) = \Gamma_{p}\langle g \rangle(A)$
	\end{enumerate}
\end{thm}
\begin{rmk*}
	An improved version of the chain rule of self-similar $p$-energy measures is proved in \cite{KS24+}. Also, for each Borel set $A$ of $K$, one can show that $(\Gamma_{p}\langle \,\cdot\, \rangle(A),\mathcal{F}_{p})$ satisfies the generalized $p$-contraction property. See \cite[Section 5]{KS24+} for details.  
\end{rmk*}

We describe another approach to defining Sobolev space motivated by a work of Korevaar and Schoen \cite{KS}. This work describes classical Sobolev spaces in terms of  Besov--Lipschitz spaces at the critical exponent (also called Korevaar--Schoen space). On a metric space $(X, \mathsf{d})$, we denote by $B_{\mathsf{d}}(x, r)= \{y \in X: \mathsf{d}(x,y)<r\}$ the open ball centered at $x \in X$ and radius $r>0$. Our next result identifies our Sobolev space obtained using rescaled discrete energies in Theorem \ref{t:main-sob-psc} as the critical Besov--Lipschitz or Korevaar--Schoen type space with comparable seminorms.
\begin{defn}\label{defn.LB}
	Let $(X, \mathsf{d})$ be a connected metric space with $\#X \ge 2$ and let $\mathfrak{m}$ be a Borel-regular measure on $X$ such that $\mathfrak{m}(B_{\mathsf{d}}(x, r)) \in (0,\infty)$ for any $x \in X, r > 0$.
	For $p \in (1, \infty)$ and $s > 0$, the Besov--Lipschitz space $B_{p, \infty}^{s} = B_{p, \infty}^{s}(X, \mathsf{d}, \mathfrak{m})$ is defined as
	\begin{align*}
		B_{p, \infty}^{s}
		\coloneqq \biggl\{ f \in L^{p}(X, \mathfrak{m}) \biggm| \sup_{r \in (0, 2\diam(X, \mathsf{d}))}\int_{X}\fint_{B_{\mathsf{d}}(x, r)}\frac{\abs{f(x) - f(y)}^{p}}{r^{sp}}\,\mathfrak{m}(dy)\mathfrak{m}(dx) < \infty \biggr\}.
	\end{align*}
\end{defn}
Korevaar and Schoen show that the Sobolev space $W^{1,p}(\mathbb{R}^n)$ coincides with $B_{p,\infty}^1(\mathbb{R}^n,d,\lambda)$ where $d$ is the Euclidean metric and $\lambda$ is the Lebesgue measure
\cite[Theorem 1.6.2]{KS}. Furthermore there exists $C \in (0,\infty)$ such that the distributional gradient $\nabla f$ of any function $f \in W^{1,p}(\mathbb{R}^n)$ satisfies
\[
C^{-1} \int_{\mathbb{R}^n} \abs{\nabla f}^p\,d\lambda \le  \sup_{r \in (0,  \infty)  }\int_{\mathbb{R}^n}\fint_{B_{d}(x, r)}\frac{\abs{f(x) - f(y)}^{p}}{r^{p}}\,\lambda(dy)\lambda(dx)\le C  \int_{\mathbb{R}^n} \abs{\nabla f}^p\,d\lambda.
\]
This result was later extended to spaces satisfying doubling property and Poincar\'e inequality by Koskela and MacManus \cite[Theorem 4.5]{KoMa}. In these settings, it turns out that the exponent $s=1$ is critical in that for every $s>1$ every function $f \in B_{p,\infty}^s$ is constant almost everywhere and for every $s \le 1$, the space $B_{p,\infty}^s$ contains non-constant functions.

This motivates the definition of the \emph{critical exponent for  Besov--Lipschitz space}
\begin{equation}\label{critical}
	s_p \coloneqq \sup\{s>0: B_{p,\infty}^s \mbox{ contains non-constant functions}\}
\end{equation}
and the \emph{Korevaar-Schoen space} as the \emph{critical Besov--Lipschitz space}   $B_{p,\infty}^{s_p}$.
This approach to define Sobolev space was recently proposed by Baudoin \cite{Bau24}. Our next result is that the  Sobolev spaces defined using rescaled discrete energies coincides with the one defined using critical Besov--Lipschitz space with comparable seminorms. Furthermore, we describe the scaling constant $\rho(p)$  in Theorem \ref{t:main-sob-psc} in terms of the critical scaling exponent for $B^s_{p,\infty}$.
\begin{thm}[Self-similar Sobolev space is a Korevaar--Schoen space] \label{thm.LB.psc}
 Let $(K, \metric, \measure)$ be the Sierpi\'{n}ski carpet. Let $\sF_p,\abs{\,\cdot\,}_{\mathcal{F}_{p}},\rho(p)$ be  the Sobolev space, seminorm and scaling constant respectively as given in Theorem \ref{t:main-sob-psc}. Set $\pwalk \coloneqq \frac{\log(8 \rho(p))}{\log 3}$.
	Then, there exists a constant $C \ge 1$ such that
	\begin{align*}
		C^{-1}\abs{f}_{\mathcal{F}_{p}}^{p}
		&\le \liminf_{r \downarrow 0}\int_{K}\fint_{B_{\metric}(x, r)}\frac{\abs{f(x) - f(y)}^{p}}{r^{\pwalk}}\,\measure(dy)\measure(dx) \nonumber\\
		&\le \sup_{r > 0}\int_{K}\fint_{B_{\metric}(x, r)}\frac{\abs{f(x) - f(y)}^{p}}{r^{\pwalk}}\,\measure(dy)\measure(dx)
		\le C\,\abs{f}_{\mathcal{F}_{p}}^{p} \quad \text{for all $f \in L^{p}(K, \measure)$}, 
	\end{align*}
	and $\pwalk/p = s_{p}$. 
	In particular, $\mathcal{F}_{p}(K,d,m) = B_{p, \infty}^{\pwalk/p}(K,d,m)$ and
	\begin{align} \label{e:bau}
		\sup_{r > 0}
		&\int_{K}\fint_{B_{\metric}(x, r)}\frac{\abs{f(x) - f(y)}^{p}}{r^{\pwalk}}\,\measure(dy)\measure(dx) \nonumber \\
		&\le C^{2}\liminf_{r \downarrow 0}\int_{X}\fint_{B_{\metric}(x, r)}\frac{\abs{f(x) - f(y)}^{p}}{r^{\pwalk}}\,\measure(dy)\measure(dx) \quad \text{for all $f \in L^{p}(K, \measure)$.}
	\end{align} 
\end{thm}
This result was previously obtained under the additional assumption $p> \dim_{ \on{ARC}}(K,d)$.
The above result answers a question of F.~Baudoin as he asks if \eqref{e:bau} is true for the Sierpi\'nski carpet \cite{Bau24}.
Recently, Yang also proves \eqref{e:bau} for generalized Sierpi\'{n}ski carpets in the case $p > \dim_{ \on{ARC}}$ \cite[Theorem 2.8]{Yan+}.
If \eqref{e:bau} were true, then \cite{Bau24} obtains number of useful consequences such as Sobolev embeddings and Gagliardo--Nirenberg inequalities.
Our notation $\pwalk$ in Theorem \ref{thm.LB.psc} is inspired by the notion of walk dimension studied for $p=2$ in the context of diffusion on fractals \cite{KM23}. Similar to that setting, $\pwalk$ also plays a role as the exponent governing Poincar\'e inequality and capacity bounds as shown in the following theorem.
\begin{thm}[Poincar\'e inequality and capacity upper bound]\label{t:main-cap_PI}
	Let $p \in (1, \infty)$ and let $(K, \metric, \measure)$ be the Sierpi\'{n}ski carpet.
	Let $ \mathcal{E}_{p},\sF_p$ be the $p$-energy and Sobolev space in Theorem \ref{t:main-sob-psc}. Let $\pwalk = \frac{\log(8 \rho(p))}{\log 3}$ be as defined in Theorem \ref{thm.LB.psc} and let $\Gamma_{p}\langle \,\cdot\, \rangle$ denote the $p$-energy measure constructed in Theorem \ref{t:main-em}.
	Then there exist $C, A \ge 1$ such that for all $x \in K$, $r > 0$ and $f \in \mathcal{F}_{p}$, we have
	\[
	\int_{B_{\metric}(x, r)}\abs{f - f_{B_{\metric}(x, r)}}^{p}\,d\measure \le Cr^{\pwalk}\int_{B_{\metric}(x,Ar)}\,d\Gamma_{p}\langle f \rangle,
	\]
	and
	\[
	\inf\bigl\{ \mathcal{E}_{p}(f) \bigm| f \in \mathcal{F}_{p} \cap \contfunc(K), \restr{f}{B_{\metric}(x, r)} \equiv 1, \supp[f] \subseteq B_{\metric}(x, 2r) \bigr\} \le C\frac{\measure(B_{\metric}(x,r))}{r^{\pwalk}},
	\]
	where $f_{B_d(x,r)} \coloneqq \frac{1}{m(B_d(x,r))} \int_{B_d(x,r)} f\,dm$.
\end{thm}
Theorems \ref{t:main-sob-psc} and \ref{thm.LB.psc} suggest that the Sobolev space we construct is \emph{canonical} since two different approaches lead to the same space. As further evidence, we present the following axiomatic description of our Sobolev space and self-similar $p$-energy.
\begin{prop}[Axiomatic description of the self-similar Sobolev space]\label{p.axiom}
		Let $p \in (1, \infty)$ and let $(K, \metric, \measure)$ be the Sierpi\'{n}ski carpet.
	Let $ \mathcal{E}_{p},\sF_p, \rho(p)$ be the $p$-energy, Sobolev space and scaling constant in Theorem \ref{t:main-sob-psc}.

		Let $\mathscr{F}_{p}$ be a subspace of $L^p(K,\measure)$ and let $\mathscr{E}_{p}\colon \mathscr{F}_{p} \to [0, \infty)$ be a functional.
		Suppose that the pair $(\mathscr{E}_{p}, \mathscr{F}_{p})$ satisfies the following properties:
		\begin{enumerate}[\rm(a)]
		\item\label{it:aciom.const} $\{f \in \mathscr{F}_{p} : \mathscr{E}_{p}(f) = 0 \} = \{f \in L^{p}(K, \measure) : \text{$f$ is constant $m$-almost everywhere} \}$. For any $a \in \bR$ and $f \in \mathscr{F}_{p}$, we have
		\[
		\mathscr{E}_{p}(f+a\indicator{K})= \mathscr{E}_{p}(f), \qq \mathscr{E}_{p}(af)= \abs{a}^{p}\mathscr{E}_{p}(f).
		\]

		\item The functional $f \mapsto \mathscr{E}_p(f)^{1/p}$ satisfies the triangle inequality on $\mathscr{F}_{p}$. In addition, the function $\norm{\,\cdot\,}_{\mathscr{F}_{p}}\colon \mathscr{F}_{p} \to [0,\infty)$ defined by $\norm{\,\cdot\,}_{\mathscr{F}_{p}}(f) \coloneqq \left( \norm{f}_{L^p(\measure)}^{p} + \mathscr{E}_{p}(f) \right)^{1/p}$ is a norm on $\mathscr{F}_{p}$ and $(\mathscr{F}_{p},\norm{\,\cdot\,}_{\mathscr{F}_{p}})$ is a uniformly convex Banach space.
		\item \textup{(Regularity)} The subspace $\mathscr{F}_{p} \cap \contfunc(K)$ is dense in $\contfunc(K)$ with respect to the uniform norm and is dense in the Banach space $(\mathscr{F}_{p},\norm{\,\cdot\,}_{\mathscr{F}_{p}})$.
		\item \textup{(Symmetry)} For every $\Phi \in D_{4}$ and for all $f \in \mathscr{F}_{p}$, we have $f \circ \Phi \in \mathscr{F}_{p}$ and $\mathscr{E}_{p}(f \circ \Phi)= \mathscr{E}_{p}(f)$.
		\item \textup{(Self-similarity)} There exists $\widetilde{\rho} \in (0,\infty)$ such that the following hold:
		For every $f \in \mathscr{F}_{p}, i \in S$, we have $f \circ F_i \in \mathscr{F}_{p}$, and
		\begin{equation*}
			\widetilde{\rho}\sum_{i \in S}\mathscr{E}_{p}(f \circ F_i) = \mathscr{E}_{p}(f).
		\end{equation*}
		Furthermore, $\mathscr{F}_{p} \cap \contfunc(K)= \set{f \in \contfunc(K) \mid f \circ F_i \in \mathscr{F}_{p} \mbox{ for all $i \in S$}}$.
		\item \textup{(Unit contractivity)} $f^{+} \wedge 1 \in \mathscr{F}_{p}$ for all $f \in \mathscr{F}_{p}$ and $\mathscr{E}_{p}(f^{+} \wedge 1) \le \mathscr{E}_{p}(f)$.
		\item\label{it:axiom.sg} \textup{(Spectral gap)} There exists a constant $C_{\textup{gap}} \in (0, \infty)$ such that
		\begin{equation*}
			\norm{f - f_{K}}_{L^{p}(\measure)}^{p} \le C_{\textup{gap}}\mathscr{E}_{p}(f) \quad \text{for all $f \in \mathscr{F}_{p}$.}
		\end{equation*}
	\end{enumerate}
	Then $\widetilde{\rho} = \rho(p), \sF_p= \mathscr{F}_{p}$ and there exists $C\in [1,\infty)$ such that
	\[
	C^{-1} \sE_p(f) \le \mathscr{E}_{p}(f) \le C  \sE_p(f) \quad \mbox{for all $f \in \sF_p= \mathscr{F}_{p}$.}
	\]
\end{prop}
By the above result, the assumptions \ref{it:aciom.const}-\ref{it:axiom.sg} in Proposition \ref{p.axiom} determine the Sobolev space uniquely and the self-similar $p$-energy up to a bi-Lipchitz transformation. In light of the uniqueness result of \cite{BBKT}, we conjecture that the $p$-energy is unique up to a multiplicative constant (Conjecture \ref{con.uniq}).
Note that the self-similar $p$-energy $\sE_p$ constructed in Theorem \ref{t:main-sob-psc} satisfies the properties of $\mathscr{E}_p$ in Proposition \ref{p.axiom}. For instance, the unit contractivity is a special case of Lipschitz contractivity.

The most widely used definition of Sobolev space on a metric measure space  relies on the notion of   upper gradient introduced by Heinonen and Koskela \cite{HK98}. Two different definitions of Sobolev space (the \emph{Newton-Sobolev space})  based on upper gradient were proposed by Shanmugalingam \cite{Sha00} and Cheeger  \cite{Che99} but these two definitions lead to the same Sobolev space on any metric measure space \cite[Theorem 10.1.1]{HKST}. The Newton-Soboev space $N^{1,p}(K,d,m)$ for the Sierpi\'{n}ski carpet is known to be trivial, that is, $N^{1,p}(K,d,m)=L^p(K,m)$ with equal norms, because the minimal weak upper gradient of any function is $0$. We refer to Remark \ref{r:np-psc} for further details and references. The triviality of Sobolev space based on upper gradient suggests the need for an alternate method to construct Sobolev spaces on fractals such as the one considered in this work.

An important motivation for our work is quasisymmetric uniformization and the related attainment problem for Ahlfors regular conformal dimension. A recent work predicts that Sobolev spaces and energy measures are relevant to the  attainment problem for Ahlfors regular conformal dimension \cite[p.395-396]{KM23}. Our work confirms this prediction. To describe our results in this direction, we recall the relevant definitions of conformal gauge and Ahlfors regular conformal dimension. Ahlfors regular conformal dimension is a slight variant of Pansu's conformal dimension \cite{Pan} and first appeared in \cite{BP03, BK05}. Conformal dimension of boundary of hyperbolic groups and Julia sets of complex dynamical systems are widely studied.   We refer the reader to \cite{MT} for a comprehensive account of conformal dimension.

 \begin{definition}[Conformal gauge] \label{d:cgauge}
	Let $(X,d)$ be a metric space and $\theta$ be another metric on $X$. We say that $d$ is \emph{quasisymmetric} to $\theta$, if there exists a homeomorphism $\eta:[0,\infty) \to [0,\infty)$ such that
	\[
	\frac{\theta(x,y)}{\theta(x,z)} \le \eta\left(\frac{d(x,y)}{d(x,z)}\right)   \q \mbox{for all triples of points $x,y,z \in X$, $x \neq z$.}
	\]
	The \emph{conformal gauge} of a metric space $(X,d)$ is defined as
	\begin{equation} \label{e:cgauge-dfn}
		\sJ(X,d) \coloneqq \{ \theta: X \times X \to [0,\infty) \mid \mbox{$\theta$ is a metric on $X$, $d$ is quasisymmetric to $\theta$} \}.
	\end{equation}

 	A Borel measure $\mu$ on $X$ is said to be \emph{$p$-Ahlfors regular} with respect to $d$ if there exists $C \ge 1$ such that
\begin{equation*} 
	C^{-1} r^p \le \mu(B_{d}(x,r)) \le C r^p \quad \mbox{for all $x \in X, 0< r < 2\diam(X,d)$.}  
\end{equation*} 
The Ahlfors regular conformal dimension is defined as
\begin{equation*}\label{e:defdarc}
	\dim_{ \on{ARC}}(X,d) \nonumber
	= \inf \biggl\{ p>0 \biggm| 
	\begin{minipage}{194pt}
		$\theta \in \sJ(X,d)$, there is $p$-Ahlfors regular measure $\mu$ on $X$ with respect to $\theta$ 
	\end{minipage}
	\biggr\}.
\end{equation*}
\end{definition}
The infimum in the definition of $\dim_{ \on{ARC}}(X,d)$ need not be attained in general \cite[\textsection 6]{BK05}.
The \emph{attainment problem for Ahlfors regular conformal dimension} asks if the infimum in the definition of $\dim_{ \on{ARC}}(X,d)$ is attained by a `optimal' metric and measure.
\emph{Quasisymmetric uniformization problem} asks if there is a metric in the conformal gauge isometric to a model space with more desirable properties.  These two problems are often related.
For instance, it is a well-known open problem to determine whether or not the conformal gauge of the standard Sierpi\'{n}ski carpet contains a Loewner metric \cite[p. 408]{HKST}, \cite[Question 8.3]{Kle} (we recall the definition of Loewner metric in Definition \ref{d:loewner}).
Another related question is to determine if the Ahlfors regular conformal dimension of the Sierpi\'{n}ski carpet is attained \cite[Problem 6.2]{BK05}.
As pointed out by Cheeger and Eriksson-Bique, these two questions are essentially the same due to the combinatorial Loewner property of the Sierpi\'{n}ski carpet \cite[Theorem 4.1]{BK13}, \cite[\textsection 1.6]{CE}.

As a motivation for the attainment problem for Ahlfors regular conformal dimension, we recall a long-standing conjecture in geometric group theory, namely Cannon's conjecture.
It asserts that any Gromov hyperbolic group $G$ whose boundary at infinity $\partial_\infty G$ is homeomorphic to $\mathbb{S}^2$ admits an action on the hyperbolic $3$-space $\mathbb{H}^3$ that is isometric, properly discontinuous and cocompact. Bonk and Kleiner show Cannon's conjecture under the additional assumption that the Ahlfors regular conformal dimension of the boundary at infinity $\partial_\infty G$ is attained \cite{BK05}. Thus Cannon's conjecture is reduced to an attainment problem for the Ahlfors regular conformal dimension of $\partial_\infty G$. We refer the reader to ICM 2006 proceedings of Bonk for further context and details \cite{Bon}.

Another related motivation for the attainment problem for Ahlfors regular conformal dimension is to better understand Loewner spaces. Since Loewner spaces enjoy desirable properties, it is useful to know if a given metric space contains a Loewner metric in its conformal gauge. To this end, Kleiner formulated a combinatorial version of Loewner property that is necessary for such a Loewner metric to exist and is easier to check. Bourdon and Kleiner verify combinatorial Loewner property for a number of examples including the Sierpi\'{n}ski carpet \cite{BK13}. Kleiner conjectured that the combinatorial Loewner property for a self-similar space is equivalent to the existence of Loewner metric in the conformal gauge \cite[Conjecture 7.5]{Kle}. Due to an observation of Cheeger and Eriksson-Bique \cite[\textsection 1.6]{CE}, Kleiner's conjecture can be rephrased as a conjecture about the attainment problem as follows: combinatorial Loewner property for a self-similar space implies that the Ahlfors regular conformal dimension is attained.  We refer   to the ICM 2006 proceedings of Kleiner for further details \cite{Kle}.
There has been a very recent progress on Kleiner's conjecture \cite[Conjecture 7.5]{Kle}. Anttila and Eriksson-Bique \cite{AE.kleiner} have verified that some variants of Laakso spaces \cite{Laa00,Laa02} give counterexamples to Kleiner's conjecture, i.e., these spaces do not contain any Loewner metric in their conformal gauges while these satisfy the combinatorial Loewner property. 

In contrast to the attainment problem, the value of the Ahlfors regular conformal dimension is   better understood.
Indeed, Keith and Kleiner \cite[Remark on p. 63]{BK13}
showed that the Ahlfors regular conformal dimension can be described as a critical exponent of discrete energies (or equivalently, discrete modulus).
Carrasco gave an independent proof of this fact  \cite[Theorem 1.3]{CP13} (see also \cite{Kig20,Mur23b,Sha23} for related works). Since we can understand the value of Ahlfors regular conformal dimension using discrete energies, it is also natural to try to understand the attainment problem using discrete energies.  Our next main result shows the relevance of discrete energies for the attainment problem on the Sierpi\'{n}ski carpet. We conjecture that similar results should be true on a large class of self-similar spaces.

As partial progress towards the attainment problem for Ahlfors regular conformal dimension on the Sierpi\'{n}ski carpet, we show that if an optimal  measure attaining the  Ahlfors regular conformal dimension exists then this measure is necessarily a bounded perturbation of the $p$-energy measure of some function in our $(1,p)$-Sobolev space, where $p$ is the Ahlfors regular conformal dimension.
This result  confirms the relevance of energy measures to the attainment problem for Ahlfors regular conformal dimension as  predicted earlier in \cite[p.395-396]{KM23}.
Furthermore, if the Ahlfors regular conformal dimension is attained we identify our Sobolev space $\sF_p$ with Newton-Sobolev space of the attaining metric measure space.   To state this result, we briefly recall the definition of Newton-Sobolev space $N^{1,p}(X,\theta,\mu)$ of a metric measure space $(X,\theta,\mu)$.

We define $\wt{N}^{1,p}(X,\theta,\mu)$ as the set of $p$-integrable functions with a $p$-integrable $p$-weak upper gradient (we recall the definition of weak upper gradient in Definition \ref{d:ug}). We equip $\wt{N}^{1,p}(X,\theta,\mu)$ with the seminorm $\norm{u}_{N^{1,p}(X,\theta,\mu)}:=\norm{u}_{L^p(\mu)}+\norm{g_u}_{L^p(\mu)}$, where $g_u$ denotes the minimal $p$-weak upper gradient of $u$ in $(X,\theta,\mu)$ (Heuristically, the minimal  $p$-weak upper gradient of $u$ is an analogue of $\abs{\nabla u}$). Two functions $f,g \in  \wt{N}^{1,p}(X,\theta,\mu)$ are said to be equivalent if $\norm{f-g}_{N^{1,p}(X,\theta,\mu)}=0$.
 The \emph{Newton-Sobolev space} $N^{1,p}(X,\theta,\mu)$ is defined to the set of equivalence classes equipped with the norm $\norm{\,\cdot\,}_{\wt{N}^{1,p}(X,\theta,\mu)}$. Our final result below identifies the Newton-Sobolev space for any metric and measure attaining the Ahlfors regular conformal dimension of $(K,d)$ with our Sobolev space $\sF_p(K,d,m)$. Moreover, the attaining measure is essentially equal to the energy measure $\Gamma_p\langle h\rangle$ for some function $h \in \contfunc(K)\cap \sF_p(K,d,m)$. The following result relates the Sobolev space based on upper gradient to the self-similar Sobolev space under the attainment of Ahlfors regular conformal dimension. Moreover, the attaining measures are essentially energy measures.

\begin{theorem} \label{t:attain}
	Let $(K,d,m)$ denote the Sierpi\'{n}ski carpet and let $p=\dim_{ \on{ARC}}(K,d)$. Suppose that there exists $\theta \in \sJ(K,d)$ and a measure $\mu$ on $K$ attaining the Ahlfors regular conformal dimension; that is, $\mu$ is a $p$-Ahlfors regular measure on $(K,\theta)$. Let $\sF_p=\sF_p(K,d,m), \sE_p$ and $\Gamma_p\langle \,\cdot\, \rangle$ denote the Sobolev space, $p$-energy and $p$-energy measure as given in Theorem \ref{t:main-em}. Then we have the following:
	\begin{enumerate}[\rm(i)]
		\item\label{attain-core} The spaces $\sF_p(K,d,m)$ and $N^{1,p}(K, \Lmetric, \Lmeasure)$ are equal  with comparable norms, semi-norms, and energy measure. More precisely, it holds that $\contfunc(K)\cap \sF_p(K,d,m)= \contfunc(K)\cap N^{1,p}(K,\theta,\mu)$, there exist a bijective linear map $\iota:\sF_p(K,d,m) \to N^{1,p}(K,\theta,\mu)$ and $C_1>1$ such that $\iota(f)=f$ for any $f \in 	 \contfunc(K)\cap \sF_p(K,d,m)= \contfunc(K)\cap N^{1,p}(K,\theta,\mu)$\footnote{more precisely, the equivalence class containing $f$ in $\sF_p(K,d,m)$ is mapped to the equivalence class containing $f$ in $N^{1,p}(K,\theta,\mu)$.} and
		 \[
		 C_1^{-1} \Gamma_p \langle f \rangle(B) \le \int_{B} g^p_{\iota(f)}\,d\mu \le  C_1  \Gamma_p \langle f  \rangle(B)
		 \]
		 for any Borel set $B \subset K, f \in \sF_p(K,d,m)$, where $g_{\iota(f)}$ denotes the minimal $p$-weak upper gradient of $\iota(f)$. In particular, $ C_1^{-1} \sE_p(f) \le \int_{K} g^p_{\iota(f)}\,d\mu \le  C_1  \sE_p(f)$ for all $f \in \sF_p(K,d,m)$.  Furthermore, the corresponding norms are comparable; that is,
		 \[
		 C^{-1} \norm{f}_{\sF_p(K,d,m)} \le \norm{\iota(f)}_{N^{1,p}(K,\theta,\mu)} \le  C \norm{f}_{\sF_p(K,d,m)} \quad \mbox{for all  $f \in \sF_p(K,d,m)$.}
		 \]
		\item \label{attain-ecomp} There exist $h \in \sF_p(K,d,m) \cap \contfunc(K)$ and $C_2 \in (0,\infty)$ such that
		\[
		C_2^{-1} \Gamma_p\langle h\rangle(B) \le \mu(B)\le C_2 \Gamma_p\langle h\rangle(B) \quad \mbox{for any Borel set $B \subset K$.}
		\]
		In particular, $\Gamma_p\langle h \rangle$ is a $p$-Ahlfors regular measure on $(K,\theta)$.
	\end{enumerate}
\end{theorem}
Let us briefly explain how Theorem \ref{t:attain} could be potentially used to solve the attainment problem. Although the attainment problem requires us to find optimal metrics and measures, it is well known that the metrics and measures determine each other (see Lemmas \ref{l:meas-met} and \ref{l:met-meas}). Therefore it suffices to look for optimal measure and use Lemma \ref{l:met-meas} to construct the corresponding metric. By Theorem \ref{t:attain}, it suffices to look for optimal measures among energy measures of continuous functions. We conjecture that it suffices to look for optimal measure among energy measures of $p$-harmonic functions (see Conjecture \ref{con:harmonic}).  One could then hope to find a `good' function whose energy measure is optimal or rule out the existence of such function by a careful analysis of energy measures.
In fact, Theorem \ref{t:attain}\ref{attain-ecomp} was inspired by a similar result for the attainment problem for conformal walk dimension \cite[Theorem 6.16]{KM23}. Such a result was successfully used to solve a similar attainment problem in \cite{KM23}.

More generally, we believe that Sobolev spaces and energy measures are relevant to similar quasisymmetric uniformization problems and the attainment problem for Ahlfors regular conformal dimension on other `self-similar spaces' such as boundaries of hyperbolic groups and Julia sets in conformal dynamics. It would be interesting to construct Sobolev space, energy measures and prove analogues of Theorem \ref{t:attain} for fractals arising from hyperbolic groups and conformal dynamics \cite{Bon,Kle}.
Another obvious question is to use Theorem \ref{t:attain} to solve the attainment problem. This motivates further study of energy measures.

Although we discussed three approaches towards defining Sobolev space based on discrete energies, Korevaar--Schoen energies, and upper gradients, there are several omissions. Among them, we  mention Sobolev spaces constructed using two-point estimates by Haj\l asz (Haj\l asz--Sobolev space) \cite{Haj96}, Poincar\'e inequalities by Haj\l asz--Koskela (Poincar\'e--Sobolev space)    \cite{HK95,HK00}, using weak $L^p$-estimates of gradient on hyperbolic fillings by Bonk--Saksman \cite{BS18}, and using the modulus of gradient based on a strongly local regular Dirichlet form \cite{BBR24,Kuw24}. It would be interesting to understand if these spaces or their variants are related to our Sobolev spaces constructed using discrete energies.

\subsection{Overview for the rest of the paper.}

In \textsection \ref{sec.preli}, we introduce basic notions concerning capacity, modulus and volume growth of graphs.

In \textsection \ref{sec.BCL}, we introduce variants of the ball Loewner property due to  Bonk and Kleiner and of Loewner-type modulus lower bounds between connected sets. The main result  (Theorem \ref{thm.pGCL-gamma}) shows that lower bounds of modulus between balls   imply lower bounds of modulus between any pair of connected sets.

In \textsection \ref{sec.PI}, we use the  lower bounds of modulus from  \textsection \ref{sec.BCL} to obtain a discrete Poincar\'e inequality. The proof of the Poincar\'e inequality in Theorem \ref{thm.PI-discrete} follows an idea of Heinonen and Koskela \cite[Proof of Theorem 5.12]{HK98}.

In \textsection \ref{sec.EHI}, we show that discrete Poincar\'e inequality along with capacity upper bounds on graphs imply elliptic Harnack inequality for $p$-harmonic functions on graphs. In this regime, to show the regularity of the Sobolev space we obtain Harnack inequality in a discrete setting by using a geometric argument in \cite{Mur18+} which improves the results  of \cite{BCK05}  in the case $p=2$.
The Harnack inequality is then used to prove existence of H\"older continuous cutoff functions with controlled energy. 

In \textsection \ref{sec.unif}, we introduce a framework describing the approximation of a metric space by a sequence of graphs. We then define the Sobolev space using discrete graph energies under the assumption that the sequence of graphs satisfy uniform Poincar\'e inequality and capacity upper bounds. We obtain many    basic properties of this Sobolev space such as completeness, separability, reflexivity, and the existence of a dense set of continuous functions in the Sobolev space.

In \textsection \ref{sec.domain}, we identify our Sobolev space as the   Korevaar-Schoen space with comparable energies. We express the critial exponent for Besov--Lipschitz space in terms of the scaling exponent for discrete energies.

In \textsection \ref{sec.ss}, we introduce the setting of self-similar sets and construct a natural approximation of a self-similar set by a seuence of graphs. obtain a sufficient condition for the existence of a self-similar $p$-energy in our Sobolev space (Theorem \ref{thm.fix}).

In \textsection \ref{sec.emeas}, we describe the construction of the energy measure associated to a self-similar $p$-energy and obtain its basic properties.

In \textsection \ref{sec.PSC}, we apply the results from previous sections to the  planar Sierpi\'nski carpet. To this end, we check the assumptions imposed on the graph approximations for the construction of the Sobolev space in \textsection \ref{sec.unif} and the assumptions imposed for the existence of a self-similar $p$-energy in \textsection \ref{sec.ss}.

In \textsection \ref{sec.attain}, we show that any optimal measure for Ahlfors regular conformal dimension on the Sierpi\'nski carpet must necessarily be comparable to a energy measure. If the Ahlfors regular conformal dimension is attained, we identify the Newton-Sobolev space of the attaining space with our Sobolev space.

In \textsection \ref{sec.conj}, we collect some conjectures and open problems related to our work.

%
%

\noindent
\textbf{Notations.}  In this paper, we use the following notation and conventions.
\begin{enumerate}[label=\textup{(\arabic*)},align=left,leftmargin=*,topsep=2pt,parsep=0pt,itemsep=2pt]
	\item $\mathbb{N} \coloneqq \{ n \in \mathbb{Z} \mid n > 0\}$ and $\mathbb{Z}_{\ge 0} \coloneqq \mathbb{N} \cup \{ 0 \}$.
	\item For a set $A$, we write $\#A$ to denote the cardinality of $A$.
	\item Let $X$ be a non-empty set. For disjoint subsets $A$ and $B$ of $X$, we use $A \sqcup B$ to denote the disjoint union of $A$ and $B$.
	\item For $p > 1$, we write $p' = \frac{p}{p - 1}$, i.e. $p$ is the H\"{o}lder conjugate index of $p$ so that $\frac{1}{p} + \frac{1}{p'} = 1$.
	\item For $a \in \mathbb{R}$, define $\mathrm{sgn}(a)\coloneqq\begin{cases} 1 \quad &\text{if $a > 0$,} \\ 0 \quad &\text{if $a = 0$,} \\ -1 \quad &\text{if $a < 0$.}\end{cases}$. 
	\item For $a, b \in \mathbb{R}$, we write $a \vee b = \max\{ a, b \}$ and $a \wedge b = \min\{ a, b \}$.
		For simplicity, we also write $a^{+} = a \vee 0$ and $a^{-} = a \wedge 0$.
		We also use these notations for real-valued functions.
	\item For $a \in \mathbb{R}$, define $\lceil a \rceil, \lfloor a \rfloor \in \mathbb{Z}$ by
	\[
	\lceil a \rceil = \max\{ n \in \mathbb{Z} \mid n \le a \} \quad \text{and} \quad \lfloor a \rfloor = \min\{ n \in \mathbb{Z} \mid a \le n \}.
	\]
	\item\label{it:nota.discrete} For arbitrary countable set $V$, define
		\[
		\mathbb{R}^{V} = \{ f \mid f\colon V \to \mathbb{R} \}, \quad \ell^{+}(V) = [0, +\infty)^{V} = \{ f \mid f\colon V \to [0, +\infty) \},
		\]
		and
		\[
		\ell_{c}^{+}(V) = \{ f \in [0, +\infty)^{V} \mid \#\supp[f] < +\infty \},
		\]
		where $\supp[f] \coloneqq \{ x \in V \mid f(x) \neq \emptyset \}$. 
		Also, set $\osc_{V}[f] \coloneqq \sup_{x,y\in V}\abs{f(x) - f(y)}$ for $f \in \mathbb{R}^{V}$. 
	\item Let $(X, d)$ be a metric space.
		The open ball with center $x \in X$ and radius $r > 0$ is denoted by $B_{d}(x, r)$, that is,
		\[
		B_{d}(x, r) \coloneqq \{ y \in X \mid d(x, y) < r \}.
		\]
		If the metric $d$ is clear in context, then we write $B(x, r)$ for short.
		We write $\closure{B}(x, R)$ for $\{ y \in X \mid d(x, y) \le r \}$.
		For a metric ball $B$, let $\mathrm{rad}(B)$ denote the radius of $B$.
		For $\lambda \ge 0$ and a ball $B = B(x, R)$, define $\lambda B = B(x, \lambda R)$.
	\item Let $(X, d)$ be a metric space. We define $\diam(A,d) \coloneqq \sup_{x,y \in A}d(x,y)$ and $\dist_{d}(A,B) \coloneqq \inf\{ d(x,y) \mid x \in A, y \in B \}$ for $A,B\subseteq X$. 
	We also use $\diam_{d}(A)$ to denote $\diam(A, d)$.
	If no confusion can occur, we omit the metric $d$ in these notations.
	\item Let $(X, \mathscr{A}, \mu)$ be a measure space. For $f \in L^{1}_{\text{loc}}(X, \mu)$ and $A \in \mathscr{A}$ with $\mu(A) < +\infty$, we use $\fint_{A}f\,d\mu$ to denote the averaged integral of $f$ over $A$, i.e.
		\[
		\fint_{A}f\,d\mu  = \frac{1}{\mu(A)}\int_{A}f(x)\,\mu(dx).
		\]
		We also write $f_{A}$ or $(f)_{A}$ to denote $\fint_{A}f\,d\mu$ if the underlying measure $\mu$ is clear.
	\item Let $(X, \mathscr{A}, \mu)$ be a measure space and let $1 \le p \le \infty$. For $f \in L^{p}(X, \mu)$, we use $\norm{f}_{p}$ to denote the $L^{p}$-norm of $f$. In addition, for any $A \in \mathscr{A}$, define
	\[
	\norm{f}_{p, A} \coloneqq \norm{f\indicator{A}}_{p} = \left(\int_{A}\abs{f(x)}^{p}\,\mu(dx)\right)^{1/p}.
	\]
	\item Let $X$ be a topological space. We use $\mathcal{B}(X)$ to denote the Borel $\sigma$-algebra of $X$ and $\mathscr{B}(X)$ (resp. $\mathscr{B}_{+}(X)$) to denote the set of $[-\infty, \infty]$-valued (resp. $[0, \infty]$-valued) Borel measurable functions on $X$. (Note that each element in $\mathscr{B}(X)$ or $\mathscr{B}_{+}(X)$ is defined on every points of $X$.)
\end{enumerate}

\section{Preliminaries}\label{sec.preli}

\subsection{Basic facts and terminologies of graphs}

Throughout this section, let $G = (V, E)$ be a locally finite connected simple non-directed graph, i.e. $G = (V, E)$ is a simple connected graph, where $V$ is a countable set (the set of vertices) and $E \subseteq \bigl\{ \{ x, y \} \bigm| x, y \in V, x \neq y \bigr\}$ (the set of edges), satisfying
\[
\deg_{G}(x) \coloneqq \#\{ y \in V \mid \{ x, y \} \in E \} < +\infty \quad \text{for all $x \in V$.}
\]
We always consider $G$ as a metric space equipped with the graph distance $d = d_{G}$.
In this paper, we suppose that $G$ has bounded degree, i.e.
\begin{equation*}
	\deg(G) \coloneqq \sup_{x \in V}\deg_{G}(x) < +\infty.
\end{equation*}

A sequence of vertices $\theta = [x_{0}, \dots, x_{n}]$ for some $n \in \mathbb{N}$ is said to be a \emph{(finite) path} in $G$ if $x_{i} \in V$ and $\{ x_{i}, x_{i + 1} \} \in E$ for each $i \in \{ 0, \dots, n - 1 \}$.
We frequently regard a path $\theta$ as a subset $\{ x_i \}_{i = 0}^{n}$ of $V$.
Define the \emph{length} of $\theta = [x_{0}, \dots, x_{n}]$ by
\[
\len_{G}(\theta) \coloneqq n.
\]
A finite path $\theta = [x_{0}, \dots, x_{n}]$ is said to be \emph{simple} if there is no loops, i.e. $x_i \neq x_j$ for any distinct $i, j \in \{ 0, \dots, n \}$.
Note that our definition excludes the case where a one point set $\{ x \}$ becomes a path (since $G$ has no self-loops).
In particular, $\len(\theta) \in \mathbb{N}$ for any finite path $\theta$.

For any subset $A \subseteq V$, we define
\[
E(A) \coloneqq \bigl\{ \{ x, y \} \in E \bigm| x, y \in A \bigr\}.
\]
A subset $A \subseteq V$ is called a \emph{connected subset of $V$ (with respect to $G$)} if $d_{(A, E(A))}(x, y) < \infty$ for all $x, y \in A$.

For arbitrary $A \subseteq V$, define
\[
\partial_{i}A = \{ x \in A \mid \text{there exists $y \in V \setminus A$ such that $\{ x, y \} \in E$} \},
\]
\[
\partial A = \{ x \in V \setminus A \mid \text{there exists $y \in A$ such that $\{ x, y \} \in E$} \},
\]
and
\[
\closure{A} = A \cup \partial A.
\]
The set $\partial_{i}A$ (resp. $\partial A$) is called the interior (resp. exterior) boundary of $A$ in $G$.
The set $\closure{A}$ is a kind of closure of $A$ in $G$.

\subsection{Combinatorial \texorpdfstring{$p$-modulus}{p-modulus} of path families}
We recall the notion of combinatorial modulus of discrete path families on a graph and a few basic properties.
For a path $\theta$ in $G=(V,E)$ and $\rho \in \ell^{+}(V)$, define the  \emph{$\rho$-length of $\theta$, $L_{\rho}(\theta)$,} by
\[
L_{\rho}(\theta) = \sum_{v \in \theta}\rho(v).
\]
For arbitrary path family $\Theta$ on $G$, define the \emph{$\rho$-length of $\Theta$, $L_{\rho}(\Theta)$,} by
\[
L_{\rho}(\Theta) = \inf_{\theta \in \Theta}L_{\rho}(\theta).
\]
The set of \emph{admissible functions $\ADM(\Theta)$ for $\Theta$} is given by
\[
\ADM(\Theta) = \{ \rho \in \ell^{+}(V) \mid L_{\rho}(\Theta) \ge 1 \}.
\]

\begin{defn}\label{dfn.pMod}
	Let $\Theta$ be a family of paths in $G$ and let $p > 0$.
	The \emph{(combinatorial) $p$-modulus $\MOD_{p}^{G}(\Theta)$ of $\Theta$} is
	\begin{equation*}
		\MOD_{p}^{G}(\Theta) = \inf_{\rho \in \ADM(\Theta)}\norm{\rho}_{\ell^{p}(V)}^{p} = \inf_{\rho \in \ADM(\Theta)}\sum_{v \in V}\rho(v)^{p}.
	\end{equation*}
	We also use $\MOD_{p}(\Theta)$ to denote $\MOD_{p}^{G}(\Theta)$ when no confusion can occur.
\end{defn}
\begin{rmk}
	For a  path family  $\Theta$, define
	\[
	V[\Theta] \coloneqq \{ v \in V \mid v \in \theta' \text{ for some } \theta' \in \Theta \}.
	\]
	We easily see that $\rho \in \ADM(\Theta)$ implies $\rho\indicator{V[\Theta]} \in \ADM(\Theta)$.
	This observation yields
	\[
	\MOD_{p}^{G}(\Theta) = \inf_{\rho \in \ADM(\Theta)}\norm{\rho}_{p, V[\Theta]}^{p}.
	\]
\end{rmk}

The following properties of $p$-modulus is well known.
\begin{lemma}[e.g., {\cite[Section 5.2]{HKST}}]\label{lem.basic-pMod}
	Let $p > 0$.
	\begin{enumerate}[\rm(i)]
		\item\label{it:mod.empty} $\Mod_{p}^{G}(\emptyset) = 0$.
		\item\label{it:mod.mono} If path families $\Theta_{i} \, (i = 1, 2)$ satisfy $\Theta_{1} \subseteq \Theta_{2}$, then $\Mod_{p}^{G}(\Theta_{1}) \le \Mod_{p}^{G}(\Theta_{2})$.
		\item\label{it:mod.subadd} For any sequence of path families $\{ \Theta_{n} \}_{n \in \mathbb{N}}$,
		\[
		\MOD_{p}^{G}\left(\bigcup_{n \in \mathbb{N}}\Theta_{n}\right) \le \sum_{n = 1}^{\infty}\MOD_{p}^{G}(\Theta_{n}).
		\]
		\item\label{it:mod.majority} Let $\Theta, \Theta_{\#}$ be families of paths. If all path $\theta \in \Theta$ has a sub-path $\theta_{\#} \in \Theta_{\#}$ (i.e. $\theta_{\#} \subseteq \theta$), then
		\[
		\MOD_{p}^{G}(\Theta) \le \MOD_{p}^{G}(\Theta_{\#}).
		\]
	\end{enumerate}
\end{lemma}
If $p >1$, then by the strict convexity of $\ell^p$,  there exists a unique $\rho \in \ADM(\Theta)$ such that $	\MOD_{p}^{G}(\Theta) = \sum_{v \in V}\rho(v)^{p}$.
%

For subsets $A_{i} \subseteq V \, (i = 0, 1, 2)$ with $A_{0} \cup A_{1} \subseteq A_{2}$, define
\[
\PATH(A_{0}, A_{1}; A_{2}) = \left\{ [x_{0}, \dots, x_{n}] \;\middle|\;
\begin{array}{c}
	n \in \mathbb{N}, \{ x_{i}, x_{i + 1} \} \in E \text{ for any $i = 0, \dots, n - 1$}, \\ x_{i} \in A_{2} \, (i = 0, \dots, n), x_{0} \in A_{0}, x_{n} \in A_{1}
\end{array}
\right\},
\]
and we write $\MOD_{p}(A_{0}, A_{1}; A_{2})$ for $\MOD_{p}\bigl( \PATH(A_{0}, A_{1}; A_{2}) \bigr)$.
We use $\PATH(A_{0}, A_{1})$ and $\MOD_{p}(A_{0}, A_{1})$ to denote $\PATH(A_{0}, A_{1}; V)$ and $\MOD_{p}(A_{0}, A_{1}; V)$ respectively.
When we need to specify the underlying graph $G$, we will use $\PATH_{G}(A_{0}, A_{1}; A_{2})$ and so on.

The following lemma is used to obtain lower bounds on modulus. Roughly speaking, modulus lower bound of a curve family is equivalent to existence of shortcuts.
\begin{lem}\label{lem.shortPath}
	Let $p > 0$.
	Let $\Theta$ be a family of paths in $G$ and let $c > 0$.
	If $\MOD_{p}(\Theta) \ge c$, then for any $\varepsilon > 0$ and $\rho \in \ell^{+}(V)$ there exists a path $\theta \in \Theta$ such that
	\begin{equation}\label{e.short-path}
		L_{\rho}(\theta) \le (1 + \varepsilon)c^{-1/p}\norm{\rho}_{p, V[\Theta]}.
	\end{equation}
	(If the infimum in the definition of $L_{\rho}(\Theta)$ is attained, then $\varepsilon$ can be replaced with $0$.)
	Conversely, if for any $\rho \in \ell^{+}(V)$ there exists a path $\theta \in \Theta$ such that $L_{\rho}(\theta) \le c^{-1/p}\norm{\rho}_{p}$, then $\MOD_{p}(\Theta) \ge c$.
	In particular, if $p \ge 1$, $L \in \mathbb{N}$ and there exists $\theta \in \Theta$ such that $\len(\theta) \le L$, then
	\begin{equation}\label{smallscale}
		\MOD_{p}^{G}(\Theta) \ge L^{1 - p}.
	\end{equation}
\end{lem}
\begin{proof}
	First, we observe that
	\begin{equation}\label{mod-another}
		\MOD_{p}^{G}(\Theta) = \inf_{\rho \in \ell^{+}(V); L_{\rho}(\Theta) > 0}\frac{\norm{\rho}_{\ell^{p}(V)}^{p}}{L_{\rho}(\Theta)^{p}}.
	\end{equation}
	Set $\widetilde{\rho} = \rho\indicator{V[\Theta]}$ for any $\rho \in \ell^{+}(V)$.
	Since $L_{\rho}(\theta) = L_{\widetilde{\rho}}(\Theta)$ and $\norm{\widetilde{\rho}}_{\ell^{p}(V)} \le \norm{\rho}_{\ell^{p}(V)}$, we have
	\[
	\MOD_{p}^{G}(\Theta) = \inf_{\rho \in \ell^{+}(V); L_{\rho}(\Theta) > 0}\frac{\norm{\rho}_{p, V[\Theta]}^{p}}{L_{\rho}(\Theta)^{p}}.
	\]
	Therefore, $\MOD_{p}(\Theta) \ge c$ implies $L_{\rho}(\Theta) \le c^{-1/p}\norm{\rho}_{p, V[\Theta]}$.
	Pick $\theta \in \Theta$ so that $L_{\rho}(\theta) \in [L_{\rho}(\Theta), (1 + \varepsilon)L_{\rho}(\Theta))$.
	Then $\theta$ satisfies \eqref{e.short-path}.

	Let us prove the converse.
	Let $\rho \in \ell^{+}(V)$ with $L_{\rho}(\Theta) > 0$ and suppose that there exists $\theta \in \Theta$ such that $L_{\rho}(\theta) \le c^{-1/p}\norm{\rho}_{p}$.
	Combining with \eqref{mod-another}, we get $\MOD_{p}(\Theta) \ge c$.

	Lastly, suppose that $p \ge 1$ and that $\theta \in \Theta$ satisfies $\len(\theta) \le L$.
	For any $\rho \in \ell^{+}(V)$, by H\"{o}lder's inequality,
	\begin{align*}
		L_{\rho}(\theta) = \sum_{v \in \theta}\rho(v)
		\le \Biggl(\sum_{v \in \theta}1\Biggr)^{(p - 1)/p}\norm{\rho}_{p, \theta}
		\le L^{(p - 1)/p}\norm{\rho}_{p, V[\Theta]},
	\end{align*}
	which implies \eqref{smallscale}.
	The proof is completed.
\end{proof}

\subsection{Discrete \texorpdfstring{$p$-energy}{p-energy}, \texorpdfstring{$p$-Laplacian}{p-Laplacian} and associated capacity}\label{subsec.p-ene_discrete}
For $f \in \mathbb{R}^{V}$, the \emph{length of discrete gradient of $f$, $\abs{\nabla f} \colon E \to [0, +\infty)$,} is given by
\[
\abs{\nabla f}\bigl(\{ x, y \}\bigr) = \abs{f(y) - f(x)} \quad \text{for $\{ x, y \} \in E$.}
\]
For simplicity, we also use $\abs{\nabla f}(x, y)$ to denote $\abs{\nabla f}\bigl(\{ x, y \}\bigr)$ for each $\{ x, y \} \in E$.

\begin{defn}
	Let $p > 0$ and let $A \subseteq V$.
	For $f, g \in \mathbb{R}^{V}$, define
	\[
	\mathcal{E}_{p, A}^{G}(f; g)
	\coloneqq
	\sum_{\{ x, y \} \in E(A)}\mathrm{sgn}(f(y) - f(x))\abs{f(y) - f(x)}^{p - 1}(g(y) - g(x)).
	\]
	The $p$-energy of $f$ on $A$ is given by $\mathcal{E}_{p, A}^{G}(f) = \mathcal{E}_{p, A}^{G}(f; f)$, i.e.
	\[
	\mathcal{E}_{p, A}^{G}(f) \coloneqq \sum_{\{ x, y \} \in E(A)}\abs{\nabla f}(x, y)^{p} = \sum_{\{ x, y \} \in E(A)}\abs{f(x) - f(y)}^{p}.
	\]
	We write $\mathcal{E}_{p}^{G}(f; g)$ and $\mathcal{E}_{p}^{G}(f)$ for $\mathcal{E}_{p, V}^{G}(f; g)$ and $\mathcal{E}_{p, V}^{G}(f)$ respectively.
	We omit the underlying graph $G$ in these notations if no confusion can occur.
\end{defn}

We recall basic properties of discrete $p$-energy, which are immediate from the definition.
\begin{lem}\label{lem.basic-dEp}
	Let $p > 0$ and $A \subseteq V$.
	\begin{enumerate}[\rm(a)]
		\item\label{it:dEp.lip} $\mathcal{E}_{p, A}^{G}(\varphi \circ f) \le \mathcal{E}_{p, A}^{G}(f)$ for any $f \in \mathbb{R}^{A}$ and $1$-Lipschitz function $\varphi \in \mathcal{C}(\mathbb{R})$.
		In particular,
		\[
		\mathcal{E}_{p, A}^{G}\bigl(f^{\#}\bigr) \le \mathcal{E}_{p, A}^{G}(f) \quad \text{for any $f \in \mathbb{R}^{V}$, $a \in \mathbb{R}$, $f^{\#} \in \{ f^{+}, f^{-}, \abs{f}, (f - a)^{+} \}$,}
		\]
		\item\label{it:dEp.cut} $\mathcal{E}_{p, A}^{G}(f \wedge g) \vee \mathcal{E}_{p, A}^{G}(f \vee g) \le \mathcal{E}_{p, A}^{G}(f) + \mathcal{E}_{p, A}^{G}(g)$ for any $f, g \in \mathbb{R}^{A}$.
		\item\label{it:dEp.leibniz} $\mathcal{E}_{p,A}^{G}(f \cdot g) \le (2^{p - 1} \vee 1)\bigl(\norm{g}_{\ell^{\infty}(A)}^{p}\mathcal{E}_{p,A}^{G}(f) + \norm{f}_{\ell^{\infty}(A)}^{p}\mathcal{E}_{p,A}^{G}(g)\bigr)$ for any $f, g \in \mathbb{R}^{A}$.
		\item\label{it:dEp.local} Suppose that $f \in \mathbb{R}^{V}$ is constant on $A^c$, i.e. there exists $a \in \mathbb{R}$ such that $f(x) = a$ for every $x \not\in A$.
		Then we have $\mathcal{E}_{p}^{G}(f) = \mathcal{E}_{p, \closure{A}}^{G}(f)$.
	\end{enumerate}
\end{lem}
\begin{proof}
	\ref{it:dEp.lip} This is obvious from $\abs{\varphi(f(x)) - \varphi(f(y))}^{p} \le \abs{f(x) - f(y)}^{p}$.

	\ref{it:dEp.cut} This is immediate from the following elementary estimate. For any $a_{1},a_{2},b_{1},b_{2} \in \mathbb{R}$,
	\begin{align*}
		\abs{(a_{1} \wedge b_{1}) - (a_{2} \wedge b_{2})}^{p} \vee \abs{(a_{1} \vee b_{1}) - (a_{2} \vee b_{2})}^{p} \le \abs{a_{1} - a_{2}}^{p} + \abs{b_{1} - b_{2}}^{p}.
	\end{align*}

	\ref{it:dEp.leibniz} We easily see that
	\begin{align*}
        \mathcal{E}_{p,A}^{G}(f \cdot g)
        &\le (2^{p - 1} \vee 1)\sum_{\{ x, y \} \in E(A)}\bigl(\abs{g(x)}^{p}\abs{f(x) - f(y)}^{p} + \abs{f(y)}^{p}\abs{g(x) - g(y)}^{p}\bigr) \\
        &\le (2^{p - 1} \vee 1)\Bigl(\norm{g}_{\ell^{\infty}(A)}^{p}\mathcal{E}_{p,A}^{G}(f) + \norm{f}_{\ell^{\infty}(A)}^{p}\mathcal{E}_{p,A}^{G}(g)\Bigr).
    \end{align*}

	\ref{it:dEp.local} The assertion holds since $\abs{f(x) - f(y)} = 0$ whenever $\{ x, y \} \not\in E\bigl(\closure{A}\bigr)$.
\end{proof}

%

Next we recall the definition of discrete $p$-Laplacian using a discrete version of integration by parts.
Let $\langle\cdot, \cdot\rangle_{\ell^2(V, \deg)}$ denote the inner product of $\ell^{2}(V, \deg)$.
\begin{defn}
	Let $p > 0$.
	The \emph{$p$-Laplacian} $\Delta_{p}^{G}$ on $G$ is the operator satisfying
	\begin{align*}
		\mathcal{E}_{p}^{G}(f; g) = -\frac{1}{2}\bigl\langle\Delta_{p}^{G}f, g\bigr\rangle_{\ell^2(V, \deg)}
	\end{align*}
	for all $f, g \in \mathbb{R}^{V}$.
	Equivalently,
	\begin{equation}\label{eq.pLap_explicit}
		\bigl(\Delta_{p}^{G}f\bigr)(x) = \frac{1}{\deg(x)}\sum_{\substack{y \in V; \\ (x, y) \in E}}\mathrm{sgn}(f(y) - f(x))\abs{f(y) - f(x)}^{p - 1}.
	\end{equation}
	(See \cite[Theorem 6.4]{Shi21} for example.)
	A function $f \in \mathbb{R}^{V}$ is said to be \emph{$p$-superharmonic} (resp. \emph{$p$-subharmonic}) at $x \in V$ if $\Delta_{p}^{G}f(x) \le 0$ (resp. $\Delta_{p}^{G}f(x) \ge 0$).
	In addition, $f$ is said to be \emph{$p$-harmonic} at $x \in V$ if $\Delta_{p}^{G}f(x) = 0$.
	If $A \subseteq V$ and $\Delta_{p}^{G}f(x) = 0$ for every $x \in A$, then $f$ is said to be $p$-harmonic in $A$.
	$p$-Superharmonic and $p$-subharmonic functions in $A$ are defined in similar ways.
\end{defn}

The following lemma describes a well-known property of $p$-superharmonic (resp. $p$-subharmonic) functions, namely the \emph{minimum (resp. maximum) principle}.
\begin{lem}[{\cite[Theorem 3.14]{HS97}} or {\cite[Theorem 7.5]{MY92}}]\label{lem.max/min}
	Let $A$ be a non-empty connected subset of $G$.
	Let $f \in \mathbb{R}^{V}$ be $p$-superharmonic (resp. $p$-subharmonic) in $A$.
	\begin{enumerate}[\rm(a)]
		\item\label{it:maxprinciple.1} If there exists $x \in A$ such that $f(x) = \min_{z \in \closure{A}}f(z)$ (resp. $f(x) = \max_{z \in \closure{A}}f(z)$), then $f$ is constant on $\closure{A}$.
		\item\label{it:maxprinciple.2} If $A$ is finite, then $\min_{\partial A}f = \min_{\closure{A}}f$ (resp. $\max_{\partial A}f = \max_{\closure{A}}f$).
	\end{enumerate}
\end{lem}
\begin{proof}
	For the reader's convenience, we recall the proof by following \cite[Theorem 1.37]{Bar.RW}, where the case $p = 2$ is treated.
	Here, we discuss only the case where $f$ is $p$-superharmonic on $A$ because the maximum principle can be obtained from the minimum principle by considering $-f$ instead of $f$.

	\ref{it:maxprinciple.1} Define $A_{\ast} = \{ z \in \closure{A} \mid f(z) = \min_{\closure{A}}f \}$.
	Then $A \cap A_{\ast} \neq \emptyset$ since $x \in A \cap A_{\ast}$.
	For any $y \in A \cap A_{\ast}$ and $z \in \closure{A}$ with $(y, z) \in E$, we have $f(z) \ge f(y)$.
	Since $f$ is $p$-superharmonic in $A$,
	\begin{align*}
		0
		\ge
		\deg(y)\Delta_{p}f(y)
		=
		\sum_{z \in V, (y, z) \in E}\mathrm{sgn}(f(z) - f(y))\abs{f(z) - f(y)}^{p - 1}
		\ge
		0.
	\end{align*}
	Hence $f(z) - f(y) = 0$ for any $z \in \closure{A}$ with $(y, z) \in E$.
	This implies that $\closure{A} \cap \closure{\{ y \}} \subseteq A_{\ast}$ for any $y \in A \cap A_{\ast}$.
	Since $A$ is connected, we conclude that $A_{\ast} = \overline{A}$, which means $f|_{\closure{A}} \equiv \min_{\closure{A}}u$.

	\ref{it:maxprinciple.2} Note that $\closure{A}$ is a finite set and thus there exists $x \in \closure{A}$ such that $f(x) = \min_{\closure{A}}f$.
	If $x \in \partial A$, then there is nothing to be proved.
	If $x \in A$, then \ref{it:maxprinciple.1} implies that $f$ is constant on $\closure{A}$.
	We finish the proof.
\end{proof}

\begin{defn}
	Let $p > 0$ and let $A_{i} \subseteq V \, (i = 0, 1, 2)$ with $A_{0} \cup A_{1} \subseteq A_{2}$.
	Define the $p$-capacity between $A_{0}$ and $A_{1}$ in $A_{2}$ by
	\[
	\CAP_{p}^{G}(A_{0}, A_{1}; A_{2}) = \inf\bigl\{ \mathcal{E}_{p, A_{2}}^{G}(f) \bigm| f \in \mathbb{R}^{V}, \text{$f = 0$ on $A_{0}$ and $f = 1$ on $A_{1}$}\bigr\}.
	\]
	We write $\CAP_{p}^{G}(A_{0}, A_{1})$ for $\CAP_{p}^{G}(A_{0}, A_{1}; V)$.
	The underlying graph $G$ is omitted in these notations if no confusion can occur.
\end{defn}

The following monotonicity is immediate from the definition.
\begin{lem}\label{lem.cap-mono}
	Let $p > 0$ and let $A_{i} \subseteq V \, (i = 0, 1, 2)$.
	If $A_{i}' \subseteq A_{i} \, (i = 0, 1)$, then
	\[
	\CAP_{p}^{G}(A_{0}', A_{1}'; A_{2}) \le \CAP_{p}^{G}(A_{0}, A_{1}; A_{2})
	\]
\end{lem}

Typical $p$-harmonic functions are given as \emph{equilibrium potential of $p$-capacity:}
\begin{lem} [{\cite[Theorems 3.5 and 3.11]{HS97}}]\label{prop.equi-pot}
	Let $p > 1$.
	Let $A_{0},A_{1} \subseteq V$ and let $A_{2}$ be non-empty connected subset of $V$ with $A_{0} \cap A_{1} = \emptyset$ and $A_{0} \cup A_{1} \subseteq A_{2}$.
	There exists a unique function $\varphi\colon A_{2} \to [0, 1]$  \emph{equilibrium potential}) such that $\varphi|_{A_{i}} \equiv i$ for $i = 0, 1$ and
	\[
	\mathcal{E}_{p, A_{2}}^{G}(\varphi) = \CAP_{p}^{G}(A_{0}, A_{1}; A_{2})
	\]
	Furthermore, $\varphi$ is $p$-harmonic in $A_{2} \setminus (A_{0} \cup A_{1})$.
\end{lem}
On bounded degree graphs, the notions of modulus and capacity between sets are comparable as observed by He and Schramm \cite{HS95}.
\begin{lem}\label{lem.mod/cap}
	Let $p > 0$.
	Then there exists a constant $C \ge 1$ depending only on $p, \deg(G)$ such that the following statement is true: for any $A_{i} \subseteq V (i = 0, 1, 2)$ with $A_{0} \cup A_{1} \subseteq A_{2}$,
	\begin{equation}\label{mod/cap}
		C^{-1}\CAP_{p}^{G}(A_{0}, A_{1}; A_{2}) \le \MOD_{p}^{G}(A_{0}, A_{1}; A_{2}) \le C\CAP_{p}^{G}(A_{0}, A_{1}; A_{2}).
	\end{equation}
\end{lem}
\begin{proof}
	If we introduce the edge version of combinatorial $p$-moduli, then that $p$-moduli and $p$-capacity are the same (see \cite[Theorem 4.2]{ABPPW} or \cite[Theorem 3.17]{Shi21} for example).
	It is easy to see that vertex and edge version of modulus are comparable by a slight modification of \cite[Theorem 8.1]{HS95}.

	A direct proof of \eqref{mod/cap} can be found in \cite[Proposition 4.8.4]{Kig20}.
\end{proof}

\subsection{Volume growth conditions}\label{sec.key}
We recall  doubling properties and Ahlfors regularity on graphs and metric spaces.
\begin{defn}\label{dfn.VD}
	A metric space $(X, \mathsf{d})$ is said to be \emph{metric doubling} if there exists $N_{\textup{D}} \in \mathbb{N}$ such that any ball $B_{\mathsf{d}}(x,r)$ can be covered by at most $N_{\textup{D}}$ balls with radii $r/2$.
	A Borel measure $\mathfrak{m}$ on $X$ is said to be \emph{volume doubling} (\ref{cond.VD} for short) with respect to $\mathsf{d}$ if there exists $C_{\textup{D}} \ge 1$ such that
	\begin{equation}\label{cond.VD}
		0 < \mathfrak{m}(B_{\mathsf{d}}(x, 2r)) \le C_{\textup{D}}\mathfrak{m}(B_{\mathsf{d}}(x, r)) < \infty \quad \text{for all $x \in X$, $r > 0$.} \tag{\textup{VD}}
	\end{equation}
	A graph $G = (V, E)$ is \emph{volume doubling} if \ref{cond.VD} holds with respect to the graph distance and the counting measure.
\end{defn}

\begin{defn}\label{defn.AR}
	Let $\hdim > 0$.
	A metric space $(X, \mathsf{d})$ is said to be \emph{$\hdim$-Ahlfors regular} (\ref{cond.AR} for short) if there exist $C_{\textup{AR}} \ge 1$ and a Borel measure $\mathfrak{m}$ on $X$ with
	\begin{equation}\label{cond.AR}
		C_{\textup{AR}}^{-1}\,r^{\hdim} \le \mathfrak{m}(B_{d}(x,r)) \le C_{\textup{AR}}\,r^{\hdim} \qquad \text{for any $x \in X$ and $r \in (0, 2\diam(X, \mathsf{d}))$.} \tag{\textup{AR($\hdim$)}}
	\end{equation}
	We also say that $\mathfrak{m}$ is \emph{$\hdim$-Ahlfors regular} in such a case. The metric space $(X, d)$ or a Borel measure $\mathfrak{m}$ is said to be \emph{Ahlfors regular} if it satisfies \ref{cond.AR} for some $\hdim > 0$.
	We shall say that a graph $G = (V, E)$ is \emph{$\hdim$-Ahlfors regular} if the condition above defining \ref{cond.AR} holds with respect to the graph distance and the counting measure for all $x \in V$ and for all $r \in [1,\diam(V)+1)$.  
\end{defn}
We recall a few elementary consequences of these definitions.
\begin{remark}\label{rem.doubling}
	Let $(X, \mathsf{d})$ be a metric space.
	\begin{enumerate}[\rm(1)]
		\item If there exists a volume doubling measure $\mathfrak{m}$ on $(X, \mathsf{d})$, then $(X, \mathsf{d})$ is metric doubling whose doubling constant $N_{\textup{D}}$ depends only on the doubling constant $C_{\textup{D}}$ of $\mathfrak{m}$ \cite[Chapter 13]{Hei}. 
		\item If a Borel measure $\mathfrak{m}$ on $X$ satisfies \ref{cond.AR} for some $\hdim > 0$, then $\mathfrak{m}$ is volume doubling whose doubling constant $C_{\textup{D}}$ depends only on $C_{\textup{AR}}$ and $\hdim > 0$.
			Furthermore, \ref{cond.AR} implies that the Hausdorff dimension of $(X, \mathsf{d})$ is $\hdim$. 
	\end{enumerate}
\end{remark}

We recall the following consequence of the volume doubling property.
\begin{lem}\label{lem.VD}
	Let $(X, \mathsf{d})$ be a metric space and let $\mathfrak{m}$ be a Borel measure on $X$ satisfying \ref{cond.VD}.
	Then there exists $\alpha > 0$ depending only on the doubling constant $C_{\textup{D}}$ such that
	\begin{equation}\label{VD.growth}
		\frac{\mathfrak{m}(B_{\mathsf{d}}(x, R))}{\mathfrak{m}(B_{\mathsf{d}}(y, r))} \le C_{\textup{D}}^2\left(\frac{\mathsf{d}(x, y) + R}{r}\right)^{\alpha} \quad \text{for any $x, y \in X$ and $1 \le r \le R < \infty$.} \tag{\textup{VD($\alpha$)}}
	\end{equation}
	In particular,
	\begin{equation}\label{VD.growth.2}
		\mathfrak{m}(B_{\mathsf{d}}(x, R)) \le C_{\textup{D}}R^{\alpha} \quad \text{for any $x \in X$ and $1 \le R < \diam(X,\mathsf{d})$.}
	\end{equation}
\end{lem}
Since increasing $\alpha$ does not affect the validity of \ref{VD.growth}, we assume that $\alpha \ge 1$ for much of this work.

\section{Loewner-type lower bounds for \texorpdfstring{$p$-modulus}{p-modulus}}\label{sec.BCL}
Throughout this section, let $p \ge 1$ and let $G = (V, E)$ be a locally finite connected simple non-directed graph.

We introduce the following Loewner-type lower bounds on modulus between balls.
The case with exponent $\zeta=0$ was introduced by Bonk and Kleiner \cite[Proposition 3.1]{BK05}. This was extended by Bourdon and Kleiner   \cite[Proposition 2.9]{BK13} to a discrete setting.
\begin{defn}\label{BCL}
	Let $\zeta \in \mathbb{R}$.
	A graph $G$ satisfies \emph{$p$-combinatorial ball Loewner condition with exponent $\zeta$} (\ref{cond.BCL} for short) if there exists $A \ge 1$ such that the following hold: for any $\kappa > 0$ there exist $c_{\BCL}(\kappa) > 0$ and $L_{\BCL}(\kappa) > 0$ such that
	\begin{equation}\label{cond.BCL}
		\MOD_{p}^{G}(\{ \theta \in \PATH(B_1, B_2) \mid \diam\theta \le L_{\BCL}(\kappa)R \}) \ge c_{\BCL}(\kappa)R^{\zeta} \tag{\textup{BCL$_{p}(\zeta)$}}
	\end{equation}
	whenever $R \in [1, \diam(G)/A)$ and $B_{i} \, (i = 1, 2)$ are balls with radii $R$ satisfying $\dist(B_1, B_2) \le \kappa R$.
\end{defn}
In this section, we discuss \ref{cond.BCL} and prove a key estimate (Theorem \ref{thm.pGCL-gamma}) in this paper.
The setting of this section is given by the following condition:
\begin{equation}\label{assum.BCL}
	\text{The underlying graph $G$ satisfies \ref{cond.BCL} and $1 - p \le \zeta < 1$.} \tag{\textup{BCL$_{p}^{\textrm{low}}(\zeta)$}}
\end{equation}
We are interested in the case where $\zeta$ is the `largest' possible value. Since  \hyperref[assum.BCL]{\textup{BCL$_{p}^{\textrm{low}}(1 - p)$}} is always true by \eqref{smallscale}, there is not much loss of generality in the assumption $\zeta \ge 1-p$ but the inequality $\zeta<1$ need not be true in general but holds in many `low dimensional settings' such as the Sierpi\'{n}ski carpet.

Under \ref{assum.BCL}, we can show a generalized lower bound of $p$-modulus as in the next theorem, which is one of the main results in this section.
It states that Loewner-type lower bounds on modulus between balls imply analogous lower bound on modulus between any pair of connected sets. This result plays important roles in the proofs of Poincar\'{e} inequality in \textsection \ref{sec.PI} and elliptic Harnack inequality in \textsection \ref{sec.EHI}.
The following theorem can be viewed as an extension of a result of Bonk and Kleiner from $\zeta=0$ to more general exponent $\zeta$ \cite[Proposition 3.1]{BK05},  \cite[Proposition 2.9]{BK13}.
\begin{thm}\label{thm.pGCL-gamma}
	Let $p \in [1,\infty)$ and $\kappa_{0} \in (0,\infty)$.
	Assume that $G$ is bounded degree graph that satisfies $p$-combinatorial ball Loewner condition \ref{assum.BCL} with exponent $\zeta \in [1-p,1)$.
	Then there exist constants $c, L > 0$ depending only on the constants associated to the assumptions
	such that the following statement is true: If $F_i \, (i = 1, 2)$ are disjoint connected subsets of $V$ satisfy
	\[
	\frac{\dist(F_1, F_2)}{\diam F_1 \wedge \diam F_2} \le \kappa_{0},
	\]
	then
	\begin{equation}\label{pGCL}
		\MOD_{p}^{G}(\{ \theta \in \PATH(F_1, F_2) \mid \diam\theta \le LR_0 \}) \ge cR_{0}^{\zeta},
	\end{equation}
	where $R_0 \coloneqq 2\dist(F_1, F_2) \wedge \frac{1}{2}\diam F_1 \wedge \frac{1}{2}\diam F_2$.
\end{thm}

The proof of the above theorem is inspired by \cite[Proposition 3.1]{BK05} and \cite[Proposition 2.9]{BK13}.
Similar to \cite{BK05,BK13}, the idea behind its proof is to show the existence of a shortcut with respect to an arbitrary function $\rho \in \ell^+(V)$ and use Lemma \ref{lem.shortPath}.
To construct such a shortcut, we need two key lemmas.

The first one is a is a discrete analog of \cite[Lemma 3.5]{BK05} and provides a linear decay of measure of suitably chosen balls.
\begin{lem}\label{lem.key-cov}
	Let $(G, \nu) = (V, E, \nu)$ be a weighted graph with $\nu(V) < +\infty$ and let $A \subseteq V$ be a connected subset with respect to $G$ with $\#A \ge 2$.
	Then there exists $z \in A$ such that
	\begin{equation}\label{eq.cov}
		\nu(B(z, r)) \le \frac{8}{\diam A}(r \vee 1)\nu(V) \quad \text{for any $r > 0$}.
	\end{equation}
\end{lem}
\begin{proof}
	The proof is a straightforward modification of the proof of \cite[Lemma 3.5]{BK05}.
	We give the details for the reader's convenience.
	If \eqref{eq.cov} were false, then for any $z \in A$ there exists $r_z > 0$ such that
	\begin{equation}\label{cov-contradict}
		\nu(B(z, r_z)) > C(r_{z} \vee 1)\nu(V),
	\end{equation}
	where $C \coloneqq 8/\diam A$.
	From this estimate, we have
	\[
	\sup_{z \in A}r_z \le \frac{\nu(B(z, r_z))}{C\nu(V)} \le C^{-1} < +\infty.
	\]
	Applying the basic covering lemma (Lemma \ref{lem.3B-appendix}), we get a family of disjoint balls $\{ B(z_i, r_i) \}_{i \in I}$ (for each $i \in I$, $z_i \in A$ and $r_i = r_{z_i}$) such that $A \subseteq \bigcup_{i \in I}B(z_i, 3r_i)$.
	Since $A$ is connected, we can show that for any distinct $i, j \in I$ there exists a sequence $i = i_0, i_1, \cdots, i_{l - 1}, i_{l} = j$ in $I$ such that
	\[
	\closure{B(z_{i_{k - 1}}, 3r_{i_{k - 1}})} \cap \closure{B(z_{i_k}, 3r_{i_k})} \neq \emptyset \quad \text{for any $k \in \{ 1, \dots, l \}$.}
	\]
	By the triangle inequality, we see that
	\[
	\diam A \le \sum_{i \in I}\diam \closure{B(z_i, 3r_i)} \le \sum_{i \in I}(6r_i + 2) \le 8\sum_{i \in I}(r_{i} \vee 1),
	\]
	that is, $\sum_{i \in I}(r_{i} \vee 1) \ge C^{-1}$.
	However, by combining with \eqref{cov-contradict}, we have
	\begin{align*}
		\nu(V) \ge \sum_{i \in I}\nu(B(z_i, r_i)) > C\nu(V)\sum_{i \in I}(r_i \vee 1) \ge \nu(V),
	\end{align*}
	which is a contradiction.
\end{proof}

The next lemma is an analog of \cite[Lemma 3.7]{BK05} or \cite[Lemma 2.10]{BK13}.
Note that condition (iv) is similar to the hypothesis and is suitable for indutive application of this lemma.
\begin{lem}\label{lem.key-shrink}
	Suppose that $G = (V, E)$ satisfies \ref{cond.BCL}.
	For any $\lambda \in (0, 1/8)$, let $L_{\lambda} \coloneqq L_{\BCL}\bigl(\frac{9}{2\lambda}\bigr) + \frac{7}{8}$.
	Let $(B, F_{1}, F_{2})$ be a triple such that $B = B(x, R)$ for some $x \in V$ and $R \ge 16$ and $F_{i} \, (i = 1, 2)$ are connected subset of $V$.
	If the triple $(B, F_{1}, F_{2})$ satisfies
	\[
	F_i \cap \frac{1}{4}B \neq \emptyset \quad \text{and} \quad F_i \setminus B \neq \emptyset \quad (i = 1, 2),
	\]
	then for any $\rho \in \ell^{+}(V)$ there exist $x_i \in F_i \, (i = 1, 2)$ satisfying the following:
	\begin{enumerate}[\rm(i)]
		\item\label{it:shrink.goodball} For each $i = 1, 2$, $x_{i} \in \closure{B}(x, 3R/4)$ and $d(x, x_{1}) \wedge d(x, x_{2}) \le 3R/8$. Furthermore, $B_i \coloneqq B(x_i, \lambda R)$ satisfies $\frac{1}{8\lambda}B_{i} \subseteq \frac{7}{8}B$ and $B_{1} \cap B_{2} = \emptyset$.
		\item\label{it:shrink.lineardecay} For each $i = 1, 2$, $\norm{\rho}_{p, B_i}^{p} \le 128(\lambda \vee R^{-1})\norm{\rho}_{p, B}^{p}$.
		\item\label{it:shrink.goodpath} There exists $\theta \in \PATH\bigl(\frac{1}{4}B_1, \frac{1}{4}B_2\bigr)$ such that $\theta \subseteq L_{\lambda}B$, $\diam \theta \le L_{\BCL}\bigl(\frac{9}{2\lambda}\bigr)R$ and
		\[
		L_{\rho}(\theta) \le C_{p, \lambda}(\lambda R)^{-\zeta/p}\norm{\rho}_{p, L_{\lambda}B},
		\]
		where $C_{p, \lambda} > 0$ is a constant depending only on $p, \zeta, \lambda$ and $c_{\BCL}\bigl(\frac{9}{2\lambda}\bigr)$.
		\item\label{it:shrink.next} $F_{i} \cap \frac{1}{4}B_{i}$, $\theta \cap \frac{1}{4}B_{i}$, $F_{i} \setminus B_{i}$ and $\theta \setminus B_{i} \, (i = 1, 2)$ are non-empty.
	\end{enumerate}
\end{lem}
\begin{proof}
	Since $R \ge 16$, we can choose disjoint connected subsets $\widetilde{F}_{i} \, (i = 1, 2)$  of $V$ such that
	\[
	 \text{$\widetilde{F}_{1}$ is a connected subset of $F_1 \cap \Biggl(\closure{B}\biggl(x, \frac{3}{8}R\biggr) \setminus \closure{B}\biggl(x, \frac{1}{4}R\biggr)\Biggr)$ with $\diam\widetilde{F}_{1} \ge \frac{R}{16}$,}
	\]
	and
	\[
	\text{$\widetilde{F}_{2}$ is a connected subset of $F_2 \cap \Biggl(\closure{B}\biggl(x, \frac{3}{4}R\biggr) \setminus \closure{B}\biggl(x, \frac{5}{8}R\biggr)\Biggr)$ with $\diam\widetilde{F}_{2} \ge \frac{R}{16}$.}
	\]
	Let $\rho \in \ell^{+}(V)$ and define a measure $\nu_{\rho}$ by $\nu_{\rho}(A) = \norm{\rho}_{p, A \cap B}^{p}$ for any $A \subseteq V$, i.e. $\nu_{p}(\{ x \}) = \rho(x)^{p}$ for $x \in B$ and $\nu_{p}(\{ x \}) = 0$ for $x \not\in B$.
	Applying Lemma \ref{lem.key-cov}, we find $z_i \in \widetilde{F}_{i}$ for each $i = 1, 2$ such that
	\[
	\nu_{\rho}(B(z_i, r)) \le \frac{8}{R/16}(r \vee 1)\nu_{\rho}(V) = 128\cdot\frac{r \vee 1}{R}\nu_{\rho}(B) \quad \text{for any $r > 0$.}
	\]
	Choosing $r = \lambda R$ and setting $B_{i} \coloneqq B(z_i, \lambda R) \subseteq B$, we get
	\[
	\norm{\rho}_{p, B_i}^{p} \le 128(\lambda \vee R^{-1})\norm{\rho}_{p, B}^{p},
	\]
	which proves \ref{it:shrink.lineardecay}.
	Clearly, we have $B_{i} \subseteq \frac{7}{8}B$ by $\lambda \in (0, 1/8)$ and $z_i \in \closure{B}(x, 3R/4)$.
	Moreover, for any $y \in \frac{1}{8\lambda}B_{i}$,
	\begin{align*}
		d_{G}(x, y) \le d_{G}(x, z_{i}) + d_{G}(z_{i}, y) < \frac{3}{4}R + \frac{1}{8\lambda}(\lambda R) = \frac{7}{8}R,
	\end{align*}
	which proves $\frac{1}{8\lambda}B_{i} \subseteq \frac{7}{8}B$.
	Since $z_{1} \in \closure{B}(x, 3R/8)$, $z_{2} \not\in \closure{B}(x, 5R/8)$ and $\lambda < 1/8$, we have
	\[
	B_{1} \subseteq \frac{1}{2}B \quad \text{and} \quad B_{2} \subseteq \frac{7}{8}B \setminus \frac{1}{2}B,
	\]
	and hence $B_{1} \cap B_{2} = \emptyset$.
	This proves \ref{it:shrink.goodball}.

	The rest of the proof is proving \ref{it:shrink.goodpath} and \ref{it:shrink.next}.

	\noindent \ref{it:shrink.goodpath}
	It is clear that
	\[
	\dist\Bigl(\frac{1}{4}B_1, \frac{1}{4}B_2\Bigr) \le d_{G}(z_1, z_2) \le \frac{3}{8}R + \frac{3}{4}R = \frac{9}{8}R = \frac{9}{2\lambda}\cdot\frac{\lambda R}{4}.
	\]
	Thus \ref{cond.BCL} together with Lemma \ref{lem.shortPath} implies that there exists a path $\theta \in \PATH\bigl(\frac{1}{4}B_1, \frac{1}{4}B_2\bigr)$ satisfying the following conditions.
	\begin{itemize}
		\item [$\bullet$] $\diam \theta \le L_{\BCL}\bigl(\frac{9}{2\lambda}\bigr)R$ and hence $\theta \subseteq \Bigl(L_{\BCL}\bigl(\frac{9}{2\lambda}\bigr) + \frac{7}{8}\Bigr)B = L_{\lambda}B$;
		\item [$\bullet$] $L_{\rho}(\theta) \le 2\bigl[c_{\BCL}\bigl(\frac{9}{2\lambda}\bigr)\bigr]^{-1/p}\left(\frac{\lambda R}{4}\right)^{-\zeta/p}\norm{\rho}_{p, L_{\lambda}B} \eqqcolon C_{p, \lambda}(\lambda R)^{-\zeta/p}\norm{\rho}_{p, L_{\lambda}B}$.
	\end{itemize}

	\noindent \ref{it:shrink.next} 
	Since $B_{i}$ is centered at $F_{i}$ and $B_{i} \subseteq B$, we immediately have $F_{i} \cap \frac{1}{4}B_{i} \neq \emptyset$ and $F_{i} \setminus B_{i} \neq \emptyset$.
	Also, $\theta \cap \frac{1}{4}B_{i} \neq \emptyset$ is clear.
	Since $B_{1} \cap B_{2} = \emptyset$ and $\theta \in \PATH\bigl(\frac{1}{4}B_{1}, \frac{1}{4}B_{2}\bigr)$, we see that $\theta \setminus B_{i} \neq \emptyset$.
	We complete the proof.
\end{proof}

Finally, we shall prove the main result (Theorem \ref{thm.pGCL-gamma}) in this section.

\begin{proof}[Proof of Theorem \ref{thm.pGCL-gamma}]
	Let $\rho \in \ell^{+}(V)$.
	We will construct a $L_{\rho}$-shortcut joining $F_{1}$ and $F_{2}$.
	Let $\lambda \in (0, 1/16)$ be a fixed small parameter that will be chosen later.
	First, we consider the case $R_{0} \ge \lambda^{-1}$.
	Set
	\[
	n_\ast = n_\ast(\lambda, R_0) = \max\{ n \in \mathbb{Z}_{\ge 0} \mid \lambda^{n}R_{0} \ge \lambda^{-1} \} + 1,
	\]
	i.e. $n_\ast \in \mathbb{N}$ is the unique natural number such that
	\begin{equation}\label{eq.n_ast}
		\frac{\log{R_0}}{\log{\lambda^{-1}}} - 1 < n_\ast \le \frac{\log{R_0}}{\log{\lambda^{-1}}}.
	\end{equation}
	Then, for any $n \in \mathbb{Z}_{\ge 0}$ with $n < n_{\ast}$,
	\[
	\lambda \vee (\lambda^{n}R_{0})^{-1} = \lambda \quad \text{ and } \quad \lambda^{n}R_{0} \ge \lambda^{-1} > 16.
	\]
	Pick $x_i \in F_i$ so that $d_{G}(x_1, x_2) = \dist(F_1, F_2)$.
	Then $B_i \coloneqq B(x_i, R_0)$ satisfies $F_i \cap \frac{1}{4}B_{i} \neq \emptyset$ and $F_i \setminus B_i \neq \emptyset$ for each $i = 1, 2$.
	Furthermore, $\dist\bigl(\frac{1}{4}B_{1}, \frac{1}{4}B_{2}\bigr)$ can be estimated as follows:
	If $R_{0} = 2\dist(F_{1}, F_{2})$, then
	\begin{align*}
		\dist\Bigl(\frac{1}{4}B_{1}, \frac{1}{4}B_{2}\Bigr) \le \dist(F_1, F_2)
	= 2 \cdot \frac{R_{0}}{4}.
	\end{align*}
	If $R_{0} \neq 2\dist(F_{1}, F_{2})$ (i.e. $2R_{0} = \diam F_{1} \wedge \diam F_{2}$), then
	\begin{align*}
		\dist\Bigl(\frac{1}{4}B_{1}, \frac{1}{4}B_{2}\Bigr) \le \dist(F_1, F_2)
		= \frac{8\dist(F_1, F_2)}{\diam F_1 \wedge \diam F_2}\cdot\frac{R_0}{4}
		\le 8\kappa_{0}\cdot\frac{R_0}{4}.
	\end{align*}
	By using \ref{cond.BCL} and Lemma \ref{lem.shortPath}, we can find a path $\theta_{\emptyset} \in \PATH\bigl(\frac{1}{4}B_{1}, \frac{1}{4}B_{2}\bigr)$ satisfying the following condition \hyperref[c_1]{\textup{(c$_{1}$)}}.
	\begin{itemize}
		\item [(c)$_{1}$] $\diam\theta_{\emptyset} \le L_{\BCL}(2 \vee 8\kappa_0)R_{0}$ and $L_{\rho}(\theta_{\emptyset}) \le C \cdot c_{\BCL}(2 \vee 8\kappa_0)^{-1/p}R_{0}^{-\zeta/p}\norm{\rho}_{p}$, where $C > 0$ is a constant depending only on $p$ and $\zeta$.  \label{c_1}
	\end{itemize}
	We set $\Theta_{1} \coloneqq \{ \theta_{\emptyset} \}$, $\mathcal{B}_{1} \coloneqq \{ B_{1}, B_{2} \}$, and $\Xi_{1} \coloneqq \{ F_{1}, F_{2} \}$.

	Next we describe an essential step of this proof.
	Set $\Xi_{2} \coloneqq \Xi_{1} \sqcup \{ \theta_{\emptyset} \} = \Xi_{1} \sqcup \Theta_{1}$, and define $F_{11} \coloneqq F_{1}$, $F_{12} = F_{21} \coloneqq \theta_{\emptyset}$ and $F_{22} \coloneqq F_{2}$.
	Then $\Xi_{2} = \{ F_{\tau} \}_{\tau \in \{ 1, 2 \}^{2}}$.
	If $R_0 \ge 16$, then $\theta_{\emptyset}$ is a connected subset of $V$ with $\#\theta_{\emptyset} \ge 2$ and we can apply Lemma \ref{lem.key-shrink} for the triple $(B_{i}, F_{i1}, F_{i2})$.
	(The case $R_{0} < 16$ will be discussed in the next paragraph.)
	Indeed, we have from $R_{0} \le 2\dist(F_{1}, F_{2})$ that $\frac{1}{4}B_{i} \cap B_{j} = \emptyset$ if $\{i, j \} = \{ 1, 2 \}$.
	Combining with $\theta_{\emptyset} \in \PATH\bigl(\frac{1}{4}B_{1}, \frac{1}{4}B_{2}\bigr)$, we verify $\frac{1}{4}B_{i} \cap \theta_{\emptyset} \neq \emptyset$ and $\theta_{\emptyset} \setminus B_{i} \neq \emptyset$ for $i = 1, 2$.
	As a result of Lemma \ref{lem.key-shrink}, we get balls $B_{i1}, B_{i2}$ and paths $\theta_{i}$ satisfying the following conditions \hyperref[a_2]{\textup{(a)$_{2}$}}-\hyperref[d_2]{\textup{(d)$_{2}$}}.
	\begin{itemize}
		\item [(a)$_{2}$] $B_{i1} = B(x_{i1}, \lambda R_0)$, $B_{i2} = B(x_{i2}, \lambda R_0)$ for some $x_{i1} \in F_{i1}$, $x_{i2} \in F_{i2}$ with $x_{i1}, x_{i2} \in \closure{B}(x_{i}, 3R_{0}/4)$ and $d_{G}(x_{i}, x_{i1}) \wedge d_{G}(x_{i}, x_{i2}) \le 3R_{0}/8$. Furthermore, $\frac{1}{8\lambda}B_{i1} \cup \frac{1}{8\lambda}B_{i2} \subseteq \frac{7}{8}B_{i}$ and $B_{i1} \cap B_{i2} = \emptyset$. \label{a_2}
		\item [(b)$_{2}$] $\norm{\rho}_{p, B_{i1}}^{p} \vee \norm{\rho}_{p, B_{i2}}^{p} \le C_{\textup{shr}}\lambda\norm{\rho}_{p, B_{i}}^{p}$, where $C_{\textup{shr}} \coloneqq 128$.
		\item [(c)$_{2}$] $\theta_{i} \in \PATH\bigl(\frac{1}{4}B_{i1}, \frac{1}{4}B_{i2}\bigr)$, $\theta_{i} \subseteq L_{\lambda}B_{i}$, $\diam\theta_{i} \le L_{\lambda}R_{0}$ and
		\[
		L_{\rho}(\theta_{i}) \le C_{p, \lambda}(\lambda R_{0})^{-\zeta/p}\norm{\rho}_{p, L_{\lambda}B_{i}},
		\]
		where $L_{\lambda}$ and $C_{p,\lambda}$ are the constants in Lemma \ref{lem.key-shrink}.
		\item [(d)$_{2}$] For $i, j \in \{ 1, 2 \}$, all of $F_{ij} \cap \frac{1}{4}B_{ij}$, $\theta_{i} \cap \frac{1}{4}B_{ij}$, $F_{ij} \setminus B_{ij}$ and $\theta_{i} \setminus B_{ij}$ are non-empty. \label{d_2}
	\end{itemize}
	We set $\Theta_{2} \coloneqq \{ \theta_{1}, \theta_{2} \}$ and $\mathcal{B}_{2} \coloneqq \{ B_{\tau} \}_{\tau \in \{ 1, 2 \}^{2}}$.
	Thanks to \hyperref[d_2]{\textup{(d)$_{2}$}}, we can use Lemma \ref{lem.key-shrink} for $(B_{ij}, F_{ij1}, F_{ij2})$ when $\lambda R_{0} \ge 16$, where $\{ F_{ij1}, F_{ij2} \} = \{ F_{ij}, \theta_{i} \}$.
	Here, we select $F_{ijk} \, (i, j, k = 1, 2)$ so that $F_{111} = F_{1}$ and $F_{222} = F_{2}$.
	Inductively, for $j = 2, \dots, n_{\ast} + 1$, we can construct a path collection $\Theta_j = \{ \theta_{\omega} \}_{\omega \in \{ 1, 2 \}^{j - 1}}$, a ball collection $\mathcal{B}_{j} = \{ B_{\tau} \}_{\tau \in \{ 1, 2 \}^{j}}$, and a collection of connected sets $\Xi_{j} = \{ F_{\tau} \}_{\tau \in \{ 1, 2 \}^{j}}$ with $F_{ii \cdots i} = F_{i} \, (i = 1, 2)$ subject to the following conditions: for any $\omega = \omega_{1}\cdots\omega_{j - 1} \in \{ 1, 2 \}^{j - 1}$ (i.e. $\omega_{k} \in \{ 1, 2 \}$ for each $k = 1, \dots, j - 1$),
	\begin{enumerate}[\rm(a)$_{j}$]
		\item $B_{\omega 1} = B(x_{\omega 1}, \lambda^{j - 1}R_{0})$ and $B_{\omega 2} = B(x_{\omega 2}, \lambda^{j - 1}R_{0})$ for some $x_{\omega 1} \in F_{\omega 1}$, $x_{\omega 2} \in F_{\omega 2}$ with $x_{\omega 1}, x_{\omega 2} \in \closure{B}(x_{\omega}, 3\lambda^{j - 2}R_{0}/4)$ and $d_{G}(x_{\omega}, x_{\omega 1}) \wedge d_{G}(x_{\omega}, x_{\omega 2}) \le 3\lambda^{j - 2}R_{0}/8$. Furthermore, $\frac{1}{8\lambda}B_{\omega 1} \cup \frac{1}{8\lambda}B_{\omega 2} \subseteq \frac{7}{8}B_{\omega}$ and $B_{\omega 1} \cap B_{\omega 2} = \emptyset$. \label{a_j}
		\item $\norm{\rho}_{p, B_{\omega 1}}^{p} \vee \norm{\rho}_{p, B_{\omega 2}}^{p} \le C_{\textup{shr}}\lambda\norm{\rho}_{p, B_{\omega}}^{p}$. \label{b_j}
		\item $\theta_{\omega} \in \PATH\bigl(\frac{1}{4}B_{\omega 1}, \frac{1}{4}B_{\omega 2}\bigr)$, $\theta_{\omega} \subseteq L_{\lambda}B_{\omega}$, $\diam\theta_{\omega} \le L_{\lambda}\lambda^{j - 2}R_{0}$ and \label{c_j}
		\[
		L_{\rho}(\theta_{\omega}) \le C_{p, \lambda}(\lambda^{j - 1}R_{0})^{-\zeta/p}\norm{\rho}_{p, L_{\lambda}B_\omega}.
		\]
		\item For $i \in \{ 1, 2 \}$, all of $F_{\omega i} \cap \frac{1}{4}B_{\omega i}$, $\theta_{\omega} \cap \frac{1}{4}B_{\omega i}$, $F_{\omega i} \setminus B_{\omega i}$ and $\theta_{\omega} \setminus B_{\omega i}$ are non-empty. \label{d_j}
	\end{enumerate}
	Note that a combination of \hyperref[a_j]{\textup{(a)$_{j}$}} and \hyperref[c_j]{\textup{(c)$_{j}$}} implies $\bigcup_{w \in \{ 1, 2 \}^{j - 1}}\theta_{\omega} \subseteq L_{\lambda}B_{1} \cup L_{\lambda}B_{2}$.
	Indeed, for $j \in \{ 1, \dots, n_{\ast} + 1 \}$, $\omega = \omega_{1}\cdots\omega_{j} \in \{ 1, 2 \}^{j}$ and $y \in L_{\lambda}B_{\omega}$, we have from \hyperref[a_j]{\textup{(a)$_{j}$}} that
	\begin{align*}
		d_{G}(x_{\omega_{1}\cdots\omega_{j - 1}}, y)
		&\le d_{G}(x_{\omega_{1}\cdots\omega_{j - 1}}, x_{\omega}) + d_{G}(x_{\omega}, y) \\
		&< \frac{3}{4}\lambda^{j - 2}R_{0} + \lambda^{j - 1}L_{\lambda}R_{0} \le \left(\frac{3}{4} + \frac{L_{\lambda}}{8}\right)\lambda^{j - 2}R_{0} < L_{\lambda}\lambda^{j - 2}R_{0},
	\end{align*}
	where we used $L_{\lambda} \ge \frac{7}{8} > \frac{6}{7}$ in the last inequality.
	Combining with \hyperref[c_j]{\textup{(c)$_{j}$}}, we obtain
	\[
	\theta_{\omega} \subseteq L_{\lambda}B_{\omega} \subseteq L_{\lambda}B_{\omega_{1}\cdots\omega_{j - 1}}.
	\]
	Hence we conclude that $\theta_{\omega} \subseteq L_{\lambda}B_{\omega_{1}} \subseteq L_{\lambda}B_{1} \cup L_{\lambda}B_{2}$.

	Next we will fill ``gaps'' between $\theta_{\omega}$ and the center of $\frac{1}{4}B_{\omega i}$ for each $\omega \in \{ 1, 2 \}^{n_{\ast}}$ and $i = 1, 2$.
	Since $G$ is connected, we can find a (shortest) path $\widetilde{\theta}_{\omega i} \in \PATH(\theta_{\omega}, x_{\omega i})$ such that $\widetilde{\theta}_{\omega i} \subseteq B_{\omega i}$ and $\len\bigl(\widetilde{\theta}_{\omega i}\bigr) \le \lambda^{n_{\ast}}R_{0}/4 < (4\lambda)^{-1}$.
	By H\"{o}lder's inequality, we also have
	\begin{equation}\label{fill-gap}
		L_{\rho}\bigl(\widetilde{\theta}_{\omega i}\bigr)
		\le \len\bigl(\widetilde{\theta}_{\omega i}\bigr)^{(p - 1)/p}\norm{\rho}_{p, B_{\omega i}}
		\le \left(\frac{1}{4\lambda}\right)^{(p - 1)/p}\norm{\rho}_{p, B_{\omega i}}.
	\end{equation}
	Concatenating paths $\{ \theta_{\omega} \mid \omega \in \{ 1, 2 \}^{j}, j = 0, \dots, n_{\ast}  \}$ and $\bigl\{\widetilde{\theta}_{\tau} \mid \tau \in \{ 1, 2 \}^{n_{\ast} + 1} \bigr\}$ in a suitable way, we can obtain a path $\theta_{\ast}$ satisfying the following conditions \eqref{concat-behav}-\eqref{concat-Lrho}:
	\begin{equation}\label{concat-behav}
		\text{$\theta_{\ast} \in \PATH(F_{1}, F_{2})$ with $\theta_{\ast} \subseteq L_{\lambda}B_{1} \cup L_{\lambda}B_{2} \cup \theta_{\emptyset}$,}
	\end{equation}
	\begin{equation}\label{concat-diam}
		\diam \theta_{\ast} \le \diam \theta_{\emptyset} + \sum_{j = 1}^{n^{\ast}}\sum_{\omega \in \{ 1, 2 \}^{j}}\diam \theta_{\omega} + \sum_{\omega \in \{ 1, 2 \}^{n_{\ast}}}\Bigl(\len\bigl(\widetilde{\theta}_{\omega 1}\bigr) + \len\bigl(\widetilde{\theta}_{\omega 2}\bigr)\Bigr),
	\end{equation}
	and
	\begin{equation}\label{concat-Lrho}
		L_{\rho}(\theta_{\ast}) \le \sum_{\omega \in \{ 1, 2 \}^{n_{\ast}}}L_{\rho}(\theta_{\omega}) + \sum_{\tau \in \{ 1, 2 \}^{n_{\ast} + 1}}L_{\rho}\bigl(\widetilde{\theta}_{\tau}\bigr).
	\end{equation}
	From \eqref{concat-diam} and \hyperref[c_j]{\textup{(c)$_{j}$}}, we can give an upper bound for $\diam \theta_{\ast}$ as follows:
	\begin{align}\label{gBCL.length}
		\diam\theta_{\ast}
		&\le \left(L_{\BCL}(\kappa_{0}) + L_{\lambda}\sum_{j = 1}^{n^{\ast}}2^{j}\lambda^{j - 1}\right)R_{0} + \frac{2^{n_{\ast}}}{2\lambda} \nonumber \\
		&\le \left(L_{\BCL}(\kappa_{0}) + \frac{2}{1 - 2\lambda}L_{\lambda} + \frac{1}{2\lambda}\right)R_{0}
		\eqqcolon LR_{0}.
	\end{align}
	We will give an upper bound on $L_{\rho}(\theta_{\ast})$ by using \eqref{concat-Lrho}.
	We start by introducing
	\[
	l_{\ast} = l_{\ast}(\lambda, L_{\lambda}) \coloneqq \max\bigl\{ l \le n_{\ast} \bigm| (8\lambda)^{-l} \le L_{\lambda} \bigr\}.
	\]
	Here, if $\{ l \le n_{\ast} \mid (8\lambda)^{-l} \le L_{\lambda} \} = \emptyset$, then we define $l_{\ast}$ as $0$.
	By \eqref{concat-Lrho}, we have
	\begin{equation}\label{shor-cut.divided}
		L_{\rho}(\theta_{\ast}) \le \mathcal{L}_{1} + \mathcal{L}_{2} + \mathcal{L}_{3}.
	\end{equation}
	where
	\[
	\mathcal{L}_{1} \coloneqq \sum_{j= 0}^{l_{\ast}}\sum_{\omega \in \{ 1, 2 \}^{j}}L_{\rho}(\theta_{\omega}), \quad
	\mathcal{L}_{2} \coloneqq \sum_{j= l_{\ast} + 1}^{n_{\ast}}\sum_{\omega \in \{ 1, 2 \}^{j}}L_{\rho}(\theta_{\omega}), \quad \text{and} \quad
	\mathcal{L}_{3} \coloneqq \sum_{\tau \in \{ 1, 2 \}^{n_{\ast} + 1}}L_{\rho}\bigl(\widetilde{\theta}_{\tau}\bigr)
	\]
	Each term can be estimated as follows.

\noindent
\underline{\textbf{The first term $\mathcal{L}_{1}$}.}\,
For any $j \in \{ 1, \dots, l_{\ast} \}$ and $\omega \in \{ 1, 2 \}^{j}$, by \hyperref[c_j]{\textup{(c)$_{j + 1}$}},
\begin{align*}
	L_{\rho}(\theta_{\omega})
	\le C_{p, \lambda}(\lambda^{j}R_{0})^{-\zeta/p}\norm{\rho}_{p, L_{\lambda}B_{\omega}}
	\le C_{p, \lambda}(\lambda^{j}R_{0})^{-\zeta/p}\norm{\rho}_{p}.
\end{align*}
Combining with \hyperref[c_1]{\textup{(c)$_{1}$}}, we see that
\begin{align}
	\mathcal{L}_{1} = \sum_{j= 0}^{l_{\ast}}\sum_{\omega \in \{ 1, 2 \}^{j}}L_{\rho}(\theta_{\omega})
	&\le L_{\rho}(\theta_{\emptyset}) + \sum_{j= 1}^{l_{\ast}}\sum_{\omega \in \{ 1, 2 \}^{j}}L_{\rho}(\theta_{\omega}) \nonumber \\
	&\le \Bigl(c(\kappa_{0})^{-1/p} + C_{p, \lambda}\sum_{j = 1}^{l_{\ast}}2^{j}\lambda^{-j\zeta/p} \Bigr)R_{0}^{-\zeta/p}\norm{\rho}_{p}.
\end{align}
Moreover, if we suppose $\lambda < C_{\textup{shr}}^{-1} (= 128^{-1})$, then since $(C_{\textup{shr}}\lambda)^{(j - l_{\ast} - 1)/p} \ge 1$ for $j \le l_{\ast}$
\begin{equation}\label{short-cut.I}
	\mathcal{L}_{1} \le \Bigl(c_{\BCL}(\kappa_{0})^{-1/p} + C_{p, \lambda}\sum_{j = 1}^{l_{\ast}}2^{j}\lambda^{-j\zeta/p}(C_{\textup{shr}}\lambda)^{(j - l_{\ast} - 1)/p}\Bigr)R_{0}^{-\zeta/p}\norm{\rho}_{p}.
\end{equation}

\noindent
\underline{\textbf{The second term $\mathcal{L}_{2}$}.}\,
Note that $L_{\lambda} \le (8\lambda)^{-l_{\ast}}$.
For any $j \in \{ l_{\ast} + 1, \dots, n_{\ast} \}$, $k \in \{ 1, \dots, j - 1 \}$ and $\omega = \omega_{1}\cdots\omega_{j} \in \{ 1, 2 \}^{j}$, define $[\omega]_{-k} = \omega_{1} \cdots \omega_{j - k} \in \{ 1, 2 \}^{j - k}$.
From \hyperref[a_j]{\textup{(a)$_{j}$}}, we observe that
\[
L_{\lambda}B_{\omega} \subseteq (8\lambda)^{-l_{\ast}} B_{\omega} \subseteq (8\lambda)^{-l_{\ast} + 1} B_{[\omega]_{-1}} \cdots \subseteq B_{[\omega]_{-l_{\ast}}}.
\]
By using \hyperref[b_j]{\textup{(b)$_{j}$}} repeatedly, we obtain
\begin{align}\label{lp-shrink}
	\norm{\rho}_{p, B_{[\omega]_{-l_{\ast}}}}
	\le (C_{\textup{shr}}\lambda)^{1/p}\norm{\rho}_{p, B_{[\omega]_{-l_{\ast} - 1}}}
	\le \cdots
	&\le (C_{\textup{shr}}\lambda)^{(j - l_{\ast} - 1)/p}\norm{\rho}_{p, B_{[\omega]_{-j + 1}}} \nonumber \\
	&\le (C_{\textup{shr}}\lambda)^{(j - l_{\ast} - 1)/p}\norm{\rho}_{p}.
\end{align}
Therefore, by \hyperref[c_j]{\textup{(c)$_{j + 1}$}}, we have
\begin{align*}
	L_{\rho}(\theta_{\omega})
	\le C_{p, \lambda}(\lambda^{j}R_{0})^{-\zeta/p}\norm{\rho}_{p, L_{\lambda}B_{\omega}}
	\le C_{p, \lambda}\lambda^{-j\zeta/p}(C_{\textup{shr}}\lambda)^{(j - l_{\ast} - 1)/p}R_{0}^{-\zeta/p}\norm{\rho}_{p},
\end{align*}
and hence
\begin{equation}\label{short-cut.II}
	\mathcal{L}_{2}
	= \sum_{j= l_{\ast} + 1}^{n_{\ast}}\sum_{\omega \in \{ 1, 2 \}^{j}}L_{\rho}(\theta_{\omega})
	\le C_{p, \lambda}\Biggl(\sum_{j = l_{\ast} + 1}^{n_{\ast}}2^{j}\lambda^{-j\zeta/p}(C_{\textup{shr}}\lambda)^{(j - l_{\ast} - 1)/p}\Biggr)R_{0}^{-\zeta/p}\norm{\rho}_{p}.
\end{equation}

\noindent
\underline{\textbf{The third term $\mathcal{L}_{3}$}.}\,
By \eqref{fill-gap} and the same argument to obtain \eqref{lp-shrink}, we have
\begin{equation}\label{fill-gap.2}
	\mathcal{L}_{3}
	= \sum_{\tau \in \{ 1, 2 \}^{n_{\ast} + 1}}L_{\rho}\bigl(\widetilde{\theta}_{\tau}\bigr)
	\le 2\left(\frac{1}{4\lambda}\right)^{(p - 1)/p}(2^{p}C_{\textup{shr}}\lambda)^{n_{\ast}/p}\norm{\rho}_{p}.
\end{equation}
Recall that we suppose $\zeta < 1$.
Pick $\delta \in (\zeta, 1)$ and define $n_{\ast}'$ as the unique non-negative integer such that
\[
\frac{\log{R_0}}{\log{(2^{p}C_{\textup{shr}}\lambda)^{-1}}} - 1 < n_{\ast}' \le \frac{\log{R_0}}{\log{(2^{p}C_{\textup{shr}}\lambda)^{-1}}}.
\]
We now suppose $\lambda \le (2^{p}C_{\textup{shr}})^{-1/(1 - \delta)}$.
Then $\log{\lambda^{-1}} \le \log{(2^{p}C_{\textup{shr}}\lambda)^{-1/\delta}}$ and hence
\begin{align*}
	\delta n_{\ast}' \le \frac{\log{R_{0}}}{\log{\lambda^{-1}}} \le n_{\ast} + 1.
\end{align*}
Therefore, we have
\begin{equation}\label{abort}
	(2^{p}C_{\textup{shr}}\lambda)^{n_{\ast}}
	\le (2^{p}C_{\textup{shr}}\lambda)^{\delta n_{\ast}' - 1}
	\le (2^{p}C_{\textup{shr}}\lambda)^{-1}R_{0}^{-\delta}
	\le (2^{p}C_{\textup{shr}}\lambda)^{-1}R_{0}^{-\zeta}.
\end{equation}
Combining \eqref{fill-gap.2} and \eqref{abort}, we obtain
\begin{equation}\label{short-cut.III}
	\mathcal{L}_{3} \le \widetilde{C}_{p, \lambda}R_{0}^{-\zeta/p}\norm{\rho}_{p},
\end{equation}
where $\widetilde{C}_{p, \lambda} \coloneqq 2(1/4\lambda)^{(p - 1)/p}(2^{p}C_{\textup{shr}}\lambda)^{-1/p}$ that depends only on $p, \lambda$ and $\deg(G)$.

Consequently, if we fix $\delta \in (\zeta, 1)$ and $\lambda < (2^{p}C_{\textup{shr}})^{-1/(1 - \delta)}$ (e.g. $\lambda = \frac{1}{2}(2^{p}C_{\textup{shr}})^{-1/(1 - \delta)}$), then  \eqref{short-cut.I}, \eqref{short-cut.II} and \eqref{short-cut.III} imply that
\begin{equation}\label{gBCL.rho-len}
	L_{\rho}(\theta_{\ast}) \le C_{\ast}R_{0}^{-\zeta/p}\norm{\rho}_{p},
\end{equation}
where
\[
C_{\ast} \coloneqq c(\kappa_{0})^{-1/p} +  C_{p, \lambda}(C_{\textup{shr}}\lambda)^{-(l_{\ast} + 1)/p}\sum_{j = 0}^{+\infty}\bigl(2\cdot C_{\textup{shr}}^{1/p} \cdot \lambda^{(1 - \zeta)/p}\bigr)^{j} + \widetilde{C}_{p, \lambda}.
\]
By $\lambda < (2^{p}C_{\textup{shr}})^{-1/(1 - \delta)}$, we have $2\cdot C_{\textup{shr}}^{1/p} \cdot \lambda^{(1 - \zeta)/p} < 1$ and, by $C_{\textup{shr}} = 128$,
\[
(C_{\textup{shr}}\lambda)^{-l_{\ast}} \le (8\lambda)^{-l_{\ast}} \le L_{\lambda}.
\]
Hence $C_{\ast} \le c^{-1/p}$ if we put
\[
c \coloneqq \left(c_{\BCL}(\kappa_{0})^{-1/p} +  C_{p, \lambda}L_{\lambda}^{-1/p}(C_{\textup{shr}}\lambda)^{-1/p}\sum_{j = 0}^{+\infty}\bigl(2\cdot C_{\textup{shr}}^{1/p} \cdot \lambda^{(1 - \zeta)/p}\bigr)^{j} + \widetilde{C}_{p, \lambda}\right)^{-p}.
\]
We conclude from Lemma \ref{lem.shortPath}, \eqref{gBCL.length} and \eqref{gBCL.rho-len} that
\[
\MOD_{p}^{G}(\{ \theta \in \PATH(F_{1}, F_{2}) \mid \diam\theta \le LR_{0} \}) \ge cR_{0}^{\zeta},
\]
which finishes the proof when $R_{0} \ge \lambda^{-1}$.

Lastly, we consider the case $R_{0} < \lambda^{-1}$.
If $R_0 = 0$ (i.e. $\#F_0 = 1$ or $\#F_1 = 1$), then the required estimate is trivial.
So we can assume that $R_0 \ge 1$ and $\#F_i \ge 2$.
Since $\dist(F_0, F_1) \le (2^{-1} \wedge 2\kappa_0)R_{0} \eqqcolon c_{1}R_{0}$, there exists a shortest path $\theta_{\ast} \in \PATH(F_0, F_1)$ such that $\len(\theta_{\ast}) \le c_{1}R_{0}$.
Applying H\"{o}lder's inequality, we obtain
\[
L_{\rho}(\theta_{\ast}) \le \len(\theta_{\ast})^{(p - 1)/p}\norm{\rho}_{p, \theta_{\ast}} \le c_{1}^{(p - 1)/p}R_{0}^{(p - 1)/p}\norm{\rho}_{p}.
\]
Since $\zeta \ge 1 - p$ and $1 \le R_{0} < \lambda^{-1}$, we see that
\[
R_{0}^{(p - 1)/p} = R_{0}^{-\zeta/p}R_{0}^{(\zeta + p - 1)/p} \le \lambda^{-(\zeta + p - 1)/p}R_{0}^{-\zeta/p}.
\]
By Lemma \ref{lem.shortPath}, we obtain
\[
\MOD_{p}^{G}(\{ \theta \in \PATH(F_{1}, F_{2}) \mid \diam\theta \le LR_{0} \}) \ge cR_{0}^{\zeta},
\]
where $c > 0$ depends only on $p$, $\zeta$, $c_1$ and $\lambda$.
\end{proof}

We also frequently use the following consequence of Theorem \ref{thm.pGCL-gamma}.
\begin{cor}\label{cor.pGCL-gamma.useful}
	Assume that $G$ is bounded degree graph that satisfies $p$-combinatorial ball Loewner condition \ref{assum.BCL} with exponent $\zeta \in [1-p,1)$.
	There exist constants $c > 0$ and $L \ge 1$ depending only on the constants associated with the assumptions
	such that if $F_i\, (i = 1, 2)$ are connected subsets of $V$ satisfying $\#F_{i} \ge 2$, $F_i \cap B \neq \emptyset$ and $F_{i} \setminus 4B \neq \emptyset$ for some ball $B$ with radius $R > 0$, then
	\begin{equation}\label{pGCL-useful}
		\MOD_{p}^{G}(F_{1}, F_{2}; 4LB) \ge c(R \vee 1)^{\zeta}.
	\end{equation}
\end{cor}
\begin{proof}
	We first consider the case $R \ge 2$.
	Notice that $V \setminus 4B \neq \emptyset$.
	Since $F_{i}$ is connected, we can find a connected subset $\widetilde{F}_{i}$ of $F_{i}$ satisfying the following conditions (i)-(iii):
	\begin{enumerate}[(i)]
		\item $\widetilde{F}_{1} \subseteq F_{1} \cap \bigl(\closure{2B} \setminus B\bigr)$ and $\widetilde{F}_{2} \subseteq F_{2} \cap \bigl(\closure{4B} \setminus 3B\bigr)$.
		\item $\widetilde{F}_{1} \cap \closure{B} \neq \emptyset$ and $\widetilde{F}_{2} \cap \closure{3B} \neq \emptyset$.
		\item $\widetilde{F}_{1} \setminus 2B \neq \emptyset$ and $\widetilde{F}_{2} \setminus 4B \neq \emptyset$.
	\end{enumerate}
	Then we immediately see that $3R \ge \diam\widetilde{F}_{1} \ge \diam\widetilde{F}_{2} = \lceil 4R \rceil - \lceil 3R \rceil \ge \frac{1}{2}R$ and
	\[
	8R \ge \dist\bigl(\widetilde{F}_{1}, \widetilde{F}_{2}\bigr) \ge \lceil 3R \rceil - \lceil 2R \rceil \ge \frac{1}{2}R.
	\]
	Hence, by applying Theorem \ref{thm.pGCL-gamma} to $\widetilde{F}_{i}$, there exist $c, L > 0$ (depending only on the constants associated with the assumptions) such that
	\[
	\MOD_{p}\Bigl(\bigl\{ \theta \in \PATH\bigl(\widetilde{F}_{1}, \widetilde{F}_{2}\bigr) \bigm| \diam\theta \le LR \bigr\}\Bigr) \ge cR^{\zeta}.
	\]
	By Lemma \ref{lem.basic-pMod}\ref{it:mod.mono},
	\[
	\MOD_{p}^{G}(\{ \theta \in \PATH(F_{1}, F_{2}) \bigm| \theta \subseteq (L + 1)B \}) \ge \MOD_{p}^{G}\Bigl(\bigl\{ \theta \in \PATH\bigl(\widetilde{F}_{1}, \widetilde{F}_{2}\bigr) \bigm| \diam\theta \le LR \bigr\}\Bigr),
	\]
	which implies our assertion in this case.

	Next we consider the case $R \le 2$.
	Let $L > 0$ be the same as in the previous paragraph.
	Then, by \eqref{smallscale} in Lemma \ref{lem.shortPath}, we have
	\begin{align*}
		&\MOD_{p}^{G}(\{ \theta \in \PATH(F_{1}, F_{2}) \bigm| \theta \subseteq (L + 4)B \}) \\
		&\ge \MOD_{p}^{G}(\{ \theta \in \PATH(F_{1}, F_{2}) \bigm| \text{$\theta$ is a shortest path} \}) \\
		&\ge 4^{1 - p} = 4^{1 - p}(R \vee 1)^{-\zeta} \cdot (R \vee 1)^{\zeta} \ge 4^{1 - p}\bigl(2^{-1} \wedge 1\bigr)(R \vee 1)^{\zeta},
	\end{align*}
	where we used $(R \vee 1)^{-\zeta} \ge (R \vee 1)^{-1} \wedge 1^{p - 1}$ and $R \le 2$ in the last inequality.
\end{proof}

\section{Discrete \texorpdfstring{$(p, p)$-Poincar\'{e}}{(p,p)-Poincare} inequality}\label{sec.PI}
Throughout this section, let $p \ge 1$ and let $G = (V, E)$ be a locally finite connected simple non-directed graph.

The goal of this section is to show that the `low-dimensional' $p$-ball combinatorial Loewner type property \ref{assum.BCL} implies a Poincar\'e inequality.
We shall give the definition of (weak) $(p, p)$-Poincar\'{e} inequality in our setting.
\begin{defn}\label{dfn.PI}
	Let $\beta > 0$.
	A graph $G$ satisfies \emph{$(p, p)$-Poincar\'{e} inequality of order $\beta$} (\ref{cond.PI} for short) if there exist constants $C_{\textup{PI}}, A_{\textup{PI}} \ge 1$ such that for any $x \in V$, $R \ge 1$ and $f \in \mathbb{R}^{V}$,
	\begin{equation}\label{cond.PI}
		\sum_{y \in B(x, R)}\abs{f(y) - f_{B(x, R)}}^{p} \le C_{\textup{PI}}R^{\beta}\mathcal{E}_{p, B(x, A_{\textup{PI}}R)}^{G}(f). \tag{\textup{PI$_{p}(\beta)$}}
	\end{equation}
\end{defn}
The main result in this section (Theorem \ref{thm.PI-discrete}) tells us that the \emph{$(p, p)$-Poincar\'{e} inequality} follows from the the combinatorial ball Loewner-type property \ref{assum.BCL} and \ref{cond.VD}. This result and its proof are inspired by a similar theorem of Heinonen and Koskela \cite[Theorem 5.12]{HK98}. Although the result in \cite{HK98}  corresponds to the case $\zeta=0$ the proof there works when $\zeta<1$.
\begin{thm}\label{thm.PI-discrete}
	Let $G = (V, E)$ be a graph satisfying \ref{VD.growth} and \ref{assum.BCL}, where $\alpha \ge 1$ and $\zeta \in [1-p,1)$.
	Then $G$ satisfies \ref{cond.PI}, where $\beta = \alpha - \zeta$, $A_{\textup{PI}} = 2$ and $C_{\textup{PI}}$ depends only on the constants associated with the assumptions.
\end{thm}
The proof of Theorem \ref{thm.PI-discrete} is done in two steps. In the first step, we introduce a two-point estimate that is a sufficient condition for the Poincar\'e inequality (see Definition \ref{dfn.TPE} and Lemma \ref{lem.TP-PI}). In the second step, we show that the combinatorial ball Loewner-type property \ref{assum.BCL} implies the two-point estimate (Lemma \ref{lem.TP}).

\subsection{Equivalence with two-point estimates}
We introduce a two-point estimate and show that it is equivalent to the Poincar\'e inequality.
For $f \in \mathbb{R}^{V}$ and $R \ge 1$, we define
\[
M_{R}^{p}[f](x) = \max_{0 < r < R}\frac{\mathcal{E}_{p, B(x, r)}^{G}(f)}{\#B(x, r)}, \quad x \in V.
\]
The function $M_{R}^{p}[f]$ is the \emph{truncated maximal function of the gradient of $f$} in our setting.  Perhaps, the notation  $M_{R} (\abs{\nabla f}^{p})$  is more appropriate but we choose the above notation for brevity.
The following definition gives a discrete generalization of pointwise estimates (see \cite[(15)]{HK00} or \cite[(5.16)]{HK98} for example).

\begin{defn}\label{dfn.TPE}
	Let $\beta > 0$.
	The graph $G$ satisfies the \emph{$p$-two-point estimate of order $\beta$} (\ref{cond.TP} for short) if there exists a constant $C_{\textup{TP}} > 0$ such that for any $z \in V$, $R \ge 1$, $f \in \mathbb{R}^{V}$ and $x, y \in B(z, C_{\textup{TP}}^{-1}R)$,
	\begin{equation}\label{cond.TP}
		\abs{f(x) - f(y)}^{p} \le C_{\textup{TP}}R^{\beta}\bigl(M_{R}^{p}[f](x) + M_{R}^{p}[f](y)\bigr). \tag{\textup{TP$_{p}$($\beta$)}}
	\end{equation}
\end{defn}

It is easy to see that \ref{VD.growth}, where $\alpha \ge 1$, implies \hyperref[cond.TP]{TP$_{p}$($\alpha + p - 1$)}.
\begin{lem}\label{lem.TP-short}
	Suppose that $G$ satisfies \ref{VD.growth} for some $\alpha \ge 1$.
	Then $G$ satisfies \hyperref[cond.TP]{\textup{TP$_{p}$($\alpha + p - 1$)}} with $C_{\mathrm{TP}} > 1$ depending only on $\alpha, C_{\mathrm{D}}$.
\end{lem}
\begin{proof}
	Let $C > 1$ that will be chosen later.
	Let $[z_{0}, z_{1}, \dots, z_{l}]$ be a shortest path in $G$ such that $z_{0} = x$ and $z_{l} = y$.
	Note that $l = d_{G}(x, y) \le R_{\ast}$ and $z_{i} \in \closure{B}(y, R_{\ast})$.
	If $C \ge 2$, then by H\"{o}lder's inequality, we have
		\begin{align*}
			\abs{f(x) - f(y)}^{p}
			\le l^{p - 1}\sum_{i = 0}^{l - 1}\abs{f(z_{i}) - f(z_{i + 1})}^{p}
			\le R^{p - 1}\mathcal{E}_{p, \closure{B}(y, R/2)}^{G}(f).
		\end{align*}
	Thus \hyperref[cond.TP]{\textup{TP$_{p}$($\alpha + p - 1$)}} follows by using \ref{VD.growth}.
\end{proof}

A well-known telescoping sum argument shows that Poincar\'e inequality implies the two point estimate. This follows from a straightforward modification of the proof of \cite[Lemma 5.15]{HK98} or a discrete version of that argument in  the special case $p=2$ in \cite[Lemma 2.4]{Mur20}. We omit the proof as we will not use the lemma below.
\begin{lem}
	Let $G = (V, E)$ be a graph satisfying \ref{cond.VD} and \ref{cond.PI} for some $\beta > 0$.
	Then $G$ satisfies \ref{cond.TP}.
\end{lem}

The following lemma is a converse of the previous lemma.
Let us recall the notion of \emph{median}.
For $f \in \mathbb{R}^{V}$ and $A \subseteq V$, a median of $f$ on $A$ is a number $a \in \mathbb{R}$ such that
\[
\#\{ w \in A \mid f(z) \ge a \} \wedge \#\{ w \in A \mid f(z) \le a \} \ge \frac{1}{2}\#A.
\]
We write $\mathsf{med}(f, A)$ to denote the set of medians of $f$ on $A$.
(It is easy to show that $\mathsf{med}(f, A) \neq \emptyset$.)

\begin{lem}\label{lem.TP-PI}
	Let $G = (V, E)$ be a graph satisfying \ref{cond.VD} and \ref{cond.TP} for some $\beta > 0$.
	Then there exist constants $C > 0$ and $A > 0$ depending only on $p, C_{\textup{D}}, \deg(G)$, $C_{\textup{TP}}$ such that
	\begin{equation}\label{eq.PI-median}
		\sum_{B(x, R)}\abs{f - a}^{p} \le CR^{\beta}\mathcal{E}_{p, B(x, AR)}^{G}(f),
	\end{equation}
	for any $x \in V$, $R \ge 1$, $f \in \mathbb{R}^{V}$, $a \in \mathsf{med}\bigl(f, B(x, R)\bigr)$.
	In particular, $G$ satisfies \ref{cond.PI} by Lemma \ref{lem.p-var}.
\end{lem}
\begin{proof}
	The proof of \cite[Lemma 5.15]{HK98} applies to our setting with minor modifications.
	For the reader's convenience, we give a proof.

	Let $z \in V$, let $R \ge 1$ and let $f \in \mathbb{R}^{V}$.
	Set $B \coloneqq B(z, C_{\textup{TP}}^{-1}R)$ and fix $a \in \mathsf{med}(f, B)$.
	By considering $f - a$ instead of $f$, we can assume that $a = 0$, i.e.
	\begin{equation}\label{homogenize}
		\#\{ z \in B \mid f(z) \ge 0 \} \wedge \#\{ z \in B \mid f(z) \le 0 \} \ge \frac{1}{2}\#B.
	\end{equation}
	Let $s > 0$.
	Suppose that $x, y \in B$ satisfy
	\begin{equation}\label{weak.pair}
		f(x) \ge s \text{ and } f(y) \le 0 \qquad (\text{resp. $f(x) \le -s$ and $f(y) \ge 0$}).
	\end{equation}
	We choose $w \in \{ x, y \}$ so that $M_{R}^{p}[f](w) = M_{R}^{p}[f](x) \vee M_{R}^{p}[f](y)$.
	Then there exists $R_{\ast} = R_{\ast}(w) \in (0, R) \cap \mathbb{N}$ such that
	\begin{equation*}
		M_{R}^{p}[f](x) + M_{R}^{p}[f](y) \le 2\frac{\mathcal{E}_{p, B(w, R_{\ast})}^{G}(f)}{\#B(w, R_{\ast})}.
	\end{equation*}
	By \ref{cond.TP}, we have
	\begin{equation*}
		s \le \abs{f(x) - f(y)} \le 2^{1/p}C_{\textup{TP}}^{1/p}R^{\beta/p}\left(\frac{\mathcal{E}_{p, B(w, R_{\ast})}^{G}(f)}{\#B(w, R_{\ast})}\right)^{1/p},
	\end{equation*}
	which is equivalent to
	\begin{equation}\label{weak.basic}
		\#B(w, R_{\ast})
		\le C_{1}s^{-p}R^{\beta}\mathcal{E}_{p, B(w, R_{\ast})}^{G}(f),
	\end{equation}
	where $C_{1} \coloneqq 2C_{\textup{TP}}$.

	Next we prove the following weak $L^p$-type estimates: for any $s \in \bigl(0, \norm{f}_{\ell^{\infty}(V)}\bigr] \cap \mathbb{R}$,
	\begin{equation}\label{weakLp}
		\#\bigl(B \cap \{ \abs{f} \ge s \}\bigr) \le C_{2}R^{\beta}s^{-p}\mathcal{E}_{p, B(z, (C_{\textup{TP}}^{-1} + 1)R)}^{G}(f),
	\end{equation}
	where $C_{2} > 0$ is a constant depending only on $C_{\textup{TP}}$ and $C_{\textup{D}}$.
	The proof will be divided into the following two cases.

	\noindent
	\textbf{Case 1:} Consider the case where there exists $x \in \{ f \ge s \} \cap B$ (resp. $x \in \{f \le -s \} \cap B$) such that, for any $y \in \{ f \le 0 \} \cap B$ (resp. $y \in \{f \ge 0 \} \cap B$),
	\[
	M_{R}^{p}[f](y) = M_{R}^{p}[f](x) \vee M_{R}^{p}[f](y).
	\]
	In this case, by applying the basic covering lemma (Lemma \ref{lem.3B-appendix}) to $\{ B(y, R_{\ast}(y)) \}_{y \in \{ f \le 0 \} \cap B}$ (resp. $\{ B(y, R_{\ast}(y)) \}_{y \in \{ f \ge 0 \} \cap B}$), we obtain $N \in \mathbb{N}$ and $\{ y_{i} \}_{i = 1}^{N} \subseteq \{ f \le 0 \} \cap B$ (resp. $\{ y_{i} \}_{i = 1}^{N} \subseteq \{ f \ge 0 \} \cap B$) such that
	\begin{equation*}
		\closure{B}(y_{i}, R_{i}) \cap \closure{B}(y_{j}, R_{j}) = \emptyset \quad\text{if $i \neq j$,}
	\end{equation*}
	and
	\begin{equation*}
		\{ f \le 0 \} \cap B \subseteq \bigcup_{i = 1}^{N}B(y_{i}, 4R_{i}) \quad\text{(resp. $\{ f \ge 0 \} \cap B \subseteq \bigcup_{i = 1}^{N}B(y_{i}, 4R_{i})$),}
	\end{equation*}
	where $R_{i} \coloneqq R_{\ast}(y_{i})$ for each $i \in \{ 1, \dots, N \}$.
	Using \ref{cond.VD} and \eqref{weak.basic}, we see that
	\begin{align}\label{weak.cov}
		\#\bigl(\{ f \le 0 \} \cap B\bigr)
		&\le \sum_{i = 1}^{N}\#B(y_{i}, 4R_{i})
		\le C_{1}C_{\textup{D}}^{2}\sum_{i = 1}^{N}\#\closure{B}(y_{i}, R_{i}) \nonumber\\
		&\le C_{1}C_{\textup{D}}^{2}s^{-p}R^{\beta}\sum_{i = 1}^{N}\mathcal{E}_{p, B(w, R_{\ast})}^{G}(f) \nonumber\\
		&\le C_{1}C_{\textup{D}}^{2}s^{-p}R^{\beta}\mathcal{E}_{p, B(z, (C_{\textup{TP}}^{-1} + 1)R)}^{G}(f).
	\end{align}
	(resp. $\#\bigl(\{ f \ge 0 \} \cap B\bigr) \le C_{1}C_{\textup{D}}^{2}s^{-p}R^{\beta}\mathcal{E}_{p, B(z, (C_{\textup{TP}}^{-1} + 1)R)}(f)$.)
	By \eqref{homogenize}, we obtain
	\begin{equation*}
		\#\bigl(B \cap \{ \abs{f} \ge s \}\bigr) \le \#B \le 2C_{1}C_{\textup{D}}^{2}s^{-p}R^{\beta}\mathcal{E}_{p, B(z, (C_{\textup{TP}}^{-1} + 1)R)}^{G}(f).
	\end{equation*}

	\noindent
	\textbf{Case 2:} Consider the complement of Case 1, i.e. for any $x \in \{ f \ge s \} \cap B$ (resp. $x \in \{f \le -s \} \cap B$) there exists $y_{x} \in \{ f \le 0 \} \cap B$ (resp. $y_{x} \in \{ f \ge 0 \} \cap B$) such that
	\[
	M_{R}^{p}[f](x) = M_{R}^{p}[f](x) \vee M_{R}^{p}[f](y_{x}).
	\]
	By considering a sequence of balls $\{ B(x, R_{\ast}(x)) \}_{x \in \{ f \ge s \} \cap B}$ (resp. $\{ B(x, R_{\ast}(x)) \}_{x \in \{ f \le -s \} \cap B}$) instead of `$\{ B(y, R_{\ast}(y)) \}_{y \in \{ f \le 0 \} \cap B}$' in Case 1, a similar argument to the derivation of \eqref{weak.cov} implies that
	\begin{equation*}
		\#\bigl(\{ f \ge s \} \cap B\bigr) \le C_{1}C_{\textup{D}}^{2}s^{-p}R^{\beta}\mathcal{E}_{p, B(z, (C_{\textup{TP}}^{-1} + 1)R)}^{G}(f).
	\end{equation*}
	(resp. $\#\bigl(\{ f \le -s \} \cap B\bigr) \le C_{1}C_{\textup{D}}^{2}s^{-p}R^{\beta}\mathcal{E}_{p, B(z, (C_{\textup{TPE}}^{-1} + 1)R)}(f).$)
	Therefore, we get \eqref{weakLp}.

	The desired Poincar\'{e} inequality will be shown by a truncation method by using  \eqref{weakLp} (cf. \cite{Maz}).
	Define $J_{\ast} = J_{\ast}(f) \coloneqq \max\bigl\{ j \in \mathbb{Z} \bigm| 2^{j} \le \norm{f}_{\ell^{\infty}(V)} \bigr\}$. (If $\norm{f}_{\ell^{\infty}(V)} = +\infty$, then we define $J_{\ast} = +\infty$.)
	For each $j \in \mathbb{Z} \cap (-\infty, J_{\ast}]$, set
	\[
	A_{j} \coloneqq B \cap \{ 2^{j} \le \abs{f} < 2^{j + 1} \} \quad \text{and} \quad f_{j} \coloneqq \bigl(\abs{f} - 2^{j}\bigr) \vee 0 \wedge 2^{j}.
	\]
	Note that $\{ \abs{f} \ge 2^{j} \} =  \bigl\{ \abs{f_{j}} \ge 2^{j} \bigr\}$.
	By \eqref{weakLp} and Lemma \ref{lem.basic-dEp}(a,d), we have
	\[
	\#\bigl(B \cap \bigl\{ \abs{f_{j}} \ge 2^{j} \bigr\}\bigr)
	\le C_{2}R^{\beta}2^{-jp}\mathcal{E}_{p, B(z, (C_{\textup{TP}}^{-1} + 1)R)}^{G}(f_{j})
	\le C_{2}R^{\beta}2^{-jp}\mathcal{E}_{p, \closure{A_{j}}}^{G}(f).
	\]
	Hence,
	\begin{align*}
		\norm{f}_{p, A_{j}}^{p} = \sum_{x \in A_{j}}\abs{f(x)}^{p}
		\le 2^{(j + 1)p}\#A_{j}
		\le 2^{(j + 1)p}\#\bigl(B \cap \bigl\{ \abs{f_{j}} \ge 2^{j} \bigr\}\bigr)
		\le 2^{p}C_{2}R^{\beta}\mathcal{E}_{p, \closure{A_{j}}}^{G}(f).
	\end{align*}
	Since $\{ A_{j} \}_{j}$ are disjoint, we note that
	$\sum_{j \in J}\mathcal{E}_{p, \closure{A}_{j}}^{G}(f) \le 2\mathcal{E}_{p}^{G}(f)$.
	Hence we obtain
	\begin{equation*}
		\norm{f}_{p, B}^{p} \le 2^{p + 1}C_{1}R^{\beta}\mathcal{E}_{p, \closure{B}}(f) \le 2^{p}C_{1}\deg(G)R^{\beta}\mathcal{E}_{p, 2B}^{G}(f),
	\end{equation*}
	which proves \eqref{eq.PI-median}.

	We conclude the proof by observing that \eqref{eq.PI-median} implies \ref{cond.PI}.
	By \eqref{eq.PI-median},
	\[
	\inf_{c \in \mathbb{R}}\sum_{z \in B(x, R)}\abs{f(z) - c}^{p} \le CR^{\beta}\mathcal{E}_{p, B(x, AR)}^{G}(f).
	\]
	Combining with Lemma \ref{lem.p-var}, we get \ref{cond.PI}.
\end{proof}

\subsection{Two-point estimates are implied by Loewner bounds}
We shall see that \ref{cond.PI} holds on a graph $G$ satisfying \ref{assum.BCL}, i.e. \ref{cond.BCL} with $1 - p \le \zeta < 1$, and \ref{cond.VD} with exponent $\alpha \ge 1$, where $\beta = \alpha - \zeta > 0$.
By virtue of Lemma \ref{lem.TP-PI}, it is enough to show the following lemma.

\begin{lem}\label{lem.TP}
	Let $G = (V, E)$ be a graph satisfying \ref{VD.growth} and \ref{assum.BCL} with the exponent $\zeta \in [1-p,1)$.
	Then $G$ satisfies \hyperref[cond.TP]{\textup{TP$_{p}(\alpha - \zeta)$}} and the associated constant $C_{\mathrm{TP}}$ depends only on constants involved in the assumptions.
\end{lem}
\begin{proof}
	We adapt the argument of \cite[Lemma 5.17]{HK98} which we briefly outline.	The proof proceeds by contradiction. If the two-point estimate fails, there exists a function for which the difference $\abs{f(x)-f(y)}=1$ but the truncated maximal-function of the gradient is much smaller than $D^{-\beta}$ where $D$ is comparable to the distance between $x$ and $y$. By using Theorem \ref{thm.pGCL-gamma} repeatedly at various scales, we find a shortcut  in the $\abs{\nabla f}$ metric between $x$ and $y$ whose length is strictly less than $1$. This contradicts the triangle inequality as any such path must have length at least $1$.

	We first prepare estimates, \eqref{BCL-TPE.mean-lower}, to get `shortcuts'.
	Let $z \in V$ and let $R \ge 1$.
	Let $C \ge 1$ that will be chosen later and set $B \coloneqq B(z, C^{-1}R)$.
	Let $x, y \in B$ be distinct.
	Pick a shortest path $\theta_{xy} = [x = x_{0}, x_{1}, \dots, x_{D_{xy} - 1}, x_{D_{xy}} = y]$, i.e. $D_{xy} = d_{G}(x, y)$ and $\{ x_{i - 1}, x_{i} \} \in E$ for each $i = 1, \dots, D_{xy}$.
	Set $D \coloneqq \lceil D_{xy}/2 \rceil$.
	Note that we always have $2^{-1}D_{xy} \le D \le 2D_{xy}$ and $D_{xy} \le 2R$.
	The assertion in the case $D \le 2$ can be obtained from Lemma \ref{lem.TP-short}.
	So, we consider the case $D \ge 3$.
	Fix $\kappa \ge 9$ and define
	\[
	n_{\ast} = n_{\ast}(\kappa, D_{x, y}) \coloneqq \max\{ j \in \mathbb{Z}_{\ge 0} \mid \kappa^{-3j}D - \kappa^{-3j - 2}D \ge 2 \}.
	\]
	Note that $D \ge 3$ and $\kappa \ge 9$ imply $D - \kappa^{-2}D \ge 2$.
	Set
	\[
	A_{j}^{x} \coloneqq \closure{B}(x, \kappa^{-3j}D) \setminus B(x, \kappa^{-3j - 2}D) \quad \text{for each $j \in \mathbb{Z}_{\ge 0}$.}
	\]
	For each $j \in \{ 0, \dots, n_{\ast} \}$, let $\theta_{j}$ be the connected component of $\theta_{xy} \cap A_{j}^{x}$.
	(Since $\theta_{xy}$ is a shortest path, there exists only one connected component of $\theta_{xy} \cap A_{j}^{x}$.)
	Then we have
	\[
	\diam\theta_{j} = \lceil\kappa^{-3j}D\rceil - \lfloor\kappa^{-3j - 2}D\rfloor \ge \kappa^{-3j}D - \kappa^{-3j - 2}D \ge 2,
	\]
	and thus $\#\theta_{j} \ge 2$.
	Using the fact that $\theta_{xy}$ is a shortest path, we see that
	\begin{align*}
		\frac{\dist(\theta_{j}, \theta_{j + 1})}{\diam\theta_{j} \wedge \diam\theta_{j + 1}}
		&\le \frac{\lceil\kappa^{-3j - 2}D\rceil - \lfloor\kappa^{-3j - 3}D\rfloor}{\lceil\kappa^{-3(j + 1)}D\rceil - \lfloor\kappa^{-3(j + 1) - 2}D\rfloor} \\
		&\le \frac{\kappa^{-3j - 2}D - \kappa^{-3j - 3}D + 2}{\kappa^{-3(j + 1)}D - \kappa^{-3(j + 1) - 2}D}
		= \frac{1 - \kappa^{-1} + \frac{2}{\kappa^{-3j - 2}D}}{\kappa^{-1} - \kappa^{-3}}
		\le \frac{\kappa^{4} + \kappa^{2} - \kappa}{\kappa^{2} - 1},
	\end{align*}
	where we used $\kappa^{-3j}D \ge 2 + \kappa^{-3j - 2}D \ge 2$ in the last inequality.
	By Theorem \ref{thm.pGCL-gamma}, there exist constants $L, c > 0$ (depending only on the constants associated with the assumptions) such that
	\begin{equation}\label{BCL-TPE.lower}
		\MOD_{p}^{G}(\Theta_{j}) \ge c(\kappa^{-3j}D)^{\zeta},
	\end{equation}
	where
	\[
	\Theta_{j} \coloneqq \bigl\{ \theta \in \PATH(\theta_{j}, \theta_{j + 1}) \bigm| \diam\theta \le L\kappa^{-3j}D \bigr\}.
	\]
	(We do not define $\{ \Theta_{j} \}$ if $n_{\ast} = 0$.)
	Note that $\theta \subseteq (1 + L)B(x, \kappa^{-3j}D)  \eqqcolon B_{j}$ for any $\theta \in \Theta_{j}$.
	By Lemma \ref{lem.shortPath} and \eqref{BCL-TPE.lower}, for any $\rho \in \ell^{+}(V)$, we have
	\[
	\norm{\rho}_{p, B_{j}}^{p} \ge cL_{\rho}(\Theta_{j})^{p}(\kappa^{-3j}D)^{\zeta}.
	\]
	Combining with \eqref{VD.growth.2}, we obtain
	\begin{equation}\label{BCL-TPE.mean-lower}
		\frac{1}{\#B_{j}}\sum_{x \in B_{j}}\rho(x)^{p}
		\ge cC_{\textup{D}}^{-1}L_{\rho}(\Theta_{j})^{p}(\kappa^{-3j}D)^{\zeta - \alpha}.
	\end{equation}

	To prove \hyperref[cond.TP]{\textup{TP$_{p}(\alpha - \zeta)$}}, it suffices to show that
	\begin{equation}\label{TPE.reduction}
		1 \le C_{\ast}R^{\alpha - \zeta}\bigl(M_{C_{\ast}R}^{p}[f](x) + M_{C_{\ast}R}^{p}[f](y)\bigr)
	\end{equation}
	for any $f \in \mathbb{R}^{V}$ with $\abs{f(x) - f(y)} = 1$, where $C_{\ast}$ is a  constant depending only on the constants associated with the assumptions.
	Hereafter, we fix $f \in \mathbb{R}^{V}$ satisfying $\abs{f(x) - f(y)} = 1$.
	Define $\abs{\nabla f}_{V} \in \ell^{+}(V)$ by setting
	\[
	\abs{\nabla f}_{V}(z) \coloneqq \max_{z' \in V; \{ z, z' \} \in E}\abs{f(z) - f(z')}, \quad z \in V.
	\]
	It is enough to consider whether the following case \eqref{TPE.assum-contrary} occurs or not.
	\begin{equation}\label{TPE.assum-contrary}
		L_{\abs{\nabla f}_{V}}(\Theta_{j})^{p}\kappa^{3j(\alpha - \zeta)} < C_{\#} \quad \text{for any $j \in \{ 0, \dots, n_{\ast}-1 \}$,}
	\end{equation}
	where
	\begin{equation}\label{Csharp}
		C_{\#} \coloneqq \frac{1}{6}\left(\frac{1}{1 - \kappa^{-3(\alpha - \zeta)/p}}\right)^{-p}.
	\end{equation}
	(Here, we let $\{ 0, \dots, n_{\ast} - 1\} = \{ 0 \}$ if $n_{\ast} = 0$.)
	Indeed, if there exists $j \in \{ 0, \dots, n_{\ast}-1 \}$ such that
	\begin{equation}\label{TPE.contrary}
		L_{\abs{\nabla f}_{V}}(\Theta_{j})^{p}\kappa^{3j(\alpha - \zeta)} \ge C_{\#},
	\end{equation}
	then we see that
	\begin{align}\label{sublemma}
		1 \stackrel{\eqref{TPE.contrary}}{\le} C_{\#}^{-1}L_{\abs{\nabla f}_{V}}(\Theta_{j})^{p}\kappa^{3j(\alpha - \zeta)}
		&\stackrel{\eqref{BCL-TPE.mean-lower}}{\le} c^{-1}C_{\#}^{-1}C_{\textup{D}}\cdot D^{\alpha - \zeta}\frac{1}{\#B_{j}}\sum_{v \in B_{j}}\abs{\nabla f}_{V}(v)^{p} \nonumber \\
		&\le c^{-1}C_{\#}^{-1}C_{\textup{D}}\cdot D^{\alpha - \zeta}\frac{1}{\#B_{j}}\deg(G)\mathcal{E}_{p, \closure{B_{j}}}^{G}(f) \nonumber \\
		&\le c^{-1}C_{\#}^{-1}C_{\textup{D}}^{2}\cdot D^{\alpha - \zeta}\frac{1}{\#\closure{B_{j}}}\deg(G)\mathcal{E}_{p, \closure{B_{j}}}^{G}(f) \nonumber \\
		&\le 2^{\alpha - \zeta}c^{-1}C_{\#}^{-1}C_{\textup{D}}^{2}\deg(G)\cdot d_{G}(x, y)^{\alpha - \zeta}M_{2(1 + L)R}^{p}[f](x).
	\end{align}
	Since $d_{G}(x, y) \le 2R$ and $\alpha - \zeta \ge 0$, we have \eqref{TPE.reduction}.

	We will show that a combination of \eqref{TPE.assum-contrary} and the failure of \eqref{TPE.reduction} yields a contradiction.
	Suppose that \eqref{TPE.assum-contrary} holds.
	Then for any $j \in \{ 0, \dots, n_{\ast}-1 \}$ there exists $\widetilde{\theta}_{j} \in \Theta_{j}$ such that
	\begin{equation}\label{TPE.short-cut.1}
		L_{\abs{\nabla f}_{V}}\bigl(\widetilde{\theta}_{j}\bigr)^{p} \le C_{\#}\kappa^{-3j(\alpha - \zeta)}.
	\end{equation}
	Note that $\diam\widetilde{\theta}_{j} \ge \kappa^{-3j - 2}D - \kappa^{-3j - 3}D \ge 2$ and thus $\#\widetilde{\theta}_{j} \ge 2$.
	By using the fact that $\theta_{xy}$ is a shortest path, we have
	\begin{align*}
		\frac{\dist\bigl(\widetilde{\theta}_{j}, \widetilde{\theta}_{j + 1}\bigr)}{\diam\widetilde{\theta}_{j} \wedge \diam\widetilde{\theta}_{j + 1}}
		\le \frac{\diam\theta_{j + 1}}{\diam\widetilde{\theta}_{j} \wedge \diam\widetilde{\theta}_{j + 1}}
		\le \frac{\lceil\kappa^{-3j}D\rceil - \lfloor\kappa^{-3j - 2}D\rfloor}{\kappa^{-3j - 2}D - \kappa^{-3j - 3}D}
		\le \frac{\kappa^{2}(2\kappa - 1)}{\kappa - 1}.
	\end{align*}
	Again by Theorem \ref{thm.pGCL-gamma}, there exist constants $\widetilde{L}, \widetilde{c} > 0$ (depending only on the constants associated with the assumptions) such that,
	for any $j \in \{ 0, \dots, n_{\ast} - 1 \}$,
	\[
	\widetilde{\Theta}_{j} \coloneqq \bigl\{ \theta \in \PATH\bigl(\widetilde{\theta}_{j}, \widetilde{\theta}_{j + 1}\bigr)\bigm| \diam\theta \le \widetilde{L}\kappa^{-3j}D \bigr\}
	\]
	satisfies
	\begin{equation}\label{BCL-TPE.lower.2}
		\MOD_{p}^{G}\bigl(\widetilde{\Theta}_{j}\bigr) \ge \widetilde{c}\bigl(\kappa^{-3j}D\bigr)^{\zeta}.
	\end{equation}
	We also define
	\[
	\widetilde{\Theta}_{n_{\ast}} = \bigl\{ \theta \in \PATH(\{ x \}, \theta_{n_{\ast}}) \bigm| \text{$\theta$ is a shortest path} \bigr\}.
	\]
	By \eqref{smallscale} in Lemma \ref{lem.shortPath}, we have
	\begin{equation}\label{BCL-TPE.lower.3}
		\MOD_{p}^{G}\bigl(\widetilde{\Theta}_{n_{\ast}}\bigr) \ge \bigl(\kappa^{-3n_{\ast}}D\bigr)^{1 - p} \ge  c_{1}\bigl(\kappa^{-3n_{\ast}}D\bigr)^{\zeta},
	\end{equation}
	where
	\[
	c_{1} = c_{1}(p, \zeta, \kappa) \coloneqq \kappa^{3(p - 1)}\left\{ \left(\frac{2}{1 - \kappa^{-2}}\right)^{1 - p - \zeta} \vee \left(\frac{2\kappa^{3}}{1 - \kappa^{-2}}\right)^{1 - p - \zeta} \right\}.
	\]
	Indeed, $n_{\ast}$ satisfies
	\[
	\frac{2}{1 - \kappa^{-2}} \le \kappa^{-3n_{\ast}}D < \frac{2}{\kappa^{-3}(1- \kappa^{-2})},
	\]
	and thus $\bigl(\kappa^{-3n_{\ast}}D\bigr)^{1 - p} \ge  c_{1}\bigl(\kappa^{-3n_{\ast}}D\bigr)^{\zeta}$ holds.
	Note that $\theta \subseteq B\bigl(x, (1 + \widetilde{L})\kappa^{-3j}D\bigr)$ for any $\theta \in \widetilde{\Theta}_{j}$.
	By using \eqref{BCL-TPE.lower.3} instead of \eqref{BCL-TPE.mean-lower} in the argument of \eqref{sublemma}, we can show that the existence of $j \in \{ 0, \dots, n_{\ast} \}$ satisfying
	\[
	L_{\abs{\nabla f}_{V}}\bigl(\widetilde{\Theta}_{j}\bigr)^{p}\kappa^{3j(\alpha - \zeta)} \ge C_{\#}
	\]
	implies \eqref{TPE.reduction} (with $C_{\ast} = 2\bigl(1 + \widetilde{L}\bigr) \vee 2\bigl(4^{\alpha - \zeta}\bigl(\widetilde{c}^{-1} \vee c_{1}^{-1}\bigr)C_{\#}^{-1}C_{\textup{D}}^{2}\deg(G)\bigr)$).
	So let us suppose the following case:
	\begin{equation}\label{TPE.assum-contrary.2}
		L_{\abs{\nabla f}_{V}}\bigl(\widetilde{\Theta}_{j}\bigr)^{p}\kappa^{3j(\alpha - \zeta)} < C_{\#} \quad \text{for any $j \in \{ 0, \dots, n_{\ast} \}$.}
	\end{equation}
	We will deduce a contradiction by constructing a ``too short path joining $x$ and $y$''.
	From \eqref{TPE.assum-contrary.2}, for each $j \in \{ 0, \dots, n_{\ast} \}$, we can find a path $\widetilde{\theta}_{j}' \in \widetilde{\Theta}_{j}$ such that
	\begin{equation}\label{TPE.short-cut.2}
		L_{\abs{\nabla f}_{V}}\bigl(\widetilde{\theta}_{j}'\bigr)^{p} \le C_{\#}\kappa^{-3j(\alpha - \zeta)}.
	\end{equation}
	By concatenating $\{ \theta_{j} \}_{j = 0}^{n_{\ast}}$ and $\bigl\{ \widetilde{\theta}_{j} \bigr\}_{j = 0}^{n_{\ast}}$, we obtain a path $\theta^{(x)} = [x = v_{0}, v_{1}, \dots, v_{l_{x}}]$ for some $l_{x} \in \mathbb{N}$ such that
	\begin{itemize}
		\item [(a)] $v_{l_{x}} \in \theta_{0} \subseteq A_{0}^{x}$;
		\item [(b)] $L_{\abs{\nabla f}_{V}}\bigl(\theta^{(x)}\bigr) \le \sum_{j = 0}^{n_{\ast}}\Bigl(L_{\abs{\nabla f}_{V}}(\theta_{j}) + L_{\abs{\nabla f}_{V}}\bigl(\widetilde{\theta}_{j}\bigr)\Bigr)$.
	\end{itemize}
	By \eqref{TPE.short-cut.1}, \eqref{TPE.short-cut.2} and \eqref{Csharp}, the condition (b) implies that
	\begin{align*}
		L_{\abs{\nabla f}_{V}}\bigl(\theta^{(x)}\bigr)
		\le 2C_{\#}^{1/p}\sum_{j = 0}^{n_{\ast}}\kappa^{-3j(\alpha - \zeta)/p}
		< \frac{2C_{\#}^{1/p}}{1 - \kappa^{-3(\alpha - \zeta)/p}} \le \frac{1}{3}.
	\end{align*}
	To summarize, what we have shown in the above argument is that \eqref{TPE.reduction} holds or
	\begin{equation}\label{TPE.short-cut.x}
		\text{there exists a path $\theta^{(x)} \in \PATH(\{ x \}, A_{0}^{x})$ such that $L_{\abs{\nabla f}_{V}}\bigl(\theta^{(x)}\bigr) < \frac{1}{3}$.}
	\end{equation}
	In a similar way, we also see that \eqref{TPE.reduction} holds or
	\begin{equation}\label{TPE.short-cut.y}
		\text{there exists a path $\theta^{(y)} \in \PATH(\{ y \}, A_{0}^{y})$ such that $L_{\abs{\nabla f}_{V}}\bigl(\theta^{(y)}\bigr) < \frac{1}{3}$,}
	\end{equation}
	where $A_{0}^{y} \coloneqq \closure{B}(y, D) \setminus B(y, \kappa^{-2}D)$.

	Next we will construct a ``short-cut joining $\theta^{(x)}$ and $\theta^{(y)}$''.
	Recall that $\theta_{xy} = [x = x_{0}, \dots, x_{D_{xy}} = y]$ and $D = \lceil D_{xy}/2 \rceil = \lceil d_{G}(x, y)/2 \rceil$.
	We write $B \coloneqq B(x_{D}, (1 - \kappa_{0}^{-2})D)$.
	Then $\theta_{w} \cap B \neq \emptyset$ for $w \in \{ x, y \}$.
	If $\kappa_{0} \in \bigl(1, \sqrt{16/15}\bigr)$, then we easily see that $\theta_{w} \setminus 4B \neq \emptyset$ for $w \in \{ x, y \}$.
	Henceforth, we fix $\kappa_{0} \in \bigl(1, \sqrt{16/15}\bigr)$.
	Applying Corollary \ref{cor.pGCL-gamma.useful}, we have
	\[
	\MOD_{p}^{G}\Bigl(\theta^{(x)}, \theta^{(y)}; 4LB\Bigr) \ge cD^{\zeta}
	\]
	for some constants $c > 0$ and $L \ge 1$ (depending only on the constants associated with the assumptions).
	Since $4LB \subseteq B\bigl(x, (4L(1 - \kappa_{0}^{-2}) + 1)D \bigr)$, a combination of this lower bound of $p$-modulus and Lemma \ref{lem.shortPath} implies that there exists a path $\theta^{x \leftrightarrow y} \in \PATH\bigl(\theta^{(x)}, \theta^{(y)}; 4LB\bigr)$ such that
	\[
	\norm{\abs{\nabla f}_{V}}_{p, B(x, (4L(1 - \kappa_{0}^{-2}) + 1)D)}^{p} \ge cL_{\abs{\nabla f}_{V}}\bigl(\theta^{x \leftrightarrow y}\bigr)^{p}D^{\zeta}.
	\]
	If $L_{\abs{\nabla f}_{V}}\bigl(\theta^{x \leftrightarrow y}\bigr) \ge \frac{1}{3}$, then the arguments in \eqref{sublemma} using the above bound instead of \eqref{BCL-TPE.mean-lower} implies \eqref{TPE.reduction} (the associated constant $C_{\ast}$ depends only on the constants associated with the assumptions.)
	If $L_{\abs{\nabla f}_{V}}\bigl(\theta^{x \leftrightarrow y}\bigr) < 1/3$, then, by concatenating $\theta^{(x)}$, $\theta^{(y)}$ and $\theta^{x \leftrightarrow y}$ and using \eqref{TPE.short-cut.x} and \eqref{TPE.short-cut.y}, we can find a path $\theta_{\ast} \in \PATH(\{ x \}, \{ y \})$ such that $L_{\abs{\nabla f}_{V}}(\theta_{\ast}) < 1$, which implies a contradiction:
	\begin{align*}
		1  = \abs{f(x) - f(y)} \le L_{\abs{\nabla f}_{V}}(\theta_{\ast}) < 1.
	\end{align*}
	As a result, we obtian \eqref{TPE.reduction} and finish the proof.
\end{proof}

\begin{proof}[Proof of Theorem \ref{thm.PI-discrete}]
Combining Lemmas \ref{lem.TP-PI} and \ref{lem.TP}, we obtain Theorem \ref{thm.PI-discrete}.
\end{proof}

\section{Discrete elliptic Harnack inequality}\label{sec.EHI}
This section is devoted to Harnack type inequalities for discrete $p$-harmonic functions.
Such estimates are crucial to establish that the Sobolev space we construct has a dense set of continuous functions.

Throughout this section, let $p \in (1,\infty)$ and let $G = (V, E)$ be a locally finite connected simple non-directed graph.

\subsection{EHI for discrete \texorpdfstring{$p$-harmonic}{p-harmonic} functions}
The Poincar\'e inequality introduced in Definition \ref{dfn.PI} implies a lower bound on capacity across annulus.
Let us introduce a matching capacity upper bound which serves to identify the exponent $\beta$ introduced in  Definition \ref{dfn.PI} as the best possible one.
\begin{defn}
	Let $\beta > 0$.
	A graph $G$ satisfies \ref{cond.cap} if there exist constants $C_{\textup{cap}} > 0$ and $A_{\textup{cap}} \ge 1$ such that for any $x \in V$ and $R \in [1, \diam(G)/A_{\textup{cap}})$,
	\begin{equation}\label{cond.cap}
		\CAP_{p}^{G}\bigl(B(x, R), B(x, 2R)^c\bigr) \le C_{\textup{cap}}\frac{\#B(x, R)}{R^{\beta}}. \tag{\textup{cap$_{p, \le}(\beta)$}}
	\end{equation}
\end{defn}

The main result of this section is the following elliptic Harnack inequality.
\begin{thm}\label{thm.EHI}
	Let $p \in (1, \infty)$, $\hdim \ge 1$ and $\beta > 0$.
	Assume that $G = (V, E)$ satisfies \ref{cond.AR}, \hyperref[assum.BCL]{\textup{BCL$_{p}^{\textup{low}}(\hdim - \beta)$}} and \ref{cond.cap}.
	Then there exist constants $\delta_{\textup{H}} \in (0, 1)$ and $C_{\textup{H}} \ge 1$ depending only on the constants associated with the assumptions
	such that, for any $x \in V$ and $R \ge 1$ with $B(x, R) \neq V$, if $h \colon V \to [0, \infty)$ is $p$-harmonic on $B(x, R)$, then
	\begin{equation}\label{EHI}
		\max_{B(x, \delta_{\textup{H}}R)}h \le C_{\textup{H}}\min_{B(x, \delta_{\textup{H}}R)}h.
	\end{equation}
\end{thm}
A standard argument using Moser's oscillation lemma immediately yields the following interior H\"older regularity of harmonic functions (see \cite[\textsection 2.3.2]{Sal02} or \cite[Proposition 1.45]{Bar.RW}). 
Recall the definition of oscillation  $\osc_{B}[h]$ below from Notation \ref{it:nota.discrete}. 
\begin{cor}\label{cor.loc-Hol}
	Let $p \in (1,\infty)$, $\hdim \ge 1$ and $\beta > 0$.
	Assume that $G = (V, E)$ satisfies \ref{cond.AR}, \hyperref[assum.BCL]{\textup{BCL$_{p}^{\textup{low}}(\hdim - \beta)$}} and \ref{cond.cap}.
	For any $\lambda \in (0, 1)$ there exist constants $C_{\textup{H\"{o}l}}, \theta_{\textup{H\"{o}l}} > 0$ depending only on the constants associated with the assumptions
	such that for any non-negative function $h \in \mathbb{R}^{V}$ which is $p$-harmonic in a ball $B$ with radius $R \ge 1$,
	\begin{equation}\label{eq.Hol-summary}
		\abs{h(x) - h(y)} \le C_{\textup{H\"{o}l}}\left(\frac{d_{G}(x, y)}{R}\right)^{\theta_{\textup{H\"{o}l}}}\osc_{B}[h], \quad \text{for all $x, y \in \lambda B$.}
	\end{equation}
\end{cor}

To prove Theorem \ref{thm.EHI}, we start with a log-Caccioppoli type inequality which plays a key role in our proof of Theorem \ref{thm.EHI}.
The following lemma is a generalization of \cite[Lemma 7.5]{KZ92}.

\begin{lem}\label{lem.log}
	Let $p \in (1,\infty)$.
	Suppose that $F, G \in C^{2}((0, +\infty);\mathbb{R})$ satisfy $\abs{F'(s)} > 0$, $G'(s) = \abs{F'(s)}^{p}$ and $G(s) \le 0$ for any $s > 0$ and that $\Psi(s) \coloneqq \frac{G(s)}{\abs{F'(s)}^{p - 1}}$ is monotone (i.e. non-decreasing or non-increasing).
	Let $A \subseteq V$, let $h\colon V \to (0, \infty)$ and let $\varphi\colon V \to [0, 1]$.
	If $\supp[\varphi] \subseteq A$, then
	\begin{align}\label{ineq.preli-logC}
		\frac{1}{2}&\sum_{\{ x, y \} \in E(\closure{A})}(\varphi(x)^{p} \wedge \varphi(y)^{p})\abs{F(h(x)) - F(h(y))}^{p} - \mathcal{E}_{p}^{G}\bigl(h; \varphi^p \cdot (G \circ h)\bigr) \nonumber \\
		&\le
		\frac{2^{p - 1}(p - 1)^{p - 1}}{p}\sum_{\{ x, y \} \in E(\closure{A})}\bigl\{ \abs{\Psi(h(x))}^{p} \vee \abs{\Psi(h(y))}^{p} \bigr\}\abs{\varphi(x) - \varphi(y)}^{p}.
	\end{align}
\end{lem}
\begin{proof}
	First, we prepare a notation.
	For each $\psi\colon V \to \mathbb{R}$ and $x, y \in V$, define $\psi_{x,y}\colon [0, 1] \to \mathbb{R}$ by
	\[
	\psi_{x,y}(t) \coloneqq t\psi(x) + (1 - t)\psi(y), \quad t \in [0, 1].
	\]
	For any $h \colon V \to (0, \infty)$, $\varphi\colon V \to [0, 1]$, $x, y \in V$, we can show
	\begin{align}\label{eq.calc}
		&\int_{0}^{1}\varphi_{x,y}(t)^{p}\abs{\frac{d}{dt}(F(h_{x,y}(t)))}^{p}\,dt \nonumber \\
		= \mathrm{sgn}&(h(x) - h(y))\abs{h(x) - h(y)}^{p - 1}\bigl\{\varphi(x)^{p}G(h(x)) - \varphi(y)^{p}G(h(y))\bigr\} \nonumber\\
		&-p\cdot\mathrm{sgn}\bigl(h(x) - h(y)\bigr)(\varphi(x) - \varphi(y))\int_{0}^{1}\varphi_{x,y}(t)^{p - 1}\abs{\frac{d}{dt}F(h_{x, y}(t))}^{p - 1}\Psi(h_{x, y}(t))\,dt.
	\end{align}
	Indeed, by simple computations, we have
	\begin{align*}
		\abs{\frac{d}{dt}F(h_{x, y}(t))}^{p}
		&= G'(h_{x, y}(t))\abs{\frac{d}{dt}h_{x, y}(t)}^{p} \\
		&= \left(\frac{d}{dt}G(h_{x, y}(t))\right)\abs{h(x) - h(y)}^{p - 1}\mathrm{sgn}\bigl(h(x) - h(y)\bigr),
	\end{align*}
	and
	\begin{align*}
		&\int_{0}^{1}\varphi_{x,y}(t)^{p}\frac{d}{dt}G(h_{x,y}(t))\,dt \\
		=
		&\int_{0}^{1}\frac{d}{dt}\Bigl(\varphi_{x,y}(t)^{p}G(h_{x,y}(t))\Bigr)\,dt - p(\varphi(x) - \varphi(y))\int_{0}^{1}\varphi_{x,y}(t)^{p - 1}G(h_{x,y}(t))\,dt \\
		=
		&\bigl(\varphi(x)^{p}G(h(x)) - \varphi(y)^{p}G(h(y))\bigr) - p(\varphi(x) - \varphi(y))\int_{0}^{1}\varphi_{x,y}(t)^{p - 1}G(h_{x,y}(t))\,dt
	\end{align*}
	Using these identities, we see that
	\begin{align*}
		\int_{0}^{1}\varphi_{x,y}(t)^{p}&\abs{\frac{d}{dt}(F(h_{x,y}(t)))}^{p}\,dt \\
		= \mathrm{sgn}\bigl(h(x) - h(y)\bigr)&\abs{h(x) - h(y)}^{p - 1} \times \\
		\Bigl\{\bigl(\varphi(x)^{p}G(h(x)) - &\varphi(y)^{p}G(h(y))\bigr) - p(\varphi(x) - \varphi(y))\int_{0}^{1}\varphi_{x,y}(t)^{p - 1}G(h_{x,y}(t))\,dt \Bigr\}.
	\end{align*}
	We now get \eqref{eq.calc} since
	\begin{align*}
		G(h_{x,y}(t))\abs{h(x) - h(y)}^{p - 1}
		= \abs{\frac{d}{dt}F(h_{x,y}(t))}^{p - 1}\Psi(h_{x,y}(t)).
	\end{align*}

	On the one hand, by a simple computation: $\varphi_{x, y}(t)^{p} \ge \varphi(x)^{p} \wedge \varphi(y)^{p}$, we have
	\begin{align}\label{logC-lower}
		\int_{0}^{1}\varphi_{x,y}(t)^{p}\abs{\frac{d}{dt}F(h_{x,y}(t))}^{p}\,dt
		&\ge \bigl(\varphi(x)^{p} \wedge \varphi(y)^{p}\bigr)\int_{0}^{1}\abs{\frac{d}{dt}F(h_{x,y}(t))}^{p}\,dt \nonumber\\
		&\ge \bigl(\varphi(x)^{p} \wedge \varphi(y)^{p}\bigr)\abs{\int_{0}^{1}\frac{d}{dt}F(h_{x,y}(t))\,dt}^{p} \quad \text{(by H\"{o}lder)} \nonumber \\
		&\ge \bigl(\varphi(x)^{p} \wedge \varphi(y)^{p}\bigr)\abs{F(h(x)) - F(h(y))}^{p}.
	\end{align}
	On the other hand, by the above Claim, we see that
	\begin{align}\label{logC-upper0}
		&\sum_{\{ x,y \} \in E(\closure{A})}\int_{0}^{1}\varphi_{x,y}(t)^{p}\abs{\frac{d}{dt}F(h_{x,y}(t))}^{p}\,dt \nonumber \\
		&= \sum_{\{ x,y \} \in E(\closure{A})}\mathrm{sgn}(h(x) - h(y))\abs{h(x) - h(y)}^{p - 1}\bigl(\varphi(x)^{p}G(h(x)) - \varphi(y)^{p}G(h(y))\bigr) \nonumber \\
		&-p\sum_{\{ x,y \} \in E(\closure{A})}\mathrm{sgn}\bigl(h(x) - h(y)\bigr)(\varphi(x) - \varphi(y))\int_{0}^{1}\varphi_{x,y}(t)^{p - 1}\abs{\frac{d}{dt}F(h_{x, y}(t))}^{p - 1}\Psi(h_{x, y}(t))\,dt \nonumber \\
		&= \mathcal{E}_{p}^{G}\bigl(h; \varphi^{p}\cdot(G \circ h)\bigr) + A_{p}[\varphi, h]
		\le \mathcal{E}_{p}^{G}\bigl(h; \varphi^{p}\cdot(G \circ h)\bigr) +  \abs{A_{p}[\varphi, h]},
	\end{align}
	where
	\begin{align*}
		&A_{p}[\varphi, h] \coloneqq \\
		-p&\sum_{\{ x,y \} \in E(\overline{A})}\mathrm{sgn}\bigl(h(x) - h(y)\bigr)(\varphi(x) - \varphi(y))\int_{0}^{1}\varphi_{x,y}(t)^{p - 1}\abs{\frac{d}{dt}F(h_{x, y}(t))}^{p - 1}\Psi(h_{x, y}(t))\,dt.
	\end{align*}
	Now, a combination of H\"{o}lder's inequality on $E \times [0, 1]$ and Young's inequality implies that for any any $\varepsilon > 0$,
	\begin{align*}
		\abs{A_{p}[\varphi, h]}
		&\le p\left(\sum_{\{ x,y \} \in E(\overline{A})}\int_{0}^{1}\varphi_{x,y}(t)^{p}\abs{\frac{d}{dt}F(h_{x, y}(t))}^{p}\,dt\right)^{(p - 1)/p} \\
		&\hspace*{50pt}\times\left(\sum_{\{ x,y \} \in E(\overline{A})}\abs{\varphi(x) - \varphi(y)}^p\int_{0}^{1}\abs{\Psi(h_{x, y}(t))}^{p}\,dt\right)^{1/p} \\
		&\le (p - 1)\varepsilon^{p/(p - 1)}\sum_{\{ x,y \} \in E(\overline{A})}\int_{0}^{1}\varphi_{x,y}(t)^{p}\abs{\frac{d}{dt}F(h_{x, y}(t))}^{p}\,dt \\
		&\hspace*{50pt}+ \frac{\varepsilon^{-p}}{p}\sum_{\{ x,y \} \in E(\closure{A})}\abs{\varphi(x) - \varphi(y)}^p\int_{0}^{1}\abs{\Psi(h_{x, y}(t))}^{p}\,dt.
	\end{align*}
	By choosing $\varepsilon = \bigl(1/2(p - 1)\bigr)^{(p - 1)/p}$ and combining with \eqref{logC-upper0}, we obtain
	\begin{align}\label{logC-upper}
		\frac{1}{2}&\sum_{\{ x,y \} \in E(\closure{A})}\int_{0}^{1}\varphi_{x,y}(t)^{p}\abs{\frac{d}{dt}F(h_{x, y}(t))}^{p}\,dt - \mathcal{E}_{p}^{G}\bigl(h; \varphi^{p}\cdot(G \circ h)\bigr) \nonumber \\
		&\le \frac{2^{p - 1}(p - 1)^{p - 1}}{p}\sum_{\{ x,y \} \in E(\closure{A})}\abs{\varphi(x) - \varphi(y)}^p\int_{0}^{1}\abs{\Psi(h_{x, y}(t))}^{p}\,dt \nonumber \\
		&\le \frac{2^{p - 1}(p - 1)^{p - 1}}{p}\sum_{\{ x,y \} \in E(\closure{A})}\bigl\{\abs{\Psi(h(x))}^p \vee \abs{\Psi(h(y))}^p\bigr\}\abs{\varphi(x) - \varphi(y)}^p.
	\end{align}
	Here we used the monotonicity of $\Psi$ in the last inequality, i.e.
	\[
	\int_{0}^{1}\abs{\Psi(h_{x, y}(t))}^{p}\,dt \le \abs{\Psi(h(x))}^p \vee \abs{\Psi(h(y))}^p.
	\]
	A combination of \eqref{logC-lower} and \eqref{logC-upper} yields \eqref{ineq.preli-logC}.
	We complete the proof.
\end{proof}

Particular the case $F(t) = \log{t}$ gives an important inequality so called the log-Caccioppoli inequality in PDE theory.
See also \cite[Corollary 7.7]{KZ92} for the case $p = 2$.
\begin{cor}[Log-Caccioppoli inequality]\label{cor.logC}
	Let $p \in (1,\infty)$.
	Let $A \subseteq V$ and $\varphi\colon V \to [0, 1]$ with $\supp[\varphi] \subseteq A$.
	If $h\colon V \to (0, \infty)$ satisfies $-\Delta_{p}^{G}h \ge \eta$ for some $\eta\colon V \to \mathbb{R}$, then
	\begin{align}\label{eq.logC}
		\frac{1}{2}\sum_{(x, y) \in E(\overline{A})}(\varphi(x)^{p} \wedge \varphi(y)^{p})&\abs{\log{h(x)} - \log{h(y)}}^{p} \nonumber \\
		&+ \frac{1}{2(p - 1)}\sum_{x \in V}\frac{\eta(x)\varphi(x)^p}{h(x)^{p - 1}}\deg_{G}(x)
		\le
		C_p\mathcal{E}_{p}^{G}(\varphi),
	\end{align}
	where $C_p \coloneqq \frac{2^{p - 1}}{p(p - 1)}$.
\end{cor}
\begin{proof}
	Set $F(t) \coloneqq \log{t}$, $G(t) \coloneqq -\frac{1}{p - 1}t^{-(p - 1)}$ for $t \in (0, \infty)$.
	Note that $G\colon (0, \infty) \to (-\infty, 0)$.
	Then we easily see that
	\begin{align*}
		&F'(t) = t^{-1} > 0, \\
		&G'(t) = t^{-p} = \abs{F'(t)}^{p}, \\
		&\Psi(t) \coloneqq \frac{G(t)}{\abs{F'(t)}^{p - 1}} = -\frac{1}{p - 1}.
	\end{align*}
	Since
	\[
	-\mathcal{E}_{p}^{G}\bigl(h; \varphi^p\cdot(G \circ h)\bigr)
	= \frac{1}{2(p - 1)}\Bigl\langle -\Delta_{p}^{G}h, \frac{\varphi^p}{h^{1 - p}}\Bigr\rangle_{\ell^2(V, \mathrm{deg})}
	\ge \frac{1}{2(p - 1)}\Bigl\langle \eta, \frac{\varphi^p}{h^{p - 1}}\Bigr\rangle_{\ell^2(V, \mathrm{deg})},
	\]
	we get \eqref{eq.logC} by applying Lemma \ref{lem.log}.
\end{proof}

The next lemma is immediate by considering $p$-energies of indicator functions.
\begin{lemma}\label{lem.cap-short}
	For any $x \in V$ and $R > 0$,
	\begin{equation}\label{cap-short.pre}
		\CAP_{p}^{G}\bigl(B(x, R), B(x, 2R)^c\bigr) \le \#\bigl\{ \{ y, z \} \in E \bigm| y \in B(x, R), z \not\in B(x, R) \bigr\}.
	\end{equation}
	In particular, if $R \in (0, 1]$, then
	\begin{equation}\label{cap-short}
		\CAP_{p}\bigl(B(x, R), B(x, 2R)^c\bigr) \le \deg_{G}(x).
	\end{equation}
\end{lemma}
\begin{proof}
	Note that $\varphi \coloneqq \indicator{B(x, R)}\colon V \to \{0, 1 \}$ satisfies $\varphi|_{B(x, R)} \equiv 1$ and $\supp[\varphi] \subseteq B(x, 2R)$.
	We then have
	\begin{align*}
		\CAP_{p}^{G}\bigl(B(x, R), B(x, 2R)^c\bigr) \le \mathcal{E}_{p}(\varphi),
	\end{align*}
	which shows \eqref{cap-short.pre}.
	We also note that $\varphi = \delta_{x}$ when $R \in (0, 1]$, and hence \eqref{cap-short} holds.
\end{proof}

The following generalization of \ref{cond.cap} is done by a standard covering argument using the metric doubling property.
\begin{lem}\label{lem.cap}
	Let $\hdim \ge 1, \beta > 0$ and let $G = (V, E)$ satisfy \ref{cond.AR} and \ref{cond.cap}.
	For any $\delta \in (0, 1)$ there exists $C_{\textup{cap}}(\delta) > 0$ depending only on $\delta$ and the constants associated with the assumptions
	such that for any $x \in V$ and $R \ge \delta^{-1}$,
	\begin{equation*}
		\CAP_{p}^{G}\bigl(B(x, \delta R), B(x, R)^c\bigr) \le C_{\textup{cap}}(\delta)\frac{\#B(x, \delta R)}{R^{\beta}}.
	\end{equation*}
\end{lem}
\begin{proof}
	Note that $G$ also satisfies \hyperref[VD.growth]{\textup{VD$(\hdim)$}}.
	Let $x \in V$, $R \ge 1$ and $\delta \in (0, 1)$.
	Let $A_{\textup{cap}} \ge 1$ be the constant in \ref{cond.cap}.
	Set $\widetilde{\delta} \coloneqq \frac{1 - \delta}{4} \wedge A_{\textup{cap}}^{-1} \in (0, 1)$.
	Fix a maximal $\widetilde{\delta}R$-net $\{ x_{i} \}_{i = 1}^{N_{\delta}}$ of $B(x, \delta R)$, i.e. $x_i \in B(x, \delta R)$, $d(x_i, x_j) \ge \widetilde{\delta}R$ for each $i \neq j \in \{ 1, \dots, N_{\delta} \}$ and $B(x, \delta R) \subseteq \bigcup_{i = 1}^{N_{\delta}}B(x_{i}, \widetilde{\delta}R)$.
	Since $G$ is metric doubling, the number $N_{\delta}$ has an upper bound depending only on $\delta R/\widetilde{\delta}R = 4\delta/(1 - \delta) \vee \delta A_{\textup{cap}}^{-1}$.

	If $\widetilde{\delta}R \ge \diam(G)/A_{\textup{cap}}$, then $B(x, R) = V$ for any $x \in V$ and $\CAP_{p}^{G}\bigl(B(x, \delta R), B(x, R)^c\bigr) = 0$.
	So we consider the case $\widetilde{\delta}R < \diam(G)/A_{\textup{cap}}$.
	For each $i \in \{ 1, \dots, N_{\delta} \}$, let $\varphi_{i}\colon V \to [0, 1]$ be the minimizer of $\CAP_{p}^{G}\bigl(B(x_i, \widetilde{\delta}R), B(x_i, 2\widetilde{\delta}R)^c\bigr)$ such that $\varphi_{i}|_{B(x_i, \widetilde{\delta}R)} \equiv 1$ and $\supp[\varphi_i] \subseteq B(x_i, 2\widetilde{\delta}R)$.
	Since $\delta + 2\widetilde{\delta} < 1$, we also have $\supp[\varphi_i] \subseteq B(x, R)$.
	Define $\varphi\colon V \to [0, 1]$ by
	\[
	\varphi \coloneqq \left(\sum_{i = 1}^{N_{\delta}}\varphi_i\right) \wedge 1.
	\]
	If $\widetilde{\delta}R \ge 1$, then we see from \ref{cond.cap} and \ref{cond.VD} that,
	\begin{align*}
		\CAP_{p}\bigl(B(x, \delta R), B(x, R)^c\bigr)
		\le \mathcal{E}_{p}(\varphi)
		&\le \mathcal{E}_{p}\left(\sum_{i = 1}^{N_{\delta}}\varphi_{i}\right) \\
		&\le N_{\delta}^{p - 1}\sum_{i = 1}^{N_{\delta}}\mathcal{E}_{p}(\varphi_{i}) \\
		&\le C_{\textup{cap}}N_{\delta}^{p - 1}\sum_{i = 1}^{N_{\delta}}\frac{\#B(x_i, \widetilde{\delta}R)}{(\widetilde{\delta}R)^{\beta}} \\
		&\le C_{\textup{cap}}C_{\textup{D}}N_{\delta}^{p - 1}\left(\frac{3\delta + 1}{4\delta}\right)^{\hdim}\left(\frac{4}{1 - \delta}\right)^{\beta}\frac{\#B(x, \delta R)}{R^{\beta}}.
	\end{align*}
	If $\widetilde{\delta}R < 1$, then we have from Lemma \ref{lem.cap-short} that
	\begin{align*}
		\CAP_{p}^{G}\bigl(B(x, \delta R), B(x, R)^c\bigr)
		\le \mathcal{E}_{p}^{G}(\indicator{B(x, \delta R)})
		&\le \#\bigl\{ \{ y, z \} \in E \bigm| y \in B(x, \delta R), z \not\in B(x, \delta R) \bigr\} \\
		&\le \deg(G)^{\delta R + 1}
		\le \deg(G)^{\delta\widetilde{\delta}^{-1} + 1}
		= \deg(G)^{4\delta/(1 - \delta) + 1}.
	\end{align*}
	Note that, by \ref{cond.AR},
	\[
	\frac{\#B(x, \delta R)}{R^{\beta}} \ge C_{\text{AR}}^{-1}\delta^{\hdim}R^{\hdim -\beta} \ge C_{\text{AR}}^{-1}\delta^{\hdim}\left(\delta^{\hdim - \beta} \wedge \left(\frac{1 - \delta}{4}\right)^{\hdim - \beta}\right)
	\]
	Therefore, if we put
	\begin{align*}
		&C_{\textup{cap}}(\delta) \\
		&= C_{\textup{cap}}C_{\textup{D}}N_{\delta}^{p - 1}\left(\frac{3\delta + 1}{4\delta}\right)^{\hdim}\left(\frac{4}{1 - \delta}\right)^{\beta} \vee C_{\text{AR}}\delta^{-\hdim}\deg(G)^{4\delta/(1 - \delta) + 1}\left(\delta^{\hdim - \beta} \wedge \left(\frac{1 - \delta}{4}\right)^{\hdim - \beta}\right)^{-1},
	\end{align*}
	then the required bound holds.
\end{proof}

Now we show Theorem \ref{thm.EHI}. 
A similar proof in the classical setting can be found in \cite[Theorem 4.3]{Hol03}. 
\begin{proof}[Proof of Theorem \ref{thm.EHI}]
	Fix $\delta_{\textup{H}} \in \bigl(0, (4L)^{-1}\bigr)$, where $L$ is the constant appeared in Corollary  \ref{cor.pGCL-gamma.useful}.
	By Lemma \ref{lem.basic-pMod}, we can assume that $L \ge 2$ without loss of generality.
	Let $\varepsilon > 0$ and set $h_{\varepsilon} \coloneqq h + \varepsilon$.
	Note that $h_{\varepsilon}$ is also $p$-harmonic on $B \coloneqq B(x, R)$.
	Define
	\[
	m \coloneqq \min_{B(x, \delta_{\textup{H}}R)}h_{\varepsilon} \quad \text{and} \quad M \coloneqq \max_{B(x, \delta_{\textup{H}}R)}h_{\varepsilon}.
	\]
	If $R \le 4L$, then $B(x, \delta_{\textup{H}}R) = \{ x \}$ and thus $m = M$.
	Hence it is enough to consider the case $R \ge 4L$.
	In this case, we always have $R - \delta_{\textup{H}}R > 4L - 1 > 2$, in particular $B(x, R) \setminus B(x, \delta_{\textup{H}}R) \neq \emptyset$.
	Using the maximum/minimum principles (Lemma \ref{lem.max/min}), we can find paths $\theta_{\textup{min}}, \theta_{\textup{max}}$ in $G$ satisfying the following conditions (i) and (ii) (see Figure \ref{fig.EHI}).
	\begin{enumerate}[(i)]
		\item $\theta_{\textup{min}} \subseteq \{ h_{\varepsilon} \le m \}$ and $\theta_{\textup{max}} \subseteq \{ h_{\varepsilon} \ge M \}$;
		\item $\theta_{\textup{min}}, \theta_{\textup{max}} \in \PATH\bigl(\partial_{i}B(x, \delta_{\textup{H}}R), \partial_{i}B(x, R); B(x, R)\bigr)$.
	\end{enumerate}
	Set $\delta \coloneqq 4\delta_{\textup{H}}L \in (0, 1)$. 
	Since $B(x, 4\delta_{\textup{H}}R) \subseteq B(x, \frac{1}{2}B)$ by $L \ge 2$, it follows from Corollary \ref{cor.pGCL-gamma.useful} that there exists $c > 0$ depending only on the constants associated with the assumptions such that
	\begin{equation}\label{EHI.pMod-path}
		\MOD_{p}^{G}(\theta_{\textup{min}}, \theta_{\textup{max}}; \delta B) \ge cR^{\hdim - \beta}. 
	\end{equation} 

	\begin{figure}\centering
		\includegraphics[height=200pt]{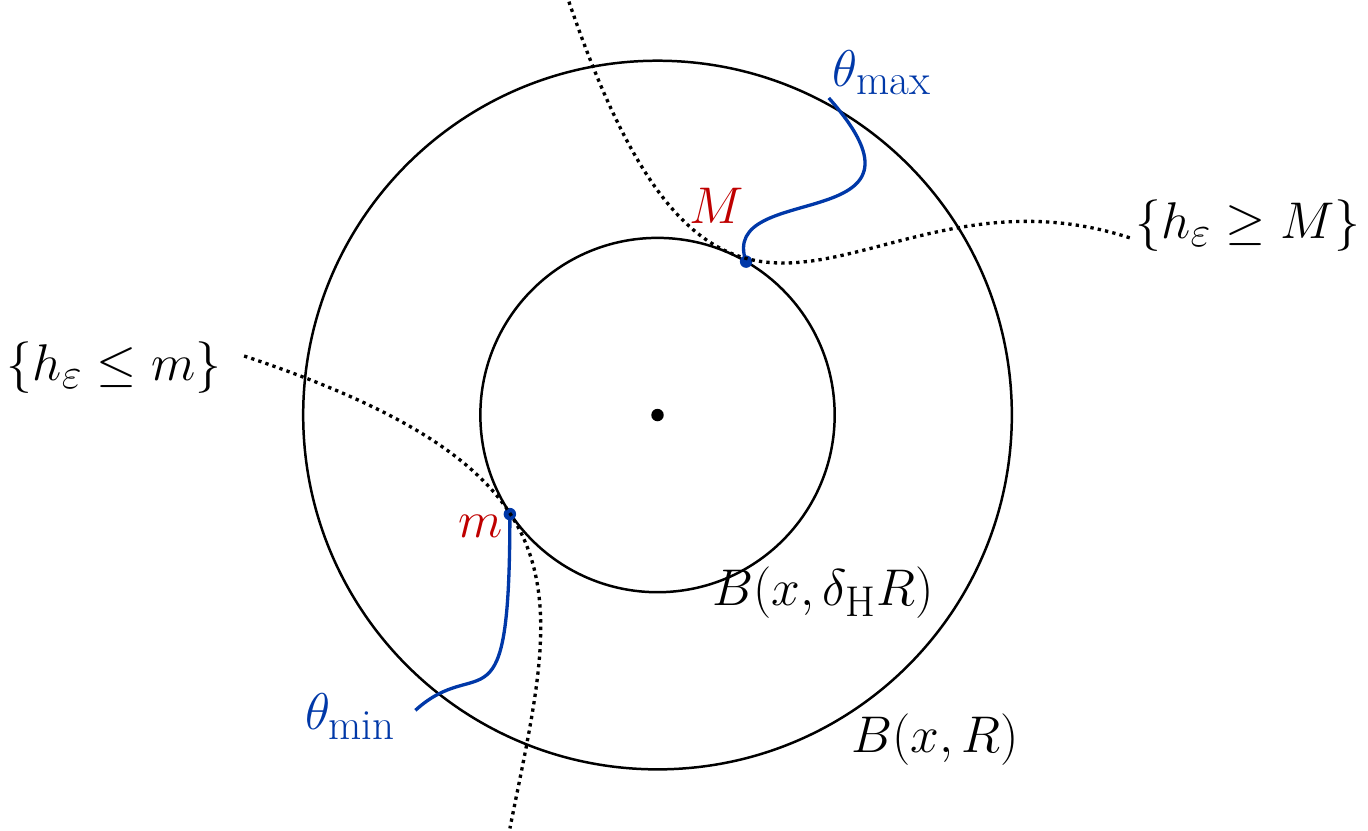}
		\caption{The paths $\theta_{\textup{min}}$ and $\theta_{\textup{max}}$}\label{fig.EHI}
	\end{figure}

	In order to show \eqref{EHI}, it suffices to consider the case $m < M$.
	Define $h_{\varepsilon}' = \frac{1}{\log{M} - \log{m}}(\log{h_{\varepsilon}} - \log{m})$ and $h_{\varepsilon}^{\ast} = (h_{\varepsilon}' \vee 0) \wedge 1$.
	Then we easily see that $\widetilde{h_{\varepsilon}^{\ast}} \in \ADM(\theta_{\textup{min}}, \theta_{\textup{max}})$, where $\widetilde{h_{\varepsilon}^{\ast}} \colon V \to [0, \infty)$ is defined as
	\[
	\widetilde{h_{\varepsilon}^{\ast}}(x) \coloneqq \max_{y \in V; \{ x, y \} \in E}\abs{h_{\varepsilon}^{\ast}(x) - h_{\varepsilon}^{\ast}(y)} \quad \text{for $x \in V$.}
	\]
	Noting that $m \ge \varepsilon > 0$, we have
	\begin{equation}\label{EHI.adm}
		\MOD_{p}^{G}(\theta_{\textup{min}}, \theta_{\textup{max}}; \delta B)
		\le C\mathcal{E}_{p, \delta B}^{G}\bigl(\widetilde{h_{\varepsilon}^{\ast}}\bigr)
		\le  C\deg(G)\left(\log{\frac{M}{m}}\right)^{-p}\mathcal{E}_{p, \delta B}^{G}(\log{h_{\varepsilon}}),
	\end{equation}
	where $C \ge 1$ is the constant in Lemma \ref{lem.mod/cap}.
	Let $\varphi$ be the equilibrium potential of $\CAP_{p}^{G}(\delta B, B^c)$ such that $\restr{\varphi}{\delta B} \equiv 1$ and $\restr{\varphi}{B^c} \equiv 0$.
	Since $h_{\varepsilon}$ is positive and $p$-harmonic in $B$, the log-Caccioppoli inequality (Corollary \ref{cor.logC}) to the tuple $(h, \varphi)$ yields
	\begin{equation}\label{EHI.logh}
		\mathcal{E}_{p, \delta B}^{G}(\log{h_{\varepsilon}}) \le C_{p}\CAP_{p}^{G}(\delta B, B^c).
	\end{equation}
	From \eqref{EHI.pMod-path}, \eqref{EHI.adm}, \eqref{EHI.logh}, \ref{cond.cap}, Lemma \ref{lem.cap} and \eqref{VD.growth.2}, we obtain
	\[
	cR^{\hdim - \beta} \le C_{p}C_{\textup{cap}}(\delta) \cdot C_{\textup{AR}}\delta^{\hdim}\deg(G)\left(\log{\frac{M}{m}}\right)^{-p}R^{\hdim - \beta},
	\]
	which implies
	\[
	\log{\frac{M}{m}} = \log\frac{\max_{\delta_{\textup{H}}B}h + \varepsilon}{\min_{\delta_{\textup{H}}B}h + \varepsilon} \le \Bigl(c^{-1}C_{p}C_{\textup{cap}}(4L\delta_{\textup{H}}) \cdot C_{\textup{AR}}(4L\delta_{\textup{H}})^{\hdim}\deg(G)\Bigr)^{1/p} \coloneqq \log{C_{\textup{H}}}.
	\]
	Hence,
	\[
	\max_{\delta_{\textup{H}}B}h + \varepsilon \le C_{\textup{H}}\Bigl(\min_{\delta_{\textup{H}}B}h + \varepsilon\Bigr).
	\]
	Since $\varepsilon > 0$ is arbitrary, \eqref{EHI} holds.
\end{proof}
\begin{rmk}\label{rmk.EHI}
	The above proof tells us that we can choose $\delta_{\textup{H}} \in (0, 1)$ arbitrarily small.
	Obviously, the constant $C_{\textup{H}}$ depends on this choice.
\end{rmk}

\subsection{H\"older continuous cut-off functions with controlled energy}
In this subsection, we  construct \emph{globally} H\"older continuous cutoff functions with controlled energy. Although energy minimizers for capacity are $p$-harmonic, the local H\"older regularity given by Corollary \ref{cor.loc-Hol} is not sufficient to conclude  the desired global  H\"older regularity asserted in Theorem \ref{thm.global.cut-off}.
This requires an additional Harnack-type estimate near boundaries.

The following theorem asserts the existence of  H\"older continuous cut-off functions with controlled energy and is the main result in this subsection.
This will in turn be used to show that our Sobolev spaces have a dense set of continuous functions.
\begin{thm}\label{thm.global.cut-off}
	Let $p \in (1,\infty)$, $\hdim \ge 1$, $\beta > 0$ and $K > 1$.
	Assume that $G = (V, E)$ satisfies \ref{cond.AR}, \hyperref[assum.BCL]{\textup{BCL$_{p}^{\textup{low}}(\hdim - \beta)$}} and \ref{cond.cap}.
	Then there exist constants $\theta_{\ast}, C_{\ast} > 0$ depending only on the constants associated to the assumptions
	such that the following hold: for any $z \in V$ and $R \ge 1$ with $B(z, KR) \neq V$, there exists a function $\varphi_{z, R} \colon V \to [0, 1]$ satisfies
	\begin{equation}\label{e.cut-off.1}
		\restr{\varphi_{z, R}}{B(z, R)} \equiv 1, \quad \supp\bigl[\varphi_{z, R}\bigr] \subseteq B(z, KR),
	\end{equation}
	\begin{equation}\label{e.cut-off.2}
		\mathcal{E}_{p}^{G}(\varphi_{z, R}) \le C_{\ast}R^{\hdim - \beta},
	\end{equation}
	and
	\begin{equation}\label{e.cut-off.3}
		\abs{\varphi_{z, R}(x) - \varphi_{z, R}(y)}
		\le C_{\ast}\left(\frac{d_{G}(x, y)}{R}\right)^{\theta_{\ast}} \quad \text{for any $x, y \in V$.}
	\end{equation}
\end{thm}
\begin{proof}
	Let $\delta_{\textup{H}} \in (0,1)$ be the constant in Theorem \ref{thm.EHI}.  
	Then we let
	\[
	\delta_{\ast} \coloneqq \frac{K - 1}{4\delta_{\textup{H}} + \delta_{\textup{H}}^{-1} + 1} \wedge \frac{K - 1}{1 + 6\delta_{\textup{H}}^{-1}} \wedge \frac{\delta_{\textup{H}}^{2}}{10} > 0,
	\]
	fix $\varepsilon \in [\frac{1}{10}\delta_{\ast}, \delta_{\ast})$, and set $R_{\ast} \coloneqq \varepsilon^{-1}$.
	The case $1 \le R \le R_{\ast}$ is easy.
	Indeed, let
	\[
	\varphi_{z, R}(x) \coloneqq \left(\frac{\lceil KR \rceil - d_{G}(z, x)}{\lceil KR \rceil - \lfloor R \rfloor}\right)^{+} \wedge 1.
	\]
	Then it is immediate that $\varphi_{z, R}$ satisfies \eqref{e.cut-off.1}.
	Furthermore, we see that
	\begin{align*}
		\mathcal{E}_{p}^{G}(\varphi_{z, R})
		&\le \bigl(\lceil KR \rceil - \lfloor R \rfloor\bigr)^{-p}\mathcal{E}_{p, B(z, KR)}^{G}\bigl(d_{G}(z, \cdot)\bigr) \\
		&\le \bigl(\lceil KR \rceil - \lfloor R \rfloor\bigr)^{-p}\deg(G)\#B(z, KR) \\
		&\le C_{\textup{AR}}K^{\hdim}\deg(G)R^{\hdim}
		\le C_{\textup{AR}}K^{\hdim}\deg(G)R_{\ast}^{\beta} \cdot R^{\hdim - \beta},
	\end{align*}
	and that
	\begin{align*}
		\abs{\varphi_{z, R}(x) - \varphi_{z, R}(y)}
		\le \frac{\abs{d_{G}(z, x) - d_{G}(z, y)}}{\lceil KR \rceil - \lfloor R \rfloor} \le d_{G}(x, y)
		\le R_{\ast}\frac{d_{G}(x, y)}{R}.
	\end{align*}

	Hereafter, we consider the case $R \ge R_{\ast}$.
	Define
	\[
	D \coloneqq B(z, KR) \setminus \left(\bigcup_{w \in \partial_{i}B(z, KR)}B\bigl(w, 2\varepsilon\delta_{\textup{H}}^{-1}R\bigr)\right),
	\]
	and let $\varphi = \varphi\big|_{z, R}$ be the equilibrium potential with respect to $\CAP_{p}^{G}\bigl(B(z, R), D^{c}\bigr)$ satisfying $\varphi_{B(z, R)} \equiv 1$ and $\supp[\varphi] \subseteq D$.
	(The condition $B(z, KR) \neq V$ implies $\partial_{i}B(z, KR) \neq \emptyset$.)
	For any $w \in \partial_{i}B(z, KR)$ and $y \in B(w, 2\varepsilon\delta_{\textup{H}}^{-1}R)$,
	\begin{align*}
		d_{G}(z, y)
		\ge d_{G}(z, w) - d_{G}(w, y)
		&> \lfloor KR \rfloor - 2\varepsilon\delta_{\textup{H}}^{-1}R \\
		&\ge (K - R^{-1} - 2\varepsilon\delta_{\textup{H}}^{-1})R \ge (K - \varepsilon - 2\varepsilon\delta_{\textup{H}}^{-1})R,
	\end{align*}
	which implies $B(z, K'R) \subseteq D$, where $K' \coloneqq K'(\varepsilon, \delta_{\textup{H}}, K) \coloneqq K - \varepsilon - 2\varepsilon\delta_{\textup{H}}^{-1} > 1$.
	Here we used $\varepsilon \le (K - 1)/(1 + 6\delta_{\textup{H}}^{-1}) < (K - 1)/(1 + 2\delta_{\textup{H}}^{-1})$ to ensure that $K' > 1$.
	By Lemma \ref{lem.cap-mono}, \ref{cond.cap}, \ref{cond.AR} and Lemma \ref{lem.cap},
	\[
	\mathcal{E}_{p}^{G}(\varphi) = \CAP_{p}^{G}\bigl(B(z, R), D^{c}\bigr) \le \CAP_{p}^{G}\bigl(B(z, R), B(z, K'R)^{c}\bigr) \le C'R^{\hdim - \beta},
	\]
	where $C' > 0$ depends only on the constants associated to the assumptions.

	The rest is proving \eqref{e.cut-off.3}.
	We shall prove that there exist constants $C, \theta > 0$ depending only on the constants associated with the assumptions such that
	\begin{equation}\label{e.Hol.eachball}
		\abs{\varphi(x) - \varphi(y)} \le C\left(\frac{d_{G}(x, y)}{R}\right)^{\theta} \quad \text{for all $z' \in D$ and $x, y \in B(z', \varepsilon R)$.}
	\end{equation}
	Fix $z' \in D$ and set $B_{\ast} \coloneqq B(z', 2\varepsilon R)$.
	We consider the following three cases.

\noindent
\underline{\textbf{Case 1:} $\delta_{\textup{H}}^{-1}B_{\ast} \subseteq D \setminus B(z, R)$.}\,
Note that $\osc_{V}[\varphi] = 1$ and that $\varphi$ is $p$-harmonic on $\delta_{\textup{H}}^{-1}B_{\ast}$. The estimate \eqref{e.Hol.eachball} follows from Corollary \ref{cor.loc-Hol}.

\noindent
\underline{\textbf{Case 2:} $\delta_{\textup{H}}^{-1}B_{\ast} \cap B(z, R) \neq \emptyset$.}\,
Since $\diam\bigl(\delta_{\textup{H}}^{-1}B_{\ast}\bigr) \le 4\varepsilon\delta_{\textup{H}}^{-1} < K' - 1$ by $\varepsilon < (K - 1)/(1 + 6\delta_{\textup{H}}^{-1})$, we have from $\delta_{\textup{H}}^{-1}B_{\ast} \cap B(z, R) \neq \emptyset$ that $\delta_{\textup{H}}^{-1}B_{\ast} \subseteq B(z, K'R) \subseteq D$.
If $B_{\ast} \subseteq B(z, R)$, then
\[
\max_{x, y \in B_{\ast}}\abs{\varphi(x) - \varphi(y)} = \abs{1 - 1} = 0.
\]
In the rest of this part, we suppose $B(z, R) \setminus B_{\ast} \neq \emptyset$.
Define
\[
m_{\ast} \coloneqq \min_{B_{\ast}}\varphi\quad \text{and} \quad M_{\ast} \coloneqq \max_{B_{\ast}}\varphi.
\]
Clearly, $0 \le m_{\ast} \le M_{\ast} \le 1$.
By $B(z, KR) \neq V$, we note that $\partial_{i}\delta_{\textup{H}}^{-1}B_{\ast} \neq \emptyset$.
Since $\varphi$ is $p$-superharmonic on $D$, by the minimum principle (Lemma \ref{lem.max/min}), we can seek a path $\gamma_{\text{min}}$ in $G$ satisfying
\[
\gamma_{\textup{min}} \in \PATH(\partial_{i}B_{\ast}, \partial_{i}\delta_{\textup{H}}^{-1}B_{\ast}; \delta_{\textup{H}}^{-1}B_{\ast}) \quad \text{and} \quad \gamma_{\textup{min}} \subseteq \{ \varphi \le m_{\ast} \}.
\]
Since
\[
\diam B_{\ast} + \mathrm{rad}\bigl(\delta_{\textup{H}}^{-1}B_{\ast}\bigr) \le \bigl(4 + \delta_{\textup{H}}^{-1}\bigr)\varepsilon R < \frac{\delta_{\textup{H}}}{2} \cdot R < R,
\]
where we used $\varepsilon < \delta_{\textup{H}}^{2}/10 < \delta_{\textup{H}}^{2}/(2+ 8\delta_{\textup{H}})$ to ensure $(4 + \delta_{\textup{H}}^{-1})\varepsilon < 2^{-1}\delta_{\textup{H}}$, we obtain  $z \not\in \delta_{\textup{H}}^{-1}B_{\ast}$.
This together with $\restr{\varphi}{B(z, R)} \equiv \max_{V}\varphi = 1$ deduces that there exists a path $\gamma_{\textup{max}}$ in $G$ such that
\[
\gamma_{\textup{max}} \in \PATH(\partial_{i}B_{\ast}, \partial_{i}\delta_{\textup{H}}^{-1}B_{\ast}; \delta_{\textup{H}}^{-1}B_{\ast}) \quad \text{and} \quad \gamma_{\textup{max}} \subseteq \{ \varphi \ge M_{\ast} \},
\]
where we used the maximum principle on $D \setminus B(z, R)$ (Lemma \ref{lem.max/min}) if necessary.
Indeed, for any $x_{0} \in \partial_{i}B(z, R) \cap \delta_{\textup{H}}^{-1}B_{\ast}$, we can easily find a path $\gamma_{0} \in \PATH(\{ x_0 \}, \partial_{i}\delta_{\textup{H}}^{-1}B_{\ast}; \delta_{\textup{H}}^{-1}B_{\ast})$, which automatically satisfies $\gamma_{0} \subseteq \{ \varphi = 1 \} \subseteq \{ \varphi \ge M_{\ast} \}$.
If $B_{\ast} \cap B(z, R) \neq \emptyset$, then $\gamma_{\textup{max}} = \gamma_{0}$ is enough.
Suppose $B_{\ast} \cap B(z, R) = \emptyset$.
Since $\varphi$ is $p$-harmonic on $\delta_{\textup{H}}^{-1}B_{\ast} \setminus B(z, R)$, an application of the maximum principle yields a path $\gamma_{1} \in \PATH(\partial_{i}B_{\ast}, \partial B(z, R); \delta_{\textup{H}}^{-1}B_{\ast})$ satisfying $\gamma_{1} \subseteq \{ \varphi \ge M_{\ast} \}$.
Let us denote the endpoint of $\gamma_{1}$ in $\partial B(z, R)$ by $x_{1}$.
By choosing $x_{0} \in \partial_{i}B(z, R) \cap \delta_{\textup{H}}^{-1}B_{\ast}$ so that $\{ x_{0}, x_{1} \} \in E$, we get the desired path $\gamma_{\textup{max}}$ by concatenating $\gamma_{0}, \{ x_{0}, x_{1} \}$ and $\gamma_{1}$.

Using these paths $\gamma_{\textup{min}}$ and $\gamma_{\textup{max}}$, we can carry out the same argument as in the proof of Theorem \ref{thm.EHI}.
Indeed, since $\varphi$ is a non-negative $p$-superharmonic function on $D$, the log-Caccioppoli inequality (Corollary \ref{cor.logC}) yields
\[
\mathcal{E}_{p, B_{\ast}}^{G}(\log{\varphi}) \le C_{p}\CAP_{p}^{G}\bigl(B_{\ast}, \bigl(\delta_{\textup{H}}^{-1}B_{\ast}\bigr)^{c}\bigr).
\]
Similar to Theorem \ref{thm.EHI}, we can obtain
\begin{equation*}
	\max_{B_{\ast}}\varphi \le C_{\textup{H}}\min_{B_{\ast}}\varphi,
\end{equation*}
where $C_{\textup{H}}$ is the constant in Theorem \ref{thm.EHI}.
The desired H\"older regularity \eqref{e.cut-off.3} follows from the above Harnack inequality using the standard Moser's oscillation lemma argument similar to Corollary \ref{cor.loc-Hol}.

\noindent
\underline{\textbf{Case 3:} $\delta_{\textup{H}}^{-1}B_{\ast} \cap D^{c} \neq \emptyset$.}\,
Let us consider $1 - \varphi$ instead of $\varphi$.
Note that $\osc_{A}[\varphi] = \osc_{A}[1 - \varphi]$ for any subset $A \subseteq V$ and that $1 - \varphi$ is a non-negative $p$-superharmonic function on $B(z, R)^c$.
For $x \in \delta_{\textup{H}}^{-1}B_{\ast}$ and $y \in \delta_{\textup{H}}^{-1}B_{\ast} \cap D^{c}$, we have
\begin{align*}
	d_{G}(z, x) \ge d_{G}(z, y) - d_{G}(y, x) \ge K'R - 4\varepsilon\delta_{\textup{H}}^{-1}R = (K - \varepsilon - 6\delta_{\textup{H}}^{-1}\varepsilon)R \ge R.
\end{align*}
Here we used $\varepsilon < (K - 1)/(1 + 6\delta_{\textup{H}}^{-1})$ to ensure $K - \varepsilon - 5\delta_{\textup{H}}^{-1} \ge 1$.
In particular, $B(z, R) \cap \delta_{\textup{H}}^{-1}B_{\ast} = \emptyset$.
Also, we observe from the definition of $D$ that $\delta_{\textup{H}}^{-1}B_{\ast} \cap \partial_{i}B(z, KR) = \emptyset$ and thus $\delta_{\textup{H}}^{-1}B_{\ast} \subseteq B(z, KR)$.
Indeed, if there exists $x \in \delta_{\textup{H}}^{-1}B_{\ast} \cap \partial_{i}B(z, KR)$, then $d_{G}(x, z') < 2\varepsilon\delta_{\textup{H}}^{-1}R$, i.e. $z' \in B(x, 2\varepsilon\delta_{\textup{H}}^{-1}R)$.
This is a contradiction since $x \in \partial_{i}B(z, KR)$ and $z' \in D \subseteq B(z, KR) \setminus B(x, 2\varepsilon\delta_{\textup{H}}^{-1}R)$.

Similar to Case 2, we define
\[
m_{\ast} \coloneqq \min_{B_{\ast}}(1 - \varphi) \quad \text{and} \quad M_{\ast} \coloneqq \max_{B_{\ast}}(1 - \varphi).
\]
Then, by the minimum principle (Lemma \ref{lem.max/min}), we can seek a path $\sigma_{\textup{min}}$ in $G$ such that
\[
\sigma_{\textup{min}} \in \PATH\bigl(\partial_{i}B_{\ast}, \partial_{i}\delta_{\textup{H}}^{-1}B_{\ast}; \delta_{\textup{H}}^{-1}B_{\ast}\bigr) \quad \text{and} \quad \sigma_{\textup{min}} \subseteq \{ 1 - \varphi \le m_{\ast} \}.
\]
Since $\delta_{\textup{H}}^{-1}B_{\ast} \cap D^{c} \neq \emptyset$ and we know that $1 - \varphi$ takes its maximum on $D^{c}$, by using maximum principle if necessary, we can find a path $\sigma_{\textup{max}}$ such that
\[
\sigma_{\textup{max}} \in \PATH\bigl(\partial_{i}B_{\ast}, \partial_{i}\delta_{\textup{H}}^{-1}B_{\ast}; \delta_{\textup{H}}^{-1}B_{\ast}\bigr) \quad \text{and} \quad \sigma_{\textup{max}} \subseteq \{ 1 - \varphi \ge M_{\ast} \}.
\]
Indeed, we can construct $\sigma_{\textup{max}}$ as follows.
If $B_{\ast} \subseteq D$, then, by an application of the maximum principle (Lemma \ref{lem.max/min}) , we can get a path $\sigma_{1}$ such that
\[
\sigma_{1} \in \PATH\bigl(\partial_{i}B_{\ast}, \partial_{i}D \cap \delta_{\textup{H}}^{-1}B_{\ast}, \delta_{\textup{H}}^{-1}B_{\ast}\bigr) \quad \text{and} \quad \sigma_{1} \subseteq \{ 1 - \varphi \ge M_{\ast} \}.
\]
Since the endpoint of $\sigma_{1}$, say $x_{1}$, is in $\partial_{i}D \cap \delta_{\textup{H}}^{-1}B_{\ast}$, there exist $w \in \partial_{i}B(z, KR)$ and $y_{1} \in B\bigl(w, 2\varepsilon\delta_{\textup{H}}^{-1}R\bigr)$ satisfying $\{ x_{1}, y_{1} \} \in E$.
By concatenating $\sigma_{1}$, $\{ x_{1}, y_{1} \}$ and a path joining $y_{1}$ and $w$ in $B\bigl(w, 2\varepsilon\delta_{\textup{H}}^{-1}R\bigr)$ in a suitable way, we get a path containing the required path $\sigma_{\textup{max}}$.
If $B_{\ast} \cap D^{c} \neq \emptyset$, then a path joining $x_{2} \in \partial_{i}B_{\ast} \cap D^{c}$ and $w \in \partial_{i}B(z, KR)$ in $B\bigl(w, 2\varepsilon\delta_{\textup{H}}^{-1}R\bigr)$, where $w$ satisfies $x_{2} \in B\bigl(w, 2\varepsilon\delta_{\textup{H}}^{-1}R\bigr)$, satisfies the required properties of $\sigma_{\textup{max}}$ since $\delta_{\textup{H}}^{-1}B_{\ast} \subseteq B(z, KR)$.

The same argument as Case 2 using these paths $\sigma_{\textup{min}}$ and $\sigma_{\textup{max}}$ gives Harnack inequality for $1 - \varphi$, which in turn yields the desired H\"older regularity.
\end{proof}

\section{Sobolev space via a sequence of discrete energies}\label{sec.unif}
We consider a sequence of finite graphs that can be regarded as approximations of a metric space on a sequence of increasingly finer scales.
The Sobolev space on a metric space is then defined using this sequence of discrete energies.

\subsection{Approximating a metric space by a sequence of graphs}\label{sec.seq-graph}
We introduce our assumptions on a sequence of graphs.
\begin{definition}\label{defn.pf}
	Let $\{ \mathbb{G}_n= (V_{n}, E_{n}) \}_{n \in \mathbb{N}}$ be a sequence of finite, connected simple non-directed graphs. We say that a family of surjective maps $\{\pi_{n,k} \colon V_n \to V_k \mid 1 \le k < n, (n,k) \in \mathbb{N}^{2} \}$ is \emph{projective} if $\pi_{n,k}$ is surjective for all $k <n$ and
	\[
	\pi_{l,k} \circ \pi_{n,l}= \pi_{n,k}, \quad \mbox{for all $k<l<n$ with $k,l,n \in \mathbb{N}$.}
	\]
	Given $\{ \mathbb{G}_n \}_{n \in \mathbb{N}}$ and a projective family of maps $\{\pi_{n,k}: k<n\}$, we say that a sequence of probability measures $\{ \measure_n \in \mathcal{P}(V_n) \}_{n \in \mathbb{N}}$, where $\mathcal{P}(V_n)$ denotes the set of probability measure on $V_{n}$, is \emph{consistent} if
	\[
	(\pi_{n,k})_* \measure_n = \measure_k \quad \mbox{for all $k<n$.}
	\]
	Given a sequence of finite connected graphs $\{ \mathbb{G}_n \}_{n \in \mathbb{N}}$, a projective family of maps $\{\pi_{n,k} \mid k < n \}$, and a consistent family of probability measures $\{ \measure_n \}_{n \in \mathbb{N}}$, we say that a sequence of functions $\{ f_n \colon V_n \to \bR \}_{n \in \bN}$ is \emph{conditional} with respect to $\{ \measure_n \}_{n \in \bN}$ if
	\begin{equation} \label{e:conditional}
		f_k(v)= \frac{1}{\measure_k(v)}\sum_{w \in \pi_{n,k}^{-1}(\{v\})} f_n(w) \measure_n(w) \quad \mbox{for all $k<n, v \in V_k$}.
	\end{equation}
	Equivalently, $f_k$ is the conditional expectation $f_k(v) = \mathbb{E}_{\measure_n}[f_n(W)\mid\pi_{n,k}(W)=v]$, where $\measure_n$ is the law of $W$.
\end{definition}
In the above definition, the graphs $\mathbb{G}_n$ can be regarded as approximating a metric space $(K,d)$ at a sequence of increasingly finer scales, while the measures $m_n$ can be considered to approximate a measure $m$ on $K$. A conditional sequence of functions can be considered to approximate a function $f$ on the metric space $(K,d)$.

The sequence of measures $\{ \measure_n \}_{n \in \mathbb{N}}$ in the above definition is often assumed to satisfy the condition given by the following definition.
\begin{definition} \label{d:rough-uniform}
	Let $\{ \measure_n \in \mathcal{P}(V_n) \}_{n \in \mathbb{N}}$ be a sequence of probability measures on a  family of finite sets $V_n$.
	We say that such a sequence  $\{ \measure_n \}_{n \in \mathbb{N}}$ is \emph{roughly uniform} if there exists $C_{\textup{u}} \ge 1$ such that
	\begin{equation}\label{e:ru}
		C_{\textup{u}}^{-1}\measure_n(v) \le \frac{1}{\# V_n} \le C_{\textup{u}}\measure_n(v), \quad \mbox{for all $n \in \bN, v \in V_n$.}
	\end{equation}
\end{definition}

We introduce a geometric condition on the sequence of graphs which relates different graphs in the sequence.
Roughly speaking, the following condition states that $\diam(\mathbb{G}_n)$ grows like $R_*^n$ and $\pi_{n+k,k}^{-1}(w)$ are `roundish' in an uniform fashion; that is $\pi_{n+k,k}^{-1}(w)$ behave like balls in the graph $\mathbb{G}_{n+k}$ for all $w \in V_k$.
\begin{definition} \label{d:scaled}
	Let $R_* \in (1,\infty)$, let $\{ \mathbb{G}_n=(V_n,E_n) \}_{n \in \mathbb{N}}$ be a sequence of finite, connected simple non-directed graphs, and let $\{\pi_{n,k} \colon V_n \to V_k \mid 1 \le k < n\}$ be a family of projective maps.
	We say that the sequence of graphs $\{ \mathbb{G}_n \}_{n \in \mathbb{N}}$ equipped with the projective maps $\{ \pi_{n,k} \colon V_n \to V_k \mid k < n \}$ is $R_*$-scaled if there exist constants $A_1,A_2 \in (1,\infty)$ so that the following holds: for any $n, k \in \mathbb{N}$, for all $w \in V_k$, there exists $c_{n}(w) \in V_{n+k}$ such that
	\begin{equation} \label{e:sc1}
		B_{d_{n+k}}(c_n(w),A_1^{-1}R_*^n) \subset \pi_{n+k,k}^{-1}(w) \subset B_{d_{n+k}}(c_n(w),A_1R_*^n)
	\end{equation}
 	and
 	\begin{equation} \label{e:sc2}
 		d_{n+k}(c_n(w),c_n(w')) \le A_2 R_*^n \quad \mbox{whenever $w,w' \in V_k$ satisfy $d_k(w,w') = 1$,}
 	\end{equation}
 	where $d_n$ denotes the graph distance of $\mathbb{G}_n$.
\end{definition}

We next discuss discrete approximations of a metric space.
Any compact metric space can be approximated by a sequence of graphs on increasing finer scales.
This idea is present in various (closely related) notions such as hyperbolic filling \cite{BBS22, BP03, BS18, BS,  Ele97}, $K$-approximation \cite{BK02}, quasi-visual approximation \cite{BM22}, generalized dyadic cubes \cite{Chi90, Dav88, HK12, KRS12, Sas23}, and partitions of a metric space indexed by tree \cite{Kig20}.
The following definition describes yet another way in which a sequence of graphs `approximate' a compact metric space.
\begin{definition}[compatibility] \label{d:compatible}
	Consider a compact metric space $(K,d)$   and let $R_* \in (1,\infty), \theta \in (0,1]$.
	Let $\{ \mathbb{G}_n=(V_n,E_n) \}_{n \in \bN}$ be a sequence of finite, connected simple non-directed graphs and let $\{\pi_{n,k} \colon V_n \to V_k \mid 1 \le k < n\}$ be a family of projective maps.
	Let $d_n \colon V_n \times V_n \to \bZ_{\ge 0}, n \in \bN$ denote the corresponding graph metrics.
	We say that $\{ \mathbb{G}_n \}$ along with $\{\pi_{n,k}: V_n \to V_k \mid 1 \le k < n\}$ is \emph{$R_{\ast}$-compatible} with $(K,d)$ if there exists a sequence of maps $\{ p_n \colon V_n \to K \}_{n \in \bN}$, a collection of Borel set $\bigl\{\wt{K}_v \mid v \in V_n, n \in \bN \bigr\}$ and $C \in [1,\infty)$ such that the following hold:
	\begin{enumerate}[(i)]
		\item\label{it:compat.comp} (comparison of metrics)
		\begin{equation} \label{e:holder}
			C^{-1}\frac{d_n(x,y)}{R_*^n} \le d(p_n(x),p_n(y)) \le C \frac{d_n(x,y)}{R_*^n}
		\end{equation}
		for all $x,y \in V_n$ and for all $n \in \bN$.
		\item\label{it:compat.parti} (partition) For all $n \in \mathbb{N}$, the collection of sets $\bigl\{\wt{K}_v \bigr\}_{v \in V_n}$ form a partition of $K$; that is $\bigcup_{v \in V_n}\wt{K}_v = K$ and $\wt{K}_u \cap \wt{K}_w = \emptyset$ for all $u,w \in V_n$ with $u \neq w$.
		\item\label{it:compat.proj} (compatibility with projections)
		For all $1 \le k <n$ and for all $v \in V_k$, we have
		\begin{equation} \label{e:pcompatible}
			\wt{K}_v = \bigcup_{w \in \pi_{n,k}^{-1}(v)} \wt{K}_w.
		\end{equation}
		\item\label{it:compat.round} (roundness of partition) For all $n \in \bN, v \in V_n$, we have
		\begin{equation} \label{e:round}
			B_d(p_n(v), C^{-1}R_*^{-n})\subset	\wt{K}_v \subset B_d(p_n(v),C R_*^{-n}).
		\end{equation}
	\end{enumerate}
\end{definition}

Note that \eqref{e:holder} implies that the points $\{p_n(v) \mid v \in V_n\}$ are $C^{-1} R_*^{-n}$-separated and that $\diam(V_n,d_n) \asymp R_*^n$.

We introduce a uniform notion of \ref{cond.AR} for a sequence of graphs.
\begin{defn}
  We shall say that the sequence $\{ \mathbb{G}_{n} \}_{n \in \mathbb{N}}$ satisfies \emph{$\hdim$-Ahlfors regularity condition uniformly}, \ref{cond.UAR} for short, if there exists $C_{\textup{AR}} \ge 1$ such that for any $n \in \mathbb{N}$, $x \in V_{n}$, $R \in [1, \diam(\mathbb{G}_{n})+1)$,
	\begin{equation}\label{cond.UAR}
		C_{\textup{AR}}^{-1}\,R^{\hdim} \le \#B_{d_{n}}(x, R) \le C_{\textup{AR}}\,R^{\hdim}. \tag{\textup{U-AR($\hdim$)}}
	\end{equation}
\end{defn}

The following elementary lemma explains the relationship between a metric space and a sequence of graphs approximating it in the sense of Definition \ref{d:compatible} and the notions in Definitions \ref{defn.pf} and \ref{d:rough-uniform}.
\begin{lem} \label{l:bp}
	Let $(K,d)$ be a compact metric space and let $\measure$ be a $d_f$-Ahlfors regular probability measure on $(K,d)$.
	Let $\{ \mathbb{G}_n=(V_n,E_n) \}_{n \in \mathbb{N}}$ be a sequence of finite, connected simple non-directed graphs and let $\{\pi_{n,k}: V_n \to V_k \mid 1 \le k <n \}$ be a projective family of maps.
	Suppose that $\{ \mathbb{G}_n \}$ along with $\{ \pi_{n,k} \mid 1 \le k < n \}$ is $R_*$-compatible with $(K,d)$.
	Let $\bigl\{ \wt{K}_{v} \in \mathcal{B}(K) \bigm| v \in V_n, n \in \mathbb{N} \bigr\}$ be a collection of Borel sets as given in Definition \ref{d:compatible}.
	Let
	\[
	\measure_n(v) \coloneqq \measure(\wt{K}_v)
	\]
	for all $n \in \mathbb{N}, v \in V_n$.
	Then
	\begin{enumerate}[\rm(i)]
		\item  The sequence of graphs $\{ \mathbb{G}_n \}$ satisfies \ref{cond.UAR}.
		\item  The family of measures $\{ \measure_n \}$ is roughly uniform, and is consistent with respect to $\{\pi_{n,k} \mid 1 \le k < n \}$.
		\item For any $f \in L^1(K,\measure)$, the family of functions $M_nf \colon V_n \to \mathbb{R}$ defined by
		\begin{equation}\label{e:Mn}
			(M_nf)(v) = \frac{1}{\measure(\wt{K}_v)} \int_{\wt{ K}_v} f\,d\measure, \quad \mbox{for all $n \in \mathbb{N}, v \in V_n$,}
		\end{equation}
		is conditional with respect to $\{ \measure_n \}$ and $\{ \pi_{n,k} \mid 1 \le k < n \}$.
	\end{enumerate}
\end{lem}
The operator $M_n$ converts a function on $K$ to a function on $V_n$. We would sometimes like to construct functions on $K$ using functions on $V_n$ by defining
\begin{equation} \label{e:Jn}
	J_nf(\cdot) \coloneqq \sum_{v \in V_n} f(v) \indicator{\wt{K}_v}(\cdot),\quad \mbox{for all $f \colon V_n \to \mathbb{R}, n \in \mathbb{N}$.}
\end{equation}

\subsection{Hypotheses on a sequence of graphs}\label{sec.hyp-seq-graph}
A sequence of graphs approximating a metric space often satisfies some analytic properties in an uniform manner.
To this end, we   introduce uniform versions of \emph{analytic conditions} such as   \ref{cond.cap}, \ref{cond.BCL}, and \ref{cond.PI}.
\begin{defn}
	Let $\{ \mathbb{G}_{n} = (V_n, E_n) \}_{n \in \mathbb{N}}$ be a sequence of finite, connected simple non-directed graphs and let $d_{n}$ be the graph metric of $\bG_n$.
	Let $p \in (1, \infty)$, $\hdim > 0$, $\beta > 0$ and $\zeta \in \mathbb{R}$.
	\begin{enumerate}[\rm(1)]
		\item We shall say that the sequence $\{ \mathbb{G}_{n} \}_{n \in \mathbb{N}}$ satisfies \emph{$p$-capacity upper bound with order $\beta$ uniformly}, \ref{cond.Ucap} for short, if there exist constants $C_{\textup{cap}} > 0$ and $A_{\textup{cap}} \ge 1$ such that for any $n \in \mathbb{N}$, $x \in V_{n}$ and $R \in [1, \diam(\mathbb{G}_{n})/A_{\textup{cap}})$,
		\begin{equation}\label{cond.Ucap}
			\CAP_{p}^{G_{n}}\bigl(B_{d_{n}}(x, R), B_{d_{n}}(x, 2R)^{c}\bigr) \le C_{\textup{cap}}\frac{\#B_{d_{n}}(x, R)}{R^{\beta}}. \tag{\textup{U-cap$_{p, \le}(\beta)$}}
		\end{equation}
		\item We shall say that the sequence $\{ \mathbb{G}_{n} \}_{n \in \mathbb{N}}$ satisfies \emph{ball combinatorial $p$-Loewner property with order $\zeta$ uniformly}, \ref{cond.UBCL} for short, if there exists $A \ge 1$ such that the following hold: for any $\kappa > 0$ there exist $c_{\UBCL}(\kappa) > 0$ and $L_{\UBCL}(\kappa) > 0$ such that
		\begin{equation}\label{cond.UBCL}
			\MOD_{p}^{\mathbb{G}_{n}}(\{ \theta \in \PATH_{\mathbb{G}_{n}}(B_1, B_2) \mid \diam(\theta, d_n) \le L_{\UBCL}(\kappa)R \}) \ge c_{\UBCL}(\kappa)R^{\zeta} \tag{\textup{U-BCL$_{p}(\zeta)$}}
		\end{equation}
		whenever $n \in \mathbb{N}$, $R \in [1, \diam(\mathbb{G}_{n})/A)$ and $B_{i} \, (i = 1, 2)$ are balls in $\mathbb{G}_{n}$ with radii $R$ satisfying $\dist_{d_{n}}(B_1, B_2) \le \kappa R$.
		We also say that $\{ \mathbb{G}_{n} \}_{n \in \mathbb{N}}$ satisfies \hyperref[cond.UBCL]{\textup{U-BCL$_{p}^{\textup{low}}(\zeta)$}} if $\{ \mathbb{G}_{n} \}_{n \in \mathbb{N}}$ satisfies \ref{cond.UBCL} with $\zeta < 1$.
		\item We shall say that the sequence of graphs $\{ \mathbb{G}_n \}_{n \in \mathbb{N}}$ satisfies \emph{$(p,p)$-Poincar\'{e} inequality with order $\beta$ uniformly}, \ref{cond.UPI} for short, if there exist constants $C_{\textup{PI}}, A_{\textup{PI}} \ge 1$ such that for any $n \in \mathbb{N}, x \in V_n$, $R \ge 1$ and $f\colon V_n \to \bR$,
		\begin{equation}\label{cond.UPI}
			\sum_{y \in B_{d_n}(x, R)}\abs{f(y) - f_{B_{d_n}(x, R)}}^{p} \le C_{\textup{PI}}R^{\beta}\mathcal{E}_{p, B_{d_n}(x, A_{\textup{PI}}R)}^{\mathbb{G}_{n}}(f). \tag{\textup{U-PI$_{p}(\beta)$}}
		\end{equation}
	\end{enumerate}
\end{defn}

Using the above definition, we can rephrase Theorem \ref{thm.PI-discrete} for a sequence of graphs as follows.
\begin{prop}\label{prop.UPI}
	Let $\{ \bG_{n} = (V_n, E_n) \}_{n \in \mathbb{N}}$ be a sequence of finite,  connected simple non-directed graphs.
	Let $p \in (1, \infty), \hdim \ge 1$ and $\beta > 0$.
	Suppose that $\{ \bG_{n} \}$ satisfies \ref{cond.UAR} and \hyperref[cond.UBCL]{\textup{U-BCL$_{p}^{\textup{low}}(\hdim - \beta)$}}.
	Then $\{ \bG_{n} \}_{n \in \mathbb{N}}$ satisfies \ref{cond.UPI} (the associated constants $C_{\mathrm{PI}} > 0$ and $A_{\mathrm{PI}} \ge 1$ depend only on the constants involved in the assumptions).
\end{prop}

The following definition gives a uniform notion of the metric doubling property for a sequence of graphs.
\begin{defn}\label{defn:unifdeg}
	Let $\{ \mathbb{G}_{n} = (V_n, E_n) \}_{n \in \mathbb{N}}$ be a sequence of finite, connected simple non-directed graphs and let $d_{n}$ be the graph metric of $\bG_n$.
	\begin{enumerate}[\rm(1)]
		\item Define $L_{\ast} \coloneqq L_{\ast}(\{ \mathbb{G}_{n} \}_{n \in \mathbb{N}}) \coloneqq \sup_{n \in \mathbb{N}}\deg(\mathbb{G}_{n})$.
		\item\label{UMD} We shall say that $\{ \mathbb{G}_{n} \}_{n \in \mathbb{N}}$ is \emph{uniformly metric doubling}, \hyperref[UMD]{\textup{U-MD}} for short, if there exists $N_{\textup{D}} \ge 2$ such that given $n \in \mathbb{N}$, $x \in V_{n}$, $R \ge 1$ there exist $y_{1}, \dots, y_{N} \in V_{n}$ satisfying $B_{d_{n}}(x, R) \subseteq \bigcup_{i = 1}^{N_{\textup{D}}}B_{d_{n}}(y_{i}, R/2)$.
	\end{enumerate}
\end{defn}

Then the following property is an easy consequence of Remark \ref{rem.doubling}.
\begin{lem}\label{lem.U-geom}
	Let $\{ \mathbb{G}_{n} \}_{n \in \mathbb{N}}$ be a sequence of graphs satisfying \ref{cond.UAR} for some $\hdim > 0$.
	Then $L_{\ast} < \infty$ and $\{ \mathbb{G}_{n} \}_{n \in \mathbb{N}}$ is \hyperref[UMD]{\textup{U-MD}}.
	In addition, the doubling constant $N_{\textup{D}}$ can be chosen so that $N_{\textup{D}}$ depends only on $C_{\textup{AR}}$.
\end{lem}

In order to state a version of Theorem \ref{thm.global.cut-off} for a sequence of graphs, we introduce the following definition.
\begin{definition}\label{cond.UCF}
	Let $\{ \mathbb{G}_n=(V_n,E_n) \}_{n \in \mathbb{N}}$ be a sequence of finite, connected simple non-directed graphs.
	Let $p \in (1,\infty), \beta>0, \vartheta \in (0,1]$.
	We say that the sequence of graphs $\{ \mathbb{G}_n \}$  satisfies \hyperref[cond.UCF]{\textup{U-CF$_{p}(\vartheta, \beta)$}} if there exists $C_* \in (0,\infty)$ so that the following holds: for all $n \in \bN, v \in V_n, R \ge 1$ there exists $\varphi_{v,R} \colon V_{n} \to [0, 1]$, so that
	\begin{align}
		\restr{\varphi_{v, R}}{B_{d_n}(v, R)} &\equiv 1, \quad \supp[\varphi_{v, R}] \subseteq B_{d_{n}}(v, 2R) \label{e:cf} \\
		\mathcal{E}_p^{\mathbb{G}_n}(\varphi_{v,R}) &\le C_*\frac{\#B_{d_n}(v, R)}{R^{\beta}}, \label{e:cfe} \\
		\abs{\varphi_{v,R}(x) - \varphi_{v,R}(y)} &\le C_* \left( \frac{d_n(x,y)}{R}\right)^\vartheta \quad \mbox{for all $x,y \in V_n$.} \label{e:cfr}
	\end{align}
\end{definition}
The next result provides a family of  H\"older continuous cut-off functions whose energies are controlled in a uniform manner. This is an immediate consequence of Theorem \ref{thm.global.cut-off}.
\begin{prop}\label{prop.UCF}
	Let $\{ \bG_{n} = (V_n, E_n) \}_{n \in \mathbb{N}}$ be a sequence of finite, connected simple non-directed graphs.
	Let $p \in (1, \infty), \hdim \ge 1$ and $\beta > 0$.
	Suppose that $\{ \bG_{n} \}$ satisfies \ref{cond.UAR}, \hyperref[cond.UBCL]{\textup{U-BCL$_{p}^{\textup{low}}(\hdim - \beta)$}} and \ref{cond.Ucap}.
	Then $\{ \bG_{n} \}$ satisfies \hyperref[cond.UCF]{\textup{U-CF$_{p}(\vartheta, \beta)$}} (the associated constants $C_*, \vartheta > 0$ depend only on the constants involved in the assumptions).
\end{prop}

We would like to define $p$-energy as limit of re-scaled discrete energies. The following result suggests the re-scaling factor.
The main result of this section is the weak monotonicity of energy.
\begin{theorem} \label{t:wm}
	Let $\{ \bG_n=(V_n,E_n) \}_{n \in \mathbb{N}}$ be a sequence of finite, connected simple non-directed graphs equipped with the projective maps $\{\pi_{n,k} \colon V_n \to V_k; k < n\}$ and let $\{ \measure_n \in \mathcal{P}(V_n) \}_{n \in \mathbb{N}}$ be a consistent sequence of probability measures.
	Suppose that $\{ \mathbb{G}_{n} \}$ along with $\{ \pi_{n, k}; k < n \}$ is $R_{\ast}$-scaled for some $R_{\ast} \in (1, \infty)$ and the sequence $\{ \measure_n \}$ is roughly uniform.
	Let $p \in (1,\infty), d_f \ge 1, \beta > 0$ and we further suppose that the sequence $\{ \mathbb{G}_n \}_{n \in \bN}$ satisfies \ref{cond.UAR} and \ref{cond.UPI}.
 There exists $ C_{\textup{WM}} \in (1,\infty)$ depending only on the constants associated to the assumptions
 such that for any  conditional sequence of functions $\{ f_n \colon V_n \to \bR \}_{n \in \bN}$ (with respect to $\measure_n, \pi_{n,k}$), we have
	\begin{equation} \label{e:wm}
		\sE_p^{\bG_k}(f_k) \le C_{\textup{WM}} R_{*}^{l(\beta-\hdim)}  	\sE_p^{\bG_{k+l}}(f_{k+l}) \quad \mbox{for all $k,l \in \bN$.}
	\end{equation}
\end{theorem}
\begin{proof}
	Let $f_n \colon V_n \to \bR, n \in \bN$ denote an arbitrary conditional sequence of functions as above.
  Let $A_1,A_2 \in (1,\infty)$ be the constants as given in Definition \ref{d:scaled}, $C_u \in (1,\infty)$ be the constant in Definition \ref{d:rough-uniform}.
  Set $A_3=2A_1+A_2$.
  For any $v,w \in V_k$ such that $d_k(v,w)=1$, we have
  \begin{equation}
  	\pi_{k+l,k}^{-1}(v) \cup 	\pi_{k+l,k}^{-1}(w) \subset B_{d_{k+l}}(c_l(v),A_{3}R_*^l) \quad \mbox{(by \eqref{e:sc1} and \eqref{e:sc2}).} \label{e:wm1}
  \end{equation}
  There is $C_1  \in [1,\infty)$ depending only on the constants involved in \ref{cond.UAR}, roughly uniform, and $R_*$-scaled properties such that
  \begin{equation} \label{e:wm2}
  	C_1^{-1} R_*^{- n\hdim} \le \measure_n(v) \le C_1 R_*^{-n\hdim} \quad \mbox{for all $n \in \bN, v \in V_n$.}
  \end{equation}

 For any $v,w \in V_k$ such that $d_k(v,w)=1$ and for all $\alpha \in \bR$, we have
 \begin{align} \label{e:wm3}
 	\abs{f_k(v)-f_k(w)}  &\le 	\abs{f_k(v)-\alpha} + 	\abs{f_k(w)-\alpha}  \nonumber \\
 	&\le \abs{\sum_{v_1 \in \pi_{k+l,k}^{-1}(v)} f_{k+l}(v_1) \frac{\measure_{k+l}(v_1)}{\measure_k(v)} -  \alpha} + \abs{\sum_{w_1 \in \pi_{k+l,k}^{-1}(w)} f_{k+l}(w_1) \frac{\measure_{k+l}(w_1)}{\measure_k(w)} -  \alpha} \nonumber \\
 	&\le  \sum_{v_1 \in \pi_{k+l,k}^{-1}(v)} \frac{\measure_{k+l}(v_1)}{\measure_k(v)} \abs{f_{k+l}(v_1) -  \alpha} + \sum_{w_1 \in \pi_{k+l,k}^{-1}(w)} \frac{\measure_{k+l}(w_1)}{\measure_k(w)} \abs{f_{k+l}(w_1) -  \alpha} \nonumber \\
 	&\stackrel{\eqref{e:wm2}}{\le} C_1^2 R_*^{-l\hdim}\left( \sum_{v_1 \in \pi_{k+l,k}^{-1}(v)}   \abs{f_{k+l}(v_1) -  \alpha} + \sum_{w_1 \in \pi_{k+l,k}^{-1}(w)}  \abs{f_{k+l}(w_1) -  \alpha}  \right) \nonumber \\
 	&\stackrel{\eqref{e:wm1}}{\le} 2 C_1^2 R_*^{-l\hdim} \sum_{v_1 \in  B_{d_{k+l}}(c_l(v),A_{3}R_*^l)} \abs{f_{k+l}(v_1) -  \alpha}  \nonumber \\
 	&\lesssim \frac{1}{\# B_{d_{k+l}}(c_l(v),A_{3}R_*^l)} \sum_{v_1 \in  B_{d_{k+l}}(c_l(v),A_{3}R_*^l)} \abs{f_{k+l}(v_1) -  \alpha},
 \end{align}
 where in the last line, we used \ref{cond.UAR}.
 Let us choose $\alpha=  (f_{k+l})_{B_{d_{k+l}}(c_l(v),A_{3}R_*^l)}$ in \eqref{e:wm3} and  use Poincar\'e inequality  \ref{cond.UPI} to obtain
 \begin{align} \label{e:wm4}
 	\abs{f_k(v)-f_k(w)}^p & \lesssim \frac{1}{\# B_{d_{k+l}}(c_l(v),A_{3}R_*^l)}   \sum_{v_1 \in  B_{d_{k+l}}(c_l(v),A_{3}R_*^l)} \abs{f_{k+1}(v_1) - (f_{k+1})_{B_{d_{k+l}}(c_l(v),A_{3}R_*^l)}}^p \nonumber \\
 	&\lesssim \frac{R_*^{l \beta}}{\# B_{d_{k+l}}(c_l(v),A_{3}R_*^l)}\mathcal{E}^{\bG_{k+l}}_{p, B_{d_{k+l}}(c_l(v), A_{\textup{PI}} A_{3} R_*^l )}(f_{k+l}) \quad \mbox{(by \ref{cond.UPI})} \nonumber \\
 	&\lesssim R_*^{l (\beta-\hdim)} \mathcal{E}^{\bG_{k+l}}_{p, B_{d_{k+l}}(c_l(v), A_{\textup{PI}} A_{3} R_*^l )}(f_{k+l})
 \end{align}
 for any $v,w \in V_k$ such that $d_k(v,w)=1$.
Using Lemma \ref{lem.U-geom}, we obtain
 \begin{equation} \label{e:wm5}
 	\sE_{p}^{\bG_k}(f_k) = \sum_{ \{v,w\} \in  E_k} 	\abs{f_k(v)-f_k(w)}^p  \stackrel{\eqref{e:wm4}}{\lesssim}  R_*^{l (\beta-\hdim)} \sum_{v \in V_k} \mathcal{E}_{p, B_{d_{k+l}}(c_l(v), A_{\textup{PI}} A_{3} R_*^l )}^{\bG_{k+l}}(f_{k+l}).
 \end{equation}
 By \eqref{e:sc1}, the points $\{c_l(v) \mid v \in V_k\}$ are $2 A_1^{-1} R_*^l$-separated for all $k, l \in \bN$.
 Since $\{ \mathbb{G}_n \}_{n \in \bN}$ are \hyperref[UMD]{\textup{U-MD}} by Lemma \ref{lem.U-geom}, there exists $C_2 >1$ (depending only on $A_{\textup{PI}}, A_{1}, A_{2}$ and the constants involved in \ref{cond.UAR}) such that
 \begin{equation} \label{e:wm6}
 	\sum_{v \in V_k} \indicator{B_{d_{k+l}}(c_l(v), A_{\textup{PI}} A_{3} R_*^l )} \le C_2, \quad \mbox{for all $k,l \in \bN$.}
 \end{equation}
 The desired estimate \eqref{e:wm} follows immediately from \eqref{e:wm5} and \eqref{e:wm6}.
\end{proof}
\begin{rmk}
	In the work \cite{Kig23}, the notion of \emph{conductive homogeneity} plays an important role to develop the theory of $(1,p)$-Sobolev spaces via discretizations.
	The estimate \eqref{e:wm4} can be regarded as a variant of this condition.
\end{rmk}

\subsection{Sobolev space and cutoff functions}\label{sec.Sob}
We now explain our strategy to construct $p$-energy as a scaling limit of discrete $p$-energies in a general setting.
The following assumption guarantees that our Sobolev space satisfies good properties.
\begin{assumption}\label{a:reg}
	Let $p \in (1, \infty)$, $\hdim \in [1,\infty)$, $\beta>0$ and $\vartheta \in (0,1]$.
	Let $(K,d)$ be a connected compact metric space with $\#K \ge 2$ and let $\measure$ be a $\hdim$-Ahlfors regular probability measure on $(K,d)$.
	Let $\{ \mathbb{G}_n=(V_n,E_n) \}_{n \in \bN}$ be a sequence of finite, connected simple non-directed graphs and let $\{\pi_{n,k} \mid 1 \le k < n\}$ denote a projective family of maps.
	There exists $R_* \in (1,\infty)$ such that $\{ \mathbb{G}_n \}$ along with $\{\pi_{n,k} \}$ is $R_*$-scaled and $R_*$-compatible with $(K,d)$.
	Furthermore, $\{ \mathbb{G}_n \}$ satisfies \ref{cond.UPI} and \hyperref[cond.UCF]{\textup{U-CF$_{p}(\vartheta, \beta)$}}.
\end{assumption}
The weak monotonicity of discrete energies (Theorem \ref{t:wm}) suggests the following definition of Sobolev space.
\begin{definition} \label{d:sob}
	Under the setting of Assumption \ref{a:reg}, we define the normalized energy of $f \in L^p(K,\measure)$ for any $n \in \mathbb{N}$ and $A \subseteq V_n$ as
	\begin{equation} \label{e:normalizedE}
		\wt{\mathcal{E}}^{(n)}_{p, A}(f) \coloneqq R_*^{n(\beta-\hdim)} \mathcal{E}_{p, A}^{\mathbb{G}_n}(M_nf),
	\end{equation}
	where $M_nf$ is as given in \eqref{e:Mn}.
	For simplicity, $\wt{\mathcal{E}}^{(n)}_{p}(f) \coloneqq \wt{\mathcal{E}}^{(n)}_{p, V_n}(f)$.
	Define our $(1, p)$-Sobolev space $\mathcal{F}_{p}(K, \metric, \measure)$ by
	\begin{equation}\label{e:Fp}
		\mathcal{F}_{p}(K, \metric, \measure) \coloneqq \biggl\{ f \in L^{p}(K, \measure) \biggm| \sup_{n  \in \mathbb{N}}\widetilde{\mathcal{E}}_{p}^{(n)}(f) < \infty \biggr\}.
	\end{equation}
	We also set $\abs{f}_{\mathcal{F}_{p}(K, \metric, \measure)} \coloneqq \left(\sup_{n \in \bN} \widetilde{\mathcal{E}}_{p}^{(n)}(f)\right)^{1/p}$ and $\norm{f}_{\mathcal{F}_{p}(K, \metric, \measure)} \coloneqq \norm{f}_{L^p(\measure)} + \abs{f}_{\mathcal{F}_{p}(K, \metric, \measure)}$.
	For simplicity, we use $\mathcal{F}_{p}$ instead of $\mathcal{F}_{p}(K, \metric, \measure)$ in these notations when no confusion can occur.
\end{definition}
%
Hereafter in this section, we always assume that Assumption \ref{a:reg} holds.
Thanks to Theorem \ref{t:wm} and Lemma \ref{l:bp}, we have
\begin{equation}\label{CH-comparable}
	\liminf_{n \to \infty}\widetilde{\mathcal{E}}_{p}^{(n)}(f) \asymp \limsup_{n  \to \infty}\widetilde{\mathcal{E}}_{p}^{(n)}(f) \asymp \sup_{n \in \bN} \widetilde{\mathcal{E}}_{p}^{(n)}(f), \quad \mbox{for all $f \in L^p(K,\measure)$.}
\end{equation}
In particular,
\[
\mathcal{F}_{p}
= \biggl\{ f \in L^{p}(K, \measure) \biggm| \liminf_{n  \to \infty}\widetilde{\mathcal{E}}_{p}^{(n)}(f) < \infty \biggr\}
= \biggl\{ f \in L^{p}(K, \measure) \biggm| \limsup_{n  \to \infty}\widetilde{\mathcal{E}}_{p}^{(n)}(f) < \infty \biggr\}.
\]

Some properties of $\mathcal{F}_{p}$ are already mentioned in \cite[Section 3.2]{Kig23} in the framework of weighted partition theory developed in \cite{Kig20}.
We summarize the basic properties of the Sobolev space $(\sF_p,\norm{\,\cdot\,}_{\sF_p})$ in the following theorem.
\begin{theorem} \label{t:Fp}
	Let $(K,d)$ be a connected compact metric space with a $\hdim$-Ahlfors regular probability measure $m$ and let $\{ \mathbb{G}_n=(V_n,E_n) \}_{n \in \bN}$ be a sequence of finite connected graphs  satisfying Assumption \ref{a:reg}. Let $(\sF_p,\norm{\,\cdot\,}_{\sF_p})$ denote the normed linear space   in Definition \ref{d:sob}. Then $(\sF_p,\norm{\,\cdot\,}_{\sF_p})$ satisfies the following properties.
	\begin{enumerate}[\rm(i)]
		\item \label{t:Fp-banach} $(\sF_p,\norm{\,\cdot\,}_{\sF_p})$ is a Banach space.
		\item \label{t:Fp-ref} $(\sF_p,\norm{\,\cdot\,}_{\sF_p})$ admits an equivalent uniformly convex norm. In particular,  $(\sF_p,\norm{\,\cdot\,}_{\sF_p})$ is a reflexive Banach space.
		\item \label{t:Fp-sep} The Banach space $(\sF_p,\norm{\,\cdot\,}_{\sF_p})$ is separable.
		\item \label{t:Fp-regcont} $\sF_p \cap \mathcal{C}(K)$ is dense in $\mathcal{C}(K)$ with respect to the uniform norm.
		\item \label{t:Fp-regFp} $\sF_p \cap \mathcal{C}(K)$ is dense in the Banach space $(\sF_p,\norm{\,\cdot\,}_{\sF_p})$.
	\end{enumerate}
\end{theorem}
The combination of properties \ref{t:Fp-regcont} and \ref{t:Fp-regFp} is referred to as \emph{regularity} in the theory of Dirichlet forms \cite{FOT,CF}.
The proof of Theorem \ref{t:Fp} will be completed over this section and the next.

\begin{proof}[Proof of Theorem \ref{t:Fp}\ref{t:Fp-banach}]
	We will give a complete proof because known detailed proofs for the required statement (see \cite[Lemmas 3.15 and 3.16]{Kig23} or \cite[Theorem 5.2]{Shi24}) are limited to the case where $\mathcal{F}_{p}$ is continuously embedded into $\mathcal{C}(K)$ and \cite[Lemma 3.24]{Kig23} is just a sketch.
	Let $\{ f_n \}_{n \ge 1}$ be a Cauchy sequence in $(\mathcal{F}_{p}, \norm{\,\cdot\,}_{\mathcal{F}_p})$.
    Since the convergence in $\mathcal{F}_{p}$ implies the convergence in $L^{p}$, the sequence $\{ f_n \}_{n \ge 1}$ converges in $L^{p}$ to some $f \in L^{p}(K, \measure)$.
    By the dominated convergence theorem, for any $k \in \bN$ and $w \in V_{k}$, we have $M_{k}f_{n}(w) \to M_{k}f(w)$ as $n \to \infty$.
    Also, since $\{ f_n \}_{n \ge 1}$ is a Cauchy sequence in $\mathcal{F}_{p}$, for any $\varepsilon > 0$ there exists $N(\varepsilon) \in \mathbb{N}$ such that
    \[
    \sup_{n \wedge l \ge N(\varepsilon)}\sup_{k \in \bN}\widetilde{\mathcal{E}}_{p}^{(k)}(f_{n} - f_{l}) \le \varepsilon.
    \]
    Letting $l \to \infty$ in the estimate $\widetilde{\mathcal{E}}_{p}^{(k)}(f_{n} - f_{l}) \le \varepsilon$ and taking the supremum over $k \in \bN$ and $n \ge N(\varepsilon)$, we obtain
    \begin{align}\label{e.fatou}
        \sup_{n \ge N(\varepsilon)}\sup_{k \in \bN}\widetilde{\mathcal{E}}_{p}^{(k)}(f_{n} - f) \le \varepsilon.
    \end{align}
    Therefore, for any $k \in \bN$,
    \begin{align*}
        \widetilde{\mathcal{E}}_{p}^{(k)}(f)^{1/p}
        \le \widetilde{\mathcal{E}}_{p}^{(k)}(f_{N(\varepsilon)} - f)^{1/p} + \widetilde{\mathcal{E}}_{p}^{(k)}(f_{N(\varepsilon)})^{1/p}
        \le \varepsilon^{1/p} + \sup_{n \ge 1}\abs{f_{n}}_{\mathcal{F}_{p}}.
    \end{align*}
    This implies $\abs{f}_{\mathcal{F}_{p}} \le \sup_{n \ge 1}\abs{f_{n}}_{\mathcal{F}_{p}} < \infty$ and thus $f \in \mathcal{F}_{p}$.
    The required convergence $f_{n} \to f$ in $\mathcal{F}_{p}$ is also deduced from the $L^{p}$-convergence of $f_{n}$ and \eqref{e.fatou}.
\end{proof}

Next, we will prove \emph{reflexivity} and \emph{separability} of the Banach space $\mathcal{F}_{p}$.
The reflexivity of such a function space is proved by the second-named author in \cite{Shi24} by showing the existence a comparable \emph{uniform convex} norm.
To construct a uniformly convex norm on $\mathcal{F}_{p}$ which is equivalent to $\norm{\,\cdot\,}_{\mathcal{F}_{p}}$, we need the notion of \emph{$\Gamma$-convergence}; see \cite{DalMaso} for details.
We first recall the definition.
\begin{defn}[{\cite[Definition 4.1 and Proposition 8.1]{DalMaso}}]
	Let $X$ be a first-countable topological space and let $F \colon X \to \mathbb{R} \cup \{ \pm\infty \}$.
	A sequence of functionals $\{ F_{n} \colon X \to \mathbb{R} \cup \{ \pm\infty \}\}_{n \in \mathbb{N}}$ \emph{$\Gamma$-converges} to $F$ if the following hold for any $x \in X$:
	\begin{itemize}
		\item (liminf inequality) If $x_{n} \to x$ in $X$, then $F(x) \le \liminf_{n \to \infty}F_{n}(x_{n})$.
		\item (limsup inequality) There exists a sequence $\{ x_n \}_{n \in \mathbb{N}}$ in $X$ such that
        \begin{equation}\label{eq.limsup}
            \text{$x_n \to x$ in $X$} \quad  \text{and} \quad  \text{$\displaystyle\limsup_{n \to \infty}F_{n}(x_{n}) \le F(x)$.}
        \end{equation}
	\end{itemize}
	A sequence $\{ x_{n} \}_{n \in \mathbb{N}}$ satisfying \eqref{eq.limsup} is called a \emph{recovery sequence of $\{ F_{n} \}_{n \in \mathbb{N}}$ at $x$}.
\end{defn}

The following compactness result is fundamental and useful.
\begin{prop}[{\cite[Theorem 8.5]{DalMaso}}]\label{prop.gamma-cpt}
	Suppose that $X$ is a  topological space with a countable base.
	Then any sequence of functionals $\{ F_{n} \colon X \to \mathbb{R} \cup \{ \pm\infty \} \}_{n \in \mathbb{N}}$ has a $\Gamma$-convergent subsequence.
\end{prop}

Now we can establish the reflexivity of $\mathcal{F}_{p}$.
\begin{proof}[Proof of Theorem \ref{t:Fp}\ref{t:Fp-ref}]
	The proof is essentially the same as in \cite[Theorem 5.9]{Shi24}, so we briefly outline the proof.
	By Proposition \ref{prop.gamma-cpt}, we have a $\Gamma$-cluster point $\mathsf{E}_{p}$ of the sequence of functionals $\bigl\{ \widetilde{\mathcal{E}}_{p}^{(n)} \bigr\}_{n \in \bN}$ on $L^{p}(K, \measure)$.
	It is easy to show that $\mathsf{E}_{p}(\,\cdot\,)^{1/p}$ is a semi-norm on $\mathcal{F}_{p}$.
	The liminf inequality implies $\mathsf{E}_{p}(\,\cdot\,)^{1/p} \le \abs{\,\cdot\,}_{\mathcal{F}_{p}}$.
	A combination of limsup inequality and weak monotonicity (Theorem \ref{t:wm}) implies the converse estimate $\mathsf{E}_{p}(\,\cdot\,)^{1/p} \gtrsim \abs{\,\cdot\,}_{\mathcal{F}_{p}}$.
	Hence,
	\[
	\trinorm{f} \coloneqq \bigl(\norm{f}_{L^{p}}^{p} + \mathsf{E}_{p}(f)\bigr)^{1/p} \quad \text{for $f \in L^{p}(K, \measure)$}
	\]
	defines a norm on $\mathcal{F}_{p}$ which is equivalent to $\norm{\,\cdot\,}_{\mathcal{F}_{p}}$.
	Noting that $\trinorm{\,\cdot\,}$ is a $\Gamma$-cluster point of $\norm{\,\cdot\,}_{p, n} \coloneqq \Bigl(\norm{\,\cdot\,}_{L^p}^{p} + \widetilde{\mathcal{E}}_{p}^{(n)}(\,\cdot\,)\Bigr)^{1/p}$, which can be regarded as the $L^{p}$-norm on $K \sqcup E_{n}$, we easily obtain $p$-Clarkson's inequality of $\trinorm{\,\cdot\,}$, i.e., for all $f, g \in L^{p}(K, \measure)$,
	\begin{align}\label{Cp}
    \begin{cases}
        \trinorm{f + g}^{p/(p - 1)} + \trinorm{f - g}^{p/(p - 1)} \le 2\bigl(\trinorm{f}^{p} + \trinorm{g}^{p})\bigr)^{1/(p - 1)} \quad &\text{if $p \le 2$,} \\
        \trinorm{f + g}^{p} + \trinorm{f - g}^{p} \le 2\bigl(\trinorm{f}^{p/(p - 1)} + \trinorm{g}^{p/(p - 1)}\bigr)^{p - 1} \quad &\text{if $p \ge 2$.}
    \end{cases}
	\end{align}
	Since $p$-Clarkson's inequality implies the uniform convexity \cite[p. 403]{Cla36}, the Milman--Pettis theorem (see \cite[Theorem 2.49]{HKST} for example) deduces the reflexivity of $\mathcal{F}_{p}$.
\end{proof}

In \cite[Theorem 5.10]{Shi24}, the separability of $\mathcal{F}_{p}$ has shown by using its reflexivity in the situation that $\mathcal{F}_{p}$ is continuously embedded into $\mathcal{C}(K)$ (cf. \cite[Theorem 3.22]{Kig23} or \cite[Theorem 5.1]{Shi24}).
The proof of \cite[Theorem 5.10]{Shi24} essentially relies on this embedding.
Here, we will adopt another simple way to show the separability by using an idea in \cite{AHM23}.
\begin{proof}[Proof of Theorem \ref{t:Fp}\ref{t:Fp-sep}]
	The Banach space $\mathcal{F}_{p}$ is reflexive by Theorem \ref{t:Fp}(ii), and $L^{p}(K, \measure)$ is separable since $K$ is separable.
	Clearly, the identity mapping $i \colon \mathcal{F}_{p} \to L^{p}(K, \measure)$ is a bounded linear injective map, so $\mathcal{F}_{p}$ is separable by \cite[Proposition 4.1]{AHM23}.
\end{proof}

We will next show the density of $\mathcal{F}_{p} \cap \mathcal{C}(K)$ in $\mathcal{C}(K)$ with respect to the uniform norm.
To show such the density, a standard idea is to use Stone--Weierstrass theorem by showing that  $\mathcal{F}_{p} \cap \mathcal{C}(K)$  is an algebra that separates points of $K$.
We recall Arzel\'{a}--Ascoli type theorem for (possibly) discontinuous functions in order to construct a function in $\mathcal{F}_{p} \cap \contfunc(K)$ that separates two distinct points (a cutoff function).
The proof that $\mathcal{F}_{p} \cap \contfunc(K)$ is an algebra will be done in the next subsection.
\begin{lem}\label{lem.AA}
	Let $(X, \mathsf{d})$ be a totally bounded metric space.
	Let $u_{n} \colon X \to \mathbb{R}$ for any $n \in \mathbb{N}$.
	Assume that there exist a non-decreasing function $\eta \colon [0, \infty) \to [0, \infty)$ and a sequence $\{ \delta_{n} \}_{n \in \mathbb{N}}$ of non-negative numbers such that $\lim_{t \downarrow 0}\eta(t) = 0$, $\lim_{n \to \infty}\delta_{n} = 0$, $\sup_{n \in \mathbb{N}, x \in X}\abs{u_{n}(x)} < \infty$ and
	\begin{equation}\label{AA.equi}
		\abs{u_{n}(x) - u_{n}(y)} \le \eta(\mathsf{d}(x, y)) + \delta_{n} \quad \text{for all $x, y \in X$ and $n \in \mathbb{N}$.}
	\end{equation}
	Then there exist a subsequence $\{ u_{n_{k}} \}_{k \in \mathbb{N}}$ and $u \in \mathcal{C}(X)$ with
	\[
	\abs{u(x) - u(y)} \le \eta(\mathsf{d}(x, y)) \quad \text{for all $x, y \in X$,}
	\]
	such that $\sup_{x \in X}\abs{u_{n_{k}}(x) - u(x)} \to 0$ as $k \to \infty$.
\end{lem}
\begin{proof}
	This is a simplified version of \cite[Lemma D.1]{Kig23}.
	Indeed, the case $(Y, d_{Y}) = (\mathbb{R}, \abs{\,\cdot\,})$ in \cite[Lemma D.1]{Kig23} is enough to obtain the required statement.
\end{proof}

The next proposition constructs cutoff functions with controlled energy in  $\mathcal{F}_{p}\cap \mathcal{C}(K)$.
We use the following useful notation.
For $A \subseteq K$, we define
\[
V_{n}(A) \coloneqq \bigl\{ w \in V_{n} \bigm| \widetilde{K}_{w} \cap A \neq \emptyset \bigr\}.
\]
\begin{prop} \label{p:cutoff}
	There exists $C \in (1,\infty)$ depending only on the constants associated with Assumption \ref{a:reg} such that for any $r>0, x \in K$ such that $B_{\metric}(x,2r) \neq K$, we have a function $\psi_{x,r} \in \mathcal{F}_{p} \cap \mathcal{C}(K)$ such that $\restr{\psi_{x, r}}{B_{\metric}(x, r)} = 1$, $\supp[\psi_{x, r}] \subseteq B_{\metric}(x, 2r)$ and
	\[
	\sup_{n \in \bN} \wt{\mathcal{E}}^{(n)}_p(\psi_{x,r}) \le C r^{\hdim-\beta}.
	\]
\end{prop}
\begin{proof}
	Let $\{\wt{K}_v \mid v \in V_n , n \in \bN\}, C \in (1,\infty)$ be as given in Definition \ref{d:compatible}.
	By \eqref{e:holder} and \eqref{e:round}, we have
	\begin{equation} \label{e:cf1}
		\wt{K}_w \subset B_d(x,r+2C R_*^{-n} + CRR_*^{-n}) \quad \mbox{	for any $w \in \bigcup_{v \in V_n(B_d(x,r))} B_{d_n}(v,R)$.}
	\end{equation}
	We choose $R_n>0$ so that $CR_nR_*^{-n}= r/2$ and a maximal $R_n/2$-separated subset $N$ of $V_n(B_d(x,r))$ (with respect to the metric $d_{n}$), so that
	$\bigcup_{w \in N}  B_{d_n}(w,R_n/2) \supset V_n(B_d(x,r))$.
	Since $\{p_n(w) \mid w \in N\}$ is $C^{-1}(R_n/2)R_*^{-n}$-separated and satisfies $\{ p_{n}(w) \}_{w \in N} \subset B_d(x,r+CR_*^{-n})$, by the $\hdim$-Ahlfors regularity of $\measure$, we obtain
	\begin{equation} \label{e:cf2}
	\# N \lesssim \left(\frac{r+cR_*^{-n}}{R_n R_*^{-n}}\right)^{\hdim} \lesssim \left(\frac{R_nR_*^{-n}+ R_*^{-n}}{R_nR_*^{-n}}\right)^{\hdim} \lesssim 1
	\end{equation}
	for all $n$ large enough so that $R_n \ge 1$.

	For $n$ large enough so that $2 C R_*^{-n} < r/2$, we have $R_n \ge 2$ and $\wt{K}_w \subset B_d(x,2r)$ for any $w \in \bigcup_{v \in V_n(B_d(x,r))} B_{d_n}(v,R_n)$ (by \eqref{e:cf1}). Therefore by applying \hyperref[cond.UCF]{\textup{U-CF$_{p}(\vartheta, \beta)$}}, for each $w \in N$, there exists $\varphi_{w,R_n/2}\colon V_{n} \to [0, 1]$ such that $\restr{\varphi_{w, R_{n}/2}}{B_{d_{n}}(w, R_{n}/2)} \equiv 1$, $\supp[\varphi_{w, R_{n}/2}] \subseteq B_{d_{n}}(w, R_{n})$,
	\[
	\sE_{p}^{\mathbb{G}_n}(\varphi_{w,R_n/2}) \lesssim R_n^{\hdim-\beta},
	\]
	and $\varphi_{w, R_{n}/2}$ satisfies the H\"older regularity condition \eqref{e:cfr}.
 	Hence by \eqref{e:cf1} and \eqref{e:cf2}, the function $\varphi_n\colon V_n \to \mathbb{R}$ defined by
 	\[
 	\varphi_n \coloneqq \max_{w \in N} \varphi_{w,R_{n}/2}
 	\]
 	satisfies $\restr{J_{n}\varphi_{n}}{B_{\metric}(x, r)} \equiv 1$, $\supp_{\measure}[J_{n}\varphi_{n}] \subseteq B_{\metric}(x, 2r)$,
 	\begin{equation}\label{e:cf3}
 	\varphi_n \equiv 1 \mbox{ on $V_n(B_d(x,r))$}, \quad 	\sE_{p}^{\mathbb{G}_n}(\varphi_{n}) \lesssim R_n^{\hdim-\beta}  \lesssim r^{\hdim-\beta} R_*^{n(\hdim-\beta)},
 	\end{equation} and
 	\begin{equation} \label{e:cf4}
	\abs{\varphi_n(v_1) - \varphi_n(v_2)} \lesssim \left( \frac{d_n(v_1,v_2)}{R_n}\right)^\vartheta, \quad \mbox{for all $v_1,v_2 \in V_n$,}
	 \end{equation}
	 for all $n \in \bN$ so that $2 C R_*^{-n} < r/2$.
	To estimate the energy, we used the elementary inequality $\sE_{p}^{\mathbb{G}_n}(\max_{w \in N} \varphi_{w,R_{n}/2}) \le \sum_{w \in N} \sE_{p}^{\mathbb{G}_n}(\varphi_{w,R_{n}/2})$ (see Lemma \ref{lem.basic-dEp}(b)).
	By Lemma \ref{lem.AA}, \eqref{e:cf4}, \eqref{e:holder}, and \eqref{e:round}, there exists a subsequence $\{ J_{n_k}\varphi_{n_k} \}_{k}$ of $\{ J_n\varphi_n \}_{n}$ which converges uniformly to $\psi_{x,r} \in \mathcal{C}(K)$.
	Then it is clear that $\restr{\psi_{x, r}}{B_{\metric}(x,r)} \equiv 1$ and $\supp[\psi_{x, r}] \subseteq B_{\metric}(x, 2r)$.
 	Using weak monotonicity (Theorem \ref{t:wm}) and dominated convergence theorem, we obtain
 	\begin{align}
 		\wt{\mathcal{E}}^{(n)}_p(\psi_{x,r}) & =   R_*^{n(\beta-\hdim)} \mathcal{E}_p^{\mathbb{G}_n}(M_n \psi_{x,r}) = \lim_{n_k \to \infty}  R_*^{n(\beta-\hdim)} \mathcal{E}_p^{\mathbb{G}_n}(M_n J_{n_k} \varphi_{n_k}) \nonumber \\
 	 	&\stackrel{\eqref{e:wm}}{\lesssim} \liminf_{n_k \to \infty} R_*^{n_k(\beta-d_f)} \mathcal{E}_p^{\mathbb{G}_{n_k}}(\varphi_{n_k}) \stackrel{\eqref{e:cf3}}{\lesssim}  r^{\hdim-\beta}. \nonumber
 	\end{align}
 	Therefore $\psi_{x,r} \in \mathcal{F}_{p} \cap \mathcal{C}(K)$ and it satisfies the desired bound on energy.
\end{proof}

\subsection{Scaling limit of discrete energies and regularity}\label{sec.sc-limit}
In the rest of this section, we suppose that Assumption \ref{a:reg} holds as in the previous subsection.
In this setting, we will construct an `improved' $p$-energy type functionals on $(K, \metric, \measure)$, which verifies that $\mathcal{F}_{p} \cap \contfunc(K)$ is an algebra.
In the following main theorem of this subsection, such a good $p$-energy is constructed as a sub-sequential $\Gamma$-limit of the re-scaled discrete $p$-energies $\bigl\{ \widetilde{\mathcal{E}}_{p}^{(n)} \bigr\}_{n \in \bN}$ .
\begin{thm}\label{thm.Epgamma}
	There exist a constant $C \ge 1$ (depending only on the constants associated with Assumption \ref{a:reg}) and $\mathcal{E}_{p}^{\Gamma} \colon \mathcal{F}_{p} \to [0, \infty)$ such that the following hold:
	\begin{enumerate}[\rm(i)]
		\item\label{it:Epgamma.Cp} The functional $\mathcal{E}_{p}^{\Gamma}(\,\cdot\,)^{1/p}$ is a semi-norm on $\mathcal{F}_{p}$ and
		\begin{equation}\label{Epgamma-comp}
		C^{-1}\abs{f}_{\mathcal{F}_{p}} \le \mathcal{E}_{p}^{\Gamma}(f)^{1/p} \le \abs{f}_{\mathcal{F}_{p}} \quad \text{for all $f \in \mathcal{F}_{p}$;}
		\end{equation}
		Moreover, $\mathcal{E}_{p}^{\Gamma}$ satisfies \emph{$p$-Clarkson's inequality}: for any $f,g \in \mathcal{F}_{p}$,
			\begin{align}
    			\begin{cases}\label{Cp-gamma}
        		\mathcal{E}_{p}^{\Gamma}(f + g)^{1/(p - 1)} + \mathcal{E}_{p}^{\Gamma}(f - g)^{1/(p - 1)} \le 2\bigl(\mathcal{E}_{p}^{\Gamma}(f) + \mathcal{E}_{p}^{\Gamma}(g)\bigr)^{1/(p - 1)} \quad &\text{if $p \le 2$,} \\
        		\mathcal{E}_{p}^{\Gamma}(f + g) + \mathcal{E}_{p}^{\Gamma}(f - g) \le 2\bigl(\mathcal{E}_{p}^{\Gamma}(f)^{1/(p - 1)} + \mathcal{E}_{p}^{\Gamma}(g)^{1/(p - 1)}\bigr)^{p - 1} \quad &\text{if $p \ge 2$,}
    			\end{cases}
			\end{align}
		In particular, $\mathcal{E}_{p}^{\Gamma}(\,\cdot\,)^{1/p}$ is uniformly convex.
		\item\label{it:Epgamma.lip} For any $f \in \mathcal{F}_{p}$ and $1$-Lipschitz function $\varphi \in \mathcal{C}(\mathbb{R})$, $\varphi \circ f \in \mathcal{F}_{p}$ and
		\[
		\mathcal{E}_{p}^{\Gamma}\bigl(\varphi \circ f\bigr) \le \mathcal{E}_{p}^{\Gamma}(f).
		\]
		\item\label{it:Epgamma.leibniz} If $f, g \in \mathcal{F}_{p} \cap L^{\infty}(K, \measure)$, then $f \cdot g \in \mathcal{F}_{p}$ and
		\begin{equation}\label{e:weakLeibniz}
			\mathcal{E}_{p}^{\Gamma}(f \cdot g) \le c_{p}\bigl(\norm{f}_{L^{\infty}}^{p} \vee \norm{g}_{L^{\infty}}^{p}\bigr)\Bigl(\mathcal{E}_{p}^{\Gamma}(f)  + \mathcal{E}_{p}^{\Gamma}(g)\Bigr).
		\end{equation} 
		where $c_p \in (0,\infty)$ is a constant determined solely by $p$. 
		\item\label{it:Epgamma.lsc} $\mathcal{E}_{p}^{\Gamma}$ is lower semi-continuous on $L^{p}(K, \measure)$. (Here we regard $\mathcal{E}_{p}^{\Gamma}$ as a $[0, \infty]$-valued functional by defining $\mathcal{E}_{p}^{\Gamma}(f) \coloneqq \infty$ for $f \in L^{p}(K, \measure) \setminus \mathcal{F}_{p}$.)
		\item\label{it:Epgamma.inv} Let $T \colon (K,\mathcal{B}(K),\measure) \to (K,\mathcal{B}(K),\measure)$ be a measure preserving transformation, i.e., $T$ is Borel measurable and $\measure(T^{-1}(A)) = \measure(A)$ for any Borel set $A$ of $K$.
		If $\widetilde{\mathcal{E}}_{p}^{(n)}(f \circ T) = \widetilde{\mathcal{E}}_{p}^{(n)}(f)$ for any $n \in \mathbb{N}$ and $f \in L^p(K,m)$, then $f \circ T \in \mathcal{F}_{p}$ for any $f \in \mathcal{F}_{p}$ and $\mathcal{E}_{p}^{\Gamma}(f \circ T) = \mathcal{E}_{p}^{\Gamma}(f)$.
	\end{enumerate}
\end{thm}
\begin{rmk}\label{rmk:Epgamma.GCP}
	One can show the \emph{generalized $p$-contraction property}, which was introduced in \cite{KS24+} and gives a generalization of the properties \ref{it:Epgamma.Cp}-\ref{it:Epgamma.leibniz} above, for $(\mathcal{E}_{p}^{\Gamma},\mathcal{F}_{p})$; see \cite[Theorem 8.19 and Remark 8.20]{KS24+}. In particular, \eqref{e:weakLeibniz} can be updated as follows: 
	\[
	\mathcal{E}_{p}^{\Gamma}(f \cdot g)^{1/p} \le \norm{g}_{L^{\infty}}\mathcal{E}_{p}^{\Gamma}(f)^{1/p} + \norm{f}_{L^{\infty}}\mathcal{E}_{p}^{\Gamma}(g)^{1/p};
	\] 
	see \cite[Proposition 2.2(d)]{KS24+}. 
	We will not deal with the generalized $p$-contraction property in this paper because these properties are not needed for our purposes. 
\end{rmk}
\begin{proof}[Proof of Theorem \ref{thm.Epgamma}]
	Let $\mathcal{E}_{p}^{\Gamma} = \mathsf{E}_{p}$ be a $\Gamma$-cluster point of $\bigl\{ \widetilde{\mathcal{E}}_{p}^{(n)} \bigr\}_{n \in \bN}$ as in the proof of Theorem \ref{t:Fp}\ref{t:Fp-ref}.
	The comparability \eqref{Epgamma-comp} is already shown there.
	If we consider $\widetilde{\mathcal{E}}_{p}^{(n)}(\,\cdot\,)^{1/p}$ instead of $\norm{\,\cdot\,}_{p,n}$ in the argument showing \eqref{Cp}, then we obtain $p$-Clarkson's inequality of $\mathcal{E}_{p}^{\Gamma}$ \eqref{Cp-gamma} and \ref{it:Epgamma.Cp}.

	\ref{it:Epgamma.lip} The proof is very similar to \cite[Theorem 3.21]{Kig23}, but we will give the details because the embedding $\mathcal{F}_{p} \subseteq \mathcal{C}(K)$ is used in \cite{Kig23}.
	We start by an observation on $L^{p}$-approximation.
	Let $f \in L^{p}(K, \measure)$ and let $f_{n} = J_{n}\bigl(M_{n}f\bigr)$ for $n \in \bN$, where $J_{n} \colon\mathbb{R}^{V_{n}} \to L^{0}(K, \measure)$ be the operator defined in \eqref{e:Jn}.
	We will prove $\norm{f - f_{n}}_{L^{p}} \to 0$ as $n \to \infty$.
	Note that $\abs{M_{n}f(z)}^{p} \le \fint_{\widetilde{K}_{z}}\abs{f}^{p}\,d\measure$ for all $z \in V_{n}$ by Jensen's inequality.
	Then we have
	\[
	\int_{K}\abs{f_{n}}^{p}\,d\measure
	= \sum_{z \in V_{n}}\int_{\widetilde{K}_{z}}\abs{M_{n}f(z)}^{p}\,\measure(dx)
	\le \int_{K}\abs{f}^{p}\,d\measure < \infty.
	\]
	Let $\mathscr{M} \colon L^{p}(K, \measure) \to L^{p}(K, \measure)$ be the \emph{Hardy--Littlwood maximal operator}, i.e., for $f \in L^{p}(K, \measure)$ and $x \in K$,
	\[
	\mathscr{M}f(x) = \sup_{r > 0}\fint_{B_{\metric}(x, r)}f(y)\,\measure(dy).
	\]
	Since $\measure$ is Ahlfors regular (by Assumption \ref{a:reg}), $\mathscr{M}$ is $L^p$-bounded (see \cite[Theorem 3.5.6]{HKST} for example), i.e., there exists a constant $C > 0$ such that
	\[
	\norm{\mathscr{M}f}_{L^{p}} \le C\norm{f}_{L^{p}} \quad \text{for all $f \in L^{p}(K, \measure)$.}
	\]
	For $x \in K$, let $z \in V_{n}$ be the unique element such that $x \in \widetilde{K}_{z}$.
	Then, by \eqref{e:round},
	\[
	\abs{f_{n}(x)} = \abs{M_{n}f(z)} \le \frac{\measure\bigl(B_{\metric}(x, 2CR_{\ast}^{-n})\bigr)}{\measure(\widetilde{K}_{z})}\mathscr{M}\abs{f}(x),
	\]
	where $C \ge 1$ is the constant in \eqref{e:round}.
	By \ref{cond.VD} of $\measure$ and \eqref{e:round},
	\begin{equation}\label{e:vol-bdd}
		\sup_{n \in \bN, z \in V_{n}, x \in K_{z}}\frac{\measure(B_{\metric}(x, 2CR_{\ast}^{-n}))}{\measure(\widetilde{K}_{z})} < \infty.
	\end{equation}
	Thus each $f_{n}$ is dominated by $C'\mathscr{M}\abs{f} \in L^{p}(K, \measure)$ for some constant $C' > 0$.

	We next consider about $\measure$-a.e. convergence of $\{ f_{n} \}$.
	Since $\measure$ is Ahlfors regular, the Lebesgue differentiation theorem on $(K, \metric, \measure)$ holds (see \cite[Section 3.4]{HKST} for example), i.e., the set $\mathscr{L}_{f}$ (\emph{Lebesgue points of $f$}) defined by
	\[
	\mathscr{L}_{f} \coloneqq \Biggl\{ x \in K \Biggm| \lim_{r \downarrow 0}\fint_{B_{\metric}(x, r)}\abs{f(x) - f(y)}\,\measure(dy) = 0 \Biggr\}
	\]
	is a Borel set and $\measure(K \setminus \mathscr{L}_{f}) = 0$.
	Let $x \in \mathscr{L}_{f}$ and let $z \in V_{n}$ be the unique element such that $x \in \widetilde{K}_{z}$.
	Then we see that
	\[
	\abs{f(x) - f_{n}(x)} \le \fint_{\widetilde{K}_{z}}\abs{f(x) - f(y)}\,\measure(dy) \le \frac{\measure(B_{\metric}(x, 2CR_{\ast}^{-n}))}{\measure(\widetilde{K}_{z})}\fint_{B_{\metric}(x, 2CR_{\ast}^{-n})}\abs{f(x) - f(y)}\,\measure(dy).
	\]
	By \eqref{e:vol-bdd}, we get $\lim_{n \to \infty}\abs{f(x) - f_{n}(x)} = 0$ for all $x \in \mathscr{L}_{f}$.
	The dominated convergence theorem deduces $\norm{f - f_{n}}_{L^{p}} \to 0$.

	We now finish the proof of the property \ref{it:Epgamma.lip}.
	It is enough to consider the case that $\varphi \in \mathcal{C}(\mathbb{R})$ is a $1$-Lipschitz function.
	Let $\{ g_{k} \}_{k}$ be a recovery sequence of $f$ with respect to $\mathcal{E}_{p}^{\Gamma}$, i.e. $g_{k}$ converges in $L^{p}$ to $f$ and
	\[
	\limsup_{k \to \infty}\widetilde{\mathcal{E}}_{p}^{(n_{k})}(g_{k}) \le \mathcal{E}_{p}^{\Gamma}(f).
	\]
	We note that
	\begin{align*}
		&\norm{\varphi \circ f - \varphi \circ (J_{n_{k}}M_{n_{k}}g_{k})}_{L^{p}} \\
		&\le \norm{\varphi \circ f - \varphi \circ (J_{n_{k}}M_{n_{k}}f)}_{L^{p}} + \norm{\varphi \circ (J_{n_{k}}M_{n_{k}}f) - \varphi \circ (J_{n_{k}}M_{n_{k}}g_{k})}_{L^{p}} \\
		&\le \norm{f - f_{n_{k}}}_{L^{p}} + \norm{J_{n_{k}}M_{n_{k}}(f - g_{k})}_{L^{p}}
		\le \norm{f - f_{n_{k}}}_{L^{p}} + \norm{f - g_{k}}_{L^{p}} \to 0 \quad \text{as $k \to \infty$,}
	\end{align*}
	and that
	\begin{align*}
		M_{n_{k}}\bigl(\varphi \circ (J_{n_{k}}M_{n_{k}}g_{k})\bigr)(w)
		&= \fint_{\widetilde{K}_{w}}\varphi\bigl(J_{n_{k}}M_{n_{k}}g_{k}(x)\bigr)\,\measure(dx) \\
		&= \fint_{\widetilde{K}_{w}}\varphi\bigl(M_{n_{k}}g_{k}(w)\bigr)\,d\measure
		= \varphi\bigl(M_{n_{k}}g_{k}(w)\bigr) \quad \text{for all $w \in V_{n_{k}}$}.
	\end{align*}
	Therefore, we have
	\begin{align*}
		\mathcal{E}_{p}^{\Gamma}(\varphi \circ f)
		&\le \liminf_{k \to \infty}\widetilde{\mathcal{E}}_{p}^{(n_{k})}\bigl(\varphi \circ (J_{n_{k}}M_{n_{k}}g_{k})\bigr) \\
		&= \liminf_{k \to \infty}R_{\ast}^{n_k(\beta - d_f)}\mathcal{E}_{p}^{\mathbb{G}_{n_{k}}}\bigl(\varphi \circ (M_{n_{k}}g_{k})\bigr) \\
		&\le \liminf_{k \to \infty}R_{\ast}^{n_k(\beta - d_f)}\mathcal{E}_{p}^{\mathbb{G}_{n_{k}}}(M_{n_{k}}g_{k})
		\le \limsup_{k \to \infty}\widetilde{\mathcal{E}}_{p}^{(n_{k})}(g_{k})
		\le \mathcal{E}_{p}^{\Gamma}(f).
	\end{align*}

	\ref{it:Epgamma.leibniz} The case $f = g$ easily follows from \ref{it:Epgamma.lip} by considering a $1$-Lipschitz function $\varphi \in \contfunc(\mathbb{R})$ satisfying $\varphi(t) = (2\norm{f}_{L^{\infty}})^{-1}t^{2}$ for any $t \in \bigl[-\norm{f}_{L^{\infty}},\norm{f}_{L^{\infty}}\bigr]$. 
	Then the case $f \neq g$ can be proved by noting that $f \cdot g = \frac{1}{4}\bigl[(f + g)^{2} - (f - g)^{2}\bigr]$ and using $p$-Clarkson's inequality \eqref{Cp-gamma}. 

	\ref{it:Epgamma.lsc} This follows from an elementary fact on the $\Gamma$-convergence \cite[Proposition 6.8]{DalMaso}.

	\ref{it:Epgamma.inv} Let $f \in \mathcal{F}_{p}$ and let $\{ f_k \}_{k}$ be a recovery sequence at $f$.
	It is clear that $f \circ T \in \mathcal{F}_{p}$. 
	Note that $\norm{f \circ T - f_{k} \circ T}_{L^{p}} = \norm{f - f_{k}}_{L^{p}} \to 0$.
	Then
	\[
		\mathcal{E}_{p}^{\Gamma}(f \circ T)
		\le \liminf_{k \to \infty}\widetilde{\mathcal{E}}_{p}^{(n_k)}(f_{k} \circ T)
		= \liminf_{k \to \infty}\widetilde{\mathcal{E}}_{p}^{(n_k)}(f_{k}) 
		\le \mathcal{E}_{p}^{\Gamma}(f).
	\]
	The converse $\mathcal{E}_{p}^{\Gamma}(f) \le \mathcal{E}_{p}^{\Gamma}(f \circ T)$ can be shown by considering a recovery sequence at $f \circ T$.
	We complete the proof.
\end{proof}

Combining Proposition \ref{p:cutoff} and Theorem \ref{thm.Epgamma}\ref{it:Epgamma.leibniz}, we can show the density of $\mathcal{F}_{p} \cap \contfunc(K)$ in $\contfunc(K)$.
The density of $\mathcal{F}_{p} \cap \contfunc(K)$ in $\mathcal{F}_{p}$ requires a long preparation and will be shown in Section \ref{sec.domain}.
\begin{proof}[Proof of Theorem \ref{t:Fp}\ref{t:Fp-regcont}]
 	By Proposition \ref{p:cutoff}, $\sF_p \cap \mathcal{C}(K)$ separates points of $K$.
 	We note that, by Theorem \ref{thm.Epgamma}\ref{it:Epgamma.leibniz}, $\sF_p \cap \mathcal{C}(K)$ is a sub-algebra of $\mathcal{C}(K)$.
 	So by the Stone--Weierstrass theorem, $\sF_p \cap \mathcal{C}(K)$ is dense in $\mathcal{C}(K)$ with respect to the uniform norm.
\end{proof}

\subsection{Poincar\'e type inequalities and partition of unity}\label{sec.PI-like}
In this subsection, we prove a Poincar\'{e} type inequality and provide a partition of unity with low energies.

Since we have no energy measures, which play the role of ``$\abs{\nabla f}^{p}\,d\measure$'', at this stage, we need to describe ``$p$-energy on a given subset of $K$'' in terms of re-scaled discrete $p$-energies.
The following lemma allows us to get the desired Poincar\'{e} inequality from \ref{cond.UPI}.
\begin{lem}\label{lem.ave-conv}
	There exists a constant $C > 0$ (depending only on $p$ and the doubling constant of $\measure$) such that the following holds: for any $x \in K$, $r > 0$ and $f \in L^{p}(K, \measure)$,
	\begin{equation*}
		\fint_{B_{\metric}(x, r)}\abs{f(x) - f_{B_{\metric}(x, r)}}^{p}\,\measure(dx)
		\le C\liminf_{n \to \infty}\frac{1}{\measure\bigl(\widetilde{K}_{x, r}^{(n)}\bigr)}\sum_{w \in V_{n}(B_{\metric}(x, r))}\abs{M_{n}f(w) - f_{\widetilde{K}_{x, r}^{(n)}}}^{p}\measure\bigl(\widetilde{K}_{w}\bigr),
	\end{equation*}
	where we set $\widetilde{K}_{x, r}^{(n)} = \bigcup_{w \in V_{n}(B_{\metric}(x, r))}\widetilde{K}_{w} \, (n \in \mathbb{N})$ for ease of notation.
\end{lem}
\begin{proof}
	Let $x \in K$, $r > 0$ and $f \in L^{p}(K, \measure)$.
	For each $n \in \mathbb{N}$, let $f_{n} \coloneqq J_{n}(M_{n}f)$, where $J_{n} \colon \mathbb{R}^{V_{n}} \to L^{0}(K, \measure)$ is the same as in \eqref{e:Jn}.
	We observe that, for large $n \in \mathbb{N}$ so that $\widetilde{K}_{x, r}^{(n)} \subseteq B_{\metric}(x, 2r)$,
	\begin{align*}
		&\frac{1}{\measure\bigl(\widetilde{K}_{x, r}^{(n)}\bigr)}\sum_{w \in V_{n}(B_{\metric}(x, r))}\abs{M_{n}f(w) - f_{\widetilde{K}_{x, r}^{(n)}}}^{p}\measure\bigl(\widetilde{K}_{w}\bigr) \\
		&\quad= \frac{1}{\measure\bigl(\widetilde{K}_{x, r}^{(n)}\bigr)}\sum_{w \in V_{n}(B_{\metric}(x, r))}\int_{\widetilde{K}_{w}}\abs{f_{n} - f_{\widetilde{K}_{x, r}^{(n)}}}^{p}\,d\measure \\
		&\quad\gtrsim \fint_{B_{\metric}(x, r)}\abs{f_{n} - f_{\widetilde{K}_{x, r}^{(n)}}}^{p}\,d\measure 
		\gtrsim \fint_{B_{\metric}(x, r)}\abs{f_{n} - (f_{n})_{B_{\metric}(x, r)}}^{p}\,d\measure,
	\end{align*}
	where we used the volume doubling property of $\measure$ and Lemma \ref{lem.p-var} in the last line.
	Since $\norm{f - f_n}_{L^p} \to 0$ by the same argument as in Theorem \ref{thm.Epgamma}, the dominated convergence theorem yields
	\[
	\lim_{n \to \infty}\fint_{B_{\metric}(x, r)}\abs{f_{n} - (f_{n})_{B_{\metric}(x, r)}}^{p}\,d\measure = \fint_{B_{\metric}(x, r)}\abs{f - f_{B_{\metric}(x, r)}}^{p}\,d\measure,
	\]
	which proves our assertion.
\end{proof}

Now we prove a $(p, p)$-Poincar\'{e}-like inequality.
\begin{lem}\label{lem.PI-like}
	There exist constants $C > 0$ and $A \ge 1$ (depending only on the constants associated with Assumption \ref{a:reg}) such that for all $x \in K$, $r > 0$ and $f \in L^{p}(K, \measure)$,
	\begin{equation}\label{eq.PI-like}
		\int_{B_{\metric}(x, r)}\abs{f - f_{B_{\metric}(x, r)}}^{p}\,d\measure \le Cr^{\beta}\liminf_{n \to \infty}\widetilde{\mathcal{E}}_{p, V_{n}(B_{\metric}(x, Ar))}^{(n)}(f).
	\end{equation}
\end{lem}
\begin{proof}
	Let $x \in K$, $r > 0$ and $f \in \mathcal{F}_{p}$.
	Let $\widetilde{K}_{x, r}^{(n)}$ be the same as in the previous lemma for each $n \in \mathbb{N}$.
	Let $C \ge 1$ be the constant in Definition \ref{d:compatible} and choose $R_{n} > 0$ so that $R_{n}R_{\ast}^{-n} = 2Cr$.
	Note that $R_{n} \uparrow +\infty$ as $n \to \infty$.
	Since $\bigl\{ \widetilde{K}_{w} \bigr\}_{w \in V_{n}}$ is a partition of $K$, there exists a unique $c_{n} \in V_{n}(B_{\metric}(x, r))$ such that $x \in \widetilde{K}_{c_{n}}$.
	For all $w \in V_{n}(B_{\metric}(x, r))$, by \eqref{e:holder}, \eqref{e:round}, and picking a point $y \in B_{\metric}(x, r) \cap \widetilde{K}_{v}$,
	\begin{align*}
		d_{n}(c_{n}, w)
		\le CR_{\ast}^{n}\metric(p_{n}(c_{n}), p_{n}(w))
		&\le CR_{\ast}^{n}\bigl(\metric(x, p_{n}(c_{n})) + \metric(x, y) + \metric(y, p_{n}(v))\bigr) \\
		&< CR_{\ast}^{n}\bigl(CR_{\ast}^{-n} + r + CR_{\ast}^{-n}\bigr)
		= 2C^{2} + \frac{R_{n}}{2}.
	\end{align*}
	Hence we have $V_{n}(B_{\metric}(x, r)) \subseteq B_{d_{n}}(c_{n}, R_{n})$ for all large enough $n \in \mathbb{N}$.
	By \ref{cond.UPI}, for all large $n \in \mathbb{N}$,
	\begin{align*}
		&\frac{1}{\measure\bigl(\widetilde{K}_{x, r}^{(n)}\bigr)}\sum_{w \in V_{n}(B_{\metric}(x, r))}\abs{M_{n}f(w) - (M_{n}f)_{B_{d_{n}}(c_{n}, R_{n})}}^{p}\measure\bigl(\widetilde{K}_{w}\bigr) \\
		&\le \frac{1}{\measure(B_{\metric}(x, r))}\sum_{w \in B_{d_{n}}(c_{n}, R_{n})}\abs{M_{n}f(w) - (M_{n}f)_{B_{d_{n}}(c_{n}, R_{n})}}^{p}\measure\bigl(\widetilde{K}_{w}\bigr) \\
		&\lesssim r^{-\hdim}R_{\ast}^{-n\hdim}\sum_{w \in B_{d_{n}}(c_{n}, R_{n})}\abs{M_{n}f(w) - (M_{n}f)_{B_{d_{n}}(c_{n}, R_{n})}}^{p} \\
		&\lesssim r^{-\hdim}R_{\ast}^{-n\hdim}R_{n}^{\beta}\mathcal{E}_{p, B_{d_{n}}(c_{n}, A_{\textup{PI}}R_{n})}^{\bG_{n}}(M_{n}f)
		\lesssim r^{-\hdim + \beta}R_{\ast}^{n(\beta - \hdim)}\mathcal{E}_{p, B_{d_{n}}(c_{n}, A_{\textup{PI}}R_{n})}^{\bG_{n}}(M_{n}f).
	\end{align*}
	For any $v \in B_{d_{n}}(c_{n}, A_{\textup{PI}}R_{n})$, by \eqref{e:holder} and \eqref{e:round},
	\[
	\widetilde{K}_{v} \subseteq B_{\metric}\bigl(x, 2CR_{\ast}^{-n} + CA_{\textup{PI}}R_{n}R_{\ast}^{-n}\bigr) \subseteq B_{\metric}\bigl(x, (2C^{2}A_{\textup{PI}} + 1)r\bigr),
	\]
	for all large $n \in \mathbb{N}$ so that $2CR_{\ast}^{-n} \le r$.
	Let $A_{\textup{PI}}' \coloneqq 2C^{2}A_{\textup{PI}} + 1$.
	Combining with Lemma \ref{lem.p-var}, we get
	\[
	\frac{1}{\measure\bigl(\widetilde{K}_{x, r}^{(n)}\bigr)}\sum_{w \in V_{n}(B_{\metric}(x, r))}\abs{M_{n}f(w) - f_{\widetilde{K}_{x, r}^{(n)}}}^{p}\measure\bigl(\widetilde{K}_{w}\bigr)
	\lesssim r^{-\hdim + \beta}\widetilde{\mathcal{E}}_{p, V_{n}(B_{\metric}(x, A_{\textup{PI}}'r))}^{(n)}(f).
	\]
	Letting $n \to \infty$ and using Lemma \ref{lem.ave-conv} complete the proof.
\end{proof}

We conclude this section by constructing  partition of unity with continuous functions of controlled energy.

We need the following elementary properties of $\mathcal{E}_{p}^{\Gamma}$, which are consequences of \eqref{Cp-gamma} and `Leibniz rule' in Theorem \ref{thm.Epgamma}\ref{it:Epgamma.leibniz}.
This is an analogue of   \cite[Theorem 1.4.2(i)]{FOT} and \cite[Exercise I.4.16]{MR} in the theory of Dirichlet forms.
\begin{prop}\label{prop.pgamma}
	\begin{enumerate}[\rm(i)]
		\item\label{it:pgamma.wsubadd} For $f \in \mathcal{F}_{p}$, we have
		\[
		\mathcal{E}_{p}^{\Gamma}(h) \le \mathcal{E}_{p}^{\Gamma}(f), \quad \forall h \in \{ \abs{f}, f^{+}, f^{-} \}.
		\]
		Furthermore, there exists a constant $C_{p} \ge 1$ depending only on $p$ such that for any $f, g \in \mathcal{F}_{p}$,
		\begin{equation}\label{min-max}
			\mathcal{E}_{p}^{\Gamma}(f \wedge g) + \mathcal{E}_{p}^{\Gamma}(f \vee g) \le C_{p}\bigl(\mathcal{E}_{p}^{\Gamma}(f) + \mathcal{E}_{p}^{\Gamma}(g)\bigr).
		\end{equation}
		\item\label{it:pgamma.divide} Let $c, M > 0$ and let $f, g \in \mathcal{F}_{p}$ be non-negative functions such that $(f + g)|_{\{ f \neq 0 \}} \ge c$ and $f \le M$.
		Then there exists a constant $D_{c, M}$ depending only on $p, c, M$ such that
		\begin{equation}\label{negative-power}
			\mathcal{E}_{p}^{\Gamma}\biggl(\frac{f}{f + g}\biggr) \le D_{c, M}\bigl(\mathcal{E}_{p}^{\Gamma}(f) + \mathcal{E}_{p}^{\Gamma}(g)\bigr).
		\end{equation}
	\end{enumerate}
\end{prop}
\begin{proof}
	\ref{it:pgamma.wsubadd} The first assertion immediately follows from the Lipschitz contractivity since $\abs{h(x) - h(y)} \le \abs{f(x) - f(y)}$ for all $h \in \{ \abs{f}, f^{+}, f^{-} \}$ and $x, y \in K$.
	The estimate \eqref{min-max} can be shown by noting that
	\[
	f \wedge g = \frac{1}{2}\bigl(f + g - \abs{f - g}\bigr), \quad f \vee g = \frac{1}{2}\bigl(f + g + \abs{f - g}\bigr),
	\]
	and using \eqref{Cp-gamma}.

	\ref{it:pgamma.divide} Define $\varphi\colon\mathbb{R}\to\mathbb{R}$ by
	\[
	\varphi(x) \coloneqq (-c^{2}x + c^{-1} + c^{3})\indicator{\{ x < c \}} + x^{-1}\indicator{\{ x \ge c \}}, \quad (x \in \mathbb{R}).
	\]
	Then we easily see that $\varphi \in C^{1}(\mathbb{R})$ and $\abs{\varphi'(x)} \le c^{2}$ for all $x \in \mathbb{R}$.
	Since $f + g \ge c$ on $\{ f \neq 0 \}$, we have $f \cdot \varphi(f + g) = \frac{f}{f + g}$. 
	By the properties \ref{it:Epgamma.lip} and \ref{it:Epgamma.leibniz} in Theorem \ref{thm.Epgamma},
	\begin{align*}
		\mathcal{E}_{p}^{\Gamma}\biggl(\frac{f}{f + g}\biggr)
		= \mathcal{E}_{p}^{\Gamma}\bigl(f \cdot \varphi(f + g)\bigr)
		&\le 2^{p - 1}\Bigl(\norm{f}_{L^{\infty}}^{p}\mathcal{E}_{p}^{\Gamma}\bigl(\varphi(f + g)\bigr) + \norm{\varphi(f + g)}_{L^{\infty}}^{p}\mathcal{E}_{p}(f)\Bigr) \\
		&\le 2^{p - 1}M^{p}c^{2p}\mathcal{E}_{p}^{\Gamma}(f + g) + 2^{p - 1}c^{-p}\mathcal{E}_{p}^{\Gamma}(f) \\
		&\le 2^{p - 1}(c^{-p} + 2^{p - 1}c^{2p}M^{p})\mathcal{E}_{p}^{\Gamma}(f) + 4^{p - 1}M^{p}\mathcal{E}_{p}^{\Gamma}(g),
	\end{align*}
	which shows \eqref{negative-power}.
\end{proof}

Following a standard argument (for example, \cite[Lemma 2.5]{Mur20}), we construct a good partition of unity using the cutoff functions of Proposition \ref{p:cutoff}.
\begin{lem}\label{lem.unity}
    Let $\varepsilon \in (0, 1)$ and let $V$ be a maximal $\varepsilon$-net of $(K, d)$.
    Then there exists a family of functions $\{ \psi_z \}_{z \in V}$ that satisfies the following properties:
    \begin{enumerate}[\rm(i)]
        \item\label{it:parunity.1} As a function $\sum_{z \in V}\psi_z \equiv 1$;
        \item\label{it:parunity.2} For any $z \in V$, we have $\psi_z \in \mathcal{F}_{p} \cap \mathcal{C}(K)$ with $0 \le \psi_z \le 1$, $\restr{\psi_z}{B_{\metric}(z, \varepsilon/4)} \equiv 1$ and $\supp[\psi_z] \subseteq B_{\metric}(z, 5\varepsilon/4)$;
        \item\label{it:parunity.3} If $z \in V$ and $z' \in V \setminus \{ z \}$, then $\restr{\psi_{z'}}{B_{\metric}(z, \varepsilon/4)} \equiv 0$.
        \item\label{it:parunity.4} There exists a constant $C \ge 1$ (depending only on the constants associated with Assumption \ref{a:reg}) such that for all $z \in V$,
        \begin{equation}\label{e:low-energy}
            \abs{\psi_z}_{\mathcal{F}_{p}}^{p} \le C\varepsilon^{\hdim - \beta}.
        \end{equation}
    \end{enumerate}
\end{lem}
\begin{proof}
    For $z \in V$, we define the `Voronoi cell' $\mathcal{R}_{z}$ as
    \[
    \mathcal{R}_{z} = \Bigl\{ x \in K \;\Bigm|\; \metric(x, z) = \min_{v \in V}\metric(x, v) \Bigr\},
    \]
    and write $\mathcal{R}_{z}^{\varepsilon/4}$ for its $\varepsilon/4$-neighborhood, i.e. $\mathcal{R}_{z}^{\varepsilon/4} = \bigcup_{x \in \mathcal{R}_{z}}B_{\metric}(x, \varepsilon/4)$.
    As shown in \cite[Lemma 2.5]{Mur20}, we know that $\bigcup_{z \in V}\mathcal{R}_{z} = K$,
    \[
    B_{\metric}(z, \varepsilon/2) \subseteq \mathcal{R}_{z} \subseteq \closure{B}_{d}(z, \varepsilon)
    \]
    and
    \[
    B_{\metric}(z, \varepsilon/4) \cap \mathcal{R}_{w}^{\varepsilon/4} = \emptyset \text{ for } v, w \in V \text{ with } v \neq w.
    \]
    For $z \in V$, we fix a maximal $\varepsilon/8$-net $N_z$ of $\mathcal{R}_{z}$.
    Then, by $\mathcal{R}_{z} \subseteq \closure{B}_{d}(z, \varepsilon)$, there exists a constant $M > 0$ (depending only on the doubling constant) such that $\sup_{z \in V}\#N_{z} \le M$.
    By Proposition \ref{p:cutoff}, for any $z \in V$ and any $w \in N_z$, we have a non-negative function $\rho_w \in \mathcal{F}_{p} \cap \mathcal{C}(K)$ satisfying
    \[
    \restr{\rho_w}{B_{\metric}(w, \varepsilon/8)} \equiv 1, \quad \supp[\rho_w] \subseteq B_{\metric}(w, \varepsilon/4), \quad 0 \le \rho_w \le 1, \quad \mathcal{E}_{p}^{\Gamma}(\rho_w) \lesssim \varepsilon^{\hdim - \beta}.
    \]
    Next, define $\phi_{z} \coloneqq \max_{w \in N_{z}}\rho_{w}$.
    Since $\bigcup_{w \in N_{z}}B_{\metric}\bigl(w, \varepsilon/8\bigr) \supseteq \mathcal{R}_{z}$, we have $\restr{\phi_{z}}{\mathcal{R}_{z}} \equiv 1$.
    From $\supp[\rho_{w}] \subseteq B_{\metric}(w, \varepsilon/4)$ and $N_{z} \subseteq \mathcal{R}_{z}$, we have $\supp[\phi_{z}] \subseteq \mathcal{R}_{z}^{\varepsilon/4}$.
    Using the triangle inequality of $\mathcal{E}_{p}^{\Gamma}(\,\cdot\,)^{1/p}$ and \eqref{min-max}, we see that
    \begin{align}\label{phi-upper}
        \mathcal{E}_{p}^{\Gamma}(\phi_{z})
        = \mathcal{E}_{p}^{\Gamma}\biggl(\max_{w \in N_{z}}\rho_{w}\biggr)
        \le (4 \vee 4^{p - 1})^{M}\sum_{w \in N_{z}}\mathcal{E}_{p}^{\Gamma}(\rho_{w})
        \lesssim \varepsilon^{d_f - \beta}.
    \end{align}
    Note that $\sum_{w \in V}\phi_{w} \ge 1$ since $\restr{\phi_{w}}{\mathcal{R}_{w}} \equiv 1$.
    Now we define $\{ \psi_{z} \}_{z \in V}$ by
    \[
    \psi_{z} \coloneqq \frac{\phi_{z}}{\sum_{w \in V}\phi_{w}}, \quad z \in V.
    \]
    Then the property \ref{it:parunity.1} is clear.
    The conditions \ref{it:parunity.2} and \ref{it:parunity.3} follow from $B_{\metric}(z, \varepsilon/4) \cap \mathcal{R}_{z'}^{\varepsilon/4} = \emptyset$ whenever $z, z' \in V$ satisfy $z \neq z'$.
    We will show the condition \ref{it:parunity.4}.
    Note that $\phi_{w}(x) = 0$ whenever $x \in B_{\metric}(z, 5\varepsilon/4)$ and $B_{\metric}(w, 5\varepsilon/4) \cap B_{\metric}(z, 5\varepsilon/4) \neq \emptyset$.
    Hence
    \[
    \psi_{z} = \phi_{z}\cdot\left(\sum_{w \in V; B_{\metric}(w, 5\varepsilon/4) \cap B_{\metric}(z, 5\varepsilon/4) \neq \emptyset}\phi_{w}\right)^{-1}.
    \]
    The metric doubling property implies that there exists a constant $M_{2}$ (depending only on the doubling constant) such that
    \[
    \sup_{z \in V}\#\{ w \in V \mid B_{\metric}(w, 5\varepsilon/4) \cap B_{\metric}(z, 5\varepsilon/4) \neq \emptyset \} \le M_{2}.
    \]
    Set $V(z) \coloneqq \{ w \in V \mid B_{\metric}(w, 5\varepsilon/4) \cap B_{\metric}(z, 5\varepsilon/4) \neq \emptyset \} \setminus \{ z \}$.
    By \eqref{negative-power} and \eqref{phi-upper},
    \begin{align*}
    	\mathcal{E}_{p}^{\Gamma}(\psi_{z})
    	&\lesssim \mathcal{E}_{p}^{\Gamma}(\phi_{z}) + \mathcal{E}_{p}^{\Gamma}\Biggr(\sum_{w \in V(z)}\phi_{w}\Biggl) \\
    	&\le \mathcal{E}_{p}^{\Gamma}(\phi_{z}) + M_{2}^{p - 1}\sum_{w \in V(z)}\mathcal{E}_{p}^{\Gamma}(\phi_{w})
    	\lesssim \varepsilon^{\hdim - \beta}.
    \end{align*}
    This completes the proof. 
\end{proof}

\section{Comparison with Korevaar--Schoen energies}\label{sec.domain}
In this section, we will give a characterization of $\mathcal{F}_{p}$ in terms of fractional Korevaar--Schoen energies.
The associated function spaces are also called \emph{Lipschitz--Besov spaces}.
For Dirichlet forms on fractals endowed with nice heat kernel estimates, such characterizations are well known; see, e.g., \cite{GHL03, Jon96, Kum00, PP99}\footnote{The proofs in these works rely on two-sided heat kernel estimates.}.

In this section, we will always assume that the metric measure space $(K,\metric,\measure)$ satisfies Assumption \ref{a:reg}.
The following main result in this section   claims that our $(1, p)$-Sobolev space $\mathcal{F}_{p}$ coincides with the critical fractional Korevaar-Schoen space $B_{p, \infty}^{\beta/p}$ in this setting (recall Definition \ref{defn.LB}).
\begin{thm}\label{thm.LB}
	Let $(K,d,m)$ be a metric measure space satisfying Assumption \ref{a:reg}.
Then, there exists a constant $C \ge 1$ (depending only on the constants associated with Assumption \ref{a:reg}) such that
    \begin{align}\label{FKS-comp}
        C^{-1}\abs{f}_{\mathcal{F}_{p}}^{p}
        &\le \liminf_{r \downarrow 0}\int_{K}\fint_{B_{\metric}(x, r)}\frac{\abs{f(x) - f(y)}^{p}}{r^{\beta}}\,\measure(dy)\measure(dx) \nonumber\\
        &\le \sup_{r > 0}\int_{K}\fint_{B_{\metric}(x, r)}\frac{\abs{f(x) - f(y)}^{p}}{r^{\beta}}\,\measure(dy)\measure(dx)
        \le C\abs{f}_{\mathcal{F}_{p}}^{p} \quad \text{for all $f \in L^{p}(K, \measure)$}.
    \end{align}
    In particular, $\mathcal{F}_{p} = B_{p, \infty}^{\beta/p}$ and
    \begin{align*}
        &\sup_{r > 0}\int_{K}\fint_{B_{\metric}(x, r)}\frac{\abs{f(x) - f(y)}^{p}}{r^{\beta}}\,\measure(dy)\measure(dx) \\
        &\hspace*{50pt}\le C^{2}\liminf_{r \downarrow 0}\int_{K}\fint_{B_{\metric}(x, r)}\frac{\abs{f(x) - f(y)}^{p}}{r^{\beta}}\,\measure(dy)\measure(dx) \quad \text{for all $f \in L^{p}(K, \measure)$.}
    \end{align*}
    Moreover, $\beta/p = s_{p}$, where $s_{p}$ is the critical exponent defined in \eqref{critical}.
\end{thm}

Before moving to the proof, let us make a remark on Sobolev embeddings for $\mathcal{F}_{p}$.
\begin{rmk}\label{rmk.sob-emb}
	Combining the above characterization of $\mathcal{F}_{p}$ and the methods of \cite{BCLS}, we immediately obtain analogues of the classical Sobolev embeddings (see also \cite[Theorem 4.3]{Bau24}).
	Indeed, we can easily check the truncation properties, namely the conditions $(H_{\infty}^{+})$ and $(H_{p})$ in \cite{BCLS}, of $\abs{\,\cdot\,}_{\mathcal{F}_{p}}$ from \eqref{FKS-comp}, and apply \cite[Theorems 3.4 and 9.1]{BCLS} by choosing a family of operators $\{ \widetilde{\mathcal{M}}_{r} \}$ as
	\[
	\widetilde{\mathcal{M}}_{r}f(x) \coloneqq \fint_{B_{\metric}(x, r^{p/\beta})}f\,d\measure \quad \text{for $r > 0$, $f \in L^{p}(K, \measure)$, $x \in K$.}
	\]
	We will not write the details because we do not use these results in this paper.
	Furthermore, a straightforward modification of \cite[Theorems 4.2 and 4.3]{AB23+}, where the authors modify the arguments in \cite[Section 8]{HK00} to fit with fractal settings in the case $p = 2$, yields Rellich--Kondrachov type compactness results in this context.
\end{rmk}

The proof of Theorem \ref{thm.LB} will be divided into two parts.
We start by showing
\[
\sup_{r > 0}\int_{K}\fint_{B_{\metric}(x, r)}\frac{\abs{f(x) - f(y)}^{p}}{r^{\beta}}\,\measure(dy)\measure(dx)
\lesssim \abs{f}_{\mathcal{F}_{p}}^{p}.
\]
To get this bound, we will use a standard argument using ``Poincar\'{e} inequality''.
\begin{lem}\label{lem.lower}
	There exists a constant $C > 0$ (depending only on the constants associated with Assumption \ref{a:reg}) such that for all Borel set $U$ of $K$ and $f \in L^{p}(K, \measure)$,
	\[
	\varlimsup_{r \downarrow 0}\int_{U}\fint_{B_{\metric}(x, r)}\frac{\abs{f(x) - f(y)}^{p}}{r^{\beta}}\,\measure(dy)\measure(dx)
        \le C\adjustlimits\varlimsup_{r \downarrow 0}\varliminf_{n \to \infty}\widetilde{\mathcal{E}}_{p, V_{n}(U_{r})}^{(n)}(f),
	\]
	where $U_{\delta}$ denotes the $\delta$-neighborhood of $U$, i.e., $U_{\delta} = \bigcup_{y \in U}B_{\metric}(y, \delta)$ for $\delta > 0$.
	Moreover,
	\[
	\sup_{r > 0}\int_{K}\fint_{B_{\metric}(x, r)}\frac{\abs{f(x) - f(y)}^{p}}{r^{\beta}}\,\measure(dy)\measure(dx)
        \le C\abs{f}_{\mathcal{F}_{p}}^{p}.
	\]
\end{lem}
\begin{proof}
    Let $r > 0$ and let $N_r \subseteq U$ be a maximal $r$-net of $U$ (with respect to the metric $\metric$).
    Note that $B_{\metric}(x, r) \subseteq B_{\metric}(y, 2r)$ for $y \in N_{r}$ and $x \in B_{\metric}(y, r)$.
    We see that
    \begin{align}\label{preKS-local}
        &\int_{U}\fint_{B_{\metric}(x, r)}\frac{\abs{f(x) - f(y)}^{p}}{r^{\beta}}\,\measure(dy)\measure(dx) \nonumber \\
        &\le \sum_{y \in N_{r}}\int_{B_{\metric}(y, r)}\fint_{B_{\metric}(x, r)}\frac{\abs{f(x) - f(y)}^{p}}{r^{\beta}}\,\measure(dy)\measure(dx) \nonumber \\
        &\lesssim \sum_{y \in N_{r}}\int_{B_{\metric}(y, 2r)}\fint_{B_{\metric}(y, 2r)}\frac{\abs{f(x) - f(y)}^{p}}{r^{\beta}}\,\mu(dy)\mu(dx) \qquad \text{(by \ref{cond.VD})} \nonumber \\
        &\lesssim \sum_{y \in N_{r}}\int_{B_{\metric}(y, 2r)}\fint_{B_{\metric}(y, 2r)}\Biggl\{ \frac{\abs{f(x) - f_{B_{\metric}(y, 2r)}}^{p}}{r^{\beta}} + \frac{\abs{f(y) - f_{B_{\metric}(y, 2r)}}^{p}}{r^{\beta}}\Biggr\}\,\measure(dy)\measure(dx) \nonumber \\
        &\lesssim \sum_{y \in N_{r}}\liminf_{n \to \infty}\widetilde{\mathcal{E}}_{p, V_{n}(B_{\metric}(y, 2Ar))}^{(n)}(f). \qquad \text{(by Lemma \ref{lem.PI-like})}
    \end{align}
    For any $y \in N_{r}$ and $w \in V_{n}(B_{\metric}(y, 2Ar))$, it is immediate that $w \in V_{n}(U_{2Ar})$.
    The overlap of $\bigl\{ V_{n}(B_{\metric}(y, 2Ar)) \bigr\}_{y \in N_{r}}$ can be controlled in the following manner.
    Let $y \in N_{r}$ and let $n \in \mathbb{N}$ be large enough so that $CR_{\ast}^{-n} < r$, where $C \ge 1$ is the constant in Definition \ref{d:compatible}.
    Then we easily see that $\{ p_{n}(w) \}_{w \in V_{n}(B_{\metric}(y, 2Ar))} \subseteq B_{\metric}(y, (2A + 1)r)$.
    In particular, we have
    \begin{align}\label{KSlower.overlap}
    	\max_{w \in V_{n}}\#\bigl\{ y \in N_{r} \bigm| w \in V_{n}(B_{\metric}(y, 2Ar)) \bigr\}
    	\le  \sup_{x \in K}\#
    	\bigl\{ y \in N_{r} \bigm| x \in B_{\metric}(y, (2A + 1)r) \bigr\}
    	\lesssim 1,
    \end{align}
    where we used the metric doubling property in the last inequality.

    Let us go back to the estimate on $\sum_{y \in N_{r}}\liminf_{n \to \infty}\widetilde{\mathcal{E}}_{p, V_{n}(B_{\metric}(y, 2Ar))}^{(n)}(f)$.
    By \eqref{KSlower.overlap},
    \begin{equation}\label{dEp-overlap}
    	\sum_{y \in N_{r}}\varliminf_{n \to \infty}\widetilde{\mathcal{E}}_{p, V_{n}(B_{\metric}(y, 2Ar))}^{(n)}(f)
    	\le \varliminf_{n \to \infty}\sum_{y \in N_{r}}\widetilde{\mathcal{E}}_{p, V_{n}(B_{\metric}(y, 2Ar))}^{(n)}(f)
    	\lesssim \varliminf_{n \to \infty}\widetilde{\mathcal{E}}_{p, V_{n}(U_{2Ar})}^{(n)}(f).
    \end{equation}
    Combining \eqref{dEp-overlap} with \eqref{preKS-local} and taking the limsup as $r \downarrow 0$, we get the first assertion.

    In the case $U = K$, by considering $\abs{f}_{\mathcal{F}_{p}}^{p}$ instead of $\liminf_{n \to \infty}\widetilde{\mathcal{E}}_{p, V_{n}(U_{2Ar})}^{(n)}(f)$ in \eqref{dEp-overlap}, we get
    \[
    \int_{K}\fint_{B_{\metric}(x, r)}\frac{\abs{f(x) - f(y)}^{p}}{r^{\beta}}\,\measure(dy)\measure(dx) \lesssim \abs{f}_{\mathcal{F}_{p}}^{p}.
    \]
    Taking the supremum completes the proof.
\end{proof}

Next we move to the converse bound:
\[
\liminf_{r \downarrow 0}\int_{K}\fint_{B(x, r)}\frac{\abs{f(x) - f(y)}^{p}}{r^{\beta}}\,\measure(dy)\measure(dx)
        \gtrsim \abs{f}_{\mathcal{F}_{p}}^{p}.
\]
Our approach is  similar to \cite[Theorem 5.2]{Bau24}  but we give a local version as well.
\begin{lem}\label{lem.upper}
	There exists a constant $C > 0$ (depending only on the constants associated with Assumption \ref{a:reg}) such that the following hold.
	For all $U \subseteq K$ and $f \in \mathcal{F}_{p}$,
	\begin{equation}\label{LBu.local}
		\limsup_{n \to \infty}\widetilde{\mathcal{E}}_{p, V_{n}(U)}^{(n)}(f) \le C\adjustlimits\lim_{\delta \downarrow 0}\liminf_{r \downarrow 0}\int_{U_{\delta}}\fint_{B_{\metric}(x, r)}\frac{\abs{f(x) - f(y)}^{p}}{r^{\beta}}\,\measure(dy)\measure(dx),
	\end{equation}
	where $U_{\delta}$ denotes the $\delta$-neighborhood of $U$.
	Furthermore, for all $f \in L^{p}(K, \measure)$,
	\begin{equation}\label{LBu.all}
		\abs{f}_{\mathcal{F}_{p}}^{p} \le C\varliminf_{r \downarrow 0}\int_{K}\fint_{B_{\metric}(x, r)}\frac{\abs{f(x) - f(y)}^{p}}{r^{\beta}}\,\measure(dy)\measure(dx).
	\end{equation}
\end{lem}
\begin{proof}
    Let $r \in (0, 1)$ and fix a maximal $r$-net $N_{r}(U) \subseteq U$ of $U$.
    Let $N_{r}$ be a maximal $r$-net of $(K, d)$ such that $N_{r}(U) \subseteq N_{r}$.
    We first observe that, by \eqref{e:holder} and \eqref{e:round}, there exists $\varepsilon>0$ such that for all   $n \in \mathbb{N}$ satisfying $R_*^{-n} < \varepsilon r$, we have 
    \[
    \widetilde{K}_{v} \cup \widetilde{K}_{w} \subseteq B_{\metric}(z, 5r/4) \quad \text{whenever $z \in K$, $\{ v, w \} \in E_{n}$ and $v \in V_{n}(B_{\metric}(z, r))$ .}
    \]
    Therefore, for all large $n \in \mathbb{N}$ and $f \in L^{p}(K, \measure)$,
    \[
    \widetilde{\mathcal{E}}_{p, V_{n}(U)}^{(n)}(f) \le \sum_{z \in N_{r}(U)}\widetilde{\mathcal{E}}_{p, V_{n}(B_{\metric}(z, 5r/4))}^{(n)}(f).
    \]

    To estimate $\widetilde{\mathcal{E}}_{p, V_{n}(B_{\metric}(z, 5r/4))}^{(n)}(f)$, we consider `discrete convolution operators'.
    (Such type approximation is originally considered by Coifman and Weiss \cite{CW}.)
    Let $\{ \psi_{z, r} \}_{z \in N_{r}}$ satisfy the conditions \ref{it:parunity.1}-\ref{it:parunity.4} in Lemma \ref{lem.unity} and define a linear operator $A_{r} \colon L^{p}(K, \measure) \to L^{p}(K, \measure)$ by
    \[
    A_{r}f \coloneqq \sum_{z \in N_{r}}f_{B_{\metric}(z, r/4)}\psi_{z, r}, \quad f \in L^{p}(K, \measure).
    \]
    Note that $A_{r}f \in \mathcal{F}_{p} \cap \mathcal{C}(K)$.
    We can show that $A_{r}$ is a bounded linear operator whose norm $\norm{A_{r}}_{L^p \to L^p}$ has a uniform bound with respect to $r$.
    Indeed, for any $f \in L^{p}(K, \measure)$,
    \begin{align*}
    	\norm{A_{r}f}_{L^p}^{p}
    	&= \int_{K}\abs{\sum_{z \in N_{r}}f_{B_{\metric}(z, r/4)}\psi_{z, r}(x)}^{p}\,\measure(dx) \\
    	&\le \int_{K}\left(\sum_{z \in N_{r}}\abs{f_{B_{\metric}(z, r/4)}}^{p}\psi_{z, r}(x)\right)\left(\sum_{z \in N_{r}}\psi_{z, r}(x)\right)^{p - 1}\,\measure(dx) \qquad \text{(by H\"{o}lder's inequality)} \\
    	&\le \int_{K}\left(\sum_{z \in N_{r}}\frac{\psi_{z, r}(x)}{\measure\bigl(B_{\metric}(z, r/4)\bigr)}\int_{B_{\metric}(z, r/4)}\abs{f}^{p}\,d\measure\right)\,\measure(dx) \qquad \text{(by H\"{o}lder's inequality)} \\
    	&\le \int_{K}\left(\sum_{z \in N_{r}}\frac{\indicator{B_{\metric}(z, r)}(x)}{\measure\bigl(B_{\metric}(z, r/4)\bigr)}\int_{B_{\metric}(z, r/4)}\abs{f}^{p}\,d\measure\right)\,\measure(dx) \\
    	&= \sum_{z \in N_{r}}\frac{\measure\bigl(B_{\metric}(z, r)\bigr)}{\measure\bigl(B_{\metric}(z, r/4)\bigr)}\int_{B_{\metric}(z, r/4)}\abs{f}^{p}\,d\measure
    	\lesssim \left(\sup_{x \in K}\#\{ z \in N_{r} \mid x \in B_{\metric}(z, r/4) \}\right)\norm{f}_{L^{p}}^{p}.
    \end{align*}
    Since $\sup_{x \in X}\#\{ z \in N_{r} \mid x \in B_{\metric}(z, r/4) \} \lesssim 1$ by the metric doubling property, we get $\norm{A_{r}}_{L^p \to L^p} \le C_{0}$, where $C_{0} > 0$ is a constant depending only on the doubling constant of $\measure$.

    For $g \in \mathcal{C}(K)$, we easily show that $A_{r}g \to g$ in the uniform norm as $r \downarrow 0$ by virtue of the uniform continuity of $g$.
    Indeed, for any $\varepsilon > 0$ there exists $r(\varepsilon) > 0$ such that $\abs{g(x) - g(y)} < \varepsilon$ whenever $\metric(x, y) < 3r(\varepsilon)/2$.
    Then for all $r < r(\varepsilon)$,
    \[
    \abs{g(x) - A_{r}g(x)} \le \sum_{z \in N_r}\abs{g(x) - g_{B_{\metric}(z, r/4)}}\psi_{z, r}(x) = \sum_{z \in N_{r}; \metric(z, x) < 5r/4}\abs{g(x) - g_{B_{\metric}(z, r/4)}}\psi_{z, r}(x).
    \]
    Let $x \in K$ and $z \in N_{r}$ such that $\metric(x, z) < 5r/4$.
    Since $\metric(x, y) < 3r/2$ for any $y \in B_{\metric}(z, r/4)$, we have
    \[
    \abs{g(x) - g_{B_{\metric}(z, r/4)}} \le \fint_{B_{\metric}(z, r/4)}\abs{g(x) - g(y)}\,\measure(dy) < \varepsilon.
    \]
    Hence
    \[
    \abs{g(x) - A_{r}g(x)} < \varepsilon\sum_{z \in N_r}\psi_{z, r}(x) = \varepsilon, \quad \forall r < r(\varepsilon),
    \]
    which implies $\sup_{x \in K}\abs{g(x) - A_{r}g(x)} \to 0$ as $r \downarrow 0$.
    In particular, $\norm{g - A_{r}g}_{L^p} \to 0$ as $r \downarrow 0$ when $g \in \mathcal{C}(K)$.

	Now we can show that $\norm{f - A_{r}f}_{L^p} \to 0$ as $r \downarrow 0$.
	Let $\varepsilon > 0$, $f \in L^{p}(K, \measure)$ and $g_{\varepsilon} \in \mathcal{C}(K)$ such that $\norm{f - g_{\varepsilon}}_{L^p} < \varepsilon$.
	Then we have
	\begin{align*}
		\norm{f - A_{r}f}_{L^p}
		&\le \norm{f - g_{\varepsilon}}_{L^p} + \norm{g_{\varepsilon} - A_{r}g_{\varepsilon}}_{L^p} + \norm{A_{r}g_{\varepsilon} - A_{r}f}_{L^p} \\
		&\le \varepsilon + \norm{g_{\varepsilon} - A_{r}g_{\varepsilon}}_{L^p} + C_{0}\varepsilon,
	\end{align*}
	and hence
	\[
	\limsup_{r \downarrow 0}\norm{f - A_{r}f}_{L^p} \le (1 + C_{0})\varepsilon.
	\]
	This shows $\norm{f - A_{r}f}_{L^p} \to 0$.

	With these preparations, we can estimate $\widetilde{\mathcal{E}}_{p, V_{n}(B_{\metric}(z, 5r/4))}^{(n)}(f)$.
    For $z \in N_{r}$ and $x \in B_{\metric}(z, 3r/2)$, we easily see that
    \[
    A_{r}f(x) = f_{B_{\metric}(z, r/4)} + \sum_{w \in N_{r} \cap B_{\metric}(z, 11r/4)}\bigl(f_{B_{\metric}(w, r/4)} - f_{B_{\metric}(z, r/4)}\bigr)\psi_{w, r}(x).
    \]
    We note that there exists a constant $M \in \mathbb{N}$ depending only on the metric doubling property such that
    \[
    \sup_{w \in N_{r}}\#\bigl(N_{r} \cap B_{\metric}(w,11r/4)\bigr) \le M.
    \]
    Also, since $\bigcup_{w \in V_{n}(B_{\metric}(z, 5r/4))}\widetilde{K}_{w} \subseteq B_{\metric}(z, 3r/2)$ for all large $n \in \mathbb{N}$, we see that
    \[
    M_{n}(A_{r}f) = f_{B_{\metric}(z, r/4)} + \sum_{w \in N_{r} \cap B_{\metric}(z, 11r/4)}\bigl(f_{B_{\metric}(w, r/4)} - f_{B_{\metric}(z, r/4)}\bigr)M_{n}\psi_{w, r} \quad \text{on $V_{n}\bigl(B_{\metric}(z, 5r/4)\bigr)$.}
    \]
    Hence we have
    \begin{align}\label{LBu.local-d}
    	&\widetilde{\mathcal{E}}_{p, V_{n}(B_{\metric}(z, 5r/4))}^{(n)}(A_{r}f) \nonumber \\
    	&= R_{\ast}^{n(\beta-\hdim)}\mathcal{E}_{p, V_{n}(B_{\metric}(z, 5r/4))}^{\mathbb{G}_{n}}\left(\sum_{w \in N_{r} \cap B_{\metric}(z, 11r/4)}\bigl(f_{B_{\metric}(w, r/4)} - f_{B_{\metric}(z, r/4)}\bigr)M_{n}\psi_{w, r}\right) \nonumber \\
    	&\le M^{p - 1}\sum_{w \in N_{r} \cap B_{\metric}(z, 11r/4)}\abs{f_{B_{\metric}(w, r/4)} - f_{B_{\metric}(z, r/4)}}^{p}\widetilde{\mathcal{E}}_{p, V_{n}(B_{\metric}(z, 5r/4))}^{(n)}(\psi_{w, r}) \nonumber \\
    	&\lesssim r^{\hdim - \beta}\sum_{w \in N_{r} \cap B_{\metric}(z, 11r/4)}\abs{f_{B_{\metric}(w, r/4)} - f_{B_{\metric}(z, r/4)}}^{p}.
    \end{align}
    For $z, w \in N_{r}$ with $w \in B_{\metric}(z, 11r/4)$, we note that $B_{\metric}(z, r/4) \cup B_{\metric}(w, r/4) \subseteq B_{\metric}(w, 3r) \cap B_{\metric}(z, 3r)$.
    Let $v \in \{ z, w \}$.
    By H\"{o}lder's inequality and $\hdim$-Ahlfors regularity of $\measure$,
    \begin{align*}
    	r^{\hdim}\abs{f_{B_{\metric}(v, r/4)} - f_{B_{\metric}(w, 3r)}}^{p}
    	&= r^{\hdim}\abs{\fint_{B_{\metric}(v, r/4)}\fint_{B_{\metric}(w, 3r)}(f(x) - f(y))\,\measure(dy)\measure(dx)}^{p} \\
    	&\le r^{\hdim}\fint_{B_{\metric}(v, r/4)}\fint_{B_{\metric}(w, 3r)}\abs{f(x) - f(y)}^{p}\,\measure(dy)\measure(dx) \\
    	&\lesssim \int_{B_{\metric}(w, 3r)}\fint_{B_{\metric}(w, 3r)}\abs{f(x) - f(y)}^{p}\,\measure(dy)\measure(dx) \\
    	&\lesssim \int_{B_{\metric}(w, 3r)}\fint_{B_{\metric}(x, 9r)}\abs{f(x) - f(y)}^{p}\,\measure(dy)\measure(dx).
    \end{align*}
    In particular,
    \begin{align*}
    	r^{\hdim}\abs{f_{B_{\metric}(w, r/4)} - f_{B_{\metric}(z, r/4)}}^{p}
    	&\lesssim r^{\hdim}\Bigl(\abs{f_{B_{\metric}(w, r/4)} - f_{B_{\metric}(w, 3r)}}^{p} + \abs{f_{B_{\metric}(w, 3r)} - f_{B_{\metric}(z, r/4)}}^{p}\Bigr) \\
    	&\lesssim \int_{B_{\metric}(w, 3r)}\fint_{B_{\metric}(x, 9r)}\abs{f(x) - f(y)}^{p}\,\measure(dy)\measure(dx),
    \end{align*}
    and thus \eqref{LBu.local-d} yields
    \begin{align}\label{LBu.local-d2}
    	&\widetilde{\mathcal{E}}_{p, V_{n}(B_{\metric}(z, 5r/4))}^{(n)}(A_{r}f) \nonumber \\
    	&\lesssim r^{-\beta}\sum_{w \in N_{r} \cap B_{\metric}(z, 11r/4)}\int_{B_{\metric}(w, 3r)}\fint_{B_{\metric}(x, 9r)}\abs{f(x) - f(y)}^{p}\,\measure(dy)\measure(dx).
    \end{align}
    Let us fix $\delta > 0$.
    Then, for all small enough $r > 0$ and $z \in N_{r}(U)$, we have $\bigcup_{w \in N_{r} \cap B_{\metric}(z, 11r/4)}B_{\metric}(w, 3r) \subseteq U_{\delta}$.
    Summing \eqref{LBu.local-d2} over $z \in N_{r}(U)$, we obtain
    \begin{align}\label{approx-energy}
        \widetilde{\mathcal{E}}_{p, V_{n}(U)}^{(n)}(A_{r}f)
        &\le \sum_{z \in N_{r}(U)}\widetilde{\mathcal{E}}_{p, V_{n}(B_{\metric}(z, 5r/4))}^{(n)}(A_{r}f) \nonumber \\
        &\lesssim (9r)^{-\beta}\int_{U_{\delta}}\fint_{B_{\metric}(x, 9r)}\abs{f(x) - f(y)}^{p}\,\measure(dy)\measure(dx),
    \end{align}
    where we used the metric doubling property in order to control the overlap of $\{ B_{\metric}(w, 3r) \mid w \in N_{r} \cap B_{\metric}(z, 11r/4) \}$ in the second inequality.
    We remark that \eqref{approx-energy} holds for large enough $n \in \mathbb{N}$ so that $R_{\ast}^{-n} < \varepsilon r$, where $\varepsilon > 0$ is as given in the beginning of the proof. 

    The estimate \eqref{LBu.local} is trivial when $\liminf_{r \downarrow 0} \int_{U_{\delta}}\fint_{B_{\metric}(x, r)}\frac{\abs{f(x) - f(y)}^{p}}{r^{\beta}}\,\measure(dy)\measure(dx) = \infty$, so we suppose that this liminf is finite.
    Pick a sequence $\{ r_{k} \}_{k \in \mathbb{N}}$ such that $r_{k} \downarrow 0$ as $k \to \infty$ and
    \[
    \lim_{k \to \infty}\int_{U_{\delta}}\fint_{B_{\metric}(x, r_{k})}\frac{\abs{f(x) - f(y)}^{p}}{r_{k}^{\beta}}\,\measure(dy)\measure(dx) = \liminf_{r \downarrow 0} \int_{U_{\delta}}\fint_{B_{\metric}(x, r)}\frac{\abs{f(x) - f(y)}^{p}}{r^{\beta}}\,\measure(dy)\measure(dx).
    \]
    If $f \in \mathcal{F}_{p}$, then \eqref{approx-energy} with $U = K$ and Lemma \ref{lem.lower} tell us that
    \[
    \abs{A_{r_{k}/9}f}_{\mathcal{F}_{p}}^{p}
    \lesssim \int_{K}\fint_{B_{\metric}(x, r_{k})}\frac{\abs{f(x) - f(y)}^{p}}{r_{k}^{\beta}}\,\measure(dy)\measure(dx)
    \lesssim \abs{f}_{\mathcal{F}_{p}}^{p} < \infty.
    \]
    In particular, $\{ A_{r_{k}/9}f \}_{k \in \mathbb{N}}$ is bounded in $\mathcal{F}_{p}$.
    Hence, by taking a subsequence, we can assume that $f_{k} \coloneqq A_{r_{k}/9}f$ converges weakly in $\mathcal{F}_{p}$ to some function $f_{\infty} \in \mathcal{F}_{p}$.
    Since $\mathcal{F}_{p}$ is continuously embedded in $L^{p}(K, \measure)$, we have $f_{\infty} = f$.
    By Mazur's lemma (Lemma \ref{lem.Mazur}) and \eqref{approx-energy}, we obtain \eqref{LBu.local}.

    We next consider the case $f \in L^{p}(K, \measure)$ and $U = K$.
    Similarly to the previous case, we assume that $\liminf_{r \downarrow 0} \int_{K}\fint_{B_{\metric}(x, r)}\frac{\abs{f(x) - f(y)}^{p}}{r^{\beta}}\,\measure(dy)\measure(dx) < \infty$ and pick a sequence $\{ r_{k} \}_{k \in \mathbb{N}}$ of positive numbers converging to $0$ and realizing this liminf.
    By \eqref{approx-energy},
    \[
    \abs{A_{r_{k}/9}f}_{\mathcal{F}_{p}}^{p}
    \lesssim \int_{K}\fint_{B_{\metric}(x, r_{k})}\frac{\abs{f(x) - f(y)}^{p}}{r_{k}^{\beta}}\,\measure(dy)\measure(dx),
    \]
    which implies the boundedness of $\{ A_{r_{k}/9}f \}_{k \in \mathbb{N}}$ in $\mathcal{F}_{p}$ since we suppose
    \[
    \lim_{k \to \infty}\int_{K}\fint_{B_{\metric}(x, r_{k})}\frac{\abs{f(x) - f(y)}^{p}}{r_{k}^{\beta}}\,\measure(dy)\measure(dx) < \infty.
    \]
    Similar arguments using Mazur's lemma as in the previous paragraph yield \eqref{LBu.all}.
\end{proof}

\begin{proof}[Proof of Theorem \ref{thm.LB}]
	The desired comparability follows from Lemmas \ref{lem.lower} and \ref{lem.upper}.

	We prove $\beta/p = s_{p}$.
    Since $\mathcal{F}_{p} = B_{p, \infty}^{\beta/p}$, it is immediate that
    \[
    \frac{\beta}{p} \le s_{p} = \sup\bigl\{ s > 0 \bigm| \text{$B_{p, \infty}^{s}(K, \metric, \measure)$ contains a non-constant function} \bigr\}.
    \]
    To prove the converse, let $s > \beta/p$ and let $f \in \mathcal{F}_{p} \supseteq B_{p, \infty}^{s}$ such that $\abs{f}_{\mathcal{F}_{p}} > 0$, i.e. $f$ is a function in $\mathcal{F}_{p}$ that is not constant.
    Let $\mathcal{A}_{n} \coloneqq A_{R_{\ast}^{-n}/9}$, where $A_{r} \, (r > 0)$ is the same operator as in the proof of Lemma \ref{lem.upper}.
    Then, by \eqref{approx-energy} with $r = R_{\ast}^{-n}/9$ for large enough $n \in \mathbb{N}$ and Theorem \ref{thm.Epgamma}, we have
    \[
    \frac{R_{\ast}^{-n\beta}}{R_{\ast}^{-nsp}}\mathcal{E}_{p}^{\Gamma}(\mathcal{A}_{n}f)
    \lesssim \int_{K}\fint_{B_{\metric}(x, R_{\ast}^{-n})}\frac{\abs{f(x) - f(y)}^{p}}{R_{\ast}^{-nsp}}\,\measure(dy)\measure(dx).
    \]
    Since $\liminf_{n \to \infty}\mathcal{E}_{p}^{\Gamma}(\mathcal{A}_{n}f) \gtrsim \abs{f}_{\mathcal{F}_{p}}^{p} > 0$, letting $n \to \infty$ yields
    \[
    \liminf_{r \downarrow 0}\int_{K}\fint_{B_{\metric}(x, r)}\frac{\abs{f(x) - f(y)}^{p}}{r^{sp}}\,\measure(dy)\measure(dx) = \infty \quad \text{whenever $f \in \mathcal{F}_{p} \setminus \mathbb{R}\indicator{K}$,}
    \]
    which completes the proof.
\end{proof}

Finally, we can prove the density of $\mathcal{F}_{p} \cap \mathcal{C}(K)$ in $\mathcal{F}_{p}$.
\begin{proof}[Proof of Theorem \ref{t:Fp}\ref{t:Fp-regFp}]
	For simplicity, let $\widehat{\mathcal{F}}_{p} \coloneqq \closure{\mathcal{F}_{p} \cap \mathcal{C}(K)}^{\norm{\,\cdot\,}_{\mathcal{F}_{p}}}$.
	The inclusion $\widehat{\mathcal{F}}_{p} \subseteq \mathcal{F}_{p}$ is obvious.
	So, we will prove $\mathcal{F}_{p} \subseteq \widehat{\mathcal{F}}_{p}$.

	By Theorem \ref{thm.LB}, we know that $\mathcal{F}_{p} = B_{p, \infty}^{\beta/p}$.
	Let $f \in \mathcal{F}_{p}$ and let $A_{r} \, (r > 0)$ be the operators defined in the proof of Lemma \ref{lem.upper}.
	Then $A_{r}f \in \mathcal{F}_{p} \cap \mathcal{C}(K) \subseteq \widehat{\mathcal{F}}_{p}$.
	By \eqref{approx-energy} with $U = K$, we have
	\[
	\abs{A_{r}f}_{\mathcal{F}_{p}}^{p} \lesssim \sup_{r > 0}\int_{K}\fint_{B_{\metric}(x, r)}\frac{\abs{f(x) - f(y)}^{p}}{r^{\beta}}\,\measure(dy)\measure(dx) \lesssim \abs{f}_{\mathcal{F}_{p}}^{p} < \infty.
	\]
	Since $\norm{A_{r}f}_{L^{p}} \lesssim \norm{f}_{L^{p}}$, we conclude that $\{ A_{r}f \}_{r > 0}$ is bounded in $\mathcal{F}_{p}$.
	Let $\{ A_{r_{k}}f \}_{k \in \mathbb{N}}$ be a convergent subsequence of $\{ A_{r}f \}_{r > 0}$ (with respect to the weak topology of $\mathcal{F}_{p}$).
	Applying Mazur's lemma (Lemma \ref{lem.Mazur}), we get
	\[
	f \in \closure{\bigl\{ \text{convex combinations of $\{ A_{r_{k}}f \}_{k \in \mathbb{N}}$}\bigr\}}^{\norm{\,\cdot\,}_{\mathcal{F}_{p}}} \subseteq \closure{\mathcal{F}_{p} \cap \mathcal{C}(K)}^{\norm{\,\cdot\,}_{\mathcal{F}_{p}}} = \widehat{\mathcal{F}}_{p},
	\]
	which completes the proof of Theorem \ref{t:Fp}.
\end{proof}

The following corollary concerns the case $p = 2$.
\begin{cor}\label{cor.DF}
	Suppose that Assumption \ref{a:reg} holds with $p = 2$.
	Then $(\mathcal{E}_{2}^{\Gamma}, \mathcal{F}_{2})$ is a $\measure$-symmetric regular Dirichlet form on $L^{2}(K, \measure)$.
\end{cor}
\begin{proof}
	We know that $\mathcal{E}_{2}^{\Gamma}$ is a non-negative quadratic form on $\mathcal{F}_{2}$ since $\mathcal{E}_{2}^{\Gamma}$ is a $\Gamma$-limit of non-negative quadratic forms (see \cite[Theorem 11.10]{DalMaso}).
	Since $\mathcal{F}_{2}$ is a Hilbert space, $(\mathcal{E}_{2}^{\Gamma}, \mathcal{F}_{2})$ defines a $\measure$-symmetric Dirichlet form on $L^{2}(K, \measure)$.
	By Theorem \ref{t:Fp}, the Dirichlet form $(\mathcal{E}_{2}^{\Gamma}, \mathcal{F}_{2})$ is regular.
\end{proof}

\section{Self-similar sets and self-similar energies}\label{sec.ss}
From this section, we move to the case of self-similar sets.
The main result in this section ensures the existence of a ``good'' $p$-energy reflecting geometric properties of the underlying space such as self-similarity and symmetry.
\subsection{Self-similar sets and related notations}
First, we give definitions of self-similar structure and related notations from the viewpoint of weighted partition theory by following \cite{Kig01, Kig20}.
\begin{defn}[Shift space]
	Let $S$ be a finite set with $\#S \ge 2$.
	For convention, we set $S^{0} \coloneqq \{ \phi \}$, where $\phi$ is an element called the \emph{empty word}.
    The collection of one-sided infinite sequences of symbols $S$ is denoted by $\Sigma(S)$, that is,
    \begin{equation*}
        \Sigma(S) = \{ \omega = \omega_{1}\omega_{2}\omega_{3}\cdots \mid \omega_{i} \in S \text{ for any } i \in \mathbb{N} \},
    \end{equation*}
    which is called the \emph{one-sided shift space} of symbols $S$.
    We define the \emph{shift map} $\sigma:\Sigma(S) \to \Sigma(S)$ by $\sigma(\omega_{1}\omega_{2}\cdots) = \omega_{2}\omega_{3}\cdots$ for each $\omega_{1}\omega_{2}\cdots \in \Sigma(S)$.
    The branches of $\sigma$ are denoted by $\sigma_{i} \, (i \in S)$, i.e. $\sigma_{i}:\Sigma(S) \to \Sigma(S)$ is defined as $\sigma_{i}(\omega_{1}\omega_{2}\cdots) = i\omega_{1}\omega_{2}\cdots$ for each $i \in S$ and $\omega_{1}\omega_{2}\cdots \in \Sigma(S)$.
    For $\omega = \omega_{1}\omega_{2}\cdots \in \Sigma(S)$ and $k \in \mathbb{Z}_{\ge 0}$, we define $[\omega]_{k} = \omega_{1} \cdots \omega_{k} \in S^{k}$.
    For $\omega = \omega_{1}\omega_{2}\cdots \in \Sigma(S)$ and $\tau = \tau_{1}\tau_{2}\cdots \in \Sigma(S)$, define the \emph{confluent} $\omega \wedge \tau \in \bigcup_{k \ge 0}S^{k}$ of $\omega$ and $\tau$ by
    \[
    \omega \wedge \tau = \omega_{1} \cdots \omega_{k}, \quad \text{where $k = \min\{ n \mid [\omega]_{n} \neq [\tau]_{n} \} - 1$.}
    \]
    If $k = 0$, then $\omega \wedge \tau$ is defined as the empty word $\phi$ (see also Definition \ref{defn.word}).
\end{defn}

We use $\Sigma$ to denote $\Sigma(S)$ when no confusion can occur.
We always consider $\Sigma = S^{\mathbb{N}}$ as a compact metrizable space equipped with the product topology.
It is known that, for any $\alpha \in (0, 1)$, the function $\delta_{\alpha} \colon \Sigma \times \Sigma \to [0, \infty)$ defined by
\begin{equation}\label{deltaalpha}
\delta_{\alpha}(\omega, \tau) \coloneqq
\begin{cases}
	\alpha^{\min\{ n \mid [\omega]_{n} \neq [\tau]_{n} \} - 1} \quad &\text{if $\omega \neq \tau$,} \\
	0\quad &\text{if $\omega = \tau$,}
\end{cases}
\end{equation}
gives a metric on $\Sigma$ and its topology coincides with that of $\Sigma$.

\begin{defn}[self-similar structure]\label{defn.ss-structure}
	Let $(K, \mathcal{O})$ be a compact metrizable space without isolated points, where $\mathcal{O}$ is the collection of open sets.
	Let $S$ be a finite set with $\#S \ge 2$ and let $\{ F_{i} \}_{i \in S}$ be a family of continuous injections from $K$ to itself.
	Then $(K, S, \{ F_{i} \}_{i \in S})$ is called a \emph{self-similar structure} if there exists a continuous surjection $\chi\colon \Sigma \to K$ such that $F_{i} \circ \chi = \chi \circ \sigma_{i}$ for all $i \in S$.
	The map $\chi$ is called the \emph{canonical projection} (or \emph{coding map}) of $(K, S, \{ F_{i} \}_{i \in S})$.
\end{defn}

We provide standard notations and facts about self-similar structures.
\begin{defn}\label{defn.word}
	Let $(K, S, \{ F_{i} \}_{i \in S})$ be a self-similar structure.
	Define $W_{k} \coloneqq S^{k} = \{ w_{1}\cdots w_{k} \mid \text{$w_{i} \in S$ for $i \in \{ 1, \dots, k \}$}\}$ for $k \in \mathbb{N}$ and $W_{\#} \coloneqq \bigcup_{k = 1}^{\infty}W_{k}$.
	We also set $W_{0} = \{ \phi \}$, where $\phi$ is the empty word, and $W_{\ast} \coloneqq \bigcup_{k \ge 0}W_{k}$.
	For $w = w_{1}w_{2} \cdots w_{k} \in W_{k}$, the length $\abs{w}_{W_{\ast}}$ of $w$ is defined as
	\[
	\abs{w}_{W_{\ast}} = k.
	\]
	If no confusion can occur, then we write $\abs{w}$ for $\abs{w}_{W_{\ast}}$ for simplicity.

	For $k \ge n \ge 0$ and $w = w_{1}w_{2} \cdots w_{k} \in W_{k}$, define $[w]_{n} \in W_{n}$ by
	\begin{equation}\label{defn.word-cut}
		[w]_{n} \coloneqq w_{1} \cdots w_{n}.
	\end{equation}
	We also define $i^{k} \coloneqq i\cdots i \in W_{k}$ for each $i \in S$ and $k \in \mathbb{Z}_{\ge 0}$.
	For $w \in W_{\ast}$ and $n \in \mathbb{N}$, define
	\[
	S^{n}(w) \coloneqq \bigl\{ v \in W_{n + \abs{w}} \bigm| [v]_{\abs{w}} = w \bigr\}.
	\]
	We use $S(w)$ to denote $S^{1}(w)$ for simplicity.

	For $w = w_{1}w_{2} \cdots w_{k} \in W_{\ast}$, we define
	\begin{equation}\label{defn.Fw}
		F_{w} \coloneqq F_{w_{1}} \circ F_{w_{2}} \circ \cdots \circ F_{w_{k}},
	\end{equation}
	and $K_{w} \coloneqq F_{w}(K)$.
	We also define $\sigma_{w} = \sigma_{w_{1}} \circ \sigma_{w_{2}} \circ \cdots \circ \sigma_{w_{k}}$ and $\Sigma_{w} \coloneqq \sigma_{w}(\Sigma)$.
\end{defn}
\begin{rmk}\label{rmk.symbol}
	We also use $W_{n}(S)$ and $\Sigma_{w}(S)$ to denote $W_{n}$ and $\Sigma_{w}$ respectively.
\end{rmk}

\begin{prop}[{\cite[Proposition 1.3.3]{Kig00}}]\label{prop.ss-proj}
	If $(K, S, \{ F_{i} \}_{i \in S})$ is a self-similar structure, then its canonical projection $\chi$ is uniquely determined in the following way: for any $\omega = \omega_{1}\omega_{2} \cdots \in \Sigma$,
	\begin{equation}\label{chi-unique}
		\bigl\{ \chi(\omega) \bigr\} = \bigcap_{k \ge 0}K_{\omega_{1} \cdots \omega_{k}}.
	\end{equation}
\end{prop}

We prepare fundamental notations on self-similar structures.
\begin{defn}
	Let $\mathcal{L} = (K, S, \{ F_{i} \}_{i \in S})$ be a self-similar structure.
	Define
	\[
	C_{\mathcal{L}} = \bigcup_{i \neq j \in S}(K_{i} \cap K_{j}), \quad \mathcal{C}_{\mathcal{L}} = \chi^{-1}(C_{\mathcal{L}}) \quad \text{and} \quad \mathcal{P}_{\mathcal{L}} = \bigcup_{n \ge 1}\sigma^{n}(\mathcal{C}_{\mathcal{L}}).
	\]
	Also, define $\mathcal{V}_{0} = \chi(\mathcal{P}_{\mathcal{L}})$.
\end{defn}
\begin{rmk}
	Usually the notation $V_{0}$ is used to denote $\mathcal{V}_{0}$.
	We employ $\mathcal{V}_{0}$ in order to avoid a conflict of notations.  We use $V_{n}$ to denote the vertex set of $\mathbb{G}_{n}$.
\end{rmk}

The set $\mathcal{V}_{0}$ describes the `boundary' of $K$ in the following sense.
\begin{prop}[{\cite[Proposition 1.3.5(2)]{Kig01}}]\label{prop.basic-ss}
	Let $\mathcal{L} = (K, S, \{ F_{i} \}_{i \in S})$ be a self-similar structure.
	If $\Sigma_{v} \cap \Sigma_{w} = \emptyset$, then $K_{v} \cap K_{w} = F_{v}(\mathcal{V}_{0}) \cap F_{w}(\mathcal{V}_{0})$.
\end{prop}

We next recall a class of natural measures on a self-similar structure, which is called \emph{self-similar measures}.
\begin{prop}[e.g. {\cite[Proposition 1.4.4]{Kig01}} and {\cite{Hut81}}]
	Let $(\theta_{i})_{i \in S}$ satisfy $\theta_{i} \in (0, 1)$ for all $i \in S$ and $\sum_{i \in S}\theta_{i} = 1$.
	Then there exists the unique Borel regular probability measure $\measure$ on $K$ such that, for every $A \in \mathcal{B}(K)$,
	\[
	\measure(A) = \sum_{i \in S}\theta_{i}\measure\bigl(F_{i}^{-1}(A)\bigr).
	\]
	Such the measure $\measure$ is called self-similar measure on $K$ with weight $(\theta_{i})_{i \in S}$.
\end{prop}

We introduce a useful notation.
Let $(a_{i})_{i \in S} \in (0, \infty)^{S}$ be a sequence of positive numbers.
For $w = w_{1}w_{1} \cdots w_{k} \in W_{\ast}$, define
\[
a_{w} \coloneqq a_{w_{1}}a_{w_{2}} \cdots a_{w_{k}}.
\]
\begin{prop}[{\cite[Theorems 1.2.4 and 1.2.7]{Kig09}}]\label{prop.ss-meas}
	Suppose that $K \neq \closure{\mathcal{V}_{0}}$.
	Let $\measure$ be a self-similar measure with weight $(\theta_{i})_{i \in S}$.
	Then $\measure(K_{w}) = \theta_{w}$ for any $w \in W_{\ast}$.
	Furthermore, if $v \neq w \in W_{\ast}$ with $K_{v} \cup K_{w} \neq K_{z}$ for some $z \in \{ v, w \}$, then $\measure(K_{v} \cap K_{w}) = 0$.
\end{prop}

In practice, many examples of self-similar structure are realized as \emph{self-similar sets in $\mathbb{R}^{D}$}.
The main object in this paper, namely the planar Sierpi\'{n}ski carpet in Section \ref{sec.PSC}, also belongs to this class, so we provide the setting of it here.
Let $D \in \mathbb{N}$.
Let $S$ be a non-empty finite set and let $(r_{i})_{i \in S} \in (0, 1)^{S}$.
For each $i \in S$, let $f_{i} \colon \mathbb{R}^{D} \to \mathbb{R}^{D}$ be an $r_{i}$-similitude, i.e. the map $f_{i}$ is given by $f_{i}(x) = r_{i}U_{i}x + q_{i} \, (x \in \mathbb{R}^{D})$ for some $U_{i} \in O(D)$ and $q_{i} \in \mathbb{R}^{D}$.
Here, $O(D)$ denotes the orthogonal group in dimension $D$.
Let $K$ be the unique non-empty compact subset of $\mathbb{R}^{D}$ such that $\bigcup_{i \in S}f_{i}(K) = K$ and let $F_{i} \coloneqq \restr{f_{i}}{K}$.
Such $K$ is called the self-similar set associated with the iterated function system $\{ f_{i} \}_{i \in S}$.
It is easy to check that $(K, S, \{ F_{i} \}_{i \in S})$ is a self-similar structure.

The reader can find many examples (and figures) of self-similar sets in fundamental textbooks on fractal geometry (see \cite[Section 1]{Kig01} for example), so we skip concrete examples here.

We next recall the famous \emph{open set condition}, which is introduced by Moran \cite{Mor46}.
The self-similar set $(K, S, \{ F_{i} \}_{i \in S})$ in $\mathbb{R}^{D}$ satisfies the open set condition if there exists a bounded open non-empty subset $O$ of $\mathbb{R}^{D}$ such that
\[
\bigcup_{i \in S}F_{i}(O) \subseteq O \quad \text{and} \quad F_{i}(O) \cap F_{j}(O) = \emptyset \quad \text{for $i \neq j \in S$.}
\]
This condition allows us to determine the Hausdorff dimension of $K$ with respect to the Euclidean metric.
Let $\metric$ be the normalized Euclidean metric of $\mathbb{R}^{D}$ so that $\diam(K, \metric) = 1$.
Let $\hdim > 0$ be the number satisfying
\begin{equation}\label{ss-df}
	\sum_{i \in S}r_{i}^{\hdim} = 1,
\end{equation}
and suppose that $(K, S, \{ F_{i} \}_{i \in S})$ satisfies the open set condition.
Then, by Moran's theorem (see \cite{Mor46, Hut81} or \cite[Corollary 1.5.9]{Kig01}), the Hausdorff dimension of $(K, \metric)$ is $\hdim$.
Moreover, there exists a constant $C \ge 1$ such that
\[
C^{-1}\measure(A) \le \mathcal{H}^{\hdim}(A) \le C\measure(A) \quad \text{for all $A \in \mathcal{B}(\mathbb{R}^{D})$,}
\]
where $\mathcal{H}^{\hdim}$ is the $\hdim$-dimensional Hausdorff measure (with respect to the metric $\metric$) and $\measure$ is the self-similar measure with weight $\bigl(r_{i}^{\hdim}\bigr)_{i \in S}$.
For a proof of this result, see \cite[Theorem 1.5.7]{Kig01} for example.

\subsection{Self-similar \texorpdfstring{$p$-energy}{p-energy}}
We now provide a general construction of self-similar energies.
To state the result, we introduce the notion of \emph{closed invariant sub-cone with respect to the renormalization}.
\begin{defn}
	Let $(K, S, \{ F_{i} \}_{i \in S})$ be a self-similar structure and let $m$ be a Borel-regular probability measure on $K$.
    Let $p \in (1, \infty)$ and $\rho = (\rho_{i})_{i \in S} \in (0, \infty)^{S}$.
    Let $\mathcal{F}$ be a linear subspace of $L^{p}(K, m)$ with $f \circ F_{i} \in \mathcal{F}$ for any $i \in S$ and $f \in \mathcal{F}$.
    \begin{enumerate}[\rm(1)]
    	\item For any functional $E \colon \mathcal{F} \to [0, \infty)$, define $\mathcal{S}_{\rho}E \colon \mathcal{F} \to [0,\infty)$ by
    		\[
    		\mathcal{S}_{\rho}E(f) \coloneqq \sum_{i \in S}\rho_{i}E(f \circ F_i) \quad \text{for $f \in \mathcal{F}$.}
    		\]
    	\item Let $\mathcal{U} \subseteq \{\mathcal{E}\colon \mathcal{F} \to [0, \infty) \mid \text{$\mathcal{E}^{1/p}$  is a semi-norm} \}$. The set $\mathcal{U}$ is said to be a \emph{closed invariant sub-cone with respect to $\mathcal{S}_{\rho}$} if it satisfies the following conditions (a)-(c).
    		\begin{enumerate}[\rm(a)]
 				\item $a_{1}E^{(1)} + a_{2}E^{(2)} \in \mathcal{U}$ for any $a_{1}, a_{2} \ge 0$ and $E^{(1)}, E^{(2)} \in \mathcal{U}$.
    			\item If $\bigl\{ E^{(n)} \bigr\}_{n \in \mathbb{N}} \subseteq \mathcal{U}$ and $\lim_{n \to \infty}E^{(n)}(f) \eqqcolon E(f)$ exists for any $f \in \mathcal{F}$, then $E \in \mathcal{U}$.
    			\item $\mathcal{S}_{\rho}E \in \mathcal{U}$ for any $E \in \mathcal{U}$.
    		\end{enumerate}
    \end{enumerate}
\end{defn}

The following theorem gives a self-similar energy as a fixed point of $\mathcal{S}_{\rho}$ (see \cite[Theorem 1.5]{Kig00}, \cite[Theorem 5.21]{KS24+}). 
In Section \ref{sec.PSC}, we will apply this theorem with $\mathcal{D} = \mathcal{F}_{p}$ and $\mathsf{E} = \mathcal{E}_{p}^{\Gamma}$ (in Theorem \ref{thm.Epgamma}) to get a ``canonical'' self-similar $p$-energy on the Sierpi\'{n}ski carpet.  The condition (\hyperref[PSS]{PSS}) in the following theorem plays a crucial role in the existence of a self-similar $p$-energy. It is not hard to see that this condition is necessary for the conclusion to hold and hence can be thought of as a \emph{pre-self-similarity} condition. 
\begin{thm}\label{thm.fix}
	Let $(K, S, \{ F_{i} \}_{i \in S})$ be a self-similar structure and let $m$ be a Borel-regular probability measure on $K$.
    Let $p \in (1, \infty)$ and let $\mathcal{D}$ be a linear subspace of $L^{p}(K, m)$.
    Suppose that there exists a functional $\mathsf{E} \colon \mathcal{D} \to [0, \infty)$ such that $\mathsf{E}(\,\cdot\,)^{1/p}$ is a semi-norm and $(\mathcal{D}, \norm{\,\cdot\,}_{\mathcal{D}})$ is a separable Banach space, where $\norm{f}_{\mathcal{D}} \coloneqq \norm{f}_{L^{p}(m)} + \mathsf{E}(f)^{1/p}$.
    In addition, we suppose that the following condition \textup{(\hyperref[PSS]{PSS})} holds.
   	\begin{itemize}
   		\item [\textup{(PSS)}]\label{PSS} It holds that $f \circ F_i \in \mathcal{D}$ for any $f \in \mathcal{D}$ and $i \in S$. Furthermore, there exist $\rho = (\rho_{i})_{i \in S} \in (0, \infty)^{S}$ and $C \ge 1$ such that for any $k \in \mathbb{Z}_{\ge 0}$ and $f \in \mathcal{D}$,
   			   \begin{equation}\label{ss.comparable}
               		C^{-1}\mathsf{E}(f) \le \sum_{w \in W_{k}}\rho_{w}\mathsf{E}(f \circ F_w) \le C\mathsf{E}(f),
    		   \end{equation}
    		 where we set $\rho_{\phi} \coloneqq 1$.
   	\end{itemize}
    Then there exists $\mathcal{E}_{p} \colon \mathcal{D} \to [0, \infty)$ satisfying the following conditions \ref{it:fixed.comp}-\ref{it:fixed.inv}.
    \begin{enumerate}[\rm(i)]
    	\item\label{it:fixed.comp} $\mathcal{E}_{p}(\,\cdot\,)^{1/p}$ is a semi-norm and $C^{-1}\mathsf{E}(f) \le \mathcal{E}_{p}(f) \le C\mathsf{E}(f)$ for every $f \in \mathcal{D}$, where $C \ge 1$ is the same as in \eqref{ss.comparable}.
    	\item\label{it:fixed.ss} $\mathcal{E}_{p}$ is self-similar, i.e. for every $f \in \mathcal{D}$ and $k \in \mathbb{Z}_{\ge 0}$,
    	\begin{equation}\label{ss}
        	\mathcal{E}_{p}(f) = \sum_{w \in W_{k}}\rho_{w}\mathcal{E}_{p}(f \circ F_w).
    	\end{equation}
    	\item\label{it:fixed.inv} If $\mathcal{U}$ is a closed invariant sub-cone with respect to $\mathcal{S}_{\rho}$ and $\mathsf{E} \in \mathcal{U}$, then $\mathcal{E}_{p} \in \mathcal{U}$.
    \end{enumerate}
\end{thm}
\begin{proof}[Sketch of Proof] 
	The existence of a self-similar energy with the desired properties follows from a standard argument as in \cite{Kig00,KS24+}.  
    By \eqref{ss.comparable}, there exists $C \ge 1$ such that
    \[
    C^{-1}\mathsf{E}(f) \le \mathcal{S}_{\rho}^{n}\mathsf{E}(f) \le C\mathsf{E}(f) \quad \text{for any $f \in \mathcal{D}$ and any $n \in \mathbb{N}$.}
    \] 
    Since $\mathcal{D}$ is separable, there exists $\{ n_{k} \}_{k \in \mathbb{N}} \subseteq \mathbb{N}$ with $n_{k} < n_{k + 1}$ for any $k \in \mathbb{N}$ such that the following limit exists in $[0,\infty)$ for any $f \in \mathcal{D}$:
    \begin{equation}\label{e:fixpt.explicit}
		\mathcal{E}_{p}(f) \coloneqq \lim_{k \to \infty}\frac{1}{n_{k}}\sum_{j = 0}^{n_{k} - 1}\mathcal{S}_{\rho}^{j}\mathsf{E}(f).   
	\end{equation}
	Then it is easy to see that $C^{-1}\mathsf{E} \le \mathcal{E}_{p} \le C\mathsf{E}$ and that $\mathcal{E}_{p}$ has the desired self-similarity \eqref{ss}. 
	Also, if $\mathsf{E} \in \mathcal{U}$, then $\mathcal{S}_{\rho}^{n} \in \mathcal{U}$. Hence $\mathcal{E}_{p} \in \mathcal{U}$ by \eqref{e:fixpt.explicit}.  
\end{proof}

Next we will explain how to apply Theorem \ref{thm.Epgamma} in the self-similar setting.
To this end, we introduce a sequence of finite graphs equipped with a family of projective maps (Definition \ref{defn.pf}) associated with the underlying self-similar structure.
Let us fix $R_{\ast} \in (1, \infty)$ and let $(K, S, \{ F_i \}_{i \in S})$ be a self-similar structure.
We also fix a metric $d$ on $K$ so that the metric topology induced by $d$ coincides with the original topology of $K$ and $\diam(K,d) = 1$.
Then, by \cite[proposition 1.3.6]{Kig01}, we have
\begin{equation}\label{G3}
	\adjustlimits\lim_{n \to \infty}\max_{w \in W_{n}}\diam(K_{w}, \metric) = 0.
\end{equation}
For $n \in \mathbb{N}$, define a graph $\mathbb{G}_{n} = (V_n, E_n)$ by setting
\begin{equation}\label{e:defn.Vn-ss}
	V_{n} \coloneqq \bigl\{ w \in W_{n} \bigm| R_{\ast}^{-n} \le \diam(K_{w}, \metric) < R_{\ast}^{-n + 1} \bigr\}
\end{equation}
and
\begin{equation}\label{e:defn.En-ss}
	E_{n} \coloneqq \bigl\{ \{ v, w \} \in V_{n} \times V_{n} \bigm| v \neq w, K_{v} \cap K_{w} \neq \emptyset \bigr\}.
\end{equation}
(The vertex set $V_{n}$ is the same as $\Lambda_{R_{\ast}^{-1}}^{d}$ in \cite[Definition 2.3.1]{Kig20}.)
For $k, n \in \mathbb{N}$ with $k < n$ and $w \in V_{n}$, define $\pi_{n, k}(w)$ as the unique element of $V_{k}$ such that $[w]_{\abs{v}} = v$.
Then it is immediate that the map $\pi_{n, k} \colon V_{n} \to V_{k}$ is surjective.
Also, we note that $\Sigma = \bigsqcup_{w \in V_{n}}\Sigma_{w}$ for each $n \in \mathbb{N}$.

We next introduce a partition $\widetilde{K}_{w} \, (w \in W_{\ast})$ associated with the self-similar structure.
Let $N_{\ast} \coloneqq \#S$ and enumerate $S$ as $\{ i(1), \dots, i(N_{\ast}) \}$.
Define $\widetilde{K}_{i(j)} \, (j = 1, \dots, N_{\ast})$ inductively as follows.
Let $\widetilde{K}_{i(1)} \coloneqq K_{i(1)}$.
For $j = 1, \dots, N_{\ast} - 1$, define
\begin{equation} \label{e:defKtilde}
	\widetilde{K}_{i(j + 1)} \coloneqq K_{i(j + 1)} \setminus \bigcup_{k = 1}^{j}\widetilde{K}_{i(k)}.
\end{equation}
Then $\widetilde{K}_{i(j)} \, (j = 1, \dots, N_{\ast})$ are pairwise disjoint and $\bigcup_{j = 1}^{N_{\ast}}\widetilde{K}_{i(j)} = K$.
Suppose that a family $\bigl\{ \widetilde{K}_{w} \bigr\}_{w \in \bigcup_{m \le n}W_{m}}$ is chosen so that it satisfies the following conditions:
\[
\bigcup_{w \in W_{m}}\widetilde{K}_{w} = K \quad \text{for each $m \in \{ 1, \dots, n \}$,}
\]
\[
\widetilde{K}_{v} \cap \widetilde{K}_{w} = \emptyset \quad \text{for any distinct $v, w \in \bigcup_{m \le n}W_{m}$ with $\abs{v} = \abs{w}$,}
\]
and
\[
\widetilde{K}_{w} = \bigcup_{i \in S}\widetilde{K}_{wi} \quad \text{for any $m \in \{ 1, \dots, n - 1 \}$, $w \in W_{m}$ and $i \in S$.}
\]
We now define $\bigl\{ \widetilde{K}_{v} \bigr\}_{v \in W_{n + 1}}$ as follows.
Let $w \in W_{n}$ and $\widetilde{K}_{wi(1)} \coloneqq K_{wi(1)} \cap \widetilde{K}_{w}$.
For $j = 1, \dots, N_{\ast} - 1$, we inductively define
\[
\widetilde{K}_{wi(j + 1)} \coloneqq \left(K_{wi(j + 1)} \setminus \bigcup_{k = 1}^{j}\widetilde{K}_{wi(k)}\right) \cap \widetilde{K}_{w}.
\]
This construction yields a family $\bigl\{ \widetilde{K}_{w} \bigr\}_{w \in W_{\ast}}$ satisfying the conditions \ref{it:compat.parti} and \ref{it:compat.proj} in Definition \ref{d:compatible}.

As in Lemma \ref{l:bp}, let $\measure_{n}(v) \coloneqq \measure\bigl(\widetilde{K}_{v}\bigr)$ for each $n \in \mathbb{N}$ and $v \in V_{n}$, where $\measure$ is a fixed self-similar probability measure.
We note that, by Proposition \ref{prop.ss-meas}, $\measure_{n}(v) = \measure(K_{v})$ for all $v \in V_{n}$ if $K \neq \closure{\mathcal{V}_{0}}$.
Also, the self-similarity of $\measure$ implies that $(\measure_{n})_{n \in \mathbb{N}}$ is consistent under $K \neq \closure{\mathcal{V}_{0}}$.

We now introduce the analogue of Assumption \ref{a:reg} when the underlying space is a self-similar set.
\begin{assum}\label{a:reg-ss}
	Let $p \in (1, \infty)$.
	Let $(K, S, \{ F_{i} \}_{i \in S})$ be a self-similar set such that $K$ is connected, $\#K \ge 2$ and $K \neq \closure{\mathcal{V}_{0}}$.
	Let $(r_{i})_{i \in S} \in (0,1)^{S}$ so that $F_{i}$ is an $r_{i}$-similitude.
	Let $\metric$ be the normalized Euclidean metric on $K$ so that $\diam(K, \metric) = 1$ and let $\measure$ be a self-similar probability measure with weight $\bigl(r_{i}^{\hdim}\bigr)_{i \in S} \in (0, 1)^{S}$, where $\hdim$ is the Hausdorff dimension of $(K, \metric)$.
	Let $R_{\ast} \in (1, \infty)$, let $\{ \mathbb{G}_{n} = (V_{n}, E_{n}) \}_{n \in \mathbb{N}}$, $\pi_{n, k} \, \text{($n, k \in \mathbb{N}$ with $k < n$)}$ and $\widetilde{K}_{w} \, (w \in W_{\ast})$ be  defined as above in \eqref{e:defn.En-ss}, \eqref{e:defn.Vn-ss}, and \eqref{e:defKtilde}.
	Let $\measure_{n}(w) = \measure\bigl(\widetilde{K}_{w}\bigr)$ for $w \in W_{n}$.
	We consider the following geometric and analytic conditions.

	\noindent
	\textbf{$\bullet$ Geometric conditions:} The measure $\measure$ is $\hdim$-Ahlfors regular. In addition, $\{ \mathbb{G}_{n} \}_{n \in \mathbb{N}}$ is $R_{\ast}$-scaled and $R_{\ast}$-compatible with $(K,d)$, i.e. \eqref{e:sc1}, \eqref{e:sc2}, \eqref{e:holder} and \eqref{e:round} hold.

	\noindent
	\textbf{$\bullet$ Analytic conditions:} The sequence $\{ \mathbb{G}_{n} \}_{n \in \mathbb{N}}$ satisfies \ref{cond.UPI} and \hyperref[cond.UCF]{\textup{U-CF$_{p}(\vartheta, \beta)$}} for some $\beta > 0$ and $\vartheta \in (0, 1]$.
\end{assum}

Obviously, Assumption \ref{a:reg-ss} for a self-similar set $(K, S, \{ F_{i} \}_{i \in S})$ implies Assumption \ref{a:reg}.
Note that the Banach space $\mathcal{F}_{p}$ is separable by Theorem \ref{t:Fp}\ref{t:Fp-sep}.
Now the following corollary is immediate from Theorems \ref{t:Fp}, \ref{thm.Epgamma} and \ref{thm.fix}.
\begin{cor}\label{cor.ss-energy}
	Suppose that a self-similar set $(K, S, \{ F_{i} \}_{i \in S})$ satisfies Assumption \ref{a:reg-ss} and let $(\mathcal{E}_{p}^{\Gamma}, \mathcal{F}_{p})$ be the $p$-energy on $(K, \metric, \measure)$ in Theorem \ref{thm.Epgamma}.
	In addition, assume that the $p$-energy $\mathcal{E}_{p}^{\Gamma}$ satisfies the pre-self-similarity condition \textup{(\hyperref[PSS]{PSS})} in Theorem \ref{thm.fix}.
	Then there exists a self-similar $p$-energy $(\mathcal{E}_{p}, \mathcal{F}_{p})$ satisfying the conditions \ref{it:fixed.comp}-\ref{it:fixed.inv} in Theorem \ref{thm.fix}.
	Furthermore, $\mathcal{F}_{p} \cap \mathcal{C}(K)$ is dense both in $(\mathcal{C}(K), \norm{\,\cdot\,}_\infty)$ and in $(\mathcal{F}_{p}, \norm{\,\cdot\,}_{\sF_p})$.
\end{cor}
\begin{rmk}
	In light of Theorem \ref{thm.Epgamma}\ref{it:Epgamma.Cp}, the pre-self-similarity condition \textup{(PSS)} can be regarded as a property of $(1, p)$-Sobolev space $\mathcal{F}_{p}$ and its semi-norm $\abs{\,\cdot\,}_{\mathcal{F}_{p}}$.
\end{rmk}

\section{Associated self-similar energy measures}\label{sec.emeas}
In this section, we  construct \emph{energy measures} associated with a `canonical $p$-energy' as constructed in Corollary \ref{cor.ss-energy} and study its basic properties.
Our construction follows an approach of Hino that heavily depends on the self-similarity of both the underlying space and the energy \cite[Lemma 4.1]{Hin05}.

First, we fix our framework in this section.
\begin{assum}\label{assum.ss}
	Let $(K, S, \{ F_{i} \}_{i \in S})$ be a self-similar structure equipped with a compatible metric $d$ such that $\diam(K, \metric) = 1$ and such that $K$ is connected.
	Let $\measure$ be a self-similar measure on $K$.
	Let $p \in (1, \infty)$ and let $(\mathcal{D}, \abs{\,\cdot\,}_{\mathcal{D}})$ be a  non-empty semi-normed space such that $\mathcal{D}$ is a linear subspace of $L^{p}(K, \measure)$.
	Let $\mathcal{E}_{p} \colon \mathcal{D} \to [0, \infty)$.
	\begin{enumerate}[\rm(1)]
		\item\label{it:assum.ss.domain} Let $\norm{\,\cdot\,}_{\mathcal{D}} \coloneqq \abs{\,\cdot\,}_{\mathcal{D}} + \norm{\,\cdot\,}_{L^{p}(\measure)}$, which defines a norm on $\mathcal{D}$. The normed space $(\mathcal{D}, \norm{\,\cdot\,}_{\mathcal{D}})$ is a reflexive Banach space. Furthermore, $\bigl\{ f \in \mathcal{D} \bigm| \abs{f}_{\mathcal{D}} = 0 \bigr\} = \mathbb{R}\indicator{K}$ and
		\begin{equation}\label{e:closedFw}
			f \circ F_{w} \in \mathcal{D} \quad \text{for any $f \in \mathcal{D}$ and any $w \in W_{\ast}$.}
		\end{equation}
		(See also \cite[Remark 5.5]{KS24+}.)  
		\item\label{it:assum.ss.energy} $\mathcal{E}_{p}(\,\cdot\,)^{1/p}$ is a semi-norm on $\mathcal{D}$ and there exist a constant $C \ge 1$ and a weight $\rho = (\rho_{i})_{i \in S} \in (0, \infty)^{S}$ such that, for any $f \in \mathcal{D}$ and $m \in \mathbb{Z}_{\ge 0}$,
		\[
		C^{-1}\abs{f}_{\mathcal{D}}^{p} \le \mathcal{E}_{p}(f) \le C\abs{f}_{\mathcal{D}}^{p}, \quad \text{and} \quad \mathcal{E}_{p}(f) = \sum_{w \in W_{m}}\rho_{w}\mathcal{E}_{p}(f \circ F_w).
		\]
		Furthermore, for any $f \in \mathcal{D}$ and $1$-Lipschitz function $\varphi \in \contfunc(K)$,
		\[
		\varphi \circ f \in \mathcal{D} \quad \text{and} \quad \mathcal{E}_{p}(\varphi \circ f) \le \mathcal{E}_{p}(f).
		\]
	\end{enumerate}
\end{assum}

We always suppose Assumption \ref{assum.ss} in this section.
(Note that the assumptions in Corollary \ref{cor.ss-energy}, namely Assumption \ref{a:reg-ss} and (\hyperref[PSS]{PSS}) imply Assumption \ref{assum.ss}.)
In this setting, we can introduce energy measures with respect to $(\mathcal{E}_{p}, \mathcal{D})$ in the following manner.
Let $f \in \mathcal{D}$ and $n \in \mathbb{Z}_{\ge 0}$.
Define a finite measure $\mathfrak{m}_{p}^{(n)}\langle f \rangle$ on $W_{n}$ by setting $\mathfrak{m}_{p}^{(n)}\langle f \rangle(\{w\}) \coloneqq \rho_{w}\mathcal{E}_{p}(f \circ F_{w})$ for each $w \in W_{n}$.
Due to the following equalities:
\begin{align*}
	\sum_{v \in S(w)}\mathfrak{m}_{p}^{(n + 1)}\langle f \rangle(\{v\})
	= \rho_{w}\sum_{i \in S}\rho_{i}\mathcal{E}_{p}\bigl((f \circ F_w) \circ F_{i}\bigr)
	= \mathfrak{m}_{p}^{(n)}\langle f \rangle(\{w\}),
\end{align*}
we can use Kolmogorov's extension theorem (see \cite[Theorem 12.1.2]{Dud} for example) to get a finite Borel measure $\mathfrak{m}_{p}\langle f \rangle$ on $\Sigma = S^{\mathbb{N}}$ such that
\[
\mathfrak{m}_{p}\langle f \rangle(\Sigma_{w}) = \rho_{w}\mathcal{E}_{p}(f \circ F_w) \quad \text{for any $n \in \mathbb{Z}_{\ge 0}$ and $w \in W_{n}$.}
\]
Clearly, $\mathfrak{m}_{p}\langle f \rangle(\Sigma) = \mathcal{E}_{p}(f)$.

Now we define a measure $\Gamma_{p}\langle f \rangle$ on $K$ as $\Gamma_{p}\langle f \rangle \coloneqq \chi_{\ast}\bigl(\mathfrak{m}_{p}\langle f \rangle\bigr)$, where $\chi$ is the coding map of $(K, S, \{ F_{i} \}_{i \in S})$ (recall Definition \ref{defn.ss-structure}).
Note that $\Gamma_{p}\langle f \rangle$ is a finite Borel-regular measure on $K$ (see \cite[Theorem 7.1.3]{Dud} for example).
We shall say that $\Gamma_{p}\langle f \rangle$ is the \emph{$\mathcal{E}_{p}$-energy measure of $f$}.
To summarize, the self-similarity of $\mathcal{E}_{p}$ (on a self-similar structure $(K, S, \{ F_{i} \}_{i \in S})$) is enough to define $p$-energy measure $\Gamma_{p}\langle \,\cdot\, \rangle$.

\subsection{Basic properties of self-similar energy measures}
We record some fundamental properties of energy measures $\Gamma_{p}\langle \,\cdot\, \rangle$.
\begin{prop}\label{prop.em-irreducible}
    Let $f \in \mathcal{D}$.
    Then $\Gamma_{p}\langle f \rangle \equiv 0$ if and only if $f$ is constant.
\end{prop}
\begin{proof}
    It is clear from $\Gamma_{p}\langle f \rangle(K) = \mathcal{E}_{p}(f)$, $\mathcal{E}_{p}(f) \asymp \abs{f}_{\mathcal{D}}^{p}$ and $\abs{f}_{\mathcal{D}} = 0 \Leftrightarrow f \in \mathbb{R}\indicator{K}$.
\end{proof}

It is natural to consider that $\Gamma_{p}\langle \,\cdot\, \rangle(A)^{1/p}$ also behaves like the $L^{p}$-norm.
The following proposition corresponds to the triangle inequality of ``$\Gamma_{p}\langle\,\cdot\,\rangle(dx)^{1/p}$''.
\begin{prop}\label{prop.em-norm}
    For any $f_{1}, f_{2} \in \mathcal{D}$ and $g \in \mathscr{B}_{+}(K)$,
    \begin{equation}\label{ineq.em-norm}
        \left(\int_{K}g\,d\Gamma_{p}\langle f_{1} + f_{2} \rangle\right)^{1/p}
        \le \left(\int_{K}g\,d\Gamma_{p}\langle f_{1} \rangle\right)^{1/p} + \left(\int_{K}g\,d\Gamma_{p}\langle f_{2} \rangle\right)^{1/p}.
    \end{equation}
    In particular, for all $A \in \mathcal{B}(K)$,
    \begin{equation}\label{ineq.pene-semi}
        \Gamma_{p}\langle f_{1} + f_{2} \rangle(A)^{1/p} \le \Gamma_{p}\langle f_{1} \rangle(A)^{1/p} + \Gamma_{p}\langle f_{2} \rangle(A)^{1/p}.
    \end{equation}
\end{prop}
\begin{proof}
	First, we prove \eqref{ineq.pene-semi} when $A$ is a closed set of $K$.
	Let $f_{1},f_{2} \in \mathcal{D}$ and define
	\[
	C_{n} \coloneqq \{ w \in W_{n} \mid \Sigma_{w} \cap \chi^{-1}(A) \neq \emptyset \}, \quad n \in \mathbb{N}.
	\]
    Then, as seen in the proof of \cite[Lemma 4.1]{Hin05}, one can show that $\bigl\{ \Sigma_{C_{n}} \bigr\}_{n \ge 1}$ is a decreasing sequence and $\bigcap_{n \in \mathbb{N}}\Sigma_{C_{n}} = \chi^{-1}(A)$, where $\Sigma_{C_{n}} \coloneqq \{ \omega \in \Sigma(S) \mid [\omega]_{n} \in C_{n} \}$.
    Indeed, for any $\alpha \in (0, 1)$, we easily see that
    \[
    \Sigma_{C_{n}} = \Bigl\{ \omega \in \Sigma \Bigm| \dist_{\delta_{\alpha}}\bigl(\omega, \chi^{-1}(A)\bigr) \le \alpha^{n - 1} \Bigr\},
    \]
    where $\delta_{\alpha}$ is the metric defined in \eqref{deltaalpha}.
    Hence $\bigcap_{n \in \mathbb{N}}\Sigma_{C_{n}} = \bigl\{ \omega \in \Sigma \bigm| \dist_{\delta_{\alpha}}\bigl(\omega, \chi^{-1}(A)\bigr) = 0 \bigr\} = \chi^{-1}(A)$.
    Using the triangle inequalities of $\mathcal{E}_{p}(\,\cdot\,)^{1/p}$ and of the $\ell^p$-norm on $C_{n}$, we see that
    \begin{align*}
    	\left(\sum_{w \in C_{n}}\rho_{w}\mathcal{E}_{p}\bigl((f_{1} + f_{2}) \circ F_{w}\bigr)\right)^{1/p}
    	&\le \left(\sum_{w \in C_{n}}\rho_{w}\Bigl(\mathcal{E}_{p}(f_{1} \circ F_w)^{1/p} + \mathcal{E}_{p}(f_{2} \circ F_{w})^{1/p}\Bigr)^{p}\right)^{1/p} \\
    	&\le \left(\sum_{w \in C_{n}}\rho_{w}\mathcal{E}_{p}(f_{1}\circ F_w)\right)^{1/p} + \left(\sum_{w \in C_{n}}\rho_{w}\mathcal{E}_{p}(f_{2} \circ F_{w})\right)^{1/p},
    \end{align*}
    and hence
    \[
    \mathfrak{m}_{p}\langle f_{1} + f_{2} \rangle\bigl(\Sigma_{C_{n}}\bigr)^{1/p} \le \mathfrak{m}_{p}\langle f_{1} \rangle\bigl(\Sigma_{C_{n}}\bigr)^{1/p} + \mathfrak{m}_{p}\langle f_{2} \rangle\bigl(\Sigma_{C_{n}}\bigr)^{1/p}.
    \]
    Letting $n \to \infty$, we obtain \eqref{ineq.pene-semi} for any closed set $A$.

    Next, let $A \in \mathcal{B}(K)$.
    Since $\Gamma_{p}\langle f_{1} + f_{2} \rangle$ is Borel-regular, there exists a sequence $\{ F_{n} \}_{n \ge 1}$ of closed subsets of $K$ such that $F_{n} \subseteq A$ and $\lim_{n \to \infty}\Gamma_{p}\langle f_{1} + f_{2} \rangle(F_{n}) = \Gamma_{p}\langle f_{1} + f_{2} \rangle(A)$.
    Then, for any $n \in \mathbb{N}$,
    \[
    \Gamma_{p}\langle f_{1} + f_{2} \rangle(F_{n})^{1/p}
    \le \Gamma_{p}\langle f_{1} \rangle(F_{n})^{1/p} + \Gamma_{p}\langle f_{2} \rangle(F_{n})^{1/p}
    \le \Gamma_{p}\langle f_{1} \rangle(A)^{1/p} + \Gamma_{p}\langle f_{2} \rangle(A)^{1/p}.
    \]
    We get \eqref{ineq.pene-semi} by letting $n \to \infty$.

    Finally, we prove \eqref{ineq.em-norm}.
    Let $N \in \mathbb{N}$.
    Let $a_{i} \ge 0$ and $A_{i} \in \mathcal{B}(K)$ such that $h \coloneqq \sum_{i = 1}^{N}a_{i}\indicator{A_{i}} \le g$.
    Then, \eqref{ineq.pene-semi} together with the triangle inequality of the $\ell^p$-norm on $\{ 1, \dots, N \}$ implies
    \begin{align*}
    	\left(\int_{K}h\,d\Gamma_{p}\langle f_{1} + f_{2} \rangle\right)^{1/p}
    	&\le \left(\int_{K}h\,d\Gamma_{p}\langle f_{1} \rangle\right)^{1/p} + \left(\int_{K}h\,d\Gamma_{p}\langle f_{2} \rangle\right)^{1/p} \\
    	&\le \left(\int_{K}g\,d\Gamma_{p}\langle f_{1} \rangle\right)^{1/p} + \left(\int_{K}g\,d\Gamma_{p}\langle f_{2} \rangle\right)^{1/p}.
    \end{align*}
    Taking the supremum over $h$, we obtain \eqref{ineq.em-norm}.
\end{proof}

The following proposition gives the self-similarity of our energy measures.
\begin{prop}\label{prop.peneSS}
	For any $n \in \mathbb{N}$ and $f \in \mathcal{D}$,
	\begin{equation}\label{eq.peness}
    	\Gamma_{p}\langle f \rangle = \sum_{w \in W_{n}}\rho_{w}(F_{w})_{\ast}\bigl(\Gamma_{p}\langle f \circ F_w \rangle\bigr),
	\end{equation}
	that is, $\Gamma_{p}\langle f \rangle(A) = \sum_{w \in W_{n}}\rho_{w}\Gamma_{p}\langle f \circ F_w \rangle\bigl(F_{w}^{-1}(A)\bigr)$ for any $A \in \mathcal{B}(K)$.
\end{prop}
\begin{proof}
	The proof is exactly the same as in \cite[Theorem 7.5]{Shi24} although the generalized Sierpi\'{n}ski carpets are considered in \cite{Shi24}.
\end{proof}

Energy measures inherit `nice' properties of the self-similar $p$-energy $\mathcal{E}_{p}$.
Here, we focus only on the Lipschitz contractivity.
\begin{prop}\label{prop.em-Markov}
    Let $f \in \mathcal{D}$ and let $\varphi \colon \mathbb{R} \to \mathbb{R}$ be a $1$-Lipschitz function.
    Then, for any $g \in \mathscr{B}_{+}(K)$,
    \[
    \int_{K}g\,d\Gamma_{p}\langle \varphi \circ f \rangle \le \int_{K}g\,d\Gamma_{p}\langle f \rangle.
    \]
    In particular, for any $A \in \mathcal{B}(K)$,
    \[
    \Gamma_{p}\langle \varphi \circ f \rangle(A) \le \Gamma_{p}\langle f \rangle(A).
    \]
\end{prop}
\begin{proof}
	Similar arguments in the proof of Proposition \ref{prop.em-norm} tells us that the following is enough: for any $n \in \mathbb{N}$ and $A \subseteq W_{n}$,
	\[
	\sum_{w \in A}\rho_{w}\mathcal{E}_{p}\bigl((\varphi \circ f) \circ F_{w}\bigr) \le \sum_{w \in A}\rho_{w}\mathcal{E}_{p}\bigl(f \circ F_w\bigr).
	\]
	This is immediate from Assumption \ref{assum.ss}\ref{it:assum.ss.energy}.
\end{proof}

\subsection{Chain rule of energy measures and strong locality}
We next show a \emph{chain rule} of energy measures.
The following `weak locality' of energy measures corresponds to the condition (H5) in \cite{BV05}, which is a consequence of the self-similarity of energies.
\begin{lem}\label{lem.em-wlocal} 
	Let $U$ be an open subset of $K$.
	If $f, g \in \mathcal{D}$ satisfy $f = g$ $\measure$-a.e. on $U$, then $\Gamma_{p}\langle f - g \rangle(U) = 0$. In particular, $\Gamma_{p}\langle f \rangle|_{\mathcal{B}(U)} = \Gamma_{p}\langle g \rangle|_{\mathcal{B}(U)}$.  
\end{lem} 
\begin{proof}
	By \eqref{ineq.pene-semi} and the inner regularity of $\Gamma_{p}\langle f - g \rangle$, it suffices to show $\Gamma_{p}\langle f - g \rangle(A) = 0$ for any closed set $A$ such that $A \subseteq U$.
	Pick $\delta \in (0, \dist_{\metric}(A, K \setminus U))$ and $N \in \mathbb{N}$ so that $\max_{w \in W_{n}}\diam(K_{w},d) < \delta$ for any $n \ge N$.
	For $n \in \mathbb{N}$, define $C_{n} \coloneqq \{ w \in V_{n} \mid \Sigma_{w} \cap \chi^{-1}(A) \neq \emptyset \}$.
	Since $(f - g) \circ F_w = 0 \, (\text{$\measure$-a.e. on $K$})$ for any $w \in C_{n}$ with $n \ge N$, we have
	\begin{align*}
		\mathfrak{m}_{p}\langle f - g \rangle(\Sigma_{C_{n}})
		= \sum_{w \in C_{n}}\rho_{w}\mathcal{E}_{p}((f - g) \circ F_w)
		= 0. 
	\end{align*}
	Letting $n \to \infty$ proves $\Gamma_{p}\langle f -g \rangle(A) = 0$, which completes the proof.
\end{proof}

The following theorem states the chain rule of our energy measures, which is the main result in this section. 
(See also \cite[Section 5]{KS24+} for an improved version of the chain rule of self-similar $p$-energy measures.) 
\begin{thm}[Chain rule]\label{thm.em-chain}
	For any $\Psi \in C^{1}(\mathbb{R})$ and $f \in \mathcal{D} \cap \mathcal{C}(K)$,
	\begin{equation}\label{eq.pene-chain}
    	\Gamma_{p}\langle \Psi \circ f \rangle(dx) = \abs{\Psi'(f(x))}^{p}\Gamma_{p}\langle f \rangle(dx),
	\end{equation}
	that is,
	\[
	\Gamma_{p}\langle \Psi \circ f \rangle(A) = \int_{A}\abs{\Psi'(f(x))}^{p}\,\Gamma_{p}\langle f \rangle(dx) \quad \text{for any $A \in \mathcal{B}(K)$.}
	\]
\end{thm} 
\begin{proof}
	The idea is very similar to \cite[Proposition 4.1]{BV05}. We present a complete proof because the framework of \cite{BV05} is slightly different from our setting.
	Let $f \in \mathcal{D} \cap \mathcal{C}(K)$, $\Psi \in C^{1}(\mathbb{R})$ and $\varepsilon > 0$.
	Then there exists $\delta > 0$ such that
	\[
	\abs{\Psi'(f(x)) - \Psi'(f(y))} < \varepsilon \quad \text{for any $x, y \in K$ with $\metric(x, y) < \delta$.}
	\]
	Let $\{ x_{j} \}_{j \in J}$ be a family such that $x_{j} \in K \, (j \in J)$, $\#J < \infty$ and $K = \bigcup_{j \in J}B_{\metric}(x_{j}, \delta)$.
	For $j \in J$, we define $\Psi_{j} \colon \mathbb{R} \to \mathbb{R}$ by
	\[
	\Psi_{j}(t) = \frac{\Psi(f(x_{j}))}{\abs{\Psi'(f(x_{j}))} + \varepsilon} + \int_{f(x_{j})}^{t}\biggl[\left(\frac{\Psi'(s)}{\abs{\Psi'(f(x_{j}))} + \varepsilon} \wedge 1\right) \vee (-1)\biggr]\,ds.
	\]
	Then, it is clear than $\Psi_{j} \in C^{1}(\mathbb{R})$ and $\abs{\Psi_{j}'(t)} \le 1$ for all $t \in \mathbb{R}$.
	We note that if $s \in \mathbb{R}$ satisfies $\abs{\Psi'(s) - \Psi'(f(x_{j}))} \le \varepsilon$, then
	\[
	\left(\frac{\Psi'(s)}{\abs{\Psi'(f(x_{j}))} + \varepsilon} \wedge 1\right) \vee (-1) = \frac{\Psi'(s)}{\abs{\Psi'(f(x_{j}))} + \varepsilon}.
	\]
	In particular,
	\[
	\Psi_{j}(f(x)) = \frac{\Psi(f(x))}{\abs{\Psi'(f(x_{j}))} + \varepsilon} \quad \text{and} \quad \Psi_{j}'(f(x)) = \frac{\Psi'(f(x))}{\abs{\Psi'(f(x_{j}))} + \varepsilon} \quad \text{for any $x \in B_{\metric}(x_{j}, \delta)$.}
	\]
	Set $a_{j} = \abs{\Psi'(f(x_{j}))} + \varepsilon$ for simplicity.
	By Lemma \ref{lem.em-wlocal}, Proposition \ref{prop.em-Markov} and the outer regularity of energy measures, for any $E \in \mathcal{B}(K)$ with $E \subseteq B_{\metric}(x_{j}, \delta)$, we see that
	\begin{equation*}
		\Gamma_{p}\langle \Psi \circ f \rangle(E)
		= \Gamma_{p}\bigl\langle a_{j}(\Psi_{j} \circ f) \bigr\rangle(E)
		= a_{j}^{p}\Gamma_{p}\langle \Psi_{j} \circ f \rangle(E)
		\le \bigl(\abs{\Psi'(f(x_{j}))} + \varepsilon\bigr)^{p}\Gamma_{p}\langle f \rangle(E).
	\end{equation*}
	Therefore, for $E \in \mathcal{B}(K)$ with $E \subseteq B_{\metric}(x_{j}, \delta)$,
	\begin{align}\label{em-chain.cover1}
		\Gamma_{p}\langle \Psi \circ f \rangle(E) \nonumber
		&\le \int_{E}\abs{\Psi'(f(x))}^{p}\,\Gamma_{p}\langle f \rangle(dx) + \int_{E}\Bigl[\bigl(\abs{\Psi'(f(x_{j}))} + \varepsilon\bigr)^{p} - \abs{\Psi'(f(x))}^{p}\Bigr]\,\Gamma_{p}\langle f \rangle(dx)  \nonumber \\
		&\le \int_{E}\abs{\Psi'(f(x))}^{p}\,\Gamma_{p}\langle f \rangle(dx) + \int_{E}\abs{\int_{\abs{\Psi'(f(x))}}^{\abs{\Psi'(f(x_{j}))} + \varepsilon}ps^{p - 1}\,ds}\,\Gamma_{p}\langle f \rangle(dx) \nonumber \\
		&\le \int_{E}\abs{\Psi'(f(x))}^{p}\,\Gamma_{p}\langle f \rangle(dx) + \varepsilon \cdot C_{p, \Psi, f}\Gamma_{p}\langle f \rangle(E),
	\end{align}
	where $C_{p, \Psi, f}$ is a constant depending only on $p$ and $\sup_{t \in f(K)}\abs{\Psi'(t)}$.

	Now let $A \in \mathcal{B}(K)$ and let $J = \{ 1, \dots, N \}$.
	We inductively define $A_{j}$ by $A_{1} \coloneqq A \cap B_{\metric}(x_{1}, \delta)$ and $A_{j + 1} \coloneqq \bigl(A \cap B_{\metric}(x_{j + 1}, \delta)\bigr) \setminus A_{j}$ so that $A = \bigsqcup_{j = 1}^{N}A_{j}$.
	By summing \eqref{em-chain.cover1} with $E = A_{j}$ over $j$ and letting $\varepsilon \downarrow 0$, we obtain
	\begin{equation}\label{em-chain1}
		\Gamma_{p}\langle \Psi \circ f \rangle(A) \le \int_{A}\abs{\Psi'(f(x))}^{p}\,\Gamma_{p}\langle f \rangle(dx) \quad \text{for any $A \in \mathcal{B}(K)$.}
	\end{equation}

	Next, we prove the converse inequality of \eqref{em-chain1}.
	For $n \in \mathbb{N}$, we define a closed set $F_{n}$ of $K$ by $F_{n} \coloneqq \bigl\{ x \in K \bigm| \abs{\Psi'(f(x))} \ge n^{-1} \bigr\}$.
	Note that $\bigcup_{n \ge 1}F_{n} = \{ \Psi' \circ f \neq 0 \}$.
	For each $n \in \mathbb{N}$ there exists $\delta_{n} > 0$ such that
	\[
	\abs{\Psi'(f(x)) - \Psi'(f(y))} < \frac{1}{2n} \quad \text{for any $x, y \in K$ with $\metric(x, y) < \delta_{n}$.}
	\]
	Pick $l_{n} \in \mathbb{N}$ so that $\max_{w \in W_{l_{n}}}\diam(K_{w}, d) < \delta_{n}$.
	Let
	\[
	F_{n}^{+} \coloneqq \bigl\{ x \in K \bigm| \Psi'(f(x)) \ge n^{-1} \bigr\} = (\Psi' \circ f)^{-1}\bigl(\bigl[n^{-1}, \infty\bigr)\bigr),
	\]
	\[
	F_{n}^{-} \coloneqq \bigl\{ x \in K \bigm| \Psi'(f(x)) \le -n^{-1} \bigr\} = (\Psi' \circ f)^{-1}\bigl(\bigl(-\infty, -n^{-1}\bigr]\bigr),
	\]
	and $W_{l_{n}}[F_{n}^{\pm}] \coloneqq \{ w \in W_{l_{n}} \mid K_{w} \cap F_{n}^{\pm} \neq \emptyset \}$.
	Then, we easily see that
	\[
	F_{n} = F_{n}^{+} \sqcup F_{n}^{-} \subseteq \left(\bigcup_{w \in W_{l_{n}}[F_{n}^{+}]}K_{w}\right) \cup \left(\bigcup_{w \in W_{l_{n}}[F_{n}^{-}]}K_{w}\right),
	\]
	and $\Psi'(f(y)) \ge (2n)^{-1}$ (resp. $\Psi'(f(y)) \le -(2n)^{-1}$) for any $y \in \bigcup_{w \in W_{l_{n}}[F_{n}^{+}]}K_{w}$ (resp. $y \in \bigcup_{w \in W_{l_{n}}[F_{n}^{-}]}K_{w}$).
	Since $f(K_{w})$ is a connected subset of $\mathbb{R}$, $f$ and $\Psi' \circ f$ are uniformly continuous on $K$, we can pick $\delta_{n}' > 0$ and a collection of open intervals $\{ I_{w} \}_{w \in W_{l_{n}}[F_{n}^{\pm}]}$ so that
	\[
	\text{$f\big((K_{w})_{\delta_{n}'}\bigr) \subseteq I_{w}$ and $\inf_{t \in I_{w}}\abs{\Psi'(t)} > 0$ for any $w \in W_{l_{n}}[F_{n}^{+}] \sqcup W_{l_{n}}[F_{n}^{-}]$. }
	\]
	Since $\Psi \in C^{1}(\mathbb{R})$, $\Psi'$ is strictly increasing or strictly decreasing on each $I_{w}$.
	Applying the inverse function theorem (e.g. \cite[Theorem 2.7]{Jost}), we get the inverse functions $\Upsilon_{w} \colon \Psi(I_{w}) \to \mathbb{R}$ of $\Psi$.
	For any $w \in W_{l_{n}}[F_{n}^{+}] \sqcup W_{l_{n}}[F_{n}^{-}]$ and any $E \in \mathcal{B}(K)$ with $E \subseteq K_{w}$, by Lemma \ref{lem.em-wlocal} and the inequality \eqref{em-chain1} as measures,
	\begin{align*}
		\int_{E}\abs{\Psi'(f(x))}^{p}\,\Gamma_{p}\langle f \rangle(dx)
		&= \int_{E}\abs{\Psi'(f(x))}^{p}\,\Gamma_{p}\langle \Upsilon_{w} \circ \Psi \circ f \rangle(dx) \\
		&\le \int_{E}\abs{\Upsilon_{w}'(\Psi(f(x)))}^{p}\abs{\Psi'(f(x))}^{p}\,\Gamma_{p}\langle \Psi \circ f \rangle(dx) \\
		&= \int_{E}\,d\Gamma_{p}\langle \Psi \circ f \rangle = \Gamma_{p}\langle \Psi \circ f \rangle(E).
	\end{align*}
	A similar covering argument as in the previous paragraph yields, for any $A \in \mathcal{B}(K)$,
	\[
	\int_{A \cap F_{n}}\abs{\Psi'(f(x))}^{p}\,\Gamma_{p}\langle f \rangle(dx) \le \Gamma_{p}\langle \Psi \circ f \rangle(A \cap F_{n}).
	\]
	By letting $n \to \infty$, we get
	\begin{align*}
		\int_{A}\abs{\Psi'(f(x))}^{p}\,\Gamma_{p}\langle f \rangle(dx)
		&= \int_{A \cap \{ \Psi' \circ f \neq 0 \}}\abs{\Psi'(f(x))}^{p}\,\Gamma_{p}\langle f \rangle(dx) \\
		&\le \Gamma_{p}\langle \Psi \circ f \rangle(A \cap \{ \Psi' \circ f \neq 0 \})
		\le \Gamma_{p}\langle \Psi \circ f \rangle(A),
	\end{align*}
	which together with \eqref{em-chain1} implies the assertion.
\end{proof}

As an immediate consequence of Theorem \ref{thm.em-chain}, we can prove the following theorem called \emph{energy image density property}.
\begin{cor}\label{cor.EID}
    For any $f \in \mathcal{D} \cap \mathcal{C}(K)$, it holds that the image measure of $\Gamma_{p}\langle f \rangle$ by $f$ is absolutely continuous with respect to the one-dimensional Lebesgue measure $\mathscr{L}^{1}$ on $\mathbb{R}$.
    In particular, $\Gamma_{p}\langle f \rangle(\{ x \}) = 0$ for any $x \in K$.
\end{cor}
\begin{proof}
    The proof is essentially the same as in \cite[Proposition 7.6]{Shi24} although the generalized Sierpi\'{n}ski carpets are considered in \cite{Shi24}.
    See also \cite[Theorem 4.3.8]{CF} for the case $p = 2$.
    (Let us remark that the reflexivity of $\mathcal{D}$ is needed to follow the argument of \cite[Proposition 7.6]{Shi24}.)
\end{proof}

Finally, we can show the `strong locality in a measure sense'.
\begin{cor}\label{cor.em-slocal}
  Let $f, g \in \mathcal{D} \cap \mathcal{C}(K)$.
  If $(f - g)|_{A}$ is constant for some Borel set $A \in \mathcal{B}(K)$, then $\Gamma_{p}\langle f \rangle(A) = \Gamma_{p}\langle g \rangle(A)$.
\end{cor}
\begin{proof}
  Let $f \in \mathcal{D} \cap \mathcal{C}(K)$ and let $A \in \mathcal{B}(K)$.
  Suppose that $f|_{A} \equiv c$ for some $c \in \mathbb{R}$.
  Then, by Corollary \ref{cor.EID}, we have $\Gamma_{p}\langle f \rangle\bigl(f^{-1}(\{ c \})\bigr) = 0$, which implies that $\Gamma_{p}\langle f \rangle(A) = 0$.
  Combining this result and Proposition \ref{prop.em-norm}, we finish the proof.
\end{proof}
\begin{rmk}
	Theorem \ref{thm.em-chain}, Corollaries \ref{cor.EID} and \ref{cor.em-slocal} are restricted to the functions in $\mathcal{D} \cap \mathcal{C}(K)$.
	One might expect that these statements can be extended to $\bigl(\closure{\mathcal{D} \cap \contfunc(K)}^{\norm{\,\cdot\,}_{\mathcal{D}}}\bigr) \cap L^{\infty}(K, \measure)$, but there is a possibility of $\measure \perp \Gamma_{p}\langle f \rangle$.
	Indeed, for canonical Dirichlet forms on many fractals, such a singularity is expected \cite{Hin05, KM20}.
	We need to consider \emph{quasi-continuous modification} of function in $\mathcal{D}$ with respect to our $p$-energy $\mathcal{E}_{p}$ and establish some fundamental results on \emph{nonlinear potential theory} associated with $\mathcal{E}_{p}$.
	We will not obtain such results in this paper because it is not needed for our purpose.
\end{rmk}

\subsection{Minimal energy-dominant measures}\label{sec.MED}
We conclude this section by giving a natural extension of the notion called \emph{minimal energy-dominant measures} (cf. \cite{Hin10}).
Let $\mathcal{E}_{p}$ satisfy Assumption \ref{assum.ss} and let $\Gamma_{p}\langle \,\cdot\, \rangle$ denote the associated energy measures.

\begin{defn}\label{defn.MED}
	A $\sigma$-finite Borel measure $\nu$ is called a \emph{minimal energy-dominant measure of $(\mathcal{E}_{p}, \mathcal{D})$} if the following two conditions hold.
	\begin{enumerate}[$\bullet$]
		\item (Domination) For every $f \in \mathcal{D}$, we have $\Gamma_{p}\langle f \rangle \ll \nu$.
		\item (Minimality) For another $\sigma$-finite Borel measure $\nu'$ satisfying the above `domination' property, we have $\nu \ll \nu'$.
	\end{enumerate}
\end{defn}
In Dirichlet form theory, the existence of such a measure is shown in \cite[Lemma 2.2]{Nak85}.
We   verify the existence of minimal energy-dominant measures of $(\mathcal{E}_{p}, \mathcal{D})$ in Lemma \ref{lem.exist-MED} later.
To prove it, we need the following lemma (cf. \cite[Lemma 2.2]{Hin10}).
\begin{lem}\label{lem.MED-conti}
	Let $\nu$ be a $\sigma$-finite Borel measure on $K$ and let $f, f_{n} \in \mathcal{D} \, (n \in \mathbb{N})$ such that $\mathcal{E}_{p}(f - f_{n}) \to 0$ as $n \to \infty$.
	Suppose that $\Gamma_{p}\langle f_{n} \rangle \ll \nu$ for any $n \in \mathbb{N}$.
	Then $\Gamma_{p}\langle f \rangle \ll \nu$.
\end{lem}
\begin{proof}
	Let $A \in \mathcal{B}(K)$ such that $\nu(A) = 0$.
	Then we have $\Gamma_{p}\langle f_{n} \rangle(A) = 0$ for any $n \in \mathbb{N}$.
	We also note that $\Gamma_{p}\langle f - f_{n} \rangle(A) \le \mathcal{E}_{p}(f - f_{n}) \to 0$.
	By Proposition \ref{prop.em-norm},
	\[
	\Gamma_{p}\langle f \rangle(A)^{1/p}
	\le \Gamma_{p}\langle f_{n} \rangle(A)^{1/p} + \Gamma_{p}\langle f - f_{n} \rangle(A)^{1/p}
	= \Gamma_{p}\langle f - f_{n} \rangle(A)^{1/p}
	\to 0,
	\]
	which proves our assertion.
\end{proof}

We now prove the existence of minimal energy-dominant measure (cf. \cite[Lemma 2.3]{Hin10}).
\begin{lem}\label{lem.exist-MED}
	Suppose that $\{ f_{n} \in \mathcal{D} \}_{n \in \mathbb{N}}$ is a dense subset of $\mathcal{D}$.
	Let $\{ a_{n} \}_{n \in \mathbb{N}}$ be a sequence of positive numbers such that $\sum_{n = 1}^{\infty}a_{n}\mathcal{E}_{p}(f_{n})$ converges.
	Then $\nu \coloneqq \sum_{n = 1}^{\infty}a_{n}\Gamma_{p}\langle f_{n} \rangle$ defines a minimal energy-dominant measure of $(\mathcal{E}_{p}, \mathcal{D})$.
\end{lem}
\begin{proof}
	By the definition of $\nu$, we note that $\Gamma_{p}\langle f_{n} \rangle(A) = 0$  for any $n \in \mathbb{N}$ and $A \in \mathcal{B}(K)$ with $\nu(A) = 0$.
	Hence the density of $\{ f_{n} \}_{n \in \mathbb{N}}$ and Lemma \ref{lem.MED-conti} imply $\Gamma_{p}\langle f \rangle \ll \nu$ for any $f \in \mathcal{D}$.
	So it is enough to show the minimality of $\nu$.
	Let $\nu'$ be another $\sigma$-finite Borel measure on $K$ such that $\Gamma_{p}\langle f \rangle \ll \nu'$ for any $f \in \mathcal{D}$.
	If $A \in \mathcal{B}(K)$ satisfies $\nu'(A) = 0$, then we have $\Gamma_{p}\langle f_{n} \rangle(A) = 0$ for any $n \in \mathbb{N}$.
	Now it is immediate that $\nu(A) = 0$, which means $\nu \ll \nu'$ and we finish the proof.
\end{proof}

The next proposition corresponds to \cite[Lemma 2.4]{Hin10}.
This states that any two minimal energy-dominant measures are mutually absolutely continuous.
\begin{prop}\label{prop.MED-abs}
	Suppose that $\mathcal{D}$ is separable with respect to $\norm{\,\cdot\,}_{\mathcal{D}}$.
	Let $\nu$ be a minimal energy-dominant measure of $(\mathcal{E}_{p}, \mathcal{D})$ and let $A \in \mathcal{B}(K)$.
	Then $\nu(A) = 0$ if and only if $\Gamma_{p}\langle f \rangle(A) = 0$ for any $f \in \mathcal{D}$.
\end{prop}
\begin{proof}
	It is clear that, for $A \in \mathcal{B}(K)$, $\nu(A) = 0$ implies $\Gamma_{p}\langle f \rangle(A) = 0$ by the `domination' property of $\nu$.

	For the converse, suppose that $A \in \mathcal{B}(K)$ and $\Gamma_{p}\langle f \rangle(A) = 0$ for any $f \in \mathcal{D}$.
	Let $\{ f_{n} \}_{n \in \mathbb{N}}$ be a dense subset of $\mathcal{D}$ and let $\{ a_{n} \}_{n \in \mathbb{N}}$ be a sequence of positive numbers such that $\sum_{n = 1}^{\infty}a_{n}\mathcal{E}_{p}(f_{n})$ converges.
	(For example, $a_{n} = 2^{-n}\bigl(\mathcal{E}_{p}(f_{n})^{-1} \wedge 1\bigr)$.)
	Then, by Lemma \ref{lem.exist-MED}, the new measure $\nu' \coloneqq \sum_{n = 1}^{\infty}a_{n}\Gamma_{p}\langle f_{n} \rangle$ is also a minimal energy-dominant measure of $(\mathcal{E}_{p}, \mathcal{D})$.
	Hence the minimality conditions for $\nu$ and $\nu'$ tell us that $\nu$ and $\nu'$ are mutually absolutely continuous.
	The assumption $\Gamma_{p}\langle f \rangle(A) = 0$ for any $f \in \mathcal{D}$ implies $\nu'(A) = 0$, and thus $\nu(A) = 0$.
	This completes the proof.
\end{proof}

\subsection{Estimates of energy measures}\label{sec.em-local}
In this subsection, we investigate `local behavior of $p$-energy', which will be described in terms of $\mathcal{E}_{p}$-energy measures.
Throughout this subsection, we suppose Assumption \ref{a:reg-ss} and the pre-self-similarity condition (\hyperref[PSS]{PSS}) in Theorem \ref{thm.fix} (with $\mathcal{D} = \mathcal{F}_{p}$ and $\mathsf{E}(\,\cdot\,) = \abs{\,\cdot\,}_{\mathcal{F}_{p}}^{p}$).
Hence, by Corollary \ref{cor.ss-energy}, there exists a self-similar $p$-energy $(\mathcal{E}_{p}, \mathcal{F}_{p})$ satisfying Assumption \ref{assum.ss}.
Let $\Gamma_{p}\langle f \rangle$, $f \in \mathcal{F}_{p}$, denote the energy measure with respect to $(\mathcal{E}_{p},\mathcal{F}_{p})$.

The following lemma gives behaviors of `$p$-energy on each cells'.
\begin{lem}\label{lem.em-cell}
	For any $f \in \mathcal{F}_{p}$, $w \in W_{\ast}$ and $n \in \mathbb{N}$,
	\[
	\rho_{w}\mathcal{E}_{p}(f \circ F_w) \le \Gamma_{p}\langle f \rangle(K_{w}) \le \sum_{v \in W_{n}; K_{v} \cap K_{w} \neq \emptyset}\rho_{v}\mathcal{E}_{p}(f \circ F_{v}).
	\]
\end{lem}
\begin{proof}
	The lower bound is immediate from $\Sigma_{w} \subseteq \chi^{-1}(K_{w})$.
	The upper bound follows from $\chi^{-1}(K_{w}) \subseteq \bigcup_{v \in W_{n}; K_{v} \cap K_{w} \neq \emptyset}\Sigma_{v}$.
\end{proof}

Next we present a lower estimate on $p$-energy measures under the assumption that the following $(p,p)$-Poincar\'e inequality holds. 
\begin{defn}
	We say that \emph{$(p,p)$-Poincar\'e inequality}, \ref{PI} for short, holds if and only if there exist constants $C_{\textup{P}} > 0$ and $A_{\textup{P}} \ge 1$ such that for any $f \in \mathcal{F}_{p}$, any $x \in K$ and any $r > 0$, 
	\begin{equation}\label{PI}
		\int_{B_{\metric}(x, r)}\abs{f(y) - f_{B_{\metric}(x, r)}}^{p}\,\measure(dy) \le C_{\textup{P}}r^{\beta}\int_{B_{\metric}(x, A_{\textup{P}}r)}\,d\Gamma_{p}\langle f \rangle. \tag*{\textup{PI$_{p}$($\beta$)}}
	\end{equation}
	Note that \ref{PI} along with H\"{o}lder's inequality implies the following $(1,p)$-Poincar\'e inequality: 
	\begin{equation}\label{1p-PI}
		\int_{B_{\metric}(x, r)}\abs{f(y) - f_{B_{\metric}(x, r)}}\,\measure(dy) \le C_{\mathrm{P}}^{1/p}\,r^{\beta/p}\left(\int_{B_{\metric}(x, A_{\mathrm{P}}r)}\,d\Gamma_{p}\langle f \rangle\right)^{1/p}.
	\end{equation}
\end{defn}

By using \ref{PI} instead of \eqref{eq.PI-like} in the proof of Lemma \ref{lem.lower}, we immediately achieve the following `local behavior of $p$-energy in terms of (fractional) Korevaar--Schoen expression'.
\begin{prop}\label{prop.KS-local.1}
	Assume that \ref{PI} holds. 
	Then there exists a constant $C > 0$ (depending only on $C_{\mathrm{P}},A_{\mathrm{P}}$ in \ref{PI}, $p$ and the doubling constant of $(K,\metric)$) such that for any Borel set $U$ of $K$ and $f \in \mathcal{F}_{p}$,
	\begin{equation}\label{KS-local.1}
		\limsup_{r \downarrow 0}\int_{U}\fint_{B_{\metric}(x, r)}\frac{\abs{f(x) - f(y)}^{p}}{r^{\beta}}\,\measure(dy)\measure(dx) \le C\Gamma_{p}\langle f \rangle(\closure{U}).
	\end{equation}
\end{prop}
\begin{proof}
	The same argument using a maximal $r$-net $N_{r} (\subseteq U)$ of $U$ to get \eqref{preKS-local} yields
	\[
	\int_{U}\fint_{B_{\metric}(x, r)}\frac{\abs{f(x) - f(y)}^{p}}{r^{\beta}}\,\measure(dy)\measure(dx)
	\lesssim \sum_{y \in N_{r}}\Gamma_{p}\langle f \rangle\bigl(B_{\metric}(y, 2A_{\textup{P}}r)\bigr).
	\]
	Since $\sum_{y \in N_{r}}\indicator{B_{\metric}(y, 2A_{\textup{P}}r)} \lesssim \indicator{U_{2A_{\textup{P}}r}}$ by the metric doubling property, we get \eqref{KS-local.1}.
\end{proof}

To get an upper estimate on $p$-energy measures, we will assume the following extra conditions to control the `geometry of $\{ wv \mid v \in V_{n} \}$' for all $w \in W_{\ast}$.
Recall that $(K, S, \{ F_{i} \}_{i \in S})$ is a self-similar set such that $F_{i}$ is an $r_{i}$-similitude for each $i \in S$.
We suppose that there exists a sequence $(l_{i})_{i \in S}$ of positive integers such that the following hold: for each $i \in S$,
\begin{equation}\label{rational}
	r_{i} = R_{\ast}^{-l_{i}},
\end{equation}
and
\begin{equation}\label{scaling}
	\rho_{i} = R_{\ast}^{l_{i}(\beta - \hdim)}.
\end{equation}
\begin{rmk}
	The condition \eqref{rational} involves the notion of \emph{rationally ramified self-similar structure} in \cite{Kig09}.
	It might be hard to deal with the graph approximation $\mathbb{G}_{n}$ when we have no such a condition.
	Indeed, the proof below does not work if $\{ wv \mid v \in V_{n} \}$ is not a subset of $V_{m}$ for some $m \ge n$.
	We can avoid such a situation by assuming \eqref{rational}.
	On the other hand, condition \eqref{scaling} seems to be natural once one knows how to determine $\beta$ in a practical situation.
	For details, see Section \ref{sec.PSC}.
\end{rmk} 

The following proposition gives a converse bound to Proposition \ref{prop.KS-local.1}. (A corresponding bound in the case $p = 2$ plays an important role in \cite{Mur23+}.)
\begin{prop}\label{prop.KS-local.2}
	Assume that \eqref{rational} and \eqref{scaling} hold. 
	Then there exists a constant $C > 0$ (depending only on the constant associated with Assumption \ref{a:reg-ss}) such that for any Borel set $U$ of $K$ and $f \in \mathcal{F}_{p}$,
	\begin{equation}\label{KS-local.2}
		\Gamma_{p}\langle f \rangle(U)
		\le C\adjustlimits\lim_{\delta \downarrow 0}\liminf_{r \downarrow 0}\int_{U_{\delta}}\fint_{B_{\metric}(x, r)}\frac{\abs{f(x) - f(y)}^{p}}{r^{\beta}}\,\measure(dy)\measure(dx).
	\end{equation}
\end{prop}
\begin{proof}
	Let $U \in \mathcal{B}(K)$, $\delta > 0$ and $f \in \mathcal{F}_{p}$.
	Then Lemma \ref{lem.upper} tells us that
	\begin{equation}\label{preKS-local.2}
		\limsup_{n \to \infty}\widetilde{\mathcal{E}}_{p, V_{n}(U)}^{(n)}(f) \le C_{0}\liminf_{r \downarrow 0}\int_{U_{\delta}}\fint_{B_{\metric}(x, r)}\frac{\abs{f(x) - f(y)}^{p}}{r^{\beta}}\,\measure(dy)\measure(dx),
	\end{equation}
	where $C_{0} > 0$ is independent of $U, \delta, f$.
	Let $m \in \mathbb{N}$ be large enough so that $\bigcup_{w \in V_{m}(U)} \subseteq U_{\delta}$ ($R_{\ast}^{-m + 1} < \delta$ is enough).
	For $w = w_{1} \dots w_{m} \in V_{m}$, set $l(w) \coloneqq \sum_{i = 1}^{m}l_{w_{i}}$ and $V_{n}^{w} \coloneqq \{ wv \mid v \in V_{n} \} $; note that $V_{n}^{w} \subseteq V_{n + l(w)}$. 
	Then we see that
	\begin{align*}
		\Gamma_{p}\langle f \rangle(U)
		&\le \mathfrak{m}_{p}\langle f \rangle\left(\bigcup_{w \in V_{m}(U)}\Sigma_{w}\right)
		= \sum_{w \in V_{m}(U)}\rho_{w}\mathcal{E}_{p}(f \circ F_w) \\
		&\quad\lesssim \sum_{w \in V_{m}(U)}\rho_{w}\varliminf_{n \to \infty}\widetilde{\mathcal{E}}_{p}^{(n)}(f \circ F_w)
		= \sum_{w \in V_{m}(U)}\varliminf_{n \to \infty}\widetilde{\mathcal{E}}_{p, V_{n}^{w}}^{(n + l(w))}(f). 
	\end{align*} 
	For $w \in V_{m}(U)$, we observe that $V_{n}^{w} \subseteq V_{n + l(w)}(U_{\delta})$.
	Therefore,
	\[
	\Gamma_{p}\langle f \rangle(U)
	\lesssim \liminf_{n \to \infty}\sum_{w \in V_{m}(U)}\widetilde{\mathcal{E}}_{p, V_{n}^{w}}^{(n + l(w))}(f)
	\le \limsup_{n \to \infty}\widetilde{\mathcal{E}}_{p, V_{n}(U_{\delta})}^{(n)}(f).
	\]
	Combining this with \eqref{preKS-local.2} for $U_{\delta}$, we obtain \eqref{KS-local.2}.
\end{proof} 
\begin{rmk}\label{r:embdy}
	Once we get energy measures, \ref{PI}, \eqref{KS-local.2} and good cutoff functions as given in Lemma \ref{lem.unity}, minor modifications of the proof of \cite[Theorem 2.9]{Mur23+} shows the following   result: For any uniform domain $U$ of $K$ in the sense of \cite[Definition 2.3]{Mur23+} and $f \in \mathcal{F}_{p}$, we have $\Gamma_{p}\langle f \rangle(\bord{U}) = 0$.  
\end{rmk}

\section{Self-similar energies on the \texorpdfstring{Sierpi\'{n}ski}{Sierpinski} carpet}\label{sec.PSC}

\subsection{Checking all assumptions}
In the rest of the paper, we focus on the \emph{planar standard Sierpi\'{n}ski carpet} and we will prove the main results.

First, recall the definition of the Sierpi\'{n}ski carpet.
\begin{defn} [Planar Sierpi\'{n}ski carpet]\label{defn.PSC}
	\begin{enumerate}[\rm(1)]
		\item\label{it:PSCbasic} Let $a_{\ast}=3,N_*=8, S= \{1,\ldots,N_*\}$ and define $q_i \in \bR^2$ as
		\begin{align*}
			q_1=(-1,-1)=-q_5,& \q q_2= (0,-1)=-q_6, \\
	 		q_3=(1,-1)=-q_7,& \q q_4=(1,0)=-q_8.
		\end{align*}
		Let $f_i\colon \bR^2 \to \bR^2, i \in S$ denote the similitude $f_i(x)= a_{\ast}^{-1}(x-q_i)+q_i$. Let $K$ be the unique non-empty compact subset such that $K= \bigcup_{i \in S} f_i(K)$ and set $F_i = \restr{f_i}{K}$.
		Let $d$ denote the normalized Euclidean metric on $K$ so that $\diam(K, d) = 1$.
		The self-similar structure $(K, S, \{ F_{i} \}_{i \in S})$ is called the \emph{planar standard Sierpi\'{n}ski carpet} (PSC for short).
		Let $\measure$ be the self-similar measure with weight $(1/N_{\ast}, \dots, 1/N_{\ast})$.
		\item Let
		\[
		\ell_{\textup{L}} = \{ -1 \} \times [-1, 1], \quad \ell_{\textup{T}} = [-1, 1] \times \{ 1 \}, \quad \ell_{\textup{R}} = \{ 1 \} \times [-1, 1], \quad \ell_{\textup{B}} = [-1, 1] \times \{ -1 \},
		\]
		so that $\mathcal{V}_{0} = \partial[-1, 1]^{2} = \ell_{\textup{L}} \cup \ell_{\textup{T}} \cup \ell_{\textup{R}} \cup \ell_{\textup{B}}$.

		\item Let $D_{4}$ be the dihedral group of order $8$ (the symmetry of the square), i.e.
		\begin{align*}
			D_{4} = \{ R_{k}, S_{k} \mid k = 0, 1, 2, 3 \},
		\end{align*}
		where
		\[
		R_{k} =
		\left[
		\begin{array}{cc}
		\cos{\frac{k\pi}{2}} & -\sin{\frac{k\pi}{2}} \\
		\sin{\frac{k\pi}{2}} & \cos{\frac{k\pi}{2}} \\
		\end{array}
		\right]
		\quad \text{and} \quad
		S_{k} =
		\left[
		\begin{array}{cc}
		\cos{\frac{k\pi}{2}} & \sin{\frac{k\pi}{2}} \\
		\sin{\frac{k\pi}{2}} & -\cos{\frac{k\pi}{2}} \\
		\end{array}
		\right].
		\]
		Then it is clear that $\Phi(K) = K$ for all $\Phi \in D_{4}$.
	\end{enumerate}
\end{defn}

Hereafter, we let $(K, S, \{ F_{i} \}_{i \in S})$ be PSC, $d$ be the normalized metric, and $\measure$ be the self-similar measure on $K$ as given in Definition \hyperref[it:PSCbasic]{\ref{defn.PSC}}.
Let us fix a partition $\bigl\{ \widetilde{K}_{w} \bigr\}_{w \in W_{\ast}}$ as constructed in Section \ref{sec.ss}.
Note that $\operatorname{int}_{K}K_{w} \subseteq \widetilde{K}_{w} \subseteq K_{w}$ for all $w \in W_{\ast}$.
To construct a `canonical' $p$-energy on PSC, we need to check Assumption \ref{a:reg-ss}, especially `Analytic conditions', and \eqref{ss.comparable}.
Recall that the approximating graphs $\{ \mathbb{G}_{n} = (V_n, E_n) \}_{n \in \mathbb{N}}$ in \eqref{e:defn.Vn-ss} are given by
\[
V_{n} = W_{n} = S^{n}, \quad E_{n} = \bigl\{ \{ v,w \} \in V_{n} \times V_{n} \bigm| v \neq w, K_{v} \cap K_{w} \neq \emptyset \bigr\},
\]
and that $M_{n} \colon L^{p}(K, \measure) \to \mathbb{R}^{V_{n}}$ in \eqref{e:Mn} is defined as
\[
M_{n}f(w) = \fint_{\widetilde{K}_{w}}f\,d\measure \quad \text{for $f \in L^{p}(K, \measure)$ and $w \in W_{n}$}.
\]
The following theorem is the main result of this subsection whose proof is divided into several steps.
\begin{thm}\label{thm.assum-PSC}
	PSC satisfies Assumption \ref{a:reg-ss} for all $p \in (1, \infty)$, that is,
	\begin{enumerate}[\rm(a)]
		\item\label{it:PSCassum1} $(K, \metric, \measure)$ is $\hdim$-Ahlfors regular, where $\hdim \coloneqq \log{N_{\ast}}/\log{a_{\ast}} = \log{8}/\log{3}$. In addition, the sequence of graphs $\{ \mathbb{G}_{n} = (V_n, E_n) \}_{n \in \mathbb{N}}$ equipped with the projective map $\pi_{n, k} \, (1 \le k < n)$, which is defined as $\pi_{n,k}(w) \coloneqq [w]_{k} \, (w \in V_{n})$, is $a_{\ast}$-scaled and $a_{\ast}$-compatible with $(K, \metric)$.
		\item\label{it:PSCassum2} The sequence $\{ \mathbb{G}_{n} \}_{n \in \mathbb{N}}$ satisfies  \hyperref[cond.UPI]{\textup{U-PI$_{p}(\pwalk)$}} and \hyperref[cond.UCF]{\textup{U-CF$_{p}(\vartheta, \pwalk)$}} for some $\vartheta \in (0, 1]$, where $\pwalk = \log{\bigl(N_{\ast}\rho(p)\bigr)}/\log{a_{\ast}}$ and $\rho(p) \in (0, \infty)$ is given later (see \eqref{p-factor}).
	\end{enumerate}
	Moreover, the following pre-self-similarity condition holds:
	\begin{enumerate}[\rm(a)]\setcounter{enumi}{2}
		\item\label{it:PSCassum3} $f \circ F_i \in \mathcal{F}_{p}$ for all $i \in S$ and $f \in \mathcal{F}_{p}$. Furthermore,
		\begin{equation}\label{ss-dom}
			\mathcal{F}_{p} \cap \contfunc(K) = \{ f \in \contfunc(K) \mid \text{$f \circ F_{i} \in \mathcal{F}_{p}$ for all $i \in S$} \},
		\end{equation}
		and the semi-norm $\abs{f}_{\mathcal{F}_{p}} = \left(\sup_{n \in \mathbb{N}}a_{\ast}^{n(\pwalk - \hdim)}\mathcal{E}_{p}^{\mathbb{G}_{n}}(M_{n}f)\right)^{1/p}$ satisfies the following: there exists $C \ge 1$ such that for all $n \in \mathbb{N}$ and $f \in \mathcal{F}_{p}$,
			\begin{equation}\label{renormalize}
			C^{-1}\abs{f}_{\mathcal{F}_{p}}^{p} \le \rho(p)^{n}\sum_{w \in W_{n}}\abs{f \circ F_w}_{\mathcal{F}_{p}}^{p} \le C\abs{f}_{\mathcal{F}_{p}}^{p}.
			\end{equation}
	\end{enumerate}
\end{thm}

We start by observing the geometry of PSC, namely Theorem \ref{thm.assum-PSC}\ref{it:PSCassum1}.
The next proposition gives a collection of geometric properties of PSC.
\begin{prop}\label{prop.PSC-geom}
	\begin{enumerate}[\rm(i)]
		\item\label{it:PSCgeom.1} For all $n \in \mathbb{Z}_{\ge 0}$ and distinct $v, w \in W_{n}$, we have $\measure(K_{w}) = N_{\ast}^{-n}$ and $\measure(K_{v} \cap K_{w}) = 0$.
		\item\label{it:PSCgeom.2} There exists a constant $C \ge 1$ (depending only on $a_{\ast}$) such that the following hold: for all $n \in \mathbb{Z}_{\ge 0}$ and $w \in W_{n}$, there exists $x \in K_{w}$ satisfying
		\[
		B_{\metric}(x, C^{-1}a_{\ast}^{-n}) \subseteq K_{w} \subseteq B_{\metric}(x, Ca_{\ast}^{-n}).
		\]
		In particular, \eqref{e:round} holds.
		\item\label{it:PSCgeom.3} There exists $C_{\textup{AR}}$ depending only on $a_{\ast}$ and $N_{\ast}$ such that
		\[
		C_{\textup{AR}}^{-1}r^{\hdim} \le \measure(B_{\metric}(x, r)) \le C_{\textup{AR}}r^{\hdim} \quad \text{for all $x \in K$, $r \in (0, 1]$,}
		\]
		i.e., $(K, \metric, \measure)$ is $\hdim$-Ahlfors regular.
		\item\label{it:PSCgeom.4} The sequence of graph $\{ \mathbb{G}_{n} \}_{n \in \mathbb{N}}$ equipped with the projective maps $\{ \pi_{n, k} \mid n, k \in \mathbb{N}, k < n \}$ is $a_{\ast}$-scaled.
		\item\label{it:PSCgeom.5} The sequence of graph $\{ \mathbb{G}_{n} \}_{n \in \mathbb{N}}$ equipped with the projective maps $\{ \pi_{n, k} \mid n, k \in \mathbb{N}, k < n \}$ is $a_{\ast}$-compatible with $(K,\metric)$.
		\item\label{it:PSCgeom.6} For any $\Phi \in D_{4}$, there exists a bijection $\tau_{\Phi} \colon W_{\ast} \to W_{\ast}$ such that $\abs{\tau_{\Phi}(w)} = \abs{w}$ and $\Phi\bigl(K_{w}\bigr) = K_{\tau_{\Phi}(w)}$ for any $w \in W_{\ast}$. Moreover, $U_{\Phi, w} \coloneqq F_{\tau_{\Phi}(w)}^{-1} \circ \Phi \circ F_{w} \in D_{4}$.
	\end{enumerate}
	In particular, Theorem \ref{thm.assum-PSC}\ref{it:PSCassum1} holds.
\end{prop}
\begin{proof}
	The properties \ref{it:PSCgeom.2}, \ref{it:PSCgeom.6} are easy and \ref{it:PSCgeom.3} is a consequence of \ref{it:PSCgeom.1}, \ref{it:PSCgeom.2}.
	So we will prove \ref{it:PSCgeom.1}, \ref{it:PSCgeom.4} and \ref{it:PSCgeom.5}.

	\ref{it:PSCgeom.1} This follows from $\closure{\mathcal{V}_{0}} \neq K$ and Proposition \ref{prop.ss-meas}.

	\ref{it:PSCgeom.4} Recall that $d_{n}$ denotes the graph distance of $\mathbb{G}_{n}$.
	Let $n, m \in \mathbb{N}$ and $w \in W_{m}$.
	Let $c_{n}(w) = w15^{n - 1} \in V_{n + m}$.
	Then it is clear that $B_{d_{n + m}}\bigl(c_{n}(w), a_{\ast}^{n - 1}\bigr) \subseteq \pi_{n + m, m}^{-1}(w)$.
	(The set $\pi_{n + m, m}^{-1}(w)$ is the same as $S^{n}(w)$ in \cite[Definition 3.5.3(1)]{Kig20}.)
	Since we can easily see that $\diam\bigl(\pi_{n + m, m}^{-1}(w), d_{n + m}\bigr) \le 2a_{\ast}^{n}$, we obtain $\pi_{n + m, m}^{-1}(w) \subseteq B_{d_{n + m}}\bigl(c_{n}(w), 3a_{\ast}^{n}\bigr)$.
	Hence we have \eqref{e:sc1} with $A_{1} = 3 \vee a_{\ast}$.
	Also, the bound on the diameter of $\pi_{n + m, m}^{-1}(\,\cdot\,)$ implies \eqref{e:sc2} with $A_{2} = 4$.
	This completes the proof.

	\ref{it:PSCgeom.5} Note that the conditions in Definition \ref{d:compatible}\ref{it:compat.parti},\ref{it:compat.proj} are already verified.
	Let $p_{n}(v) = F_{v}(F_{1}(q_{5})) \in K_{v}$ for $n \in \mathbb{N}$ and $v \in V_{n}$.
	Then the condition in Definition \ref{d:compatible}(iv) is evident.
	So we will prove the H\"{o}lder comparison \eqref{e:holder}.
	Let $v, w \in V_{n}$ with $v \neq w$.
	Pick a path $[z(0), \dots, z(l)]$ in $\mathbb{G}_{n}$ such that $\{ z(0), z(l) \} = \{ v, w \}$ and $l \le d_{n}(v, w)$.
	Then
	\[
	d(p_{n}(v), p_{n}(w)) \le \diam\left(\bigcup_{j = 0}^{l}K_{z(j)}, \metric\right) \le 2la_{\ast}^{-n},
	\]
	which implies the upper bound in \eqref{e:holder} (with $C = 2$).

	The desired lower bound requires a geometric observation.
	Let $\pi_{i} \colon \mathbb{R}^{2} \to \mathbb{R} \, (i = 1, 2)$ denote the projection map of $\mathbb{R}^{2}$ onto $i$-th coordinate, i.e. $\pi_{i}(x_1, x_2) = x_i$ for $(x_1, x_2) \in \mathbb{R}^{2}$.
	Then we observe that
	\[
	\abs{\pi_{1}(p_{n}(v)) - \pi_{1}(p_{n}(w))} \vee \abs{\pi_{2}(p_{n}(v)) - \pi_{2}(p_{n}(w))} \ge \frac{d_{n}(v, w)}{2} \cdot 2a_{\ast}^{-n - 1},
	\]
	which implies $\metric(p_{n}(v), p_{n}(w)) \ge (2\sqrt{2}a_{\ast})^{-1}d_{n}(v, w)a_{\ast}^{-n}$.
	Therefore, \eqref{e:holder} holds with $C = 2\sqrt{2}a_{\ast}$.
\end{proof}

We next move to Theorem \ref{thm.assum-PSC}\ref{it:PSCassum2}.
Thanks to Propositions \ref{prop.UPI} and \ref{prop.UCF}, checking \hyperref[cond.Ucap]{\textup{U-cap$_{p, \le}$($\pwalk$)}} and \hyperref[cond.UBCL]{\textup{U-BCL$_{p}^{\textup{low}}(d_f - \pwalk)$}} is enough for this purpose.
The planarity is crucial to ensure $\hdim - \pwalk < 1$ for \textbf{all} $p \in (1, \infty)$.
We start with the definition of $\pwalk$ which is the quantity called \emph{$p$-walk dimension} of PSC (see Definition \ref{defn.p-walk}).
This value is closely related with the following behavior of discrete $p$-capacities.
\begin{thm}[{\cite[Lemma 4.4]{BK13}}]\label{thm.super}
	Let $p \in [1, \infty)$.
	Define
	\begin{equation}\label{defn.cc}
		\mathcal{C}_{p}^{(n)} \coloneqq \sup_{m \in \mathbb{N}, w \in V_{m}}\CAP_{p}^{\mathbb{G}_{n + m}}\bigl(\pi_{n + m, m}^{-1}(w), V_{n + m} \setminus \pi_{n + m, m}^{-1}(B_{d_{m}}(w, 2))\bigr).
	\end{equation}
	Then there exists a constant $C \ge 1$ (depending only on $p, L_{\ast}$) such that
	\begin{equation}\label{cc-mult}
		C^{-1}\cdot\mathcal{C}_{p}^{(n)}\mathcal{C}_{p}^{(m)} \le \mathcal{C}_{p}^{(n + m)} \le C\cdot\mathcal{C}_{p}^{(n)}\mathcal{C}_{p}^{(m)} \quad \text{for all $n,m \in \mathbb{N}$.}
	\end{equation}
	In particular, the limit
	\begin{equation}\label{p-factor}
		\lim_{n \to \infty}\left(\mathcal{C}_{p}^{(n)}\right)^{-1/n} \eqqcolon \rho(p) \in (0,\infty)
	\end{equation}
	exists, and
	\begin{equation}\label{cc-comparable}
		C^{-1}\rho(p)^{-n} \le \mathcal{C}_{p}^{(n)} \le C\rho(p)^{-n} \quad \text{for all $n \in \mathbb{N}$.}
	\end{equation}
	We call $\rho(p)$ the \emph{$p$-scaling factor} of PSC.
\end{thm}
\begin{rmk}
	\begin{enumerate}[\rm(1)]
		\item The work \cite{BK13} has dealt with a slightly different version of $\mathcal{C}_{p}^{(n)}$, but this is not an issue because the value $M_{n}$ in \cite[Lemma 4.4]{BK13} is uniformly comparable with $\mathcal{C}_{p}^{(n)}$ (cf. Lemma \ref{lem.mod/cap}, Lemma \ref{lem.left-right} and the last line in \cite[page 66]{BK13}).
		\item In \cite{Kig20}, Kigami has introduced refined versions of \eqref{defn.cc}. See also the values $\mathcal{E}_{M, p, n}(w, T_{\abs{w}})$, $\mathcal{E}_{M, p, m, n}$ and $\mathcal{E}_{M, p, m}$, which are called \emph{conductance constants}, in \cite{Kig23}. Our $\mathcal{C}_{p}^{(n)}$ corresponds to $\mathcal{E}_{1, p, n}$ in his notation.
		\item The sub-multiplicative inequality in \eqref{cc-mult}:
		\[
		\mathcal{C}_{p}^{(n + m)} \le C\cdot\mathcal{C}_{p}^{(n)}\mathcal{C}_{p}^{(m)} \quad \text{for all $n,m \in \mathbb{N}$,}
		\]
		is shown in various general frameworks by using combinatorial modulus (see \cite[Proposition 3.6]{BK13}, \cite[Lemma 3.7]{CP13} and \cite[Lemma 4.9.3]{Kig20} for example). It is rather difficult to show the converse, namely the super-multiplicative inequality. Indeed, the argument in \cite[Lemma 4.4]{BK13} requires the planarity and symmetries of PSC. 
		Recently, Anttila and Eriksson-Bique \cite{AE.super} show the desired super-multiplicative inequality for any $p \in (1,\infty)$ on a large class of fractal spaces arising from iterated graph systems by using flows on graphs and a duality argument. See also Remark \hyperref[it:other.super]{\ref{rmk.other}}\ref{it:other.super}. 
	\end{enumerate}
\end{rmk}

\begin{defn}\label{defn.p-walk}
	Let $p \ge 1$.
	Define
	\begin{equation}\label{p-walk}
		\pwalk \coloneqq \frac{\log{\bigl(N_{\ast}\rho(p)}\bigr)}{\log{a_{\ast}}}.
	\end{equation}
	We call $\pwalk$ the \emph{$p$-walk dimension} of PSC.
\end{defn}

The next proposition is a collection of properties concerning analytic conditions. 
\begin{prop}\label{prop.PSC-analysis}
	\begin{enumerate}[\rm(i)]
		\item\label{it:PSCanalysis1} $\hdim - \pwalk < 1$ for all $p \in [1, \infty)$.
		\item\label{it:PSCanalysis2} The sequence $\{ \mathbb{G}_{n} \}_{n \in \mathbb{N}}$ satisfies \hyperref[cond.Ucap]{\textup{U-cap$_{p, \le}$($\pwalk$)}} for all $p \in [1, \infty)$.
		\item\label{it:PSCanalysis3} The sequence $\{ \mathbb{G}_{n} \}_{n \in \mathbb{N}}$ satisfies \hyperref[cond.UBCL]{\textup{U-BCL$_{p}$($\hdim - \pwalk$)}} for all $p \in (1, \infty)$.
	\end{enumerate}
\end{prop}
The Loewner type condition \ref{it:PSCanalysis3} requires a few steps, so we first prove \ref{it:PSCanalysis1} and \ref{it:PSCanalysis2}.
\begin{proof}[Proof of Proposition {\hyperref[it:PSCanalysis1]{\ref{prop.PSC-analysis}}}\ref{it:PSCanalysis1} and \ref{it:PSCanalysis2}]
	\ref{it:PSCanalysis1} Since $\hdim < 2$ and $\pwalk \ge p$ (see \cite[Proposition 3.5]{Shi24} or \cite[Lemma 4.6.15]{Kig20}), we have $\hdim - \pwalk < 2 - p \le 1$ for all $p \ge 1$.

	\ref{it:PSCanalysis2} By virtue of a similar argument to the last part in Lemma \ref{lem.cap}, it is enough to estimate discrete $p$-capacities for large enough $R$, say $R \ge 2a_{\ast} + 1$.
	Let $n \in \mathbb{N}$, $x \in V_{n}$ and $R \in [2a_{\ast} + 1, \diam(\mathbb{G}_{n}))$.
	Let $n(R) \in \mathbb{Z}$ be the unique integer such that
	\[
	2a_{\ast}^{n(R)} < R \le 2a_{\ast}^{n(R) + 1}.
	\]
	Then $1 \le n(R) < n$ since $R > 2a_{\ast}$ and $R \le 2a_{\ast}^{n}$.

	For each $w \in V_{n(R)}$, let $\varphi_{w} \colon V_{n} \to [0, 1]$ satisfy $\restr{\varphi_{w}}{S^{n - n(R)}(w)} \equiv 1$, $\supp[\varphi_{w}] \subseteq \bigcup_{v \in V_{n(R)}; d_{n(R)}(v, w) \le 1}S^{n - n(R)}(v)$ and $\mathcal{E}_{p}^{\mathbb{G}_{n}}(\varphi_{w}) = \CAP_{p}^{\mathbb{G}_{n + m}}\bigl(S^{n - n(R)}(w), V_{n} \setminus S^{n - n(R)}(B_{d_{n(R)}}(w, 2))\bigr)$.
	Let
	\[
	\mathcal{N}(x, R) \coloneqq \bigl\{ w \in V_{n(R)} \bigm| B_{\metric_{n}}(x, R) \cap S^{n - n(R)}(w) \neq \emptyset \bigr\}.
	\]
	Since $\mathbb{G}_{n}$ is metric doubling and its doubling constant depends only on $a, N_{\ast}$, we easily see that $\#\mathcal{N}(x, R) \lesssim 1$, where the bound also depends only on $a, N_{\ast}$.
	Let $\varphi \coloneqq \sum_{w \in \mathcal{N}(x, R)}\varphi_{w}$.
	Then $\restr{\varphi}{B_{d_{n}}(x, R)} \equiv 1$, $\supp[\varphi] \subseteq B_{d_{n}}(x, 2R)$ and $\mathcal{E}_{p}^{\mathbb{G}_{n}}(\varphi) \le (\#\mathcal{N}(x, r))^{p - 1}\mathcal{C}_{p}^{(n)} \lesssim \rho(p)^{-n}$.
	Since $\rho(p)^{-n} = a_{\ast}^{n(\hdim - \pwalk)} \lesssim \#B_{d_{n}}(x, R)/R^{\pwalk}$, we get \hyperref[cond.Ucap]{\textup{U-cap$_{p, \le}$($\pwalk$)}}.
\end{proof}

Let us introduce some useful notations and a new graph approximation as a preparation to prove \hyperref[cond.UBCL]{\textup{U-BCL$_{p}(\hdim - \pwalk)$}}.
Recall that
\[
L_{\ast} \coloneqq \sup_{n \in \mathbb{N}, w \in V_{n}}\deg_{\mathbb{G}_{n}}(w) \le 8.
\]
We also define
\[
E_{n}^{\#} \coloneqq \bigl\{ \{ v, w \} \in E_{n} \bigm| v \neq w, \#(K_{v} \cap K_{w}) \ge 2 \bigr\},
\]
and $G_{n}^{\#} \coloneqq (V_{n}, E_{n}^{\#})$ (see Figure \ref{fig.SCgraphs}).
We use $d_{n}^{\#}$ to denote the graph distances of $G_{n}^{\#}$.
Then $d_{n}^{\#}(v, w) \le 2$ for all $\{ v, w \} \in E_{n} \setminus E_{n}^{\#}$.
Therefore, by Proposition \ref{prop.URI}, discrete $p$-energies $\mathcal{E}_{p}^{\mathbb{G}_{n}}$ and $\mathcal{E}_{p}^{G_{n}^{\#}}$ are uniformly comparable.
In particular, we obtain the following comparability of discrete $p$-capacity and $p$-modulus.
\begin{prop}\label{prop.inv-cap-mod}
	Let $p > 0$.
	Then there exists a constant $C \ge 1$ (depending only on $p, L_{\ast}$) such that the following hold.
	\begin{enumerate}[\rm(i)]
		\item\label{it:cap-comp} For any $n \in \mathbb{N}$ and non-empty disjoint subsets $A, B$ of $V_{n}$,
		\begin{equation*}\label{inv-cap}
			C^{-1}\CAP_{p}^{G_{n}^{\#}}(A, B) \le \CAP_{p}^{\mathbb{G}_{n}}(A, B) \le C\CAP_{p}^{G_{n}^{\#}}(A, B).
		\end{equation*}
		\item\label{it:mod-comp} For any $n \in \mathbb{N}$ and non-empty disjoint subsets $A, B$ of $V_{n}$,
		\begin{equation*}\label{inv-mod}
			C^{-1}\MOD_{p}^{G_{n}^{\#}}(A, B) \le \MOD_{p}^{\mathbb{G}_{n}}(A, B) \le C\MOD_{p}^{G_{n}^{\#}}(A, B).
		\end{equation*}
	\end{enumerate}
\end{prop}
\begin{proof}
	The statement \ref{it:cap-comp} is immediate from Proposition \ref{prop.URI}.
	The second assertion follows from \ref{it:cap-comp} and Lemma \ref{lem.mod/cap}.
\end{proof}

\begin{figure}[t]\centering
	\includegraphics[height=120pt]{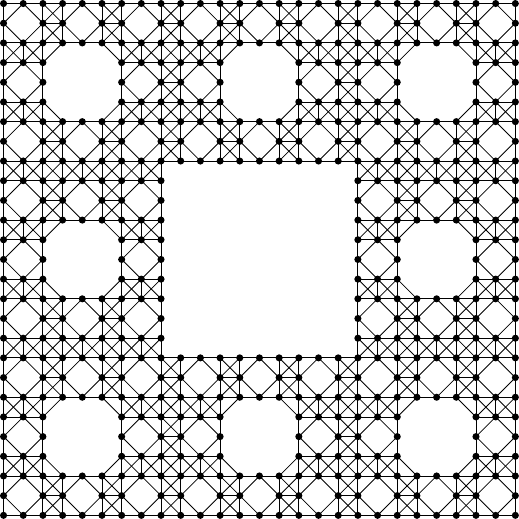}\hspace*{70pt}
	\includegraphics[height=120pt]{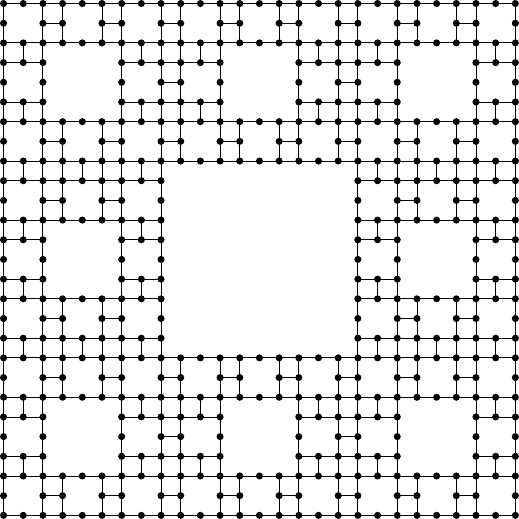}
	\caption{The graphs $\mathbb{G}_{3}$ (left) and $G_{3}^{\#}$ (right)}\label{fig.SCgraphs}
\end{figure}

We next consider the `$p$-conductance between opposite faces' whose behavior is the same as $\mathcal{C}_{p}^{(n)}$.
For $A \subseteq K$ and $n \in \mathbb{N}$, define
\[
W_{n}[A] \coloneqq \{ w \in W_{n} \mid K_{w} \cap A \neq \emptyset \}.
\]

\begin{lem}[{\cite[Lemma 4.13]{Shi24}}]\label{lem.left-right}
	There exists a constant $C \ge 1$ depending only on $p, L_{\ast}$ such that
	\[
	C^{-1}\rho(p)^{-n} \le \MOD_{p}^{\mathbb{G}_{n}}\bigl(W_{n}[\ell_{1}], W_{n}[\ell_{2}]\bigr) \le C\rho(p)^{-n} \quad \text{for all $n \in \mathbb{N}$,}
	\]
	whenever $\{ \ell_{1}, \ell_{2} \} = \{ \ell_{\textup{L}}, \ell_{\textup{R}} \}$ or $\{ \ell_{1}, \ell_{2} \} = \{ \ell_{\textup{T}}, \ell_{\textup{B}} \}$.
\end{lem}

The following notation and result are needed to describe `local symmetry' of PSC.
For $\{ v, w \} \in E_{m}^{\#}$, define
\[
\ell_{v, w} = K_{v} \cap K_{w}.
\]
We let $\mathcal{R}_{v, w} \colon \mathbb{R}^{2} \to \mathbb{R}^{2}$ be the reflection in the line containing $\ell_{v, w}$.

\begin{prop}\label{prop.locsym-word}
	For any $\{ v, w \} \in E_{m}$, there exists a bijection $\tau_{v, w} \colon \bigcup_{n \in \mathbb{N}}S^{n}( \{ v, w \}) \to \bigcup_{n \in \mathbb{N}}S^{n}( \{ v, w \})$ such that $\mathcal{R}_{v, w}(K_{z}) = K_{\tau_{v, w}(z)}$ for all $z \in \bigcup_{n \in \mathbb{N}}S^{n}( \{ v, w \})$, where $S^{n}( \{ v, w \}) \coloneqq S^{n}(v) \cup S^{n}(w) = \{ z \in W_{n + m} \mid [z]_{m} \in \{ v, w \}\}$.
	Moreover, $\tau_{v, w}\bigl(S^{n}(v)\bigr) = S^{n}(w)$ for all $n \in \mathbb{N}$.
\end{prop}
\begin{proof}
	We observe that for any $\{ v, w \} \in E_{m}$ there exists a unique $\Phi \in D_{4}$ such that $F_{w} \circ \Phi \circ F_{v}^{-1} = \mathcal{R}_{v, w}$ on $K_{v}$.
	Then it is easy to see that
	\[
	\tau_{v, w}(z) \coloneqq
	\begin{cases}
		w\tau_{\Phi}(z_{m + 1}\cdots z_{\abs{z}}) \quad &\text{if $[z]_{m} = v$,}\\
		v\tau_{\Phi}(z_{m + 1}\cdots z_{\abs{z}}) \quad &\text{if $[z]_{m} = w$,}
	\end{cases}
	\quad \text{for $z \in \bigcup_{n \in \mathbb{N}}\bigl(S^{n}(v) \cup S^{n}(w)\bigr)$,}
	\]
	is the map satisfying the required properties.
\end{proof}

We recall a useful fact on combinatorial modulus due to Kigami.
\begin{lem}[{\cite[Lemma C.4]{Kig23}}]\label{lem.mod-cover}
	Let $p > 0$.
	Let $G_{i} = (V_{i}, E_{i}) \, (i = 1, 2)$ be two graphs with $\deg(G_{i}) < \infty$ and let $H \colon V_{1} \to 2^{V_{2}}$ be a function such that $\#H(v) < \infty$ for all $v \in V_{1}$.
	Let $\Theta_{1}, \Theta_{2}$ be two path families of paths in $G_{1}, G_{2}$ respectively such that for each $\theta \in \Theta_{1}$, there exists $\theta' \in \Theta_{2}$ such that $\theta' \subseteq \bigcup_{v \in \theta}H(v)$.
	Then
	\begin{equation}\label{eq.mod-covering}
		\MOD_{p}^{G_1}(\Theta_1) \le C\biggl(\sup_{v \in V_{1}}\#H(v)\biggr)^p\sup_{v' \in V_{2}}\#\bigl\{ v \in V_{1} \bigm| v' \in H(v) \bigr\}\MOD_{p}^{G_2}(\Theta_{2}).
	\end{equation}
where $C > 0$ is a constant depending only on $p, \deg(G_{1})$ and $\deg(G_{2})$.
\end{lem}

With these preparations, we now check \hyperref[cond.UBCL]{\textup{U-BCL$_{p}$($\hdim - \pwalk$)}} for PSC.
The following lemma is a key ingredient.
\begin{lem}\label{lem.mod-cell}
	Let $p \ge 1$ and let $L \ge 1$.
	There exists a constant $c > 0$ (depending only on $p, L, L_{\ast}$) such that the following hold: for any $k, m \in \mathbb{N}$ and $v, w \in V_{m}$ with $d_{m}(v, w) \le L$,
	\begin{equation}\label{mod.min-max}
		\MOD_{p}^{\mathbb{G}_{m + k}}\bigl(\{ \theta \in \PATH(S^k(v), S^k(w); \mathbb{G}_{k + m}) \mid \diam(\theta, d_{k + m}) \le 2La_{\ast}^{k} \}\bigr) \ge c\rho(p)^{-k}.
	\end{equation}
\end{lem}
\begin{proof}
	The idea goes back to \cite[Lemma 4.4]{BK13}. (See also \cite[Theorem 4.8]{Kig23}.)
	We first note that
	\[
	\Theta_{k}(v, w) \coloneqq \{ \theta \in \PATH(S^k(v), S^k(w); \mathbb{G}_{k + m}) \mid \diam(\theta, d_{k + m}) \le 2La_{\ast}^{k} \} \neq \emptyset
	\]
	since $\diam(S^{k}(z), d_{k + \abs{z}}) \le 2a_{\ast}^{k}$ for all $k \in \mathbb{N}$ and $z \in W_{\ast}$.

	If $v = w$, then $\Theta_{k}(v, w)$ contains
	\[
	\bigl\{ [vz(0), \dots, vz(l)] \bigm| [z(0), \dots, z(l)] \in \PATH\bigl(W_{k}[\ell_{\textup{L}}], W_{k}[\ell_{\textup{R}}]; \mathbb{G}_{k}\bigr) \bigr\}.
	\]
	Therefore, by Lemmas \ref{lem.basic-pMod}\ref{it:mod.mono} and  \ref{lem.left-right}, we get
	\[
	\MOD_{p}^{\mathbb{G}_{k + m}}\bigl(\Theta_{k}(v, w)\bigr) \ge C^{-1}\rho(p)^{-k},
	\]
	where $C \ge 1$ is the same constant as in Lemma \ref{lem.left-right}.

	For $v, w \in V_{m}$ with $v \neq w$, we fix a simple path $\gamma = [z(0), z(1), \dots, z(l)]$ in $G_{m}^{\#}$ (NOT in $\mathbb{G}_{m}$!) such that $z(0) = v$ and $z(l) = w$.
	We will prove
	\[
	\rho(p)^{-k} \lesssim \MOD_{p}^{\mathbb{G}_{k + m}}\Biggl(\biggl\{ \theta \in \Theta_{k}(v, w) \biggm| \theta \subseteq \bigcup_{j = 0}^{l}S^{k}\bigl(z(j)\bigr) \biggr\}\Biggr).
	\]
	For ease of notation, set
	\[
	\Theta_{k}(v, w; \gamma) \coloneqq \biggl\{ \theta \in \PATH(S^k(v), S^k(w); \mathbb{G}_{k + m}) \biggm| \theta \subseteq \bigcup_{j = 0}^{l}S^{k}\bigl(z(j)\bigr) \biggr\}.
	\]
	For $j \in \{0, \dots, l\}$, we inductively define $H_{z(j)} \colon V_{k} \to 2^{V_{k + m}}$ in the following manner: define
	\[
	H_{z(0)}(z) = \bigl\{ z(0) \cdot z, z(0) \cdot \tau_{S_{1}}(z) \bigr\} \quad \text{for $x \in V_{k}$;}
	\]
	and
	\[
	H_{z(j + 1)}(z) = \tau_{z(j), z(j + 1)}\left(H_{z(j)}(z)\right),
	\]
	where $\tau_{z(j), z(j + 1)} \colon S^{k}\bigl(\{ z(j), z(j + 1) \}\bigr) \to S^{k}\bigl(\{ z(j), z(j + 1) \}\bigr)$ is the bijection in Proposition \ref{prop.locsym-word}.
	(Recall that $S_{1} \in D_{4}$ is the reflection in the line $\{ y = x \}$.)
	We now define $H \colon V_{k} \to 2^{V_{k + m}}$ by
	\[
	H(z) \coloneqq \bigcup_{j = 0}^{l}H_{z(j)}(z).
	\]
	Then we claim the following:
	\begin{align}\label{lr-cover}
		&\text{For any $\theta \in \mathsf{Path}\bigl(W_{k}[\ell_{\textup{L}}], W_{k}[\ell_{\textup{R}}]; \mathbb{G}_{k}\bigr)$, there exists a path $\theta' \in \Theta_{k}(v, w; \gamma)$} \nonumber \\
		&\text{such that $\theta' \subseteq \bigcup_{z \in \theta}H(z)$.}
	\end{align}
	Since $\tau_{S_{1}}\bigl(W_{k}[\ell_{\textup{L}}]\bigr) = W_{k}[\ell_{\textup{B}}]$ and $\tau_{S_{1}}\bigl(W_{k}[\ell_{\textup{R}}]\bigr) = W_{k}[\ell_{\textup{T}}]$, we observe that $\tau_{S_{1}}(\theta)$ is a path in $\mathbb{G}_{k}$ joining $W_{k}[\ell_{\textup{B}}]$ and $W_{k}[\ell_{\textup{T}}]$ for any $\theta \in \mathsf{Path}\bigl(W_{k}[\ell_{\textup{L}}], W_{k}[\ell_{\textup{R}}]; \mathbb{G}_{k}\bigr)$.
	Hence, for any $j \in \{ 0, \dots, l \}$ and $\ell_{1}, \ell_{2} \in \{ \ell_{\textup{L}}, \ell_{\textup{R}}, \ell_{\textup{B}}, \ell_{\textup{T}} \}$ with $\ell_{1} \neq \ell_{2}$, $H_{z(j)}(\theta)$ contains a path joining
	\[
	\bigl\{ z \in S^{k}\bigl(z(j)\bigr) \bigm| K_{z} \cap \ell_{1} \neq \emptyset \bigr\} \quad \text{and} \quad \bigl\{ z \in S^{k}\bigl(z(j)\bigr) \bigm| K_{z} \cap \ell_{2} \neq \emptyset \bigr\}.
	\]
	Combining with the fact that
	\[
	\bigl\{ z, \tau_{z(j), z(j + 1)}(z) \bigr\} \in E_{k + m} \quad \text{for all $z \in \bigl(S^{k}(z(j)) \cup S^{k}(z(j + 1))\bigr) \cap W_{k + m}[\ell_{z(j), z(j + 1)}]$},
	\]
	we obtain \eqref{lr-cover} 	(See also Figure \ref{fig.loewner})
	.

	\begin{figure}\centering
		\includegraphics[height=200pt]{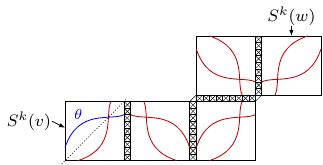}
		\caption{The subset $\bigcup_{z \in \theta}H(z)$ drawn in red. Each big box corresponds to a copy of $\mathbb{G}_{k}$ in $\mathbb{G}_{k + m}$.}\label{fig.loewner}
	\end{figure}

	Lemma \ref{lem.mod-cover} together with \eqref{lr-cover} yields
	\begin{align*}
		\Mod_{p}^{\mathbb{G}_{k}}\bigl(W_{k}[\ell_{\textup{L}}], W_{k}[\ell_{\textup{R}}]\bigr)
		\le 2^{p + 1}C' \cdot \MOD_{p}^{\mathbb{G}_{k + m}}\bigl(\Theta_{k}(v, w; \gamma)\bigr),
	\end{align*}
	where $C' > 0$ is the same constant as in Lemma \ref{lem.mod-cover}.
	Combining with Lemma \ref{lem.left-right}, we obtain $\MOD_{p}^{\mathbb{G}_{k + m}}\bigl(\Theta_{k}(v, w; \gamma)\bigr) \gtrsim \rho(p)^{-k}$.

	Since $\diam\bigl(S^{k}(z), d_{k + \abs{z}}\bigr) \le 2a_{\ast}^{k}$ for all $z \in W_{\ast}$ and $k \in \mathbb{N}$, we have
	\[
	\Theta_{k}(v, w; \gamma) \subseteq \bigl\{ \theta \in \PATH(S^k(v), S^k(w); \mathbb{G}_{k + m}) \bigm| \diam(\theta, d_{k + m}) \le 2la_{\ast}^{k} \bigr\}.
	\]
	We obtain the required estimate by choosing a path $\gamma$ with $l \in [d_{m}(v, w), L]$ and using Lemma \ref{lem.basic-pMod}\ref{it:mod.mono}.
\end{proof}

Combining Lemma \ref{lem.mod-cell} with a geometric observation, we immediately obtain \hyperref[cond.UBCL]{\textup{U-BCL$_{p}$($\hdim - \pwalk$)}}.
\begin{proof}[Proof of Proposition {\hyperref[it:PSCanalysis1]{\ref{prop.PSC-analysis}}}\ref{it:PSCanalysis3}]
	Let $\kappa > 0$, $n \in \mathbb{N}$ and $1 \le R \le \diam(\mathbb{G}_{n})$.
	Let $B_{i} \, (i = 1, 2)$ be balls in $\mathbb{G}_{n}$ with radii $R$ such that $\dist_{d_{n}}(B_{1}, B_{2}) \le \kappa R$.
	Choose $n(R) \in \mathbb{Z}$ so that
	\[
	2a_{\ast}^{n(R)} < R \le 2a_{\ast}^{n(R) + 1}.
	\]
	By $R \le \diam(\mathbb{G}_{n})$ and $\diam(\mathbb{G}_{n}) \le 2a_{\ast}^{n}$, we then have $n \ge n(R)$.

	First, we suppose $R \ge 3$.
	Then $n(R) \ge 0$.
	It is a simple observation that there exist $w(1), w(2) \in V_{n - n(R)}$ such that
	\[
	\text{`$S^{n(R)}\bigl(w(i)\bigr) \subseteq B_{i}$'} \quad \text{and} \quad \text{`$S^{n(R)}\bigl(w(i)\bigr)$ contains the center of $B_{i}$'} \quad \text{for each $i = 1, 2$.}
	\]
	Then, we have
	\[
	\dist_{d_{n}}\Bigl(S^{n(R)}\bigl(w(1)\bigr), S^{n(R)}\bigl(w(2)\bigr)\Bigr) \le R + \kappa R + R \le 2(2 + \kappa)a_{\ast} \cdot a_{\ast}^{n(R)}.
	\]
	This together with a similar argument to Proposition \hyperref[it:PSCgeom.1]{\ref{prop.PSC-geom}}\ref{it:PSCgeom.5} implies $d_{n - n(R)}(w(1), w(2)) \le 2\cdot2\lceil(2 + \kappa)a_{\ast}\rceil \eqqcolon L(\kappa)$.
	By Lemmas \ref{lem.basic-pMod}\ref{it:mod.mono} and \ref{lem.mod-cell},
	\begin{align}\label{KM-mod}
		&\MOD_{p}^{\mathbb{G}_{n}}(\{ \theta \in \PATH(B_1, B_2; \mathbb{G}_{n}) \mid \diam(\theta, d_{n}) \le L(\kappa)R \}) \nonumber \\
		&\ge \MOD_{p}^{\mathbb{G}_{n}}\Bigl(\bigl\{ \theta \in \PATH\bigl(S^{n(R)}\bigl(w(1)\bigr), S^{n(R)}\bigl(w(2)\bigr); \mathbb{G}_{n}\bigr) \bigm| \diam(\theta, d_{n}) \le L(\kappa)R \bigr\}\Bigr) \nonumber \\
		&\ge \MOD_{p}^{\mathbb{G}_{n}}\Bigl(\bigl\{ \theta \in \PATH\bigl(S^{n(R)}\bigl(w(1)\bigr), S^{n(R)}\bigl(w(2)\bigr); \mathbb{G}_{n}\bigr) \bigm| \diam(\theta, d_{n}) \le 2L(\kappa)a_{\ast}^{n(R)} \bigr\}\Bigr) \nonumber \\
		&\ge C^{-1}\rho(p)^{-n(R)}
		= C^{-1}a_{\ast}^{n(R)(\hdim - \pwalk)}
		\ge C^{-1}(2a_{\ast})^{- \hdim + \pwalk} \cdot R^{\hdim - \pwalk},
	\end{align}
	where $C > 0$ is the same constant as in \eqref{mod.min-max} (with $L = L(\kappa)$).

	Let us consider the case $1 \le R < 3$ to complete the proof.
	By \eqref{smallscale} in Lemma \ref{lem.shortPath},
	\begin{align*}
		\MOD_{p}^{\mathbb{G}_{n}}(\{ \theta \in \PATH(B_1, B_2; \mathbb{G}_{n}) \mid \diam(\theta, d_{n}) \le L(\kappa)R \})
		&\ge \bigl(L(\kappa)R\bigr)^{1 - p} \\
		&\ge 3^{-p}L(\kappa)^{1 - p}R^{\hdim - \pwalk},
	\end{align*}
	where we used $\hdim - \pwalk < 1$ (Proposition \hyperref[it:PSCanalysis1]{\ref{prop.PSC-analysis}}\ref{it:PSCanalysis1}) and $R < 3$ in the last inequality.
\end{proof}

Once we know \hyperref[cond.UBCL]{\textup{U-BCL$_{p}(d_f - \pwalk)$}} for PSC (and observe some fundamental geometric conditions), we can apply Theorem \ref{thm.Epgamma} so that we get $\mathcal{E}_{p}^{\Gamma}$ on PSC.
Our desired self-similar $p$-energy $\mathcal{E}_{p}$ will be obtained by applying Theorem \ref{thm.fix} to $\mathcal{E}_{p}^{\Gamma}$.
The important hypothesis \eqref{ss.comparable} in Theorem \ref{thm.fix} will be verified with the help of an \emph{unfolding argument}, which is heavily inspired by \cite[subsection 5.1]{Hin13}.
In order to get a self-similar $p$-energy by applying Corollary \ref{cor.ss-energy}, the remaining condition we have to check is the pre-self-similarity condition (\hyperref[PSS]{PSS}) in Theorem \ref{thm.fix}, i.e., there exists $C \ge 1$ such that
\begin{equation}\label{PSC.renormalization}
	C^{-1}\abs{f}_{\mathcal{F}_{p}}^{p} \le \rho(p)^{n}\sum_{w \in W_{n}}\abs{f \circ F_w}_{\mathcal{F}_{p}}^{p} \le C\abs{f}_{\mathcal{F}_{p}}^{p} \quad \text{for any $f \in \mathcal{F}_{p}$ and $n \in \mathbb{N}$.}
\end{equation}
In the rest of this subsection, we will prove the following stronger condition (\hyperref[PSS-strong]{\textup{PSS'}}) including the \emph{self-similarity of the domain}:
\begin{equation*}\label{PSS-strong}
	\text{(PSS') \, \eqref{PSC.renormalization} holds and $\mathcal{F}_{p} \cap \contfunc(K) = \{ f \in \contfunc(K) \mid \text{$f \circ F_i \in \mathcal{F}_{p} \cap \contfunc(K)$ for any $i \in S$} \}$.}
\end{equation*}

\begin{prop}\label{prop.renormalization}
	PSC satisfies \hyperref[PSS-strong]{\textup{(PSS')}} for any $p \in (1,\infty)$.
\end{prop}
The proof of the above proposition is long, so we will divide into several steps.
First, we prove the following easy bound:
\begin{equation}\label{renormalize-upper}
	\rho(p)^{n}\sum_{w \in W_{n}}\abs{f \circ F_w}_{\mathcal{F}_{p}}^{p} \lesssim \abs{f}_{\mathcal{F}_{p}}^{p} \quad \text{for any $f \in L^{p}(K, \measure)$}.
\end{equation}
Here we regard $\abs{\,\cdot\,}_{\mathcal{F}_{p}}$ as a $[0, \infty]$-valued functional defined on $L^{p}(K, \measure)$, which satisfies $\abs{f}_{\mathcal{F}_{p}} < \infty$ if and only if $f \in \mathcal{F}_{p}$.
\begin{proof}[Proof of \eqref{renormalize-upper}]
	Since $\measure$ is the self-similar measure with the equal weight, we have $M_{n}(f \circ F_w)(v) = M_{n + k}f(wv)$ for $n, k \in \mathbb{N}$ and $w \in V_{k}$, $v \in V_{n}$.
	Therefore,
	\begin{align*}
		\rho(p)^{n}\sum_{w \in W_{n}}\widetilde{\mathcal{E}}_{p}^{(k)}(f \circ F_w)
		= \sum_{w \in W_{n}}\widetilde{\mathcal{E}}_{p, S^{k}(w)}^{(n+k)}(f)
		\le \widetilde{\mathcal{E}}_{p}^{(n+k)}(f),
	\end{align*} 
	which together with the weak monotonicity (Theorem \ref{t:wm}) implies \eqref{renormalize-upper}.
\end{proof}

The reverse inequality is much harder and requires the notion of \emph{unfolding} of functions.
We will use a modified version of the argument using unfolding operators in \cite[subsection 5.1]{Hin13} to show the self-similarity of the domain and the converse estimate:
\begin{equation}\label{renormalize-lower}
	\abs{f}_{\mathcal{F}_{p}}^{p} \lesssim \rho(p)^{n}\sum_{w \in W_{n}}\abs{f \circ F_w}_{\mathcal{F}_{p}}^{p} \quad \text{for any $f \in \mathcal{F}_{p}^{S}$},
\end{equation}
where we set $\mathcal{F}_{p}^{S} \coloneqq \{ f \in \contfunc(K) \mid \text{$f \circ F_i \in \mathcal{F}_{p} \cap \contfunc(K)$ for any $i \in S$} \}$.

\begin{defn}[Folding maps and unfolding operators]
	\begin{enumerate}[\rm(1)]
		\item For $n \in \mathbb{N}$, let $\widehat{\varphi}_{n} \colon \mathbb{R} \to [0, \infty)$ be the periodic function with period $4a_{\ast}^{-n}$ such that
			\[
			\widehat{\varphi}_{n}(t) =
			\begin{cases}
				t + 1 \quad &\text{for $t \in [-1, -1 + 2a_{\ast}^{-n}]$,}\\
				-t - 1 + 4a_{\ast}^{-n} \quad &\text{for $t \in [-1 + 2a_{\ast}^{-n}, -1 + 4a_{\ast}^{-n}]$.}
			\end{cases}
			\]
			Define $\varphi^{(n)} \colon [-1, 1]^{2} \to [0, 2a_{\ast}^{-n}]^{2}$ by
			\[
			\varphi^{[n]}(x, y) \coloneqq \bigl(\widehat{\varphi}_{n}(x), \widehat{\varphi}_{n}(y)\bigr) \quad \text{for $(x,y) \in [-1, 1]^{2}$.}
			\]
			For $w \in V_{n}$, define $\varphi_{w} \colon K \to K_{w}$ by
			\[
			\varphi_{w}(x) \coloneqq \left(\restr{\varphi^{[n]}}{K_{w}}\right)^{-1}\bigl(\varphi^{[n]}(x)\bigr) \quad \text{for $x \in K$.}
			\]
			\item For $\{ v, w \} \in E_{n}^{\#}$, let $H_{v, w}$ be the line containing $\ell_{v, w}$.
				Then $H_{v, w}$ splits $\mathbb{R}^{2}$ into the two closed half spaces, which are denoted by $G_{v, w}$ and $G_{w, v}$ and satisfy $K_{v} \subseteq G_{v, w}$ and $K_{w} \subseteq G_{w, v}$. We remark that the order of $v$ and $w$ is important in the notations $G_{v, w}$, $G_{w, v}$.
			\item For $f \in L^{p}(K, \measure)$ and $w \in W_{n}$, define $\Xi_{w}(f) \coloneqq f \circ \varphi_{w}$. The map $\Xi_{w}$ is called an \emph{unfolding operator}. For $\{ v, w \} \in E_{n}^{\#}$, define $\Xi_{v, w}(f) \coloneqq \Xi_{v}(f)\indicator{G_{v, w}}$.
	\end{enumerate}
\end{defn}
\begin{rmk}
	For $w \in W_{\ast}$, define
	\[
	\mathsf{N}(w) \coloneqq \bigl\{ v \in W_{\ast} \bigm| \text{$\abs{v} = \abs{w}$ and $\{ v, w \} \in E_{\abs{w}}^{\#}$} \bigr\} \cup \{ w \}.
	\]
	Then $\restr{\varphi_{w}}{K_{\mathsf{N}(w)}}$ satisfies
	\[
		\restr{\varphi_{w}}{K_{\mathsf{N}(w)}}(x) =
		\begin{cases}
			x \quad &\text{if $x \in K_{w}$,} \\
			\mathcal{R}_{v, w}(x) &\text{if $x \in K_{v}$ for some $v \in \mathsf{N}(w) \setminus \{ w \}$.}
		\end{cases}
	\]
	For other basic properties of $\varphi_{w}$, we refer to \cite[Lemma 2.13]{BBKT}.
\end{rmk}

To provide a quantitative (localized) energy estimate for $\Xi_{z}(f)$ by following \cite{Hin13}, we make the help of Korevaar--Schoen type bounds given in Section \ref{sec.domain}.
(Note that we  \emph{cannot} use results in Section \ref{sec.em-local} because we have no energy measures at this stage.)
Recall that, by Theorem \ref{thm.LB}, there exists a constant $C \ge 1$ such that, for any $f \in L^{p}(K, \measure)$,
\begin{equation}\label{LB-folding}
	C^{-1}\abs{f}_{\mathcal{F}_{p}}^{p} \le \varlimsup_{r \downarrow 0}r^{-(\hdim + \pwalk)}\iint_{\{ (x, y) \in K \times K \mid d(x, y) < r \}}\abs{f(x) - f(y)}^{p}\,\measure(dx)\measure(dy) \le C\abs{f}_{\mathcal{F}_{p}}^{p}.
\end{equation}
Let us introduce some notations for simplicity.
For $f \in L^{p}(K, \measure)$ and $\delta > 0$, define
\begin{align*}
	E_{p, \delta}(f)
	&\coloneqq \delta^{-(\hdim + \pwalk)}\iint_{\{ (x, y) \in K \times K \mid d(x, y) < \delta \}}\abs{f(x) - f(y)}^{p}\,\measure(dx)\measure(dy) \\
	&=  \delta^{-(\hdim + \pwalk)}\int_{\{ (x, y) \in K \times K \mid d(x, y) < \delta \}}\abs{f(x) - f(y)}^{p}\,\measure\otimes\measure(dxdy).
\end{align*}
For $A_{1}, A_{2} \in \mathcal{B}(K)$, we also define
\[
E_{p, \delta}(f; A_{1}, A_{2}) \coloneqq \delta^{-(\hdim + \pwalk)}\iint_{\{ (x, y) \in A_{1} \times A_{2} \mid d(x, y) < \delta \}}\abs{f(x) - f(y)}^{p}\,\measure(dx)\measure(dy).
\]
For simplicity, we write $E_{p, \delta}(f; A)$ for $E_{p, \delta}(f; A, A)$.
Since $\measure$ is the self-similar measure with the weight $(a_{\ast}^{-\hdim}, \dots, a_{\ast}^{-\hdim})$, we have
\[
E_{p, \delta}(f; K_{w}) = \rho(p)^{n}E_{p, a_{\ast}^{n}\delta}(f \circ F_w) \quad \text{for any $w \in V_{n}$.}
\]
(Note that $\rho(p)^{n}a_{\ast}^{-n(\hdim + \pwalk)} = a_{\ast}^{-2n\hdim}$.)
Additionally, we have $\indicator{K_{v} \cup K_{w}}(\mathcal{R}_{v, w})_{\ast}\measure(dx) = \indicator{K_{v} \cup K_{w}}\measure(dx)$ for any $\{ v, w \} \in E_{n}^{\#}$.

The following estimate on localized energies of $\Xi_{z}(f)$ is a key ingredient.
\begin{lem}\label{lem.KS-folding1}
	Let $n \in \mathbb{N}$, $z \in W_{n}$, $\delta > 0$ and $f \in L^{p}(K, \measure)$.
	Then, for any $\{ v, w \} \in E_{n}$,
	\[
	E_{p, \delta}\bigl(\Xi_{z}(f); K_{v}, K_{w}\bigr) \le E_{p, \delta}\bigl(\Xi_{z}(f); K_{v}\bigr) \le \rho(p)^{n}E_{p, a_{\ast}^{n}\delta}(F_{z}^{\ast}f).
	\]
	In particular, there exists a constant $C > 0$ such that
	\[
	\abs{\Xi_{z}(f)}_{\mathcal{F}_{p}}^{p} \le C(\#W_{n})\rho(p)^{n}\abs{F_{z}^{\ast}f}_{\mathcal{F}_{p}}^{p} \quad \text{for any $f \in L^{p}(K, \measure)$, $n \in \mathbb{N}$ and $z \in W_{n}$,}
	\]
	where $F_{z}^{\ast}f \coloneqq f \circ F_{z}$.
\end{lem}
\begin{proof}
	This lemma corresponds to \cite[Corollary 5.4]{Hin13}.
	For $v, z \in W_{n}$, we see that
	\begin{align*}
		&E_{p, \delta}\bigl(\Xi_{z}(f); K_{v}\bigr) \\
		&= \delta^{-(\hdim + \pwalk)}\iint_{\{ (x, y) \in K_{v} \times K_{v} \mid d(x, y) < \delta \}}\abs{\bigl(F_{z}^{\ast}f \circ F_{z}^{-1} \circ \varphi_{z}\bigr)(x) - \bigl(F_{z}^{\ast}f \circ F_{z}^{-1} \circ \varphi_{z}\bigr)(y)}^{p}\,\measure(dx)\measure(dy) \\
		&= \delta^{-(\hdim + \pwalk)}\iint_{\{ (x, y) \in K_{z} \times K_{z} \mid d(x, y) < \delta \}}\abs{F_{z}^{\ast}f\bigl(F_{z}^{-1}(x)\bigr) - F_{z}^{\ast}f\bigl(F_{z}^{-1}(y)\bigr)}^{p}\,\measure(dx)\measure(dy) \\
		&= E_{p, \delta}(f; K_{z}) = \rho(p)^{n}E_{p, a_{\ast}^{n}\delta}(F_{z}^{\ast}f),
	\end{align*}
	where we used \cite[(2.22)]{BBKT} (with $\nu = \restr{\mu}{K_{v}}$) in the second equality.
	Furthermore,
	\begin{align*}
		&E_{p, \delta}\bigl(\Xi_{z}(f); K_{v}, K_{w}\bigr) \\
		&= \delta^{-(\hdim + \pwalk)}\iint_{\{ (x, y) \in K_{v} \times K_{w} \mid d(x, y) < \delta \}}\abs{\bigl(f \circ \varphi_{z}\bigr)(x) - \bigl(f \circ \varphi_{z} \circ \varphi_{v}\bigr)(y)}^{p}\,\measure(dx)\measure(dy) \\
		&= \delta^{-(\hdim + \pwalk)}\iint_{\{ (x, y) \in K_{v} \times K_{w} \mid d(x, y) < \delta \}}\abs{\Xi_{z}(f)(x) - \Xi_{z}(f)\bigl(\varphi_{v}(y)\bigr)}^{p}\,\measure(dx)\measure(dy) \\
		&\le \delta^{-(\hdim + \pwalk)}\iint_{\{ (x, y) \in K_{v} \times K_{v} \mid d(x, y) < \delta \}}\abs{\Xi_{z}(f)(x) - \Xi_{z}(f)(y)}^{p}\,\measure(dx)\measure(dy)
		= E_{p, \delta}\bigl(\Xi_{z}(f); K_{v}\bigr),
	\end{align*}
	where we used $(\varphi_{z} \circ \varphi_{v})(y) = \varphi_{z}(y)$ for $y \in K_{w}$ in the first identity, $\metric(x, \varphi_{v}(y)) \le \metric(x, y)$ for $(x, y) \in K_{v} \times K_{w}$ in the fourth line.

	Next we give an estimate for $\abs{\Xi_{z}(f)}_{\mathcal{F}_{p}}$.
	Let $n \in \mathbb{N}$ and $z \in W_{n}$.
	For small enough $\delta > 0$, we observe that
	\[
	E_{p, \delta}(\Xi_{z}(f)) = \sum_{v \in W_{n}}E_{p, \delta}(\Xi_{z}(f); K_{v}) + \sum_{\{ v, w \} \in E_{n}}E_{p, \delta}(\Xi_{z}(f); K_{v}, K_{w}).
	\]
	Therefore, we have $E_{p, \delta}(\Xi_{z}(f)) \le (1 + L_{\ast})\rho(p)^{n}\sum_{v \in W_{n}}E_{p, a_{\ast}^{n}\delta}(f \circ F_{z})$
	(recall that $L_{\ast} = \sup_{n \in \mathbb{N}}\deg(\mathbb{G}_{n}) < \infty$ as defined in Definition \ref{defn:unifdeg}). Hence 
	\[
	\varlimsup_{\delta \downarrow 0}E_{p, \delta}(\Xi_{z}(f)) \lesssim \rho(p)^{n}\varlimsup_{\delta \downarrow 0}E_{p, \delta}(f \circ F_{z})(\#W_{n}), 
	\]
	which together with \eqref{LB-folding} yields the desired conclusion.
\end{proof}

We also need the following approximation.
\begin{lem}\label{lem.approx-folding}
	Let $F$ be a non-empty subset of $K$.
	Suppose that $f \in \mathcal{F}_{p} \cap \mathcal{C}(K)$ satisfies $f(x) = 0$ for any $x \in F$.
	Then there exist $f_{n} \in \mathcal{F}_{p} \cap \mathcal{C}(K) \, (n \in \mathbb{N})$ such that $\supp[f_{n}] \subseteq K \setminus F$ for all $n \in \mathbb{N}$ and $f_{n}$ converges in $\mathcal{F}_{p}$ to $f$ as $n \to \infty$.
\end{lem}
\begin{proof}
	We first consider the case that $f$ is non-negative, i.e., let us suppose that $f \in \mathcal{F}_{p} \cap \mathcal{C}(K)$ satisfies $\restr{f}{F} = 0$ and $f \ge 0$.
	Since $f$ is uniformly continuous, for any $n \in \mathbb{N}$ there exists $r_{n} > 0$ such that
	\[
	f(x) < \frac{1}{n} \quad \text{for all $x \in F_{n} \coloneqq \bigcup_{x \in F}B_{\metric}(x, r_{n})$.}
	\]
	Define $f_{n} \in \mathcal{F}_{p} \cap \mathcal{C}(K)$ by
	\[
	f_{n} = \bigl(f - n^{-1}\bigr) \vee 0.
	\]
	Then we immediately have $f_{n}(x) = 0$ for $x \in F_{n}$ and $\supp[f_{n}] \subseteq K \setminus F$.
	Furthermore, by Theorem \ref{thm.Epgamma}\ref{it:Epgamma.lip} (or \eqref{LB-folding}), we have
	\[
	\abs{f_{n}}_{\mathcal{F}_{p}} \le C\abs{f}_{\mathcal{F}_{p}} \quad \text{for all $n \ge 1$,}
	\]
	where $C$ is independent of $f$ and $n$.
	It is also clear that $\sup_{x \in K}\abs{f(x) - f_{n}(x)} \to 0$ as $n \to \infty$ and hence $\norm{f - f_{n}}_{L^{p}} \to 0$.
	Since $\{ f_{n} \}_{n \ge 1}$ is a bounded sequence in $\mathcal{F}_{p}$, there exists a subsequence $\{ f_{n_{k}} \}_{k \ge 1}$ such that $f_{n_{k}}$ converges weakly in $\mathcal{F}_{p}$ to $f$ as $k \to \infty$.
	Applying Mazur's lemma (Lemma \ref{lem.Mazur}), we get $g_{n} \in \mathcal{F}_{p} \cap \mathcal{C}(K) \, (n \ge 1)$ such that $\supp[g_{n}] \subseteq K \setminus F$ and $\norm{f - g_{n}}_{\mathcal{F}_{p}} \to 0$, which proves our assertion.

	For general $f \in \mathcal{F}_{p} \cap \mathcal{C}(K)$ satisfying $\restr{f}{F} = 0$, we obtain the assertion by applying the above result to $f^{\pm}$.
\end{proof}

Next we prove a Fatou type lemma for localized Korevaar--Schoen energies.
\begin{lem}\label{lem.lsc-folding}
	Let $f, f_{k} \in \mathcal{F}_{p} \, (k \in \mathbb{N})$ such that $f_{k}$ converges in $L^{p}(K, \measure)$ to $f$ as $k \to \infty$.
	Suppose $\sup_{k \in \mathbb{N}}\abs{f_{k}}_{\mathcal{F}_{p}} < \infty$.
	Then, for any $n \in \mathbb{N}$ and $\{ v, w \} \in E_{n}$,
	\[
	\limsup_{\delta \downarrow 0}E_{p, \delta}(f; K_{v}, K_{w}) \le \liminf_{n \to \infty}\limsup_{\delta \downarrow 0}E_{p, \delta}(f_{n}; K_{v}, K_{w}).
	\]
\end{lem}
\begin{proof}
	First, we prove the following claim: for any $g, g_{k} \in \mathcal{F}_{p} \, (k \in \mathbb{N})$ such that $\lim_{k \to \infty}\abs{g - g_{k}}_{\mathcal{F}_{p}} = 0$, we have
	\begin{equation}\label{KS-local-conv}
		\lim_{k \to \infty}\varlimsup_{\delta \downarrow 0}E_{p, \delta}(g_{k}; K_{v}, K_{w}) = \varlimsup_{\delta \downarrow 0}E_{p, \delta}(g; K_{v}, K_{w}).
	\end{equation}
	This is immediate since
	\begin{align*}
		\abs{\varlimsup_{\delta \downarrow 0}E_{p, \delta}(g; K_{v}, K_{w})^{1/p} - \varlimsup_{\delta \downarrow 0}E_{p, \delta}(g_{k}; K_{v}, K_{w})^{1/p}}
		&\le \varlimsup_{\delta \downarrow 0}E_{p, \delta}(g - g_{n}; K_{v}, K_{w})^{1/p} \\
		&\le \varlimsup_{\delta \downarrow 0}E_{p, \delta}(g - g_{k})^{1/p}
		\lesssim \abs{g - g_{k}}_{\mathcal{F}_{p}}.
	\end{align*}

	The rest of the proof is a standard argument using Mazur's lemma (Lemma \ref{lem.Mazur}).
	Let $f_{k} \in \mathcal{F}_{p} \, (k \in \mathbb{N})$ be a sequence converging in $L^{p}$ to some $f \in \mathcal{F}_{p}$.
	By extracting a subsequence $\{ f_{k'} \}_{k'}$ if necessary, we can assume that
	\[
	\lim_{k' \to \infty}\varlimsup_{\delta \downarrow 0}E_{p, \delta}(f_{k'}; K_{v}, K_{w}) = \varliminf_{k \to \infty}\varlimsup_{\delta \downarrow 0}E_{p, \delta}(f_{k}; K_{v}, K_{w}).
	\]
	Since $\mathcal{F}_{p}$ is reflexive, there exists a subsequence, which is also denoted by $\{ f_{k'} \}_{k'}$, such that $f_{n_{k}}$ converges weakly in $\mathcal{F}_{p}$ to $f$.
	By Mazur's lemma, there exist finite subset $I_{j} \subseteq [j, \infty) \cap \mathbb{N} \, (j \in \mathbb{N})$ and
	\[
	\biggl\{ \lambda_{k'}^{(j)} \biggm| \text{$\lambda_{k'}^{(j)} \ge 0$ for $k' \in I_{j}$ and $\sum_{k' \in I_{j}}\lambda_{k'}^{(j)} = 1$} \biggr\}_{j \in \mathbb{N}}
	\]
	such that $g_{j} \coloneqq \sum_{k' \in I_{j}}\lambda_{k'}^{(j)}f_{k'} \in \mathcal{F}_{p} \, (j \in \mathbb{N})$ satisfies $\norm{f - g_{j}}_{\mathcal{F}_{p}} \to 0$ as $j \to \infty$.
	By the triangle inequality of $L^{p}$-norm, we see that
	\begin{align*}
		\varlimsup_{\delta \downarrow 0}E_{p, \delta}(g_{j}; K_{v}, K_{w})^{1/p}
		\le \sum_{k' \in I_{j}}\lambda_{k'}^{(j)}\varlimsup_{\delta \downarrow 0}E_{p, \delta}(f_{k'}; K_{v}, K_{w})^{1/p}.
	\end{align*}
	Letting $j \to \infty$ and using \eqref{KS-local-conv}, we obtain
	\[
	\varlimsup_{\delta \downarrow 0}E_{p, \delta}(f; K_{v}, K_{w})^{1/p}
		\le \lim_{k' \to \infty}\varlimsup_{\delta \downarrow 0}E_{p, \delta}(f_{k'}; K_{v}, K_{w})^{1/p},
	\]
	proving our assertion.
\end{proof}

Now we can estimate the unfolding map $\Xi_{v, w}(f)$ for $\{ v, w \} \in E_{n}^{\#}$.
\begin{lem}\label{lem.KS-folding2}
	Let $n \in \mathbb{N}$, $\{ v, w \} \in E_{n}^{\#}$ and $f \in \mathcal{F}_{p}^{S}$.
	If $\restr{f}{\ell_{v, w}} = 0$, then for any $z \in W_{n}$ with $K_{z} \subseteq G_{w,v}$ we have
	\[
	\lim_{\delta \downarrow 0}E_{p, \delta}\bigl(\Xi_{v, w}(f); K_{v}, K_{z}\bigr) = 0.
	\]
\end{lem}
\begin{proof}
	This lemma corresponds to a weaker version of \cite[Lemma 5.6]{Hin13}.
	Let $n \in \mathbb{N}$ and $\{ v, w \} \in E_{n}^{\#}$.
	Let $f \in \mathcal{F}_{p}^{S}$ satisfy $\restr{f}{\ell_{v, w}} = 0$.
	Note that $\Xi_{v}(f) \in \mathcal{F}_{p} \cap \contfunc(K)$ by Lemma \ref{lem.KS-folding1}.
	Applying Lemma \ref{lem.approx-folding} to $\Xi_{v}(f)$, we get a sequence $f_{k} \in \mathcal{F}_{p} \cap \mathcal{C}(K) \, (k \in \mathbb{N})$ such that $\supp[f_{k}] \subseteq K \setminus \ell_{v, w}$ and $f_{k}$ converges in $\mathcal{F}_{p}$ to $\Xi_{v}(f)$.
	Set $g_{k} \coloneqq \Xi_{v}(f_{k})$ and $h_{k} \coloneqq \Xi_{v, w}(f_{k})$ for $k \ge 1$.
	For $\delta < \dist_{d}(H_{v, w}, \supp[g_{k}])$, we see that
	\begin{align*}
		E_{p, \delta}(h_{k}) = E_{p, \delta}(g_{k}; K \cap G_{v, w}) \le E_{p, \delta}(g_{k}).
	\end{align*}
	Combining this with Lemma \ref{lem.KS-folding1} and \eqref{LB-folding}, we obtain
	\[
	\abs{h_{k}}_{\mathcal{F}_{p}}^{p} \lesssim \varlimsup_{\delta \downarrow 0}E_{p, \delta}(h_{k}) \le \varlimsup_{\delta \downarrow 0}E_{p, \delta}(g_{k}) \le C(\#W_{n})\rho(p)^{n}\abs{F_{v}^{\ast}f_{k}}_{\mathcal{F}_{p}}^{p}.
	\]
	By \eqref{renormalize-upper}, there exists a constant $C' > 0$ without depending on $n, k$ such that
	\[
	\abs{h_{k}}_{\mathcal{F}_{p}}^{p} \le C'(\#W_{n})\abs{f_{k}}_{\mathcal{F}_{p}}^{p}.
	\]
	In particular, for each fixed $n \in \mathbb{N}$, $\{ h_{k} \}_{k \ge 1}$ is bounded in $\mathcal{F}_{p}$.
	Note that $h_{k}$ converges in $L^{p}(K, \measure)$ to $\Xi_{v, w}(f)$ as $k \to \infty$.
	Hence, by Lemma \ref{lem.lsc-folding}, for any $z \in V_{n}$,
	\[
	\varlimsup_{\delta \downarrow 0}E_{p, \delta}(\Xi_{v, w}(f); K_{v}, K_{z}) \le \varliminf_{k \to \infty}\varlimsup_{\delta \downarrow 0}E_{p, \delta}(h_{k}; K_{v}, K_{z}).
	\]
	If $\delta < \dist_{d}(H_{v, w}, \supp[g_{k}])$ and $K_{z} \subseteq G_{w, v}$, then we have $E_{p, \delta}(h_{k}; K_{v}, K_{z}) = 0$.
	Therefore, we obtain $\varlimsup_{\delta \downarrow 0}E_{p, \delta}(\Xi_{v, w}(f); K_{v}, K_{z}) = 0$.
	This completes the proof.
\end{proof}

Finally, we can prove the  bound \eqref{renormalize-lower} and complete the proof of Proposition \ref{prop.renormalization}.
\begin{proof}[Proof of Proposition \ref{prop.renormalization}]
	The estimate \eqref{renormalize-upper} is already proved.
	In particular, $\mathcal{F}_{p} = \{ f \in \mathcal{F}_{p} \mid f \circ F_i \in \mathcal{F}_{p} \, \forall i \in S \}$.
	To prove \eqref{renormalize-lower}, let $f \in \mathcal{F}_{p}^{S}$.
	Let us fix $n \in \mathbb{N}$.
	Then, for small enough $\delta > 0$, we observe that
	\begin{equation}\label{KS-divide}
		E_{p, \delta}(f) = \sum_{w \in W_{n}}E_{p, \delta}(f; K_{w}) + \sum_{\{ v, w \} \in E_{n}}E_{p, \delta}(f; K_{v}, K_{w}).
	\end{equation}
	We obtain upper bounds for $E_{p, \delta}(f; K_{v}, K_{w})$ by dividing into the following two cases.

	\noindent
	\textbf{Case 1:}\,\underline{$\{ v, w \} \in E_{n}^{\#}$}; Define $h_{i} \in \contfunc(K) \, (i = 0, 1)$ by
	\[
	h_{0} \coloneqq \Xi_{v}(f) \quad \text{and} \quad h_{1} \coloneqq \Xi_{w, v}(f - h_{0}).
	\]
	It is easy to see that $\restr{f}{K_{v} \cup K_{w}} = \restr{(h_{0} + h_{1})}{K_{v} \cup K_{w}}$ and that $\restr{(f - h_{0})}{\ell_{v, w}} = 0$.
	Since $h_{0} \in \mathcal{F}_{p} \cap \contfunc(K)$ by Lemma \ref{lem.KS-folding1} and $f \in \mathcal{F}_{p}^{S}$, it is also immediate that $f - h_{0} \in \mathcal{F}_{p}^{S}$.
	Hence, by Lemmas \ref{lem.KS-folding1} and \ref{lem.KS-folding2},
	\begin{align*}
		\varlimsup_{\delta \downarrow 0}E_{p, \delta}(f; K_{v}, K_{w})
		&\le 2^{p - 1}\varlimsup_{\delta \downarrow 0}\bigl(E_{p, \delta}(h_{0}; K_{v}, K_{w}) + E_{p, \delta}(h_{1}; K_{v}, K_{w})\bigr) \\
		&\le 2^{p - 1}\rho(p)^{n}\varlimsup_{\delta \downarrow 0}E_{p, \delta}(F_{v}^{\ast}f).
	\end{align*}

	\noindent
	\textbf{Case 2:}\,\underline{$\{ v, w \} \in E_{n} \setminus E_{n}^{\#}$}; Clearly, there exists $z(i) \in V_{n} \, (i = 1, 2, 3)$ such that $\{ z(1), z(3) \} = \{ v, w \}$, $\{ z(i), z(i + 1) \} \in E_{n}^{\#}$ for $i = 1, 2$ and $K_{z(i)} \not\subseteq G_{z(j), z(2)}$ for $\{ i, j \} = \{ 1, 3 \}$.
	Now we define $h_{i} \in \contfunc(K) \, (i = 0, 1, 2)$ by
	\[
	h_{0} \coloneqq \Xi_{z(2)}(f), \quad h_{1} \coloneqq \Xi_{z(1), z(2)}(f - h_{0}) \quad \text{and} \quad h_{2} \coloneqq \Xi_{z(3), z(2)}(f - h_{0}).
	\]
	Then we have $\restr{f}{\cup_{i = 1}^{3}K_{z(i)}} = \restr{(h_{0} + h_{1} + h_{2})}{\cup_{i = 1}^{3}K_{z(i)}}$ and $\restr{(f - h_{0})}{\ell_{z(1), z(2)} \cup \ell_{z(2), z(3)}} = 0$.
	Hence, by Lemmas \ref{lem.KS-folding1} and \ref{lem.KS-folding2},
	\begin{align*}
		\varlimsup_{\delta \downarrow 0}E_{p, \delta}(f; K_{v}, K_{w})
		&\le 3^{p - 1}\varlimsup_{\delta \downarrow 0}\sum_{j = 0}^{2}E_{p, \delta}(h_{j}; K_{v}, K_{w}) \\
		&\le 3^{p - 1}\rho(p)^{n}\varlimsup_{\delta \downarrow 0}E_{p, \delta}\bigl(F_{z(2)}^{\ast}f\bigr).
	\end{align*}

	From \eqref{KS-divide} and above observations, we obtain
	\[
	\varlimsup_{\delta \downarrow 0}E_{p, \delta}(f) \le (1 + L_{\ast}^{2})\rho(p)^{n}\sum_{v \in W_{n}}\varlimsup_{\delta \downarrow 0}E_{p, \delta}(F_{v}^{\ast}f),
	\]
	which together with \eqref{LB-folding} proves \eqref{renormalize-lower} (recall that $L_{\ast} = \sup_{n \in \mathbb{N}}\deg(\mathbb{G}_{n}) < \infty$ as defined in Definition \ref{defn:unifdeg}).
	Note that \eqref{renormalize-lower} implies $\mathcal{F}_{p}^{S} = \mathcal{F}_{p} \cap \contfunc(K)$.
	We complete the proof.
\end{proof}

\begin{proof}[Proof of Theorem \ref{thm.assum-PSC}]
	\ref{it:PSCassum1} and \ref{it:PSCassum3} are proved in Proposition \hyperref[it:PSCgeom.1]{\ref{prop.PSC-geom}} and \ref{prop.renormalization} respectively.
	\ref{it:PSCassum2} follows from Propositions \hyperref[it:PSCanalysis1]{\ref{prop.PSC-analysis}},  \ref{prop.UPI} and \ref{prop.UCF}.
\end{proof}

We are now ready to prove the first four main results stated in the introduction (Theorems \ref{t:main-sob-psc}, \ref{t:main-em}, \ref{thm.LB.psc} and \ref{t:main-cap_PI}).
\begin{proof}[Proof of Theorem \ref{t:main-sob-psc}]
	Theorem \ref{thm.assum-PSC} implies Assumptions \ref{a:reg} and \ref{a:reg-ss}. Therefore by Theorem \ref{t:Fp} we obtain the conclusions \ref{sob-banach} and \ref{sob-reg}.
	The existence of self-similar energy with the desired properties follows from Corollary \ref{cor.ss-energy} except for properties \ref{sob-uc}, \ref{sob-lip}, \ref{sob-sg}, \ref{sob-sl} and \ref{sob-sym}.

	The properties \ref{sob-uc}, \ref{sob-lip} and \ref{sob-sym} are shown by choosing suitable closed invariant sub-cones.
	Indeed, we can show that
	\begin{equation*}
		\mathcal{U}_{p}^{\textup{Cla}} \coloneqq \left\{ \mathcal{E} \;\middle|\;
    	\begin{array}{c}
    	\text{$\mathcal{E} \colon \mathcal{F}_{p} \to [0,\infty)$, $\mathcal{E}^{1/p}$ is a semi-norm and for any $f,g \in \mathcal{F}_{p}$,} \\
    	\text{$\mathcal{E}(f + g)^{1/(p - 1)} + \mathcal{E}(f - g)^{1/(p - 1)} \le 2\bigl(\mathcal{E}(f) + \mathcal{E}(g)\bigr)^{1/(p - 1)}$ if $p \in (1,2]$,} \\
    	\text{$\mathcal{E}(f + g) + \mathcal{E}(f - g) \le 2\bigl(\mathcal{E}(f)^{1/(p - 1)} + \mathcal{E}(g)^{1/(p - 1)}\bigr)^{p-1}$ if $p \in (2,\infty)$}
    	\end{array}
    	\right\},
	\end{equation*}
	\begin{equation*}
		\mathcal{U}_{p}^{\textup{Lip}} \coloneqq \left\{ \mathcal{E} \;\middle|\;
    	\begin{array}{c}
    	\text{$\mathcal{E} \colon \mathcal{F}_{p} \to [0,\infty)$, $\mathcal{E}^{1/p}$ is a semi-norm and } \\
    	\text{$\mathcal{E}(\varphi \circ f) \le \mathcal{E}(f)$ for any $f \in \mathcal{F}_{p}$ and $1$-Lipschitz function $\varphi \in \contfunc(K)$}
    	\end{array}
    	\right\},
	\end{equation*}
	and
	\begin{equation*}
		\mathcal{U}_{p}^{\textup{sym}} \coloneqq \left\{ \mathcal{E} \;\middle|\;
    	\begin{array}{c}
    	\text{$\mathcal{E} \colon \mathcal{F}_{p} \to [0,\infty)$, $\mathcal{E}^{1/p}$ is a semi-norm and} \\
    	\text{$\mathcal{E}(f \circ \Phi) = \mathcal{E}(f)$ for any $f \in \mathcal{F}_{p}$ and $\Phi \in D_{4}$}
    	\end{array}
    	\right\}
	\end{equation*}
	are closed invariant sub-cones.
	Here we only prove $\mathcal{S}_{\rho}\bigl(\mathcal{U}_{p}^{\textup{sym}}\bigr) \subseteq \mathcal{U}_{p}^{\textup{sym}}$.
    Let $\Phi \in D_{4}$ and $f \in \mathcal{F}_{p}$.
    Note that $f \circ \Phi \in \mathcal{F}_{p}$ since $\mathcal{E}_{p}^{\mathbb{G}_{n}}(M_{n}(f \circ \Phi)) = \mathcal{E}_{p}^{\mathbb{G}_{n}}(M_{n}f)$.
	For any $\mathsf{E} \in \mathcal{U}_{p}^{\textup{sym}}$, by virtue of Proposition \hyperref[it:PSCgeom.1]{\ref{prop.PSC-geom}}\ref{it:PSCgeom.6},
	\begin{align*}
		\mathcal{S}_{\rho}\mathsf{E}(f \circ \Phi)
		&= \rho(p)\sum_{i \in S}\mathsf{E}(f \circ \Phi \circ F_{i})
		= \rho(p)\sum_{i \in S}\mathsf{E}\bigl(f \circ F_{\tau_{\Phi}(w)} \circ U_{\Phi, w}\bigr) \\
		&= \rho(p)\sum_{i \in S}\mathsf{E}\bigl(f \circ F_{\tau_{\Phi}(w)}\bigr)
		= \rho(p)\sum_{j \in S}\mathsf{E}(f \circ F_{j})
		= \mathcal{S}_{\rho}\mathsf{E}(f),
	\end{align*}
	which shows $\mathcal{S}_{\rho}\mathsf{E} \in \mathcal{U}_{p}^{\textup{sym}}$.

	Since $\mathcal{E}_{p}^{\Gamma} \in \mathcal{U}_{p}^{\textup{Cla}} \cap \mathcal{U}_{p}^{\textup{Lip}} \cap \mathcal{U}_{p}^{\textup{sym}}$ by Theorem \ref{thm.Epgamma}, we have $\mathcal{E}_{p} \in \mathcal{U}_{p}^{\textup{Cla}} \cap \mathcal{U}_{p}^{\textup{Lip}} \cap \mathcal{U}_{p}^{\textup{sym}}$ (Theorem \ref{thm.fix}\ref{it:fixed.inv}).

	\ref{sob-sg} (spectral gap) This follows from applying Lemma \ref{lem.PI-like} with $r= 2 \diam(K,d)$.

	\ref{sob-sl} (strong locality) This is a consequence of the self-similarity \ref{sob-ss}.
    Let $f, g, h \in \mathcal{W}^{p}$ and suppose that $\supp[f] \cap \supp[g - a\indicator{K}] = \emptyset$ for some $a \in \mathbb{R}$.
    Consider $n \in \mathbb{N}$ large enough so that
    \[
    \max_{w \in T_{n}^{(r_{\ast})}}\diam(K_{w},d) < \dist(\supp[f], \supp[g - a\indicator{K}]),
    \]
    and define $A_{n,f},A_{n,g} \subseteq W_{n}$ by
    \[
    A_{n,f} \coloneqq \bigl\{ w \in W_{n} \bigm| K_{w} \cap \supp[f] \neq \emptyset \bigr\}, \quad A_{n,g} \coloneqq \bigl\{ w \in W_{n} \bigm| K_{w} \cap \supp[g - a\indicator{K}] \neq \emptyset \bigr\}.
    \]
    Then, by the self-similarity of $\mathcal{E}_{p}$ and $\mathcal{E}_{p}^{-1}(0) = \mathbb{R}\indicator{K}$, we see that
    \begin{align*}
        \mathcal{E}_{p}(f + g + h)
        &= \sum_{w \in W_{n}}\sigma(p)^{j(w)}\mathcal{E}_{p}(f \circ F_{w} + g \circ F_{w} + h \circ F_{w}) \\
        &= \sum_{w \in A_{n,f}}\sigma(p)^{j(w)}\mathcal{E}_{p}(f \circ F_{w} + h \circ F_{w}) + \sum_{w \in A_{n,g}}\sigma(p)^{j(w)}\mathcal{E}_{p}(g \circ F_{w} + h \circ F_{w}) \\
        &\quad+ \sum_{w \in W_{n} \setminus (A_{n,f} \cup A_{n,g})}\sigma(p)^{j(w)}\mathcal{E}_{p}(h \circ F_{w}) \\
        &= \sum_{w \in W_{n}}\sigma(p)^{j(w)}\mathcal{E}_{p}(f \circ F_{w} + h \circ F_{w}) + \sum_{w \in W_{n}}\sigma(p)^{j(w)}\mathcal{E}_{p}(g \circ F_{w} + h \circ F_{w}) \\
        &\quad -\left(\sum_{w \in W_{n} \setminus A_{n,f}}\sigma(p)^{j(w)}\mathcal{E}_{p}(h \circ F_{w}) + \sum_{w \in W_{n} \setminus A_{n,g}}\sigma(p)^{j(w)}\mathcal{E}_{p}(h \circ F_{w})\right) \\
        &\quad + \sum_{w \in W_{n} \setminus (A_{n,f} \cup A_{n,g})}\sigma(p)^{j(w)}\mathcal{E}_{p}(h \circ F_{w}) \\
        &= \mathcal{E}_{p}(f + h) + \mathcal{E}_{p}(g + h) - \mathcal{E}_{p}(h),
    \end{align*}
    which is our assertion.
\end{proof}

\begin{proof}[Proof of Theorem \ref{t:main-em} except for \eqref{e:em.exact-cell}]
The existence of energy measures follows from the construction described after Assumption \ref{assum.ss}, which in turn follows from Theorem \ref{t:main-sob-psc}.
Properties \ref{em-tri}, \ref{em-lip}, \ref{em-ss} follow from Propositions \ref{prop.em-norm}, \ref{prop.em-Markov} and \ref{prop.peneSS} respectively.
The assertions in \ref{em-chain} follow from Theorem \ref{thm.em-chain} and Corollary \ref{cor.em-slocal}.
It remains to prove \ref{em-cell} and \ref{em-sym}.

\ref{em-cell} The property $\Gamma_p \langle f \rangle (K)=\sE_p(f)$ is immediate from the definition of $\Gamma_p\langle f \rangle$. The equality \eqref{e:em.exact-cell} will be shown later. 

\ref{em-sym} Let $A \in \mathcal{B}(K)$ be a closed set, and let $f \in \mathcal{F}_{p}$, $\Phi \in D_{4}$.
    For each $n \in \mathbb{N}$, define
    \[
    C_{n} \coloneqq \Bigl\{ w \in W_{n} \Bigm| \Sigma_{w} \cap \chi^{-1}(A) \neq \emptyset \Bigr\} \quad \text{and} \quad C_{n, \Phi} \coloneqq \Bigl\{ w \in W_{n} \Bigm| \Sigma_{w} \cap \chi^{-1}(\Phi(A)) \neq \emptyset \Bigr\}.
    \]
    Also, define
    \[
    \Sigma_{C_{n}} \coloneqq \bigl\{ \omega \in \Sigma \bigm| [\omega]_{n} \in C_{n} \bigr\}
    \quad \text{and} \quad \Sigma_{C_{n, \Phi}} \coloneqq \bigl\{ \omega \in \Sigma \bigm| [\omega]_{n} \in C_{n, \Phi} \bigr\}.
    \]
    Then $\tau_{\Phi}|_{C_{n}}$ gives a bijection between $C_{n}$ and $C_{n, \Phi}$.

    By Proposition \hyperref[it:PSCgeom.1]{\ref{prop.PSC-geom}}\ref{it:PSCgeom.6} and $\mathcal{E}_{p}(f \circ \Phi) = \mathcal{E}_{p}(f)$, we have
    \begin{align*}
    	\mathfrak{m}_{p}\langle f \circ \Phi \rangle\bigl(\Sigma_{C_{n}}\bigr)
    	&= \rho(p)^{n}\sum_{w \in C_{n}}\mathcal{E}_{p}\bigl(f \circ \Phi \circ F_{w}\bigr)
    	= \rho(p)^{n}\sum_{w \in C_{n}}\mathcal{E}_{p}\bigl(f \circ F_{\tau_{\Phi}(w)} \circ U_{\Phi, w}\bigr) \\
    	&= \rho(p)^{n}\sum_{w \in C_{n}}\mathcal{E}_{p}\bigl(f \circ F_{\tau_{\Phi}(w)}\bigr)
    	= \rho(p)^{n}\sum_{v \in C_{n, \Phi}}\mathcal{E}_{p}(f \circ F_{v})
    	= \mathfrak{m}_{p}\langle f \rangle\bigl(\Sigma_{C_{n, \Phi}}\bigr).
    \end{align*}
    Letting $n \to \infty$, we obtain $\Gamma_{p}\langle f \circ \Phi \rangle(A) = \Phi_{\ast}\Gamma_{p}\langle f \rangle(A)$ since $\bigcap_{n \in \mathbb{N}}\Sigma_{C_{n}} = \chi^{-1}(A)$ and $\bigcap_{n \in \mathbb{N}}\Sigma_{C_{n, \Phi}} = \chi^{-1}(\Phi^{-1}(A))$ as seen in the proof of Proposition \ref{prop.em-norm}.
    Hence we obtain $\Phi_{\ast}\Gamma_{p}\langle f \rangle(A) = \Gamma_{p}\langle f \circ \Phi \rangle(A)$ for any closed set $A$ of $K$.

    Recall that both measures $\Gamma_{p}\langle f \circ \Phi \rangle$ and $\Phi_{\ast}\Gamma_{p}\langle f \rangle$ are Borel-regular.
    In particular, for any $A \in \mathcal{B}(K)$, there exists a sequence $\{ A_{n} \}_{n \in \mathbb{N}}$ of closed subsets of $K$ such that $A_{n} \subseteq A$ and $\Gamma_{p}\langle f \circ \Phi \rangle(A_{n}) \to \Gamma_{p}\langle f \circ \Phi \rangle(A)$ as $n \to \infty$.
    For any $n \in \mathbb{N}$,
    \[
    \Gamma_{p}\langle f \circ \Phi \rangle(A_{n}) = \Phi_{\ast}\Gamma_{p}\langle f \rangle(A_{n}) \le \Phi_{\ast}\Gamma_{p}\langle f \rangle(A).
    \]
    Hence we have $\Gamma_{p}\langle f \circ \Phi \rangle(A) \le \Phi_{\ast}\Gamma_{p}\langle f \rangle(A)$.
    The converse inequality can be shown in a similar way.
\end{proof}

\begin{proof}[Proof of Theorem \ref{thm.LB.psc}]
As mentioned earlier, Assumption \ref{a:reg} follows from Theorem \ref{thm.assum-PSC}. The desired conclusion then follows from any application of Theorem \ref{thm.LB}.
\end{proof}

\begin{prop}\label{prop:PI.PSC}
	For any $p \in (1,\infty)$, \hyperref[PI]{\textup{PI$_{p}(\pwalk)$}} holds. 
\end{prop}
\begin{proof}
	First we define  
	\[
	\sigma(\mathcal{E}_p) \coloneqq \inf \biggl\{ \mathcal{E}_{p}(f) \biggm| f \in \mathcal{F}_{p} \cap \contfunc(K), \restr{f}{\ell_{\textup{L}}} \equiv 0, \int_{K} f\,d\measure = 1 \biggr\}, 
	\]
	and show that $\sigma(\mathcal{E}_p)  \in (0,\infty)$. 
	By Theorem \ref{t:main-sob-psc}\ref{sob-reg}, \ref{sob-lip}, we have 
	\[
	\mathcal{F}_{p, \textup{ave}}(\textup{L})
	\coloneqq \biggl\{ f \in \mathcal{F}_{p} \cap \contfunc(K) \biggm| \restr{f}{\ell_{\textup{L}}} \equiv 0, \int_{K}f\,d\measure = 1 \biggr\}
	\neq \emptyset,
	\]
	and hence $\sigma(\mathcal{E}_p) < \infty$.
	Let $\closure{\mathcal{F}}_{p, \textup{ave}}(\textup{L})$ denote the closure of $\mathcal{F}_{p, \textup{ave}}(\textup{L})$ with respect to the norm $\norm{\,\cdot\,}_{\mathcal{F}_{p}}$. 
	Clearly,  $\int_{K}v\,d\measure = 1$ for any $v \in \closure{\mathcal{F}}_{p, \textup{ave}}(\textup{L})$. 
	For each $n \in \mathbb{N}$, let $v_{n} \in \mathcal{F}_{p, \textup{ave}}(\textup{L})$ satisfy $\mathcal{E}_{p}(v_{n}) \le \inf_{f \in \mathcal{F}_{p, \textup{ave}}(\textup{L})}\mathcal{E}_{p}(f) + n^{-1}$, and define $h_{n} \in \contfunc(K)$ by
	\[
	h_{n} \coloneqq \sum_{i \in \{ 3, 4, 5  \}}(F_{i})_{\ast}v_{n}.
	\]
	By Theorem \ref{t:main-sob-psc}\ref{sob-ss}, we have $h_{n} \in \mathcal{F}_{p} \cap \contfunc(K)$ and
	\[
	\mathcal{E}_{p}(h_{n}) = 3\rho(p)\mathcal{E}_{p}(v_{n}) \le 3\rho(p)\bigl(\sigma(\mathcal{E}_{p}) + n^{-1}\bigr).
	\]
	We also have from Theorem \ref{t:main-sob-psc}\ref{sob-sg} that 
	\[
	\norm{h_{n}}_{L^{p}(\measure)}^{p} \lesssim \norm{h_{n} - (h_{n})_{K}}_{L^{p}(\measure)}^{p} + \measure(K) \le C\mathcal{E}_{p}(h_{n}) + \measure(K).
	\]
	From these estimates, $\{ h_{n} \}_{n \in \mathbb{N}}$ turns out to be a bounded sequence in $\mathcal{F}_{p}$ and hence, by Theorem \ref{t:main-sob-psc}\ref{sob-banach}, we get a subsequence $\{ h_{n_{k}} \}_{k \in \mathbb{N}}$ and $h_{\infty} \in \mathcal{F}_{p}$ so that $h_{n_{k}}$ converges weakly to $h_{\infty}$ in $\mathcal{F}_{p}$. 
	Mazur's lemma (Lemma \ref{lem.Mazur}) yields a sequence $\widetilde{h}_{j} \in \textrm{conv}\{ h_{n_{k}} \}_{k \ge j} \, (j \in \mathbb{N})$ such that $\widetilde{h}_{j}$ converges to $h_{\infty}$ in $\mathcal{F}_{p}$, and we then have
	\[
	\mathcal{E}_{p}(h_{\infty}) = \lim_{j \to \infty}\mathcal{E}_{p}(\widetilde{h}_{j})
	\le \limsup_{j \to \infty}3\rho(p)\bigl(\sigma(\mathcal{E}_{p}) + j^{-1}\bigr) = 3\rho(p)\sigma(\mathcal{E}_{p}),
	\]
	and $\int_{K}h_{\infty}\,d\measure = \lim_{j \to \infty}\int_{K}\widetilde{h}_{j}\,d\measure = 3/8$.
	If $\sigma(\mathcal{E}_{p}) = 0$, then $h_{\infty}$ should be a constant function by Theorem \ref{t:main-sob-psc}\ref{sob-sg}.
	Since $\widetilde{h}_{j} = 0$ on $\bigcup_{i \in S \setminus \{ 3, 4, 5 \}}K_{i}$, we have $h_{\infty} \equiv 0$, which contradicts $\int_{K}h_{\infty}\,d\measure = 3/8$.
	Therefore $\sigma(\mathcal{E}_{p}) > 0$. 
		
	Next we show that for any $f \in \mathcal{F}_{p}$, 
	\begin{equation}\label{e:PSC.avediffcell}
		\abs{f_{K_v}-f_{K_w}}^p \le 2^{p/(p-1)}\sigma(\mathcal{E}_{p})^{-1}\rho(p)^{\,-n}\left[\Gamma_{p}\langle f \rangle(K_v) + \Gamma_{p}\langle f \rangle(K_w) \right].
	\end{equation}
	By Theorems \ref{t:main-sob-psc}\ref{sob-reg} and \ref{t:main-em}\ref{em-tri}, it suffices to assume that $f \in \mathcal{F}_{p} \cap \contfunc(K)$.
	By replacing $f$ with $f \circ \Phi$ for some $\Phi \in D_{4}$, we may assume that $F_v^{-1}(K_v \cap K_w)= \ell_{\textup{L}}, F_w^{-1}(K_v \cap K_w) = \ell_{\textup{R}}$.
	Without loss of generality, we assume that $f_{K_v}- f_{K_w} \neq 0$.
	The function $h \coloneqq f \circ F_{v} -   (f \circ F_{w}) \circ S_{2} \in \mathcal{F}_{p} \cap \contfunc(K)$ satisfies $\int_{K} h \,d\measure = f_{K_v}- f_{K_w}, \restr{h}{\ell_{\textup{L}}} \equiv 0$ and
	\begin{equation}\label{e:PSC.nb1}
		\mathcal{E}_{p}(h)
		\le  \bigl( \mathcal{E}_{p}(f \circ F_{v})^{1/p} + \mathcal{E}_p(f \circ F_{w})^{1/p} \bigr)^p
		\le 2^{p/(p-1)}\bigl( \mathcal{E}_p(f \circ F_{v})  + \mathcal{E}_p(f \circ F_{w})\bigr).
	\end{equation}
	By Lemma \ref{lem.em-cell}, we know that $\rho(p)^{\,n}\mathcal{E}_{p}(f \circ F_{z}) \le \Gamma_{p}\langle f \rangle(K_{z})$ for any $z \in W_{n}$.
	Hence \eqref{e:PSC.nb1} yields the desired inequality. 
	
	Finally, we prove \hyperref[PI]{\textup{PI$_{p}(\pwalk)$}}. 
	For $r \in (0,\infty)$, let $n = n(r) \in \mathbb{Z}_{\ge 0}$ be the smallest non-negative integer such that $r \ge a_{\ast}^{-n}$.
	Set $W(x,r) \coloneqq W_{n}[B_{\metric}(x,r)]$ and $U(x,r) \coloneqq \bigcup_{w \in W(x,r)}K_{w}$. 
	Since $\diam(K_{w},d) = a_{\ast}^{-n}$ for any $w \in W(x,r)$, we easily see that $U(x,r) \subseteq B_{\metric}(x,A_{\mathrm{P}}r)$, where $A_{\mathrm{P}} \coloneqq 2$.   
	Also, there exists $N_{1}$ which is independent of $x,r$ such that $\#W(x,r) \le N_{1}$ by the metric doubling property of $(K,d)$ (Proposition \hyperref[it:PSCgeom.1]{\ref{prop.PSC-geom}}\ref{it:PSCgeom.2}, \ref{it:PSCgeom.3}). 
	It is easy to see that 
    \begin{align}\label{e:PSC.var}
        &\int_{U(x,r)}\abs{f(y) - f_{U(x,r)}}^{p}\,\measure(dy) \nonumber \\
        &\le 2^{p - 1}\sum_{w \in W(x,r)}\measure(K_{w})\left(\fint_{K_{w}}\abs{f(y) - f_{K_{w}}}^{p}\,\measure(dy) + \abs{f_{K_{w}} - f_{U(x,r)}}^{p}\right).  
    \end{align}
    By Theorem \ref{t:main-sob-psc}\ref{sob-sg} for $f \circ F_{w}$ and Lemma \ref{lem.em-cell}, we have 
    \begin{equation}\label{e:PSC.cellPI}
    	\fint_{K_{w}}\abs{f - f_{K_{w}}}^{p}\,d\measure \le C\mathcal{E}_{p}(f \circ F_{w}) \le C\rho(p)^{-n}\Gamma_{p}\langle f \rangle(K_{w}), 
    \end{equation}
    where $C$ is the constant in Theorem \ref{t:main-sob-psc}\ref{sob-sg}. 
    Note that, by Proposition \hyperref[it:PSCgeom.1]{\ref{prop.PSC-geom}}\ref{it:PSCgeom.1},  
    \begin{align*}
    	f_{K_{w}} - f_{U(x,r)}
    	= \frac{1}{\measure(U(x,r))}\sum_{v \in W(x,r)}(f_{K_{w}} - f_{K_{v}})\measure(K_{v}). 
    \end{align*}
    Hence, by choosing $w' \in W(x,r) \setminus \{ w \}$ so that $f_{K_{w}} - f_{K_{w'}} = \max_{v \in W(x,r)}\abs{f_{K_{w}} - f_{K_{v}}}$, 
    \[
    \abs{f_{K_{w}} - f_{U(x,r)}} \le \abs{f_{K_{w}} - f_{K_{w'}}},  
    \]
    which together with H\"{o}lder's inequality and \eqref{e:PSC.avediffcell} yields that 
    \begin{equation}\label{e:PSC.avediff.2}
        \abs{f_{K_{w}} - f_{U(x,r)}}^{p}
        \le 2^{p/(p-1) + 1}N_{1}^{p - 1}\sigma(\mathcal{E}_{p})^{-1}\rho(p)^{-n}\sum_{v \in W(x,r)}\Gamma_{p}\langle f \rangle(K_{v}). 
    \end{equation}
    Since $\measure(K_{w})\rho(p)^{-n} = a_{\ast}^{-n\pwalk} \le r^{\pwalk}$ for any $w \in W_{n}$, we obtain \hyperref[PI]{\textup{PI$_{p}(\pwalk)$}} by \eqref{e:PSC.var}, \eqref{e:PSC.cellPI} and \eqref{e:PSC.avediff.2}. 
\end{proof}

\begin{proof}[Proof of Theorem \ref{t:main-cap_PI} and \eqref{e:em.exact-cell}]\label{prf.PI}
	The capacity upper bounds follow from Proposition \ref{p:cutoff} after verifying the assumptions using Theorem \ref{thm.assum-PSC}. 
	The Poincar\'{e} inequality \hyperref[PI]{\textup{PI$_{p}(\pwalk)$}} follows from Proposition \ref{prop:PI.PSC}. 
	Let us show \eqref{e:em.exact-cell}. 
	Note that for any $w \in S^n, n \in \bN, f \in \sF_p$, by the self-similarity \ref{em-ss},
	\begin{equation} \label{e:ss1}
		\Gamma_p \langle f \rangle (F_w(K)) = \rho(p)^n \sum_{u \in W_n: F_u(K) \cap F_w(K) \neq \emptyset} \Gamma_p \langle f \circ F_u \rangle(F_u(K) \cap F_w(K)).
	\end{equation}
	If $u \neq w$ and $u, v \in W_n$, then $F_u(K) \cap F_w(K) \subset \mathcal{V}_0$ which has energy measure zero by Remark \ref{r:embdy} (we need \hyperref[PI]{\textup{PI$_{p}(\pwalk)$}} here).
	Therefore $	\Gamma_p \langle f \rangle (F_w(K)) = \rho(p)^n  \Gamma_p \langle f \circ F_w \rangle(F_w(K))= \rho(p)^n \sE_p(f \circ F_w)$  for any $w \in W_n, n \in \bN, f \in \sF_p$. 
\end{proof}

\begin{rmk}\label{rmk.other}
	\begin{enumerate}[\rm(1)]
		\item\label{it:other.super} One can see that the arguments in this section work with minor modifications for the generalized Sierpi\'{n}ski carpets (in the sense of Barlow and Bass \cite{BB99}\footnote{Precisely, the \emph{nondiagonality} condition \cite[Hypotheses 2.1(H3)]{BB99}  has been strengthened later in \cite{BBKT}. For a detail explanation on this change, we refer the reader to \cite{Kaj10}.}) embedded in $\mathbb{R}^{2}$. 
	To go beyond the planar case, there are two obstacles that we have to overcome. 
	The first one is the restriction on $p$ such that $\hdim - \pwalk < 1$, which was always assumed in Sections \ref{sec.BCL}-\ref{sec.EHI}. 
	In higher-dimensional cases, there may exist $p \in (1,\infty)$ such that $\hdim - \pwalk \ge 1$, so we need completely new techniques/arguments to cover the all case $p \in (1,\infty)$ in Sections \ref{sec.BCL}-\ref{sec.EHI}. 
	This seems to be a very challenging problem (see also Problem \ref{prob.relax}). 
	Even for $p$ with $\hdim - \pwalk < 1$, one has to verify both \ref{cond.UPI} and \ref{cond.Ucap} (with the same exponent $\beta$) for the approximating graphs, which is the second obstacle. 
	As done in Proposition \hyperref[it:PSCanalysis1]{\ref{prop.PSC-analysis}} for the Sierpi\'nski carpet, one can obtain these conditions if one knows a good behavior of the discrete $p$-capacities as in \eqref{cc-mult} or \eqref{cc-comparable}. 
	Let us emphasize that the argument in \cite[Proof of Lemma 4.4]{BK13} (see also Theorem \ref{thm.super} for the statement of this result) showing \eqref{cc-mult} (and \eqref{cc-comparable}) uses the planarity in an essential way. 
	Fortunately, a recent paper by Anttila and Eriksson-Bique \cite{AE.super}, which appeared after the submission of our paper, shows the \emph{multiplicative inequality} corresponding to \eqref{cc-mult} for any $p \in (1,\infty)$ on a large class of fractal spaces having nice symmetries including the generalized Sierpi\'nski carpet. 
	A part of their results can be regarded as a generalization of \cite{BB90,Mc02} (the case $p = 2$), and it seems to be enough to obtain \ref{cond.UPI} and \ref{cond.Ucap} for the generalized Sierpi\'nski carpets beyond the planar case as long as $\hdim - \pwalk < 1$ is satisfied. 
	See also \cite[Section 9]{AE.super}. 
		\item Similarly to Remark \ref{rmk:Epgamma.GCP}, one can show that $(\mathcal{E}_{p},\mathcal{F}_{p})$ satisfies the generalized $p$-contraction property \cite{KS24+}. In particular, $\mathcal{E}_{p}$ is differentiable in the following sense: for any $f,g \in \mathcal{F}_{p}$, the function $\mathbb{R} \ni t \mapsto \mathcal{E}_{p}(f + tg) \in [0, \infty)$ is differentiable.
			The derivative $\left.\frac{d}{dt}\mathcal{E}_{p}(f + tg)\right|_{t = 0}$ can play the role of $p\int_{\mathbb{R}^{n}}\abs{\nabla f(x)}^{p - 2}\langle \nabla f(x), \nabla g(x) \rangle\,dx$ in the Euclidean setting, so such the differentiability allows us to introduce the notion of \emph{$p$-harmonic functions} (in a weak sense).
	\end{enumerate}
\end{rmk}

\subsection{Quasi-uniqueness of energies}\label{sec.q-unique}
In this subsection, we present an axiomatic approach to our Sobolev space.
We consider self-similar $p$-energies and the corresponding Sobolev space satisfying some natural conditions. Under these conditions, we  prove that the domain of self-similar $p$-energies is uniquely determined and the corresponding semi-norm is uniquely determined up to a bi-Lipschitz modification.
We first introduce a list of desired properties for the self-similar $p$-energies (and the associated energy measures) on PSC.
\begin{assumption}[Canonical self-similar $p$-energy]\label{a:main}
	Let $(K, \metric, \measure)$ be the Sierpi\'{n}ski carpet as given in Definition \hyperref[it:PSCbasic]{\ref{defn.PSC}}.
	Let $\mathscr{F}_{p}$ be a subspace of $L^p(K,\measure)$ and let $\mathscr{E}_{p}\colon \mathscr{F}_{p} \to [0, \infty)$ be a functional (called self-similar $p$-energy) that satisfy the following conditions.
	\begin{enumerate}[\rm(a)]
		\item\label{it:PSCaxiom.const} $\{f \in \mathscr{F}_{p} : \mathscr{E}_{p}(f) = 0 \} = \{f \in L^{p}(K, \measure) : \text{$f$ is constant $m$-almost everywhere} \}$. For any $a \in \bR$ and $f \in \mathscr{F}_{p}$, we have
		\[
		\mathscr{E}_{p}(f+a\indicator{K})= \mathscr{E}_{p}(f), \qq \mathscr{E}_{p}(af)= \abs{a}^{p}\mathscr{E}_{p}(f).
		\]
		\item\label{it:PSCaxiom.banach} The functional $f \mapsto \mathscr{E}_p(f)^{1/p}$ satisfies the triangle inequality on $\mathscr{F}_{p}$. In addition, the function $\norm{\,\cdot\,}_{\mathscr{F}_{p}}\colon \mathscr{F}_{p} \to [0,\infty)$ defined by $\norm{\,\cdot\,}_{\mathscr{F}_{p}}(f) \coloneqq \left( \norm{f}_{L^p(\measure)}^{p} + \mathscr{E}_{p}(f) \right)^{1/p}$ is a norm on $\mathscr{F}_{p}$ and $(\mathscr{F}_{p},\norm{\,\cdot\,}_{\mathscr{F}_{p}})$ is a uniformly convex Banach space.
		\item\label{it:PSCaxiom.reg} \textup{(Regularity)} The subspace $\mathscr{F}_{p} \cap \contfunc(K)$ is dense in $\contfunc(K)$ with respect to the uniform norm and is dense in the Banach space $(\mathscr{F}_{p},\norm{\,\cdot\,}_{\mathscr{F}_{p}})$.
		\item\label{it:PSCaxiom.sym} \textup{(Symmetry)} For every $\Phi \in D_{4}$ and for all $f \in \mathscr{F}_{p}$, we have $f \circ \Phi \in \mathscr{F}_{p}$ and $\mathscr{E}_{p}(f \circ \Phi)= \mathscr{E}_{p}(f)$.
		\item\label{it:PSCaxiom.ss} \textup{(Self-similarity)} There exists $\widetilde{\rho} \in (0,\infty)$ such that the following hold:
		For every $f \in \mathscr{F}_{p}, i \in S$, we have $f \circ F_i \in \mathscr{F}_{p}$, and
		\begin{equation*}
			\widetilde{\rho}\sum_{i \in S}\mathscr{E}_{p}(f \circ F_i) = \mathscr{E}_{p}(f).
		\end{equation*}
		Furthermore, $\mathscr{F}_{p} \cap \contfunc(K)= \set{f \in \contfunc(K) \mid f \circ F_i \in \mathscr{F}_{p} \mbox{ for all $i \in S$}}$.
		\item\label{it:PSCaxiom.uc} \textup{(Unit contractivity)} $f^{+} \wedge 1 \in \mathscr{F}_{p}$ for all $f \in \mathscr{F}_{p}$ and $\mathscr{E}_{p}(f^{+} \wedge 1) \le \mathscr{E}_{p}(f)$.
		\item\label{it:PSCaxiom.sg} \textup{(Spectral gap)} There exists a constant $C_{\textup{gap}} \in (0, \infty)$ such that
		\begin{equation*}
			\norm{f - f_{K}}_{L^{p}(\measure)}^{p} \le C_{\textup{gap}}\mathscr{E}_{p}(f) \quad \text{for all $f \in \mathscr{F}_{p}$.}
		\end{equation*}
	\end{enumerate}
\end{assumption}
\begin{rmk}
	\begin{enumerate}[\rm(1)]
		\item We do not claim that this assumption is the ``optimal'' axiom for self-similar $p$-energies. It would be desirable weaken Assumption \ref{a:main} for the purposes of the axiomatic characterization in Proposition \ref{p:qunique}.  For instance, we conjecture that Assumption \ref{a:main}\ref{it:PSCaxiom.sg} in Proposition \ref{p.axiom} is not necessary.
		\item If $(\mathscr{E}_{p}, \mathscr{F}_{p})$ satisfies the above assumptions, especially the self-similarity condition Assumption \ref{a:main}\ref{it:PSCaxiom.ss}, then the arguments in the first part of Section \ref{sec.emeas}   yields the associated self-similar measures. We use $\Gamma_{\mathscr{E}_{p}}\langle \,\cdot\, \rangle$ to denote these measures.
	\end{enumerate}
\end{rmk}

For convenience, we set
\[
\mathfrak{E}_{p}(K, \metric, \measure) \coloneqq \mathfrak{E}_{p} \coloneqq \bigl\{ (\mathscr{E}_{p}, \mathscr{F}_{p}) \bigm| \text{$(\mathscr{E}_{p}, \mathscr{F}_{p})$ satisfies Assumption \ref{a:main}} \bigr\}
\]
By Theorem \ref{t:main-sob-psc}, we know that $\mathfrak{E}_{p} \neq \emptyset$ for any $p \in (1, \infty)$.
Recall that $\rho(p) > 0$ denotes the $p$-scaling factor of PSC (see \eqref{p-factor}) and $\pwalk$ is as defined in \eqref{p-walk}.

We shall say that a constant $C > 0$ \emph{depends only on $p$ and the geometric data of PSC} if $C$ is a constant determined by $a_{\ast}, N_{\ast}, L_{\ast}, p, \rho(p)$.

Let us introduce the notion of $p$-capacity associated with $(\mathscr{E}_{p}, \mathscr{F}_{p}) \in \mathfrak{E}_{p}$.
For two disjoint subsets $A,B \subset K$ such that $\dist_{\metric}(A,B)>0$, we define
\[
\Cap_{\mathscr{E}_{p}}(A,B) = \inf \set{\mathscr{E}_{p}(f) \mid \text{$f \in \mathscr{F}_{p} \cap \contfunc(K)$ such that $f \ge 1$ on $A$, $f \le 0$ on $B$}}.
\]
Note that by Assumption \ref{a:main}\ref{it:PSCaxiom.reg}, the set $\{ f \in \mathscr{F}_{p} \cap \contfunc(K) \mid \text{$f \ge 1$ on $A$, $f \le 0$ on $B$}\}$ is non-empty.
It is immediate from Assumption \ref{a:main}\ref{it:PSCaxiom.const} that $\Cap_{\mathscr{E}_{p}}(A,B) = \Cap_{\mathscr{E}_{p}}(B,A)$.

Now we can show non-triviality of $p$-capacities.
Recall that $\ell_{\textup{L}}$ (resp. $\ell_{\textup{R}}$) denotes the left-line (resp. right-line) segment of $K$.
\begin{lem}\label{lem.non-triv}
	Let $p \in (1, \infty)$ and $(\mathscr{E}_{p}, \mathscr{F}_{p}) \in \mathfrak{E}_{p}$.
	We have
	\begin{align}
		0 &< \Cap_{\mathscr{E}_{p}}(\ell_{\textup{L}},\ell_{\textup{R}}) < \infty, \label{e:bnd1} \\
 		0 &< \inf \Biggl\{ \mathscr{E}_{p}(f) \Biggm| f \in \mathscr{F}_{p} \cap \contfunc(K), \restr{f}{\ell_{\textup{L}}} \equiv 0, \int_K f\,d\measure = 1 \Biggr\} < \infty. \label{e:bnd2}
	\end{align}
\end{lem}
\begin{proof}
	By Assumption \ref{a:main}\ref{it:PSCaxiom.uc}, it holds that 
	\[
	\Cap_{\mathscr{E}_{p}}(\ell_{\textup{L}},\ell_{\textup{R}}) = \inf \Bigl\{ \mathscr{E}_{p}(f) \Bigm| f \in \mathscr{F}_{p} \cap \contfunc(K), \restr{f}{\ell_{\textup{L}}} \equiv 0, \restr{f}{\ell_{\textup{R}}} \equiv 1, 0 \le f \le 1 \Bigr\}.
	\]
	Note that the set $\bigl\{f \in \mathscr{F}_{p} \cap \contfunc(K) \bigm| \restr{f}{\ell_{\textup{L}}} \equiv 0, \restr{f}{\ell_{\textup{R}}} \equiv 1, 0 \le f \le 1 \bigr\} \eqqcolon \mathscr{F}_{p,c}(\textup{L}, \textup{R})$ is non-empty, which can be verified by Assumption \ref{a:main}\ref{it:PSCaxiom.reg}, \ref{it:PSCaxiom.uc}.
	In particular, $\Cap_{\mathscr{E}_{p}}(\ell_{\textup{L}},\ell_{\textup{R}}) < \infty$.
	We let $\closure{\mathscr{F}}_{p}(\textup{L}, \textup{R})$ be the closure of $\mathscr{F}_{p,c}(\textup{L}, \textup{R})$ with respect to the norm $\norm{\,\cdot\,}_{\mathscr{F}_{p}}$.

		To show $\Cap_{\mathscr{E}_{p}}(\ell_{\textup{L}},\ell_{\textup{R}}) > 0$, let $u_{n} \in \mathscr{F}_{p,c}(\textup{L}, \textup{R})$ for each $n \in \mathbb{N}$ such that $\mathscr{E}_{p}(u_{n}) \le \Cap_{\mathscr{E}_{p}}(\ell_{\textup{L}},\ell_{\textup{R}}) + n^{-1}$ and define $g_{n} \in \contfunc(K)$ by
		\[
		g_{n} \coloneqq \sum_{i \in \{ 3, 4, 5 \}}(F_{i})_{\ast}\indicator{K} + \sum_{i \in \{ 2, 6 \}}(F_{i})_{\ast}u_{n}.
		\]
		Then $g_{n} \in \mathscr{F}_{p} \cap \contfunc(K)$ by Assumption \ref{a:main}\ref{it:PSCaxiom.ss} and thus $g_{n} \in \mathscr{F}_{p, c}(\textup{L}, \textup{R})$.
		By Assumption \ref{a:main}\ref{it:PSCaxiom.const}, \ref{it:PSCaxiom.ss},
		\[
		\mathscr{E}_{p}(g_{n}) = 2\widetilde{\rho}\,\mathscr{E}_{p}(u_{n}) \le 2\widetilde{\rho}\,\bigl(\Cap_{\mathscr{E}_{p}}(\ell_{\textup{L}},\ell_{\textup{R}}) + n^{-1}\bigr),
		\]
		which together with $0 \le g_{n} \le 1$ implies that $\{ g_{n} \}_{n \in \mathbb{N}}$ is bounded in $\mathscr{F}_{p}$.
		Hence Assumption \ref{a:main}\ref{it:PSCaxiom.banach} yields a subsequence $\{ n_{k} \}_{k \in \mathbb{N}}$ and $g_{\infty} \in \closure{\mathscr{F}}_{p}(\textup{L}, \textup{R})$ such that $g_{n_{k}}$ converges weakly to $g_{\infty}$ in $\mathscr{F}_{p}$. 
		By Mazur's lemma (Lemma \ref{lem.Mazur}), there exists a sequence $\widetilde{g}_{j} \in \textrm{conv}\{ g_{n_{k}} \}_{k \ge j} \, (j \in \mathbb{N})$ such that $\widetilde{g}_{j}$ converges to $g_{\infty}$ in $\mathscr{F}_{p}$.
		Therefore, we have
		\[
		\mathscr{E}_{p}(g_{\infty}) = \lim_{j \to \infty}\mathscr{E}_{p}(\widetilde{g}_{j}) \le \limsup_{j \to \infty}2\widetilde{\rho}\bigl(\Cap_{\mathscr{E}_{p}}(\ell_{\textup{L}},\ell_{\textup{R}}) + j^{-1}\bigr) = 2\widetilde{\rho}\Cap_{\mathscr{E}_{p}}(\ell_{\textup{L}},\ell_{\textup{R}}).
		\]
		If $\Cap_{\mathscr{E}_{p}}(\ell_{\textup{L}},\ell_{\textup{R}}) = 0$, then $g_{\infty}$ should be a constant function by virtue of Assumption \ref{a:main}\ref{it:PSCaxiom.const}.
		This is a contradiction since $g_{\infty}(x) = 1$ $\measure$-a.e., on $\bigcup_{i \in \{ 3, 4, 5 \}}K_{i}$ and $g_{\infty}(x) = 0$ $\measure$-a,e., on $\bigcup_{i \in \{ 1, 7, 8 \}}K_{i}$.
		Consequently, we get $\Cap_{\mathscr{E}_{p}}(\ell_{\textup{L}},\ell_{\textup{R}}) > 0$. 
		The lower estimate in \eqref{e:bnd2} can be shown in a similar way as the argument in the proof of Proposition \ref{prop:PI.PSC}. (Note that we used Assumption \ref{a:main}\ref{it:PSCaxiom.sg} here.)
\end{proof}

For a $p$-energy $(\mathscr{E}_{p}, \mathscr{F}_{p})$, we define the quantities considered in Lemma \ref{lem.non-triv}.
\begin{definition}\label{defn.poi-const}
	Let $p \in (1, \infty)$ and $(\mathscr{E}_{p}, \mathscr{F}_{p}) \in \mathfrak{E}_{p}$.
	We define $\chi(\mathscr{E}_p), \sigma(\mathscr{E}_p) \in (0, \infty)$ by setting
	\[
	\chi(\mathscr{E}_p) = \inf \bigl\{ \mathscr{E}_{p}(f) \bigm| f \in \mathscr{F}_{p} \cap \contfunc(K), \restr{f}{\ell_{\textup{L}}} \equiv 0, \restr{f}{\ell_{\textup{R}}} \equiv 1 \bigr\},
	\]
	and
	\[
	\sigma(\mathscr{E}_p) =  \inf \biggl\{ \mathscr{E}_{p}(f) \biggm| f \in \mathscr{F}_{p} \cap \contfunc(K), \restr{f}{\ell_{\textup{L}}} \equiv 0, \int_{K} f\,d\measure = 1 \biggr\}.
	\]
\end{definition}
In the case $p = 2$, $\chi(\mathscr{E}_{2})$ in the above definition is the same as $\norm{\mathscr{E}_{2}}$ in \cite[(4.41)]{BBKT}.

The following proposition characterizes the $p$-energy using the axioms in Assumption \ref{a:main}.
\begin{prop} [Quasi-uniqueness of $p$-energy] \label{p:qunique} 
	Let $p \in (1, \infty)$. 
	Then $\mathscr{F}_{p} = B_{p, \infty}^{\pwalk/p}(K, \metric, \measure)$ for any $(\mathscr{E}_{p},\mathscr{F}_{p}) \in \mathfrak{E}_{p}$. 
	Furthermore, there exist $C_u,c_l > 0$ (that depend only on $p$ and the geometric data of PSC) such that for all $(\mathscr{E}_{p}, \mathscr{F}_{p}) \in \mathfrak{E}_{p}$ and $f \in \mathscr{F}_{p}$,
	\begin{align}
		\mathscr{E}_{p}(f) &\ge	c_l \sigma(\mathscr{E}_{p}) \sup_{r > 0} \int_K \fint_{B_{\metric}(x,r)} \frac{\abs{f(y)-f(x)}^p}{r^{\pwalk}}\,\measure(dy)\,\measure(dx), \label{e:q-uniq.lower} \\
		\mathscr{E}_{p}(f) &\le C_u \chi(\mathscr{E}_{p}) \limsup_{r \downarrow 0} \int_K \fint_{B_{\metric}(x,r)} \frac{\abs{f(y)-f(x)}^p}{r^{\pwalk}}\,\measure(dy)\,\measure(dx). \label{e:q-uniq.upper}
	\end{align}
	In particular, any two $p$-energies $(\mathscr{E}_{p},\mathscr{F}_{p})$, $(\widehat{\mathscr{E}}_{p},\widehat{\mathscr{F}}_{p}) \in \mathfrak{E}_{p}$ are comparable; that is $\widehat{\mathscr{F}}_{p} = \mathscr{F}_{p} = B_{p, \infty}^{\pwalk/p}(K, \metric, \measure)$ and there exists $C > 0$ such that $C^{-1}\widehat{\mathscr{E}}_{p}(f) \le \mathscr{E}_{p}(f) \le C\widehat{\mathscr{E}}_{p}(f)$ for all $f \in \mathscr{F}_{p}$. 
\end{prop}
We start with a comparison between $\sigma(\mathscr{E}_{p})$ and $\chi(\mathscr{E}_{p})$.
\begin{lem} \label{l:oneside}
	For any $p \in (1, \infty)$ and $(\mathscr{E}_{p}, \mathscr{F}_{p}) \in \mathfrak{E}_{p}$, we have $\sigma(\mathscr{E}_{p}) \le 2^p \chi(\mathscr{E}_{p})$.
\end{lem}
\begin{proof}
	Let $f \in \mathscr{F}_{p} \cap \contfunc(K)$ be such that $\restr{f}{\ell_{\textup{L}}} \equiv 0, \restr{f}{\ell_{\textup{R}}} \equiv 1$ and $\mathscr{E}_{p}(f) \le \chi(\mathscr{E}_{p}) + \varepsilon$.
	Then by Assumption \ref{a:main}\ref{it:PSCaxiom.const},\ref{it:PSCaxiom.sym}, the function $g \coloneqq 2^{-1}\bigl(f + (1 - f) \circ S_{2}\bigr)$ satisfies $g \in \mathscr{F}_{p} \cap \contfunc(K)$, $\restr{g}{\ell_{\textup{L}}} \equiv 0$, $\int_{K}g\,d\measure = 2^{-1}$, and $\mathscr{E}_{p}(g) \le \mathscr{E}_{p}(f)$.
	This implies $\sigma(\mathscr{E}_{p}) \le \mathscr{E}_{p}(2g) \le 2^{p}\mathscr{E}_{p}(f) \le 2^p(\chi(\mathscr{E}_{p})+\varepsilon)$.
	Letting $\varepsilon \downarrow 0$, we obtain the desired estimate.
\end{proof}

We next obtain Poincar\'{e} inequalities for $(\mathscr{E}_{p}, \mathscr{F}_{p})$ satisfying Assumption \ref{a:main}.
The following lemma is a key estimate, which can be shown in a similar way as the argument in the proof of Proposition \ref{prop:PI.PSC}.  
Recall that the self-similarity of $(\mathscr{E}_{p}, \mathscr{F}_{p})$ allows us to get the associated energy measures $\Gamma_{\mathscr{E}_{p}}\langle \,\cdot\, \rangle$.
\begin{lem} \label{l:neighbor}
	Let $n \in \bN, v, w \in W_n$ be such that $\{ v, w \} \in E_{n}^{\#}$ and let $f \in \mathscr{F}_{p}$.
	Then
	\[
	\abs{f_{K_v}-f_{K_w}}^p \le 2^{p/(p-1)}\sigma(\mathscr{E}_{p})^{-1}\widetilde{\rho}^{\,-n}  \left[ \Gamma_{\mathscr{E}_{p}}\langle f \rangle(K_v) + \Gamma_{\mathscr{E}_{p}}\langle f \rangle(K_w) \right].
	\]
\end{lem}
%

The following proposition shows the uniqueness of the scaling factor $\widetilde{\rho}$ in Assumption \ref{a:main}\ref{it:PSCaxiom.ss} and gives a global Poincar\'{e} inequality.
\begin{prop}[Poincar\'e inequality: global version] \label{p:globpoin}
	Let $p \in (1, \infty)$ and $(\mathscr{E}_{p}, \mathscr{F}_{p}) \in \mathfrak{E}_{p}$.
	Then
	\[
	\widetilde{\rho} = \rho(p),
	\]
	where $\widetilde{\rho}$ is the constants in Assumption \ref{a:main}\ref{it:PSCaxiom.ss}.
	Furthermore, there exists $C_1>0$ (depending only on $p$ and the geometric data of PSC) such that
	\begin{equation}\label{global-PI}
		\int_{K} \abs{f(x) - f_K}^{p}\,\measure(dx) \le C_1 \sigma(\mathscr{E}_{p})^{-1} \mathscr{E}_{p}(f) \qq \mbox{for all $f \in \mathscr{F}_{p}$.}
	\end{equation}
\end{prop}
\begin{proof}
	First we show $\widetilde{\rho} \le \rho(p)$.
	Let $f \in \mathscr{F}_{p} \cap \contfunc(K)$.
	Recall that $M_{n}f(w) = \fint_{K_{w}}f\,d\measure = \int_{K}f \circ F_w\,d\measure$ for $w \in W_{n}$.
	By \hyperref[cond.UPI]{\textup{U-PI$_{p}(\pwalk)$}} for $\{ \mathbb{G}_{n} = (W_{n}, E_{n}) \}_{n \in \mathbb{N}}$ and $\diam(\mathbb{G}_{n}) \asymp a_{\ast}^{n}$, there exists $C_{\textup{UPI}} > 0$ (depending only on $\rho(p)$ and other geometric data of PSC) such that
	\begin{equation} \label{e:pn1}
		\sum_{w \in W_n} \abs{f_n(w)-f_K}^p \measure_{n}(w) \le C_{\textup{UPI}}\, \rho(p)^{n} \sum_{ \{v,w\} \in E_{n}} \abs{M_{n}f(v)-M_{n}f(w)}^{p},
	\end{equation}
	where $\measure_{n}(w) = \measure(K_{w}) = a_{\ast}^{-n\hdim}$.
	We note that the dominated convergence theorem and the uniform continuity of $f$ imply
	\begin{equation}\label{e:pn2}
		\int_{K}\abs{f-f_K}^p \, d\measure = \lim_{n \to \infty} \sum_{w \in W_n} \abs{f_n(w)-f_K}^{p}\measure_{n}(w).
	\end{equation}
	By Lemma \ref{l:neighbor}, we have
	\begin{align} \label{e:pn3}
		\sum_{ \{v,w\} \in E_n} \abs{M_{n}f(v)-M_{n}f(w)}^{p}
		&\le  \sum_{ \{v,w\} \in E_n}  2^{p/(p-1)}\sigma(\mathscr{E}_{p})^{-1}\widetilde{\rho}^{\,-n} \left[ \Gamma_{\mathscr{E}_{p}}\langle f \rangle(K_{v}) + \Gamma_{\mathscr{E}_{p}}\langle f \rangle(K_{w}) \right] \nonumber \\
		&\le \sigma(\mathscr{E}_{p})^{-1} 2^{p/(p-1)}L_{\ast} \cdot \widetilde{\rho}^{\,-n} \sum_{w \in W_n} \Gamma_{\mathscr{E}_{p}}\langle f \rangle(K_{w}) \quad  \nonumber \\
		&\le  \sigma(\mathscr{E}_{p})^{-1} 2^{p/(p-1) + 2}L_{\ast} \cdot \widetilde{\rho}^{\,-n}\mathscr{E}_{p}(f),
	\end{align}
	where we used $\sup_{x \in K, n \in \mathbb{N}}\#\{ w \in W_{n} \mid x \in K_{w} \} \le 4$ in the last inequality.
	By \eqref{e:pn1}, \eqref{e:pn2} and \eqref{e:pn3}, we obtain
	\begin{equation} \label{e:poin}
		\int_{K} \abs{f-f_K}^p \, d\measure
		\le C_{1}\sigma(\mathscr{E}_{p})^{-1} \left( \lim_{n \to \infty} \bigl(\rho(p)\widetilde{\rho}^{\,-1}\bigr)^{n} \right) \mathscr{E}_{p}(f), \q \mbox{for all $f \in \mathscr{F}_{p} \cap \contfunc(K)$.}
	\end{equation}
	where $C_{1} \coloneqq 2^{p/(p-1) + 2}L_{\ast}C_{\textup{UPI}}$.
	This implies $\rho \le \rho(p)$ (otherwise, by \eqref{e:poin}, we have $f \equiv f_K$ $\measure$-a.e. for all $f \in \mathscr{F}_{p} \cap \contfunc(K)$ which contradicts Assumption \ref{a:main}\ref{it:PSCaxiom.const}, \ref{it:PSCaxiom.reg}).

	Next we show $\widetilde{\rho} \ge \rho(p)$.
	Let $\varepsilon > 0$ and choose $h \in \mathscr{F}_{p} \cap \contfunc(K)$ such that $\restr{h}{\ell_{\textup{L}}} \equiv 0,  \restr{h}{\ell_{\textup{R}}} \equiv 1$ and $\mathscr{E}_{p}(h) \le \chi(\mathscr{E}_{p})+ \varepsilon$.
	Let $W_{n,e} = \bigl\{w =w_1\cdots w_n \in W_n: w_n \in \{2,4,6,8\}\bigr\}$ and $W_{n,o} = W_n \setminus W_{n,e}$.
	For $w \in W_{n,e}$, we define $N(w)=\{u \in W_{n,e}: K_u \cap K_w \neq \emptyset \}$. Similarly for $w \in W_{n,o}$, we define $N(w) =\{u \in W_{n,o}: K_u \cap K_w \neq \emptyset \}$.
	Given any function $f \colon W_n \to \bR$, we define $\wt f \colon W_{n,o} \to \bR$ and $\wh f \in \contfunc(K)$ as
	\[
	\wt f(w)=   \frac{1}{\#N(w)} \sum_{u \in N(w)} f(w),
	\]
	and
	\begin{equation}\label{d:fhat}
		F_w^*\wh f \equiv
		\begin{cases}
			\wt f(w) & \mbox{if $w \in W_{n,o}$,} \\
			\wt f(w_1\cdots w_{n-1}1) + \left( \wt f(w_1\cdots w_{n-1}3)- \wt f(w_1\cdots w_{n-1}1) \right)h  & \mbox{if $w_n=2$,} \\
			\wt f(w_1\cdots w_{n-1}7) + \left( \wt f(w_1\cdots w_{n-1}5)- \wt f(w_1\cdots w_{n-1}7) \right)h  & \mbox{if $w_n=6$,}\\
			\wt f(w_1\cdots w_{n-1}3) + \left( \wt f(w_1\cdots w_{n-1}5)- \wt f(w_1\cdots w_{n-1}3) \right) h \circ R_{1}  & \mbox{if $w_n=4$,}    \\
			\wt f(w_1\cdots w_{n-1}1) + \left( \wt f(w_1\cdots w_{n-1}7)- \wt f(w_1\cdots w_{n-1}1) \right) h \circ R_{1}  & \mbox{if $w_n=8$.}
		\end{cases}
	\end{equation}
	We will show that $\mathscr{E}_{p}(\wh f) \lesssim \widetilde{\rho}^{\,n}\mathscr{E}_{p}(h)\mathcal{E}_{p}^{G_{n}^{\#}}(f)$.
	Note that by Assumption \ref{a:main}\ref{it:PSCaxiom.sym}, \ref{it:PSCaxiom.ss}, we have $\wh f \in \mathscr{F}_{p} \cap \contfunc(K)$ and
	\begin{align*}
		\widetilde{\rho}^{\,-n} \mathscr{E}_{p}(\wh f)
		&= \sum_{w=w_1\cdots w_n \in W_n, w_n=2} \mathscr{E}_{p}(h) \abs{\wt f(w_1\cdots w_{n-1}3)- \wt f(w_1\cdots w_{n-1}1)}^p + \\
		& \hspace{2mm}
	 	\sum_{w=w_1\cdots w_n \in W_n, w_n=6}\mathscr{E}_{p}(h) \abs{\wt f(w_1\cdots w_{n-1}5)- \wt f(w_1\cdots w_{n-1}7)}^p + \\
	 	& \hspace{2mm}
	 	\sum_{w=w_1\cdots w_n \in W_n, w_n=4}\mathscr{E}_{p}(h) \abs{\wt f(w_1\cdots w_{n-1}5)- \wt f(w_1\cdots w_{n-1}3)}^p + \\
	 	& \hspace{2mm}
	 	\sum_{w=w_1\cdots w_n \in W_n, w_n=8}\mathscr{E}_{p}(h) \abs{\wt f(w_1\cdots w_{n-1}7)- \wt f(w_1\cdots w_{n-1}1)}^p.
	\end{align*}
	For any $u,v \in W_{n,o}$ and $w \in W_{n,e}$ satisfying $ \{u,w\}, \{v,w\} \in E_{n}^{\#}$, and for any $u' \in N(u)$ we easily see that $d_{n}^{\#}(w,v') \le 3$.
	For any such $u,v,w$, Jensen's and H\"{o}lder's inequalities imply that
	\begin{align*}
		\abs{\wt f(u)-\wt f(v)}^p &\le \frac{1}{\# N(u) \#N(v)} \sum_{u' \in N(u), v' \in N(v)} \abs{f(u')-f(v')}^p \nonumber\\
		&\le \sum_{u', v' \in B_{d_{n}^{\#}}(w, 4)} \abs{f(u')-f(v')}^p \nonumber\\
		& \le 6^{p - 1} \sum_{ \{u_1,u_2\} \in E_{n}^{\#}, u_1,u_2 \in B_{d_{n}^{\#}}(w, 4)} \abs{f(u_1)-f(u_2)}^{p} \q \mbox{(since $d_{n}^{\#}(u_{1},u_{2}) \le 6$)}.
	\end{align*}
	In particular, we get
	\begin{equation}\label{e:pn4}
		\mathscr{E}_{p}(\wh f) \le \widetilde{\rho}^{\,n}\mathscr{E}_{p}(h)\Bigl(6^{p - 1}\sup_{k \in \mathbb{N}, v \in W_{k}}\#B_{d_{k}^{\#}}(v, 4)\Bigr)\mathcal{E}_{p}^{G_{n}^{\#}}(f).
	\end{equation}
	Recall that there exists $C_{\textup{face}} \ge 1$ depending only on the geometric data of PSC such that
	\[
	C_{\textup{face}}^{-1}\,\rho(p)^{-k} \le \CAP_{p}^{\mathbb{G}_{k}}\bigl(W_k[\ell_{\textup{L}}], W_k[\ell_{\textup{R}}]\bigr) \le C_{\textup{face}}\,\rho(p)^{-k}, \quad \text{for all $k \in \mathbb{N}$.}
	\]
	(See Lemmas \ref{lem.left-right} and   \ref{lem.mod/cap}.)
	Now let us choose $f \in \bR^{W_n}$ such that $\restr{f}{W_n[\ell_{\textup{L}}]} \equiv 0, \restr{f}{W_n[\ell_{\textup{R}}]} \equiv 1$ and $\sE_{p}^{G_{n}^{\#}}(f) \le \sE_{p}^{\mathbb{G}_{n}}(f) \le C_{\textup{face}}\rho(p)^{-n}$.
	Then the function $\wh f \in \mathscr{F}_{p} \cap \contfunc(K)$ defined in \eqref{d:fhat} satisfies $\restr{\wh f}{\ell_{\textup{L}}} \equiv 0, \restr{\wh f}{\ell_{\textup{R}}} \equiv 1$.
	Hence we have from \eqref{e:pn4} that
	\begin{equation} \label{e:pn5}
 		0 < \chi(\mathscr{E}_{p})
 		\le \mathscr{E}_{p}(\wh f)
		\le \widetilde{\rho}^{\,n}\rho(p)^{-n}\bigl(\chi(\mathscr{E}_{p}) + \varepsilon\bigr)\Bigl(6^{p - 1}C_{\textup{face}}\sup_{k \in \mathbb{N}, v \in W_{k}}\#B_{d_{k}^{\#}}(v, 4)\Bigr).
	\end{equation}
	By letting $n \to \infty$ in \eqref{e:pn5} and using the fact that $\sup_{k \in \bN, v \in W_k} \# B_{d_{G_k}^{\#}}(v,4) <\infty$, we obtain $\widetilde{\rho} \ge \rho(p)$.
	This concludes the proof of $\widetilde{\rho} = \rho(p)$.

	The desired global Poincar\'{e} inequality for $f \in \mathscr{F}_{p} \cap \contfunc(K)$ is evident from $\widetilde{\rho} = \rho(p)$ and \eqref{e:poin}.
	By virtue of the regularity (Assumption \ref{a:main}\ref{it:PSCaxiom.reg}), we can extend it to any function in $\mathscr{F}_{p}$.
\end{proof}

The following lemma is a Poincar\'e inequality on finite graphs.
\begin{lem} \label{l:finp}
	Let $G=(V,E)$ be a connected graph with $\#V =n$ and diameter $D$. Then \[ 	\sum_{v \in V} \abs{f(u)-\ol f}^p \le n D^{p-1}\sum_{ \{v,w\}\in E} \abs{f(v)-f(w)}^p,\]
	where $\ol f= \frac{1}{n}\sum_{v \in V} f(v)$.
\end{lem}
\begin{proof}
	By Jensen's inequality,
	\[
 	\sum_{v \in V}\abs{f(v)-\ol f}^p \le n^{-1}\sum_{v,w \in V} \abs{f(v)-f(w)}^p.
	\]
	For $v,w \in V$, by using a path of length at most $D$, we obtain
	\[
	\abs{f(v)-f(w)}^p \le D^{p-1} \sum_{ \{u_1,u_2\} \in E} \abs{f(u_1)-f(u_2)}^p.
	\]
	Combining the above two estimates implies the desired inequality.
\end{proof}

The self-similarity of the $p$-energy along with the global Poincar\'e inequality implies the following local version.
\begin{prop}\label{p:localpoin}
	Let $p \in (1,\infty)$.
	There exists $\widetilde{C}_{\textup{P}} \in (0,\infty)$ (depending only on $p$ and the geometric data of PSC) such that, for all $(\mathscr{E}_{p}, \mathscr{F}_{p}) \in \mathfrak{E}_{p}$, $f \in \mathscr{F}_{p}, x \in K, r > 0$, we have
	\begin{equation} \label{e:lp}
		\int_{B_{\metric}(x,r)} \abs{f(y)- f_{B_{\metric}(x,r)}}^{p} \,\measure(dy) \le \widetilde{C}_{\textup{P}} \sigma(\mathscr{E}_{p})^{-1} r^{\pwalk} \Gamma_{\mathscr{E}_{p}}\langle f \rangle\bigl(B_{\metric}(x,2r)\bigr).
	\end{equation}
\end{prop}
\begin{proof}
	For $r>0$, let $n(r) \in \bZ_+$ be the smallest non-negative integer $n$ such that $r \ge a_{\ast}^{-n}$ and let $W(x,r) \coloneqq W_{n(r)}\bigl(B_{\metric}(x, r) \bigr) = \{w \in W_{n(r)}: K_w \cap B(x,r) \neq \emptyset \}$ for simplicity.
	Then, there exists $N_1 \in \mathbb{N}$ (depending only on $a_{\ast}, L_{\ast}$) such that
	\begin{equation} \label{e:lp2}
		\bigcup_{w \in W(x,r)} K_w \subset B_{\metric}(x,2r), \q \#W(x,r) \le N_1, \q \mbox{for all $x \in K, r>0$}.
	\end{equation}
	For any $w \in W_n$, by Proposition \ref{p:globpoin} and Lemma \ref{lem.em-cell}, we have
	\begin{align} \label{e:glb}
		\int_{K_w} \abs{f(y)-f_{K_w}}\,\measure(dy)
		&= a_{\ast}^{-n\hdim} \int_{K} \abs{(f \circ F_{w})(y) - (f \circ F_{w})_K}^{p}\,\measure(dy) \nonumber \\
		& \le C_{1}a_{\ast}^{-n\hdim} \sigma(\mathscr{E}_{p})^{-1} \mathscr{E}_{p}(f \circ F_{w})
		\le C_{1} \bigl(a_{\ast}^{\hdim}\widetilde{\rho}\bigr)^{-n} \Gamma_{\mathscr{E}_{p}}\langle f \rangle(K_w),
	\end{align}
	where $C_{1} > 0$ is the constant in \eqref{global-PI}.
	Furthermore, for each $x \in K, r>0$, the induced subgraph of $G_{n(r)}^{\#}$ with vertex set $W(x,r)$ is connected (and hence has diameter at most $N_1$).
	For any $c \in \bR$,
 	\begin{align} \label{e:lp3}
 		&\int_{B_{\metric}(x,r)} \abs{f(y)-c}^p \,\measure(dy) \nonumber \\
 		&\hspace*{8pt}\le \sum_{w \in W(x,r)} \int_{K_w} \abs{f(y)-c}^p \, \measure(dy) \nonumber \\
 		&\hspace*{8pt}\le 2^{p-1} \sum_{w \in W(x,r)} \left( \int_{K_w} \abs{f(y)-f_{K_w}}^p \, \measure(dy) + \measure(K_w) \abs{f_{K_w}-c}^p \right)  \nonumber \\
 		&\stackrel{\eqref{e:glb}}{\le} 2^{p-1}a_{\ast}^{-n(r)\hdim}(C_{1} \vee 1) \sum_{w \in W(x,r)} \bigl( \widetilde{\rho}^{\,-n(r)} \Gamma_{\mathscr{E}_{p}}\langle f \rangle(K_w)  +  \abs{M_{n(r)}f(w)-c}^p \bigr).
 	\end{align}
	If $c = \frac{1}{\#W(x,r)}\sum_{w \in W(x, r)}M_{n(r)}f(w)$, then by Lemma \ref{l:finp}, \eqref{e:lp2}, and Lemma \ref{l:neighbor},
	\begin{align} \label{e:lp4}
		&\sum_{w \in W(x,r)} \abs{M_{n(r)}f(w) - c}^p \nonumber \\
		&\le N_1^{p} \sum_{u,v \in W(x,r): \{u,v\} \in E_{n(r)}^{\#}} \abs{M_{n(r)}f(u)-M_{n(r)}f(v)}^p  \nonumber \\
		&\le 2^{p/(p - 1)}N_1^{p}\sigma(\mathscr{E}_{p})^{-1}\widetilde{\rho}^{\,-n(r)} \sum_{u,v \in W(x,r): \{u,v\} \in E_{n(r)}^{\#}}\bigl[\Gamma_{\mathscr{E}_{p}}\langle f \rangle(K_{u}) + \Gamma_{\mathscr{E}_{p}}\langle f \rangle(K_{v})\bigr] \nonumber \\
		&\le 2^{p/(p - 1)}N_1^{p}L_{\ast}\sigma(\mathscr{E}_{p})^{-1}\widetilde{\rho}^{\,-n(r)}\Gamma_{\mathscr{E}_{p}}\langle f \rangle\bigl(B_{\metric}(x,2r)\bigr).
	\end{align}
	The desired conclusion follows from Lemma \ref{lem.p-var}, \eqref{e:lp2}, \eqref{e:lp3} and \eqref{e:lp4}.
\end{proof}

The following lemma is a lower bound on $p$-energy which is a consequence of the Poincar\'e inequality \eqref{e:lp}.
\begin{lem} \label{l:elb}
	 Let $p \in (1, \infty)$.
	 There exists $c_{l} > 0$ (depending only on $p$ and the geometric data of PSC) such that
	\begin{equation}
		c_{l} \sigma(\mathscr{E}_{p})\sup_{r>0} \int_K \fint_{B(x,r)} \frac{\abs{f(y)-f(x)}^p}{r^{\pwalk}}\,\measure(dy)\,\measure(dx)
		\le \mathscr{E}_{p}(f)
	\end{equation}
	for all $(\mathscr{E}_{p}, \mathscr{F}_{p}) \in \mathfrak{E}_{p}$ and $f \in \mathscr{F}_{p}$. 
	In particular, $\mathscr{F}_{p} \subseteq B_{p,\infty}^{\pwalk/p}(K,\metric,\measure)$. 
\end{lem}
\begin{proof}
	Let $r > 0$ and let $N  \subset K$ denote a maximal $r$-net of $(K, \metric)$.
	Then
 	\begin{align} \label{e:leb1}
 		\lefteqn{r^{-\pwalk} \int_{K} \fint_{B_{\metric}(x,r)} \abs{f(y)-f(x)}^{p} \,\measure(dy)\,\measure(dx)} \nonumber \\
 		&\le C_{\textup{AR}}r^{-\pwalk-\hdim} \sum_{n \in N} \int_{B_{\metric}(n,2r)} \int_{B_{\metric}(n,2r)} \abs{f(x)-f(y)}^p \, \measure(dx)\,\measure(dy).
 	\end{align}
	For any $n \in N, r>0$, we have from Proposition \ref{p:localpoin} that
	\begin{align} \label{e:leb2}
	 	&\int_{B(n,2r)} \int_{B(n,2r)} \abs{f(x)-f(y)}^p \, \measure(dx)\,\measure(dy) \nonumber \\
	 	&\le 2^{p - 1}\measure(B_{\metric}(n,2r)) \int_{B_{\metric}(n,2r)} \abs{f(x)-f_{B_{\metric}(n,2r)}}^p\,\measure(dx)  \nonumber \\
	 	&\le 2^{p - 1 + \hdim + \pwalk}\widetilde{C}_{\textup{P}}\sigma(\mathscr{E}_{p})^{-1}r^{\hdim + \pwalk} \Gamma_{\mathscr{E}_{p}}\langle f \rangle\bigl(B_{\metric}(n,4r)\bigr).
	\end{align}
	There exists $C$ that depends only on $a_{\ast}, L_{\ast}$ such that $\sum_{n \in N} \indicator{B_{\metric}(n,4r)} \le C$ (by the metric doubling property of $(K, d)$).
	This along with \eqref{e:leb1} and \eqref{e:leb2} implies the desired estimate.
\end{proof}

Lastly, we prove an upper bound on $p$-energy by using the self-similarity instead of a suitable partition of unity (cf. Lemma \ref{lem.upper}).
\begin{lem} \label{l:eub}
	Let $p \in (1, \infty)$.
	There exists $C_u > 0$ (depending only on $p$ and the geometric data of PSC) such that for any $(\mathscr{F}_{p}, \mathscr{E}_{p}) \in \mathfrak{E}_{p}$ and $f \in \mathscr{F}_{p}$, we have
	\[
	\mathscr{E}_p(f) \le C_u \chi(\mathscr{E}_{p}) \limsup_{r \downarrow 0} \int_{K} \fint_{B_{\metric}(x,r)} \frac{\abs{f(y)-f(x)}^p}{r^{\pwalk}}\,\measure(dy)\,\measure(dx).
	\]
	Furthermore, $\mathscr{F}_{p} = B_{p,\infty}^{\pwalk/p}(K,\metric,\measure)$. 
\end{lem}
\begin{proof}
	Let $h \in \mathscr{F}_{p} \cap \contfunc(K)$ be such that $\restr{h}{\ell_{\textup{L}}} \equiv 0, \restr{h}{\ell_{\textup{R}}} \equiv 1$ such that $\mathscr{E}_p(h) \le 2 \chi(\mathscr{E}_p)$.
	Let $W_{n,e}, W_{n,o}, N(w)$ be the same notations as in the proof of Proposition \ref{p:globpoin}.
	To any function $f \in L^{p}(K,\measure)$, we define a function $f_n \colon W_{n,o} \to \bR$ as
	\begin{equation}\label{e:defn.ave-1nbd}
		f_n(w)=   \fint_{\bigcup_{v \in N(w)}K_{v}} f\,d\measure.
	\end{equation}  
	We note that $\#N(w) \le 4$ for all $w \in W_{n,o}$.
	We define $\wh f_n \colon K \to \bR$ by specifying $F_w^*\wh f_n$ for all $w=w_1\cdots w_n \in W_{n}$ as
	\begin{equation}\label{e:q-approx}
		F_w^*\wh f_n \equiv \begin{cases}
			f_n(w) & \mbox{if $w \in W_{n,o}$,} \\
			f_n(w_1\cdots w_{n-1}1) + 	\left( f_n(w_1\cdots w_{n-1}3)- f_n(w_1\cdots w_{n-1}1) \right)h  & \mbox{if $w_n=2$,} \\
				f_n(w_1\cdots w_{n-1}7) + 	\left( f_n(w_1\cdots w_{n-1}5)- f_n(w_1\cdots w_{n-1}7) \right)h  & \mbox{if $w_n=6$,}\\
					f_n(w_1\cdots w_{n-1}3) + 	\left( f_n(w_1\cdots w_{n-1}5)- f_n(w_1\cdots w_{n-1}3) \right) h \circ R_1  & \mbox{if $w_n=4$,} \\
						f_n(w_1\cdots w_{n-1}1) + 	\left( f_n(w_1\cdots w_{n-1}7)- f_n(w_1\cdots w_{n-1}1) \right) h \circ R_1  & \mbox{if $w_n=8$.}
		\end{cases}
	\end{equation}
	Note that $\wh f_{n} \in \mathscr{F}_{p} \cap \contfunc(K)$ by Assumption \ref{a:main}\ref{it:PSCaxiom.ss}. 

	For $w \in W_n$, let $q_w = F_w(q_1)$.
	For $u,v \in W_{n,o}, w \in W_{n,e}$ such that $\{u,w\}, \{v,w\} \in E_{n}^{\#}$, we have $\sup_{x \in K_{u} \cup K_{w} \cup K_{v}}d(q_w,x) \le \sqrt{5} \cdot a_{\ast}^{-n}$.
	This along with Jensen's inequality implies that there exists $C > 0$ (depending only on $p$ and the geometric data of PSC) such that for all $u,v,w$ as above, we have
	\begin{align*}
		\abs{f_n(u)-f_n(v)}^p \le C a_{\ast}^{-2n\hdim} \int_{B_{\metric}(q_w, C a_{\ast}^{-n})} \int_{B_{\metric}(q_w,C a_{\ast}^{-n})} \abs{f(x)-f(y)}^p \, \measure(dx)\,\measure(dy).
	\end{align*}
	Therefore by Assumption \ref{a:main}\ref{it:PSCaxiom.sym}, \ref{it:PSCaxiom.ss} and the bounded overlap of $B_{\metric}(q_w,C a_{\ast}^{-n}), w \in W_{n,e}$, we have
	\begin{align}
		\mathscr{E}_{p}(\wh f_n) &= \sum_{w \in W_{n,e}} \widetilde{\rho}^{\,n} \mathscr{E}_{p}(F_w^* \wh f_n)  \nonumber \\
		& \le \sum_{w \in W_{n,e}} \widetilde{\rho}^{\,n}\mathscr{E}_{p}(h) a_{\ast}^{-2n\hdim} \int_{B_{\metric}(q_w, C a_{\ast}^{-n})} \int_{B_{\metric}(q_w,C a_{\ast}^{-n})} \abs{f(x)-f(y)}^p \, \measure(dx)\,\measure(dy)  \nonumber \\
		& \le C_{0}\chi(\mathscr{E}_{p})\widetilde{\rho}^{\,n} a_{\ast}^{-n\hdim} \int_{K} \fint_{B_{\metric}(x,2 C a_{\ast}^{-n})} \abs{f(x)-f(y)}^p \, \measure(dy)\,\measure(dx),
	\end{align}
	where $C_{0} > 0$ is a constant depending only on $p$ and the geometric data of PSC.
	By setting $r_n=2C3^{-n}$ and using $\widetilde{\rho} = \rho(p)$, we have
	\begin{equation}\label{e:q-approx.unifestimate}
		\mathscr{E}_p(\wh f_n) \le  \widetilde{C}_{0}\chi(\mathscr{E}_p) \int_{K} \fint_{B(x,r_n)}  \frac{\abs{f(x)-f(y)}^p}{r_{n}^{\pwalk}} \,\measure(dy)\,\measure(dx).
	\end{equation}  
	Let $\varepsilon > 0$, $f \in L^{p}(K,\measure)$ and $u \in \contfunc(K)$ satisfy $\norm{f - u}_{L^p(\measure)} < \varepsilon$. 
	Then $u_n$ defined in \eqref{e:defn.ave-1nbd} converges to $u$ uniformly in $K$ (by uniform continuity of $u$) and thus $\wh u_n$ defined in \eqref{e:q-approx} also converges to $u$ uniformly. 
	Similar to the proof of Theorem \ref{thm.Epgamma}, by H\"{o}lder's inequalty, Proposition \hyperref[it:PSCgeom.1]{\ref{prop.PSC-geom}}\ref{it:PSCgeom.2},\ref{it:PSCgeom.3} and the $L^p$-boundedness of the Hardy--Littlwood maximal operator (see, e.g., \cite[Theorem 3.5.6]{HKST}), we see that 
	\[
	\sup_{w \in W_{\ast}}\abs{f_{n}(w) - u_{n}(w)} 
	\le \sup_{w \in W_{\ast}}\fint_{\bigcup_{v \in N(w)}K_{v}}\abs{f - u}^{p}\,d\measure 
	\lesssim \norm{f - u}_{L^{p}(\measure)}^{p}, 
	\]
	and hence $\norm{\wh f_{n} - \wh u_{n}}_{L^{p}(\measure)} \lesssim \varepsilon$. 
	In particular, we conclude that $\wh f_{n}$ converges in $L^p(K,\measure)$ to $f$ since 
	\[
	\norm{f - \wh f_{n}}_{L^p(\measure)} 
	\le \norm{f - u}_{L^p(\measure)} + \norm{u - \wh u_{n}}_{L^p(\measure)} + \norm{\wh u_{n} - \wh f_{n}}_{L^p(\measure)} 
	\lesssim \varepsilon + \norm{u - \wh u_{n}}_{L^p(\measure)}.
	\]
	If $f \in B_{p,\infty}^{\pwalk/p}(K,\metric,\measure)$, then $\wh f_n$ has a subsequence that converges weakly in $\mathscr{F}_{p}$ by \eqref{e:q-approx.unifestimate} and the reflexivity of $(\mathscr{F}_{p},\norm{\,\cdot\,}_{\mathscr{F}_{p}})$. 
	This together with $\lim_{n \to \infty}\wh f_{n} = f$ in $L^p(K,\measure)$ implies that $f \in \mathscr{F}_{p}$, thereby $\mathscr{F}_{p} = B_{p,\infty}^{\pwalk/p}(K,\metric,\measure)$. 
	In addition, by Mazur's lemma (Lemma \ref{lem.Mazur}), we obtain
	\[
	\mathscr{E}_{p}(f) \le \limsup_{n \to \infty} \mathscr{E}_p(\wh f_n) \le \widetilde{C}_{0}\chi(\mathscr{E}_{p}) \limsup_{r \downarrow 0} \int_{K} \fint_{B_{\metric}(x,r)} \frac{\abs{f(y)-f(x)}^p}{r^{\pwalk}}\,\measure(dy)\,\measure(dx), 
	\]
	proving the desired assertion. 
\end{proof}

\begin{proof}[Proof of Proposition \ref{p:qunique}]
	This follows from Lemmas \ref{l:elb} and \ref{l:eub}. 
\end{proof}

\begin{proof}[Proof of Proposition \ref{p.axiom}]
	This follows from Proposition \ref{p:qunique} and Theorem \ref{thm.LB.psc}.
\end{proof}

\section{The attainment problem for Ahlfors regular conformal dimension on the \texorpdfstring{Sierpi\'{n}ski}{Sierpinski} carpet}\label{sec.attain}
In this section, we obtain partial results towards the attainment problem, namely the last main result Theorem \ref{t:attain}.


\subsection{Newton-Sobolev space \texorpdfstring{$N^{1, p}$}{N1p}}
We start by recalling the theory first-order Sobolev spaces on metric measure spaces based on the notion of \emph{upper gradients}. A comprehensive account of this theory can be found in \cite{HKST} (see also \cite{BB,Hei}).

Hereafter, we let $(X, \Lmetric, \Lmeasure)$ be a metric measure space in the sense of \cite{HKST}, i.e., $(X, \Lmetric)$ is a separable metric space and $\Lmeasure$ is a locally finite Borel-regular (outer) measure on $X$.
In addition, we always assume that $\Lmeasure(O) > 0$ whenever $O$ is a non-empty open subset of $X$.
\begin{defn}[Curves in a metric space]
	\begin{enumerate}[\rm(1)]
		\item A continuous map $\gamma \colon I \to X$, where $I$ is an interval of $\mathbb{R}$, is called a \emph{curve} in $X$. If $I$ is a closed interval, then $\gamma$ is called a \emph{compact curve}. For any subinterval $[a', b'] \subseteq I$, the \emph{subcurve} $\restr{\gamma}{[a', b']}$ is the restriction of $\gamma$ to $[a', b']$.
		\item For a compact curve $\gamma \colon [a, b] \to X$, its length $\ell(\gamma)$ (with respect to the metric $\rho$) is defined by
		\[
			\ell(\gamma) = \sup\Biggl\{ \sum_{i = 1}^{k}\Lmetric(\gamma(t_{i - 1}), \gamma(t_{i})) \Biggm| \text{$k \in \mathbb{N}$, $\{ t_{i} \}_{i = 0}^{k} \subseteq \mathbb{R}$ s.t. $a = t_0 < t_1 < \cdots < t_{k} = b$} \Biggr\}.
		\]
		For a curve $\gamma \colon I \to X$ ($I$ is not assumed to be a closed interval), define its length by
		\[
		\ell(\gamma) = \sup\bigl\{ \ell(\gamma') \bigm| \text{$\gamma'$ is a compact subcurve of $\gamma$} \bigr\}.
		\]
		A curve $\gamma$ is said to be \emph{rectifiable (with respect to the metric $\rho$)} if $\ell(\gamma) < \infty$.
		The set of all compact rectifiable curves is denoted by $\Gamma_{\mathrm{rect}} = \Gamma_{\mathrm{rect}}(X, \Lmetric)$.
	\end{enumerate}
\end{defn}
It is known that every compact rectifiable curve $\gamma \colon [a, b] \to X$ admits a (orientation preserving) \emph{arc-length parametrization} $\widetilde{\gamma} \colon [0, \ell(\gamma)] \to X$ that satisfies $\widetilde{\gamma}\bigl(\ell(\restr{\gamma}{[a, t]})\bigr) = \gamma(t)$ for each $t \in [a, b]$ (see \cite[(5.1.6)]{HKST} for example).
\begin{defn}[Line integral on a metric space]
	Let $\gamma \in \Gamma_{\mathrm{rect}}$ be a compact curve and let $\rho \in \mathscr{B}_{+}(X)$.
	The \emph{line integral of $\rho$ over $\gamma$} is defined by
	\begin{equation}
		\int_{\gamma}\rho\,ds \coloneqq \int_{0}^{\ell(\gamma)}\rho\bigl(\widetilde{\gamma}(t)\bigr)\,dt,
	\end{equation}
	where $\widetilde{\gamma}$ is the arc-length parametrization of $\gamma$.
	If $\gamma \in \Gamma_{\mathrm{rect}}$, then we define
	\[
	\int_{\gamma}\rho\,ds \coloneqq \sup\left\{ \int_{\gamma'}\rho\,ds \;\middle|\; \text{$\rho'$ is a compact subcurve of $\gamma$} \right\}.
	\]
\end{defn}

\begin{defn}[Modulus of curve families]
	Let $p \in (0, \infty)$ and let $\Gamma$ be a subset of $\Gamma_{\mathrm{rect}}$.
	A non-negative Borel function $\rho \in \mathscr{B}_{+}(X)$ is said to be \emph{admissible for $\Gamma$} if
	\[
	\inf_{\gamma \in \Gamma}\int_{\gamma}\rho\,ds \ge 1.
	\]
	The $p$-modulus of $\Gamma$ is defined as
	\[
	\MOD_{p}(\Gamma) = \inf\bigl\{ \norm{\rho}_{L^{p}(\Lmeasure)}^{p} \bigm| \text{$\rho$ is admissible for $\Gamma$} \bigr\}.
	\]
	We shall say that a property of curves holds for \emph{$\MOD_{p}$-a.e.\ curve} if the $p$-modulus of the set of curves for which the property fails to hold is zero.
\end{defn}
The corresponding properties to the discrete case in Lemma \ref{lem.basic-pMod} are also true for $p$-modulus on $(X, \Lmetric, \Lmeasure)$  \cite[Section 5.2]{HKST}. The next notion of minimal $p$-weak upper gradient of a function $u$ plays the role of `$\abs{\nabla u}$'. The notion of	 weak upper gradients was   introduced in \cite{HK98}, where it was called `very weak gradients'.
\begin{defn}[Upper gradients] \label{d:ug}
	Let $p \in [1, \infty)$, $u \colon X \to \mathbb{R}$ and $g \in \mathscr{B}_{+}(X)$.
	(Here, both $u$ and $g$ is defined on every points of $X$.)
	The Borel function $g$ is called a \emph{$p$-weak upper gradient of $u$} if
	\begin{equation} \label{e:wugrad}
		\abs{u(x)-u(y)} \le \int_{\gamma} g\,ds \quad \mbox{for $\MOD_{p}$-a.e.\ $\gamma \in \Gamma_{\on{rect}}$,}
	\end{equation}
	where $x, y$ are endpoints of $\gamma$.
	If \eqref{e:wugrad} holds for every compact rectifiable curve, then $g$ is called an \emph{upper gradient of $u$}.

	A $p$-weak upper gradient $g$ of $u$ is said to be a \emph{minimal $p$-weak upper gradient} if it is $p$-integrable with respect to the measure $\Lmeasure$ and if $g \le g'$ $\Lmeasure$-a.e. in $X$ whenever $g'$ is a $p$-integrable $p$-weak upper gradient of $u$.
	Such the minimal $p$-weak upper gradient of $u$ is denoted by $g_{u}$.
\end{defn}
\begin{rmk}\label{rmk:ug.others}
	There are several ways to define upper gradients, and these coincide with each other on any compact metric space equipped with a finite Borel measure. See  \cite[Section 4 and Theorem 7.4]{AGS13} for details in the case of $p \in (1,\infty)$. (Our definition of $g_{u}$ is the same as that of the upper gradient denoted by $\abs{\nabla u}_{S,p}$ in \cite{AGS13}.) 
\end{rmk}

If $\{ g \mid \text{$g$ is a $p$-integrable upper gradient of $u$} \} \neq \emptyset$, then the existence and uniqueness (up to a $\Lmeasure$-null set) of minimal $p$-weak upper gradient are established by a standard argument (so-called the direct method) in calculus of variations (see \cite[Theorem 6.3.20 and Lemma 6.2.8]{HKST}).
We also recall that $\norm{g_{u}}_{L^{p}(\Lmeasure)}^{p}$ is the smallest $L^{p}(X, \Lmeasure)$-norm among all $p$-integrable $p$-weak upper gradient of $u$.
For other basic properties on upper gradients, we refer to \cite{BB,Hei,HKST}.

For a locally Lipschitz function $u \colon X \to \bR$, we define its \emph{lower pointwise Lipschitz constant function} $\on{lip} u\colon X \to [0,\infty)$ as
\begin{equation} \label{e:lipdef}
	\on{lip} u (x) \coloneqq \liminf_{r \downarrow 0} \sup_{y \in B(x,r)} \frac{\abs{u(y)-u(x)}}{r},
\end{equation}
which gives a typical example of upper gradients (see \cite[Lemmas 6.2.5 and 6.2.6]{HKST}).
\begin{prop}\label{p:lip-ug}
	If $u \colon X \to \mathbb{R}$ is a locally Lipschitz function, then $\on{lip}u \in \mathscr{B}_{+}(X)$ is an upper gradient of $u$.
\end{prop}

Now we can define the function spaces $\widetilde{N}^{1,p}$ and $N^{1,p}$, which are called \emph{Newton-Sobolev spaces} and introduced in \cite{Sha00}.
Let $p \in [1, \infty)$ and let
\begin{align}\label{d:N-pw}
	&\widetilde{N}^{1, p}(X, \Lmetric, \Lmeasure) \nonumber \\
	&\coloneqq \Biggl\{ u \colon X \to [-\infty, \infty] \Biggm|
	\begin{array}{c}
    \text{$u$ is $p$-integrable (with respect to $\Lmeasure$) and there} \\
    \text{exists a $p$-integrable $p$-weak upper gradient $g$ of $u$} \\
    \end{array}
	\Biggr\},
\end{align}
which is clearly a vector space (over $\mathbb{R}$).
We equip $\widetilde{N}^{1, p}(X, \Lmetric, \Lmeasure)$ with the seminorm $\norm{\,\cdot\,}_{N^{1,p}(X, \Lmetric, \Lmeasure)}$ given by
\begin{equation}\label{d:N-seminorm}
	\norm{u}_{N^{1,p}(X, \Lmetric, \Lmeasure)} = \norm{u}_{L^{p}(\mu)} + \norm{g_{u}}_{L^{p}(\Lmeasure)}.
\end{equation}
To get a normed   space, we next consider a quotient space of $\widetilde{N}^{1, p}(X, \Lmetric, \Lmeasure)$.
\begin{defn}[Newton-Sobolev space $N^{1, p}$]\label{d:Newton}
	Let $p \in [1, \infty)$.
	For $f, g \in \widetilde{N}^{1,p}(X, \Lmetric, \Lmeasure)$, we define an equivalence relation $f \sim_{N^{1,p}} g$ by $\norm{f - g}_{N^{1,p}(X, \Lmetric, \Lmeasure)} =0$.
	Let us denote the equivalence class of $f$ with respect to $\sim_{N^{1,p}}$ by $[f]_{N^{1,p}}$.
	Define
	\[
	N^{1,p}(X, \Lmetric, \Lmeasure) \coloneqq \widetilde{N}^{1,p}(X, \Lmetric, \Lmeasure)/\sim_{N^{1,p}}.
	\]
	We consider $N^{1,p}(X, \Lmetric, \Lmeasure)$ as a normed space equipped with the quotient norm associated with the seminorm defined in \eqref{d:N-seminorm}, which is also denoted by $\norm{\,\cdot\,}_{N^{1,p}(X, \Lmetric, \Lmeasure)}$.
	We also use $\norm{\,\cdot\,}_{N^{1,p}}$ or $\norm{\,\cdot\,}_{N^{1,p}(\Lmeasure)}$ to denote $\norm{\,\cdot\,}_{N^{1,p}(X, \Lmetric, \Lmeasure)}$.
\end{defn}
For any $p \in [1, \infty)$, $N^{1,p}(X, \Lmetric, \Lmeasure)$ is a Banach space \cite[Theorem 7.3.6]{HKST}.
\begin{rmk} \label{r:np-psc}
	If $(K, \metric, \measure)$ is PSC given in Definition \hyperref[it:PSCbasic]{\ref{defn.PSC}}, then \cite[Proposition 7.1.33]{HKST} implies that $N^{1,p}(K, \metric, \measure)$ is trivial, i.e., $N^{1,p}(K, \metric, \measure) = L^{p}(K, \measure)$.
	This triviality is due to the fact that $\MOD_{p}(\Gamma_{\mathrm{rect}}(K, \metric)) = 0$.
	Such triviality of $1$-modulus is proved by \cite{LP04} and one can find a proof in \cite[Proposition 4.3.3]{MT} for all $p \ge 1$.
\end{rmk}

We recall   Poincar\'e inequalities based on the notion of upper gradient.
\begin{definition} \label{d:lp}
	Let $p \in [1,\infty)$.
	The metric measure space $(X, \Lmetric, \Lmeasure)$ is said to satisfy the \emph{$(p, p)$-Poincar\'e inequality} if there exist $C_{\mathrm{P}} \in (0,\infty), A_{\mathrm{P}} \in [1,\infty)$ such that for any $x \in X, r > 0, u \in \wt{N}^{1,p}(X, \Lmetric, \Lmeasure)$ and for any $p$-weak upper gradient $g$ of $u$, we have
	\begin{equation}\label{d:pp-PI}
		\int_{B_{\Lmetric}(x,r)} \abs{u(y)-u_{B_{\Lmetric}(x,r),\Lmeasure}}^{p}\,\Lmeasure(dy) \le Cr^p\int_{B_{\Lmetric}(x,A_{\mathrm{P}}r)}g^{p}\,d\Lmeasure, \tag{\textup{$(p,p)$-PI$^{\mathrm{ug}}$}}
	\end{equation}
	where $u_{B_{\Lmetric}(x,r), \Lmeasure} = \fint_{B_{\Lmetric}(x,r)}u\,d\Lmeasure$.
	In addition, $(X, \Lmetric, \Lmeasure)$ is said to satisfy the \emph{$(1, p)$-Poincar\'e inequality} (or \emph{$p$-Poincar\'{e} inequality} for short) if for any $x \in X, r > 0, u \in \wt{N}^{1,p}(X, \Lmetric, \Lmeasure)$ and for any $p$-weak upper gradient $g$ of $u$, we have
	\begin{equation}\label{d:p-PI}
		\fint_{B_{\Lmetric}(x,r)} \abs{u(y)-u_{B_{\Lmetric}(x,r),\Lmeasure}}\,\Lmeasure(dy) \le Cr\left(\fint_{B_{\Lmetric}(x,A_{\mathrm{P}}r)}g^{p}\,d\Lmeasure\right)^{1/p}. \tag{\textup{$p$-PI$^{\mathrm{ug}}$}}
	\end{equation}
	The constants $C_{\mathrm{P}}, A_{\mathrm{P}}$ in \ref{d:pp-PI} (resp. \ref{d:p-PI}) are called the \emph{data} of \ref{d:pp-PI} (resp. \ref{d:p-PI}).
	(Here `\textrm{ug}' stands for upper gradient  to distinguish it from Poincar\'e inequality corresponding to energy measures as shown in Proposition \ref{prop:PI.PSC} or Poincar\'e inequality on graphs as shown in Theorem \ref{thm.PI-discrete}.) 
	Note that, by H\"{o}lder's inequality, \ref{d:pp-PI} implies \ref{d:p-PI} without changing the constant $A_{\mathrm{P}}$. 
\end{definition}
\subsection{Lipschitz partition of unity and localized energies}
In this subsection, we provide analogue results in Section \ref{sec.em-local}.
We focus on an upper bound on the ``energy measure'' $g_{f}^{p}\,d\mu$ because we do not use lower bounds in this paper.

We work in the same settings as in the previous section, i.e., $(X, \Lmetric)$ is a separable metric space and $\Lmeasure$ is a locally finite Borel-regular (outer) measure on $X$ which is positive on any non-empty open subset of $X$.
In addition, we let $p \in (1, \infty)$ throughout this subsection.

The following Lipschitz partition of unity is a well-known tool to approximate arbitrary functions in $\wt{N}^{1,p}(X,\Lmetric,\Lmeasure)$ with Lipschitz functions (see \cite[pp. 104--105]{HKST}).
\begin{lem}\label{l:lip}
	Let $(X, \Lmetric)$ be a doubling metric space.
	Let $\{x_i: i\in I\}$ be a maximal $r$-separated subset for some $r>0$.
	Then there exists $C_1>0$ depending only on the doubling constant of $(X, \Lmetric)$ and a collection of $C_1/r$-Lipschitz functions $\varphi_i \colon X \to [0,1]$ such that $\sum_{i \in I} \varphi_i \equiv 1$ and $\supp[\varphi_i] \subset B_{\Lmetric}(x_i,2r_i)$ for all $i \in I$.
\end{lem}

The next lemma provides an estimate for upper gradients of discrete convolutions.
\begin{lem} \label{l:dconv}
	Suppose that $(X, \Lmetric, \Lmeasure)$ is volume doubling.
	Let $\{x_i: i \in I\}$ be a maximal $r$-separated subset of $(X, \Lmetric)$ and let $\{ \varphi_i \}_{i \in I}$ denote a Lipschitz partition of unity satisfying the properties described in Lemma \ref{l:lip}.
	For a $\Lmeasure$-integrable function $u \colon X \to \bR$, define $u_r \colon X \to \bR$ as
 	\begin{equation} \label{e:dconv}
 		u_r(x) \coloneqq \sum_{i \in I} u_{B_{\Lmetric}(x_i,r), \Lmeasure}\varphi_i(x), \quad \mbox{where $\displaystyle u_{B_{\Lmetric}(x_i,r), \Lmeasure} = \fint u\,d\Lmeasure$ for all $i \in I$.}
 	\end{equation}
 	There exists $C>0$ depending only on the doubling constant of $\Lmeasure$ such that
 	\begin{equation}
 		\on{lip} u_r(x) \le C r^{-1} \fint_{B_{\Lmetric}(x,4r)}\abs{u(z)- u_{B_{\Lmetric}(x,4r), \Lmeasure}}\,\Lmeasure(dz) \quad
	\mbox{for all $x \in X$.}
	\end{equation}
\end{lem}
\begin{proof}
	In this proof, we write $u_{B_{\Lmetric}(x,r)} = u_{B_{\Lmetric}(x,r), \Lmeasure}$ for simplicity.
	For any $x,y \in X$ with $\Lmetric(x,y)<r$, we have $\varphi_i(x) \vee \varphi_i(y) \neq 0$ only if $\Lmetric(x_i,x) < 3r$ and therefore $B_{\Lmetric}(x_i,r) \subset B_{\Lmetric}(x,4r)$ whenever $\varphi_i(x) \vee \varphi_i(y) \neq 0$.
	Hence for all $x,y \in X$ such that $\Lmetric(x,y)<r$, we have
 	\begin{align*}
 		\abs{u_r(x)-u_r(y)}
 		&= \abs{ \sum_{i \in I} u_{B_{\Lmetric}(x_i,r)}(\varphi_i(x)-\varphi_i(y))} =  \abs{ \sum_{i \in I} \bigl(u_{B_{\Lmetric}(x_i,r)}-u_{B_{\Lmetric}(x,4r)}\bigr) (\varphi_i(x)-\varphi_i(y))} \\
 		&\le \sum_{i \in I, \Lmetric(x,x_i) < 4r} \abs{\bigl(u_{B_{\Lmetric}(x_i,r)}-u_{B_{\Lmetric}(x,4r)}\bigr) (\varphi_i(x)-\varphi_i(y))} \quad \\
 		&\le C_1 r^{-1} \Lmetric(x,y) \sum_{i \in I, \Lmetric(x,x_i) < 4r} \fint_{B_{\Lmetric}(x_i,r)} \abs{(u(z)-u_{B_{\Lmetric}(x,4r)})}\,\Lmeasure(dz) \\
 		&\le C_2 r^{-1} \Lmetric(x,y) \fint_{B_{\Lmetric}(x,4r)}\abs{(u(z)-u_{B_{\Lmetric}(x,4r)})}\,\Lmeasure(dz).
 	\end{align*}
 	In the second and third line, we used Lemma \ref{l:lip}.
 	In the last line, we used the fact that $\Lmeasure$ is a doubling measure and that the set of $\#\{i \in I \mid \Lmetric(x_i,x)<4r\}$ is bounded by a constant that depends only on the doubling constant of $(X,\Lmetric)$.
 \end{proof}

 It is well known that the $p$-energy of a function in $\wt{N}^{1,p}(X,\Lmetric,\Lmeasure)$ is bounded from above by a Koreervaar-Schoen type energy.
 We say that a function $u \colon X \to \bR$ belongs to the \emph{Korevaar-Schoen-Sobolev space} $KS^{1,p}(X,\Lmetric,\Lmeasure)$ if $u \in L^p(X,\Lmeasure)$ and
 \[
 \limsup_{\epsilon \downarrow 0} \int_X \epsilon^{-p}\fint_{B_{\Lmetric}(x,\epsilon)} \abs{u(y)-u(x)}^p\,\Lmeasure(dy)\,\Lmeasure(dx)<\infty.
 \]
 It is know that $KS^{1,p}(X,\Lmetric,\Lmeasure) = B_{p,\infty}^{1}(X,\Lmetric,\Lmeasure)$; see \cite[Lemma 3.2]{Bau24}. 

In the following proposition, we control the $L^p$-norm of the minimal $p$-weak upper gradient on arbitrary sets using a Korevaar--Schoen type energy. The statement and its proof  is a slight extension of that of \cite[Theorem 10.4.3]{HKST} which deals with the case $B=X$.
\begin{prop} \label{p:ks}
	Let $(X,\Lmetric,\Lmeasure)$ be volume doubling.
	There exists $C>0$ such that for all $u \in KS^{1,p}(X,\Lmetric,\Lmeasure)$, there exists $\wt{u} \in \wt{N}^{1,p}(X,\Lmetric,\Lmeasure)$ such that $\wt{u}=u$ $\Lmeasure$-almost everywhere and such that its minimal $p$-weak upper gradient $g_{\wt{u}}$ satisfies, for any Borel set $B \subseteq X$,
	\begin{equation} \label{e:ksn}
		\int_{B} g_{\wt{u}}^p\,d\Lmeasure \le C \limsup_{\varepsilon \downarrow 0} \int_{B} \varepsilon^{-p} \fint_{B_{\Lmetric}(y,\varepsilon)} \abs{u(y)-u(x)}^p\,\Lmeasure(dy)\,\Lmeasure(dx).
	\end{equation}
\end{prop}
\begin{proof}
	For each $n\in \bN$, consider a maximal $n^{-1}$-separated subset of $(X,\Lmetric)$ and the corresponding Lipschitz partition of unity as given in Lemma \ref{l:lip}.
	Let $v_n \coloneqq u_{n^{-1}}$ denote the function defined in \eqref{e:dconv}. Then by  \cite[Proof of Theorem 10.4.3]{HKST}, we have $\lim_{n \to \infty}\int_{X} \abs{v_n-u}^p\,d\Lmeasure =0$ and, by Lemma \ref{l:dconv} and Jensen's inequality, there exists $C_1 > 0$ depending only on $p$ and the doubling constant of $\Lmeasure$ such that
	\begin{equation} \label{e:ksn1}
		\varlimsup_{n \to \infty} \int_{A} \on{lip}v_n (x)^p \, \Lmeasure(dx) \le C_1 \varlimsup_{\varepsilon \downarrow 0}\int_{A} \varepsilon^{-p}\fint_{B_{\Lmetric}(x,\varepsilon)} \abs{u(y)-u(x)}^p\,\Lmeasure(dy)\,\Lmeasure(dx) < \infty, 
	\end{equation}
	for any Borel set $A$ of $X$. 
	Hence $\{ v_{n} \}_{n \in \mathbb{N}}$ is bounded in $\wt{N}^{1,p}(X,\Lmetric,\Lmeasure)$.
	Therefore by Mazur's lemma and \cite[Proposition 7.3.7, Theorem 7.3.8]{HKST}, there exists $\wt{u} \in \wt{N}^{1,p}(X,\Lmetric,\Lmeasure)$ such that $\wt{u}=u$ $\Lmeasure$-almost everywhere and $g \in \mathscr{B}_{+}(X)$ satisfies the following properties.
	The function $g$ is a $p$-weak upper gradient of $\wt{u}$ and is a limit in $L^p(X,\Lmeasure)$ of a sequence $\{ g_j \}_{j \in \mathbb{N}}$ such that $g_j$ is a convex combination of elements in the sequence $\{ \on{lip} v_j \}_{j \in \bN}$ for all $j$ and for any $n \in \bN$ all but finitely many elements of $g_j$ are finite convex combinations of  $\on{lip} v_j$ with $j \ge n$.
	Hence by Lemma \ref{l:dconv}, we conclude
	\begin{align*}
	\int_{B} g_{\wt{u}}^p\,d\Lmeasure
	&\le \int_B g^p \,d\Lmeasure \le \limsup_{n \to \infty} \int_B (\on{lip} v_n)^p\,d\Lmeasure\\
	&\stackrel{\eqref{e:ksn1}}{\le}  C \limsup_{\varepsilon \downarrow 0} \int_{B} \varepsilon^{-p} \fint_{B_{\Lmetric}(y,\varepsilon)} \abs{u(y)-u(x)}^p\,\Lmeasure(dy)\,\Lmeasure(dx).
	\end{align*}
\end{proof}

\subsection{Loewner metric and measure}
Let us recall the definition of Loewner spaces. 
\begin{defn}[Loewner space; {\cite[Definition 1.8]{CE}}] \label{d:loewner}
	Let $p \in (1, \infty)$ and let $(X, \Lmetric, \Lmeasure)$ be a metric measure space such that it is metric doubling.
	The metric measure space $(X, \Lmetric, \Lmeasure)$ is said to be \emph{$p$-Loewner} if $\Lmeasure$ is $p$-Ahlfors regular with respect to $\Lmetric$ and $p$-Poincar\'{e} inequality \ref{d:p-PI} holds.
	If $(X, \Lmetric, \Lmeasure)$ is $p$-Loewner for some $p \in (1, \infty)$, then $\Lmetric$ is called a Loewner metric and $\Lmeasure$ is called a Loewner measure.
\end{defn}

	The original definition of   \emph{Loewner spaces} due to Heinonen and Koskela \cite[Definition 3.1]{HK98} is based on lower bounds on modulus.
	However, this gives an equivalent one by virtue of \cite[Theorems 5.7 and 5.12]{HK98}. This celebrated work identified Loewner spaces as the abstract setting where much of the nice properties of quasiconformal maps on Euclidean spaces are available.

	The next result is an observation due to Cheeger and Eriksson-Bique \cite{CE}. It states that, for a metric measure space satisfying the combinatorial Loewner property, any metric and measure attaining the Ahlfors regular conformal dimension yields a Loewner space. We recall this short argument as it plays a key role in rest of this section.
	\begin{prop} [{\cite[\textsection 1.6]{CE}}] \label{p:loewner}
		Let $(K, \metric, \measure)$ be the planar Sierpi\'{n}ski carpet in Definition \hyperref[it:PSCbasic]{\ref{defn.PSC}}.
		Suppose that the Ahlfors regular conformal dimension of $(K, \metric, \measure)$ ($\dim_{\on{ARC}}$ for short) is attained, i.e., there exists a metric $\Lmetric \in \mathcal{J}(K, \metric)$ equipped with a $\dim_{\on{ARC}}$-Ahlfors regular measure $\Lmeasure$ with respect to $\Lmetric$.
		Then $(K, \Lmetric, \Lmeasure)$ is a $\dim_{\on{ARC}}$-Loewner space. 
	\end{prop}
	\begin{proof}
		This result follows from the $\dim_{\on{ARC}}$-combinatorial Loewner property of PSC, which is proved in \cite[Theorem 4.1]{BK13}.
		As explained in \cite[\textsection 1.6]{CE}, $\dim_{\on{ARC}}$-combinatorial Loewner property along with $\dim_{\on{ARC}}$-Ahlfors regularity implies $\dim_{\on{ARC}}$-Loewner property in the sense of \cite[(3.2)]{HK98}.
		This is due to a result of Ha\"issinky \cite[Proposition B.2]{Hai09} comparing combinatorial and continuous versions of modulus and a different equivalent definition of the Loewner property in Heinonen and Koskela's celebrated work \cite[Definition 3.1, Theorems 5.12 and 5.7]{HK98}. 
	\end{proof}

Recall from Definition \ref{d:cgauge} that the Ahlfors regular conformal dimension concerns the existence of a metric $\theta \in \sJ(X,d)$ and $p$-Ahlfors regular measure  on $(X,\theta)$. It is well known that the measures and metrics satisfying these conditions determine each other; that is $\mu$ can be recovered from $\theta$ and $\theta$ can be recovered from $\mu$ (up to a bounded multiplicative constant). We recall this in Lemmas \ref{l:met-meas} and \ref{l:meas-met}.
\begin{lem}\label{l:met-meas}
	Let $p \in (1, \infty)$ and let $(X, \Lmetric, \Lmeasure)$ be a metric measure space.
	If $\Lmeasure$ is $p$-Ahlfors regular with respect to $\Lmetric$, then there exists a constant $C \ge 1$ (depending only on $p$ and the doubling constant of $\Lmetric$) such that
	\begin{equation}\label{e:hauscompl}
		C^{-1}\mathscr{H}^{p}_{\Lmetric}(B) \le \Lmeasure(B) \le C\mathscr{H}^{p}_{\Lmetric}(B) \quad \text{for all Borel set $B \in \mathcal{B}(X)$,}
	\end{equation}
	where $\mathscr{H}^{p}_{\Lmetric}$ denotes the $p$-dimensional Hausdorff measure with respect to the metric $\Lmetric$.
\end{lem}
We also note that, by Lemma \ref{l:met-meas}, the Ahlfors regularity can be regarded as a property on metrics (and the corresponding Hausdorff measures).

Conversely, David--Semmes deformation theory (\cite{DS90} for example) allows us to construct a corresponding metric associated to a given Ahlfors regular measure $\mu$ that is bi-Lipschitz equivalent to the original Loewner metric.
See also \cite[Chapter 14]{Hei} or \cite[Section 7.1]{MT}. To describe this we recall the definition of a maximal semi-metric.
\begin{definition}
	A function $r:X \times X \to [0,\infty)$ is said to be a \emph{semi-metric}, if it satisfies all the properties of a metric except possibly the property that $r(x,y)=0$ implies $x=y$.

	Let $h: X \times X \to [0,\infty)$ be  an arbitrary function. Then there exists a unique maximal semi-metric $d_h:X \times X \to [0,\infty)$ such that $d_h(x,y)\le h(x,y)$ for all $x,y \in X$ \cite[Lemma 3.1.23]{BBI}. We say that $d_h$ is the \emph{maximal semi-metric induced by $h$}.
	More concretely, $d_h$ can be defined as follows. Let $\wt h(x,y)= \min(h(x,y),h(y,x))$. Then
	\be \label{e:maxmetric}
	d_h(x,y) = \inf \set{ \sum_{i=0}^{N-1} \wt h(x_i,x_{i+1}): N \in \bN, x_0=x, x_N=y}.
	\ee
\end{definition}
The following lemma follows easily from the definitions.
\begin{lem}\label{l:meas-met}
	Let $p \in (1, \infty)$ and let $(X,d)$ be a metric measure space.  If $\theta \in \sJ(X,d)$ and $\mu$ be a measure such that $\mu$ is $p$-Ahlfors regular with respect to $\Lmetric$. Let $h(x,y):= \mu(B_{\metric}(x,d(x,y)))^{1/p}$ for all $x,y \in X$ and let $d_h$ denote the maximal semi-metric. Then $d_h$ is bi-Lipschitz equivalent to $\theta$, that is, there exists $C>1$ such that
	\[
	C^{-1}\theta (x,y) \le d_h(x,y) \le C\theta(x,y) \quad \text{for all $x, y \in X$.}
	\]
	In particular $d_h \in \sJ(X,d)$ and $\mu$ is $p$-Ahlfors regular on $(X,d_h)$.
\end{lem}

The key point of the above lemma is that the definition of $d_h$ depends only on the measure $\Lmeasure$ and $\metric$ and not on $\Lmetric$. Nevertheless, as the conclusion shows $d_h$ is bi-Lipschitz equivalent to $\Lmetric$. The proof relies on the observation that $\mu$ being a doubling measure and $\Lmetric \in \sJ(X,d)$ imply that $\mu(B_{\metric}(x,\metric(x,y)))$ is comparable to $\mu(B_{\Lmetric}(x,\Lmetric(x,y)))$. 

In the rest of this paper, we discuss the structures of metrics and measures that attain the Ahlfors regular conformal dimension of the Sierpi\'{n}ski carpet if exist.
In view of Lemma \ref{l:meas-met}, we focus on optimal measures.
We introduce the standing framework in the remaining part:
\begin{assum}\label{a:attain}
	Let $(K, \metric, \measure)$ be the planar Sierpi\'{n}ski carpet in Definition \hyperref[it:PSCbasic]{\ref{defn.PSC}}.
	Let $\hdim = \log{8}/\log{3}$ and $p = \dim_{\on{ARC}}(K, \metric, \measure)$.
	We suppose the attainment of $\dim_{\on{ARC}}(K, \metric, \measure)$.
	Let $\Lmetric \in \mathcal{J}(K, \metric)$ and let $\Lmeasure$ be a Borel-regular measure on $K$ such that $\Lmeasure$ is $p$-Ahlfors regular with respect to $\Lmetric$.
\end{assum}
\begin{rmk}\label{rmk.ARC-SC}
	By the results of \cite{KL04,Tys00} (see also \cite[Section 4.3]{MT} for a review of related results), we know that
	\begin{equation}\label{dARC-strict}
	1 < 1 + \frac{\log{2}}{\log{3}} \le p = \dim_{\on{ARC}}(K, \metric, \measure) < \hdim.
	\end{equation}
	Also, by \cite[Theorem 4.7.6]{Kig20}, we have $\pwalk = \hdim$.
\end{rmk}

B.\ Kleiner \cite{Kle+} observed that any optimal measure $\mu$ is mutually singular to the self-similar measure $m$. Although we don't need this fact,
it helps us to elucidate that the comparison of norms on Theorem \ref{t:attain}\ref{attain-core} does not follow comparison of corresponding semi-norms as the $L^p(\measure)$ and $L^p(\Lmeasure)$ norms are not comparable.
\begin{prop}[due to Bruce Kleiner]\label{prop.singular}
	Under Assumption \ref{a:attain},
	the measures $\measure$ and $\Lmeasure$ are mutually singular.
\end{prop}
\begin{proof}
 This proof by contradiction uses a `blow-up' argument. Assume to the contrary that $\mu$ is not singular to $m$. Let $\mu=\mu_a+\mu_s$ denote the Lebesgue decomposition of $\mu$ with respect to $m$, where $\mu_a \ll m$, $\mu_s \perp m$ and $\mu_a \neq 0$ by assumption. Let $f= \frac{d\mu_a}{dm}$. For $m$-almost every $x \in K$, we have (\cite[Proposition A.4]{KM20})
 \begin{equation} \label{e:si1}
	\lim_{r \downarrow 0} \frac{\mu_s(B_d(x,r))}{m(B_d(x,r))}=0
\end{equation}
 and for
  $m$-almost every $x \in \{y \in K: f(y)>0\}$, we have  (\cite[(2.8)]{Hei})
 \begin{equation} \label{e:si2}
 	\lim_{r \downarrow 0}  \frac{1}{m(B_d(x,r))} \int_{B_d(x,r)} \abs{f(y)-f(x)}\,m(dy)=0.
 \end{equation}
  Since $\mu_a\neq 0$, there exists $x \in  \{y \in K: f(y)>0\}$ such that both \eqref{e:si1} and \eqref{e:si2} hold. Pick $\omega \in \Sigma$ such that $\chi(\omega)=x$ and set $w_n:=[\omega]_n \in W_n$ for all $n \in \bN$.
Define a sequence of probability measures $\mu_n$ and metrics $\theta_n \colon K \times K \to [0,\infty)$ as
\[
\mu_n(A)\coloneqq \frac{\mu(F_{w_n}(A))}{\mu(K_{w_n})}, \quad \theta_n(x,y) \coloneqq \frac{\theta(F_{w_n}(x),F_{w_n}(y))}{\diam(K_{w_n},\theta)}, \quad \mbox{for all $n \in \bN$,}
\]
where $\theta \in \sJ(K,d)$ is such that $\mu$ is $p$-Ahlfors regular in $(K,\theta)$ and $p$ is as given in Assumption \ref{a:attain}. By \eqref{e:si1} and \eqref{e:si2}, the sequence of measures $\mu_n$ converges to $\frac{f(x)}{f(x)}m = m$ in the topology of weak convergence. Furthermore, it is easy to verify that there exists a homeomorphism $\eta: [0,\infty) \to [0,\infty)$ such that the identity map $\on{Id}:(K,\theta_n) \to (K,d)$ is an $\eta$-quasisymmetry for all $n \in \bN$. By the same argument as \cite[Proof of Proposition 6.18]{KM23} using Arzela--Ascoli theorem, there exists a subsequence $\{\theta_{n_k}\}_{k \in \bN}$ of $\{\theta_n\}_{n \in \bN}$ converging uniformly to $\wt\theta \in C(K \times K)$. This along with $\diam(K,\theta_n)=1$ implies that $\wt{\theta}$ is a metric on $K$, $\on{Id}:(K,\wt \theta) \to (K,d)$ is a $\eta$-quasisymmetry and hence $\wt \theta \in \sJ(K,d)$. This implies that the measure $m$ is $p$-Ahlfors regular in $(K,\wt \theta)$. Therefore by Lemma \ref{l:meas-met}, we obtain $p=\hdim$ which contradicts \eqref{dARC-strict}.
\end{proof}

\subsection{Identifying self-similar and Newtonian Sobolev spaces}
In this subsection, we will compare different notions of energies ($\mathcal{E}_{p}(f)$ and $\int_{K}g_{f}^{p}\,d\Lmeasure$) and Sobolev spaces ($\mathcal{F}_{p}$ and $N^{1,p}$) on the Sierpi\'{n}ski carpet under assuming the attainment of its Ahlfors regular conformal dimension.
Throughout this subsection, we always suppose Assumption \ref{a:attain}.\footnote{We clarify this assumption in all statements where the attainment is used because whether this assumptions is true or not is a big open problem in the field.}

We recall the following two different Poincar\'e inequalities.
 \begin{theorem}
	There exist $C,A >1$ such that for all $x \in K, r>0$, we have
	\begin{align} \label{e:pil}
		\int_{B_{\Lmetric}(x,r)} \abs{f- f_{B_{\Lmetric}(x, r), \Lmeasure}}^p \,d\Lmeasure &\le C r^{p} \int_{B_{\Lmetric}(x,Ar)} g_f^p\,d\Lmeasure \quad \mbox{for all $f \in  N^{1,p}(K,\Lmetric,\Lmeasure)$,} \\
		\int_{B_{\metric}(x,r)} \abs{f-f_{B_{\metric}(x, r), \measure}}^p \, \,d\measure &\le C r^{\hdim} \Gamma_{p} \langle f \rangle(B_{\metric}(x,Ar)) \quad \mbox{for all $f \in   \sF_p(K, \metric, \measure)$,} \label{e:pie}
	\end{align}
	where $g_f$ is the minimal $p$-weak upper gradient of $f$. 
\end{theorem}
\begin{proof}
	The first one \eqref{e:pil} follows from Proposition \ref{p:loewner} and \cite[Theorem 9.1.2]{HKST}.
	The second one \eqref{e:pie} follows from Proposition \ref{prop:PI.PSC} (see also Remark \ref{rmk.ARC-SC}).
\end{proof}

The following is a two-weight Poincar\'e type inequality, which is the key ingredient to compare two different worlds (self-similar and Loewner).
\begin{prop} \label{p:2wt}
	Suppose Assumption \ref{a:attain}.
 	There exist $C,A >1$ such that for all $x \in K, r>0$, we have
 	\begin{align} \label{e:piel}
 		\inf_{\alpha \in \bR} \int_{B_{\metric}(x,r)} \abs{f- \alpha}^p \,d\measure &\le C r^{\hdim} \int_{B_{\metric}(x,Ar)} g_f^p\,d\Lmeasure \quad \mbox{for all $f \in N^{1,p}(K,\Lmetric,\Lmeasure) \cap \contfunc(K)$,} \\
 		\inf_{\alpha \in \bR} \int_{B_{\Lmetric}(x,r)} \abs{f-\alpha}^p \, \,d\Lmeasure &\le C r^p \,\Gamma_{p} \langle f \rangle(B_{\Lmetric}(x,Ar)) \quad \mbox{for all $f \in \sF_p(K, \metric, \measure) \cap \contfunc(K)$.} \label{e:pile}
 	\end{align}
\end{prop}
\begin{proof}
	In this proof, each function in $N^{1,p}(K,\Lmetric,\Lmeasure) \cap \contfunc(K)$ (or $\sF_p(K, \metric, \measure) \cap \contfunc(K)$) is considered as a pointwisely defined continuous function on $K$.
	Fix $p_1 \in (p,\infty)$.
	To prove \eqref{e:piel}, by \cite[Lemma 4.22]{Hei} and $\hdim$-Ahlfors regularity of $(K,\metric,\measure)$, it suffices to show the following weak type estimate:
	There exist $C_1,A_1\in(1,\infty)$ such that
	\begin{equation}
		\label{e:wtype}
		\adjustlimits\inf_{\alpha \in \bR} \sup_{t >0} t^{p_1} m\left( \{y \in B_d(x,r): \abs{f(y)-\alpha}>t \} \right) \le C_1 r^{\hdim} \left(\int_{B_d(x,A_1 r)} g_f^p \,d\mu\right)^{p_1/p}
	\end{equation}
	for all $f \in N^{1,p}(K,\Lmetric,\Lmeasure) \cap \contfunc(K)$.

	Note that by Proposition \ref{p:loewner}, the space $(X,\theta,\mu)$ satisfies the Poincar\'{e} inequality \ref{d:p-PI}. 
	Let $A_{\mathrm{P}} \in [1,\infty)$ denote the constant in \ref{d:p-PI} as given in Definition \ref{d:lp}.
	Since $\Lmetric \in \sJ(K,\metric)$, by \cite[Lemma 1.2.18]{MT}, there exists $A \in (1,\infty)$ such that for all $x \in K, r>0$ , there exists $s\in (0,\diam(K,\theta)]$ satisfying
	\begin{equation} \label{e:2w1}
		B_{\metric}(x,r) \subset B_{\Lmetric}(x,s) \subset B_{\Lmetric}(x,(1+2A_{\mathrm{P}})s) \subset B_{\metric}(x,Ar).
	\end{equation}
	By \ref{d:p-PI} and $p$-Ahlfors regularity of $(K,\Lmetric,\Lmeasure)$, there exists $C_2>1$ such that for all $x \in K, s>0, y \in B_{\Lmetric}(x,s), f \in \wt{N}^{1,p}(K,\Lmetric,\Lmeasure)$, we have
	\begin{align} \label{e:2w2}
		\abs{\fint_{B_{\Lmetric}(y,s)} f \,d\Lmeasure - \fint_{B_{\Lmetric}(x,2s)}f\,d\Lmeasure} & \le  \frac{1}{\Lmeasure(B_{\Lmetric}(y,s))}\int_{B_{\Lmetric}(x,2s)} \abs{f- \fint_{B_{\Lmetric}(x,2s)}f\,d\Lmeasure} \, d\Lmeasure \nonumber \\
		&\le C_2 \left( \int_{B_{\Lmetric}(x,2A_{\mathrm{P}}s)} g_f^p\,d\Lmeasure \right)^{1/p}. 
	\end{align} 
	By a similar argument, there exists $C_3>1$ such that for all $x \in K, s>0, y \in B_{\Lmetric}(x,s), i \in \bZ_{\ge 0}, f \in \wt{N}^{1,p}(K,\Lmetric,\Lmeasure)$, we have
	\begin{equation} \label{e:2w3}
		\abs{ \fint_{B_{\Lmetric}(y,2^{-i}s)}  f \,d\Lmeasure- \fint_{B_{\Lmetric}(y,2^{-i-1}s)} f \,d\Lmeasure} \le C_3 \left( \int_{B_{\Lmetric}(y,A_{\mathrm{P}} 2^{-i}s)} g_f^p \,d\Lmeasure\right)^{1/p}.
	\end{equation}
	Note that $(K, \Lmetric)$ is connected since $(K, \Lmetric)$ is homeomorphic to $(K, \metric)$.
	By the reverse doubling property \cite[Exercise 13.1]{Hei} of $\measure$ with respect to the metric $\Lmetric$, there exists $c_4 \in (0,1)$ such that for all $y \in K, s\in (0,\diam(K,\theta)]$, we have
	\begin{equation} \label{e:2w4}
		c_4 \sum_{i=0}^\infty \left(\frac{\measure(B_{\Lmetric}(y,2^{-i}s))}{\measure(B_{\Lmetric}(y,s))}\right)^{1/p_1} < \frac{1}{2}.
	\end{equation}
	In order to show \eqref{e:wtype}, for any $f \in N^{1,p}(K,\Lmetric,\Lmeasure) \cap \contfunc(K)$, we choose $\alpha= \fint_{B_{\Lmetric}(x,2s)} f\,d\Lmeasure$.
	If $t \le 2 C_2 \left(\int_{B_{\metric}(x,Ar)} g_f^p\,d\Lmeasure \right)^{1/p}$, the estimate \eqref{e:wtype} follows from the $\hdim$-Ahlfors regularity of $(K,\metric,\measure)$.
	Therefore, it suffices to consider the case $t>2 C_2 \left(\int_{B_{\metric}(x,Ar)} g_f^p\,d\Lmeasure \right)^{1/p}$.
	By \eqref{e:2w1} and \eqref{e:2w2}, we have
	\begin{equation} \label{e:2w5}
		\bigl\{y \in B_{\metric}(x,r) : \abs{f(y)-\alpha}>t \bigr\}  \subset 	\Biggl\{y \in B_{\metric}(x,r) : \abs{f(y)- \fint_{B_{\Lmetric}(y,s)} f\,d\Lmeasure}> t/2 \Biggr\}
	\end{equation}
	for all $t>2 C_2 \left(\int_{B_{\metric}(x,Ar)} g_f^p\,d\Lmeasure \right)^{1/p}$.
	By \eqref{e:2w5}, for any $y \in B_{\metric}(x,r)$ such that $\abs{f(y)-\fint_{B_{\Lmetric}(x,2s)} f\,d\Lmeasure} > t > 2 C_2 \left(\int_{B_{\metric}(x,Ar)} g_f^p\,d\Lmeasure \right)^{1/p}$, we have
	\begin{align}
		c_4 \sum_{i=0}^\infty \left(\frac{\measure(B_{\Lmetric}(y,5A_{\mathrm{P}}2^{-i}s))}{\measure(B_{\Lmetric}(y,5 A_{\mathrm{P}}s))}\right)^{1/p_1} t &< t/2 \qquad \mbox{(by \eqref{e:2w4})} \nonumber \\
		&< \abs{f(y)-\fint_{B_{\Lmetric}(y,s)} f\,d\Lmeasure} \qquad \mbox{(by \eqref{e:2w5})} \nonumber \\
		&\le C_3 \sum_{i=0}^\infty \left( \int_{B_{\Lmetric}(y,A_{\mathrm{P}}2^{-i}s)} g_f^p \,d\Lmeasure\right)^{1/p} \qquad \mbox{(by \eqref{e:2w3}).} \nonumber
	\end{align}
 	Therefore there exists $C_5 > 1$ such that following property holds:
 	For each $y \in  B_{\metric}(x,r)$ that satisfies $\abs{f(y)-\fint_{B_{\Lmetric}(x,2s)} f\,d\Lmeasure} > t > 2 C_2 \left(\int_{B_{\metric}(x,Ar)} g_f^p\,d\Lmeasure \right)^{1/p}$, there exists $i_y \in \bZ_{\ge 0}$ such that
	\begin{equation} \label{e:2w6}
 		\measure(B_{\Lmetric}(y, 5 A_{\mathrm{P}}2^{-i_y}s)) \le C_5 t^{-p_1} r^{\hdim}  \left(\int_{B_{\Lmetric}(y,A_{\mathrm{P}}2^{-i_y}s)} g_f^p \,d\Lmeasure\right)^{p_1/p}.
 	\end{equation}
 	By the $5B$ covering lemma \cite[Theorem 1.2]{Hei}, there exists a pairwise disjoint collection of balls $\{B_{\Lmetric}(y_j,A_{\mathrm{P}}2^{-i_{y_j}}s) \mid j \in J\}$ with $y_j \in B_d(x,r)$ for all $j \in J$ such that
 	\[
 	 \left\{y \in B_{\metric}(x,r) : \abs{f(y)- \fint_{B_{\Lmetric}(x,2s)} f\,d\Lmeasure} > t \right\} \subseteq \bigcup_{j \in J}B_{\Lmetric}(y_j, 5A_{\mathrm{P}}2^{-i_{y_j}}s).
 	\]
 	Hence
 	\begin{align*}
 		\lefteqn{	\measure\left(\left\{y \in B_{\metric}(x,r) : \abs{f(y)- \fint_{B_{\Lmetric}(x,2s)} f\,d\Lmeasure} > t \right\}\right)}\\ &\le \sum_{j \in J} \measure\left(B_{\Lmetric}(y_j, 5A_{\mathrm{P}}2^{-i_{y_j}}s)\right) \\
 		&\stackrel{\eqref{e:2w6}}{\le} C_5 t^{-p_1} r^{\hdim}  \sum_{j \in J} \left(\int_{B_{\Lmetric}(y_j,A_{\mathrm{P}}2^{-i_{y_j}}s)} g_f^p \,d\Lmeasure \right)^{p_1/p} \\
 	 & \le   C_5 t^{-p_1} r^{\hdim}  \left(\sum_{j \in J} \int_{B_{\Lmetric}(y_j,A_{\mathrm{P}}2^{-i_{y_j}}s)} g_f^p \,d\Lmeasure \right)^{p_1/p} \quad \mbox{(since $p_1>p$)}  \\
 		&\le C_5 t^{-p_1} r^{\hdim}  \left(\int_{B_{\Lmetric}(x,(1+A_{\mathrm{P}})s)} g_f^p \,d\Lmeasure\right)^{p_1/p} \\
 		&\stackrel{\eqref{e:2w1}}{\le} C_5 t^{-p_1} r^{\hdim} \left(\int_{B_{\metric}(x,Ar)} g_f^p \,d\Lmeasure\right)^{p_1/p},
 	\end{align*}
 	which concludes the proof of \eqref{e:wtype} and therefore \eqref{e:piel}.

 	The proof of \eqref{e:pile} follows from a similar argument where the application of \ref{d:p-PI} in $(K,\Lmetric,\Lmeasure)$ is replaced with \eqref{1p-PI} (with $\beta = \pwalk = \hdim$), which is the $(1,p)$-Poincar\'{e} inequality for the self-similar energy on $(K,\metric,\measure)$.
\end{proof}

The following result compares energy measures and energies in the Sobolev spaces.
\begin{theorem} \label{t:compenergy}
	Suppose Assumption \ref{a:attain}.
	Then we have
	\[
	\mathcal{F}_{p}(K, \metric, \measure) \cap \contfunc(K) =  N^{1,p}(K,\Lmetric,\Lmeasure) \cap \contfunc(K).
	\]
	We let $\mathscr{C}_{p} \coloneqq \mathcal{F}_{p}(K, \metric, \measure) \cap \contfunc(K)$.
 	In addition, there exists $C>1$ such that for any Borel set $B \in \mathcal{B}(K)$ and for all $f \in \mathscr{C}_{p}$, we have
	\begin{equation} \label{e:emeascomp}
		C^{-1} \Gamma_{p} \langle f \rangle(B) \le \int_{B} g_f^p \,d\mu \le C \Gamma_{p} \langle f \rangle(B),
	\end{equation}
	where $g_f$ denotes the minimal $p$-weak upper gradient of $f$.
	In particular,
	\begin{equation} \label{e:ecomp}
		C^{-1} \mathcal{E}_{p}(f) \le  \int_{K} g_f^p \,d\Lmeasure \le C \mathcal{E}_{p}(f) \quad \mbox{for all $f \in \mathscr{C}_{p}$.}
	\end{equation}
 	Furthermore, there exists $C_1>0$ such that
	\begin{equation} \label{e:e1comp}
		C_1^{-1}\norm{f}_{N^{1,p}(\Lmeasure)} \le \norm{f}_{\sF_p} \le C_{1}\norm{f}_{N^{1,p}(\Lmeasure)}
		\quad \mbox{for all $f \in \mathscr{C}_{p}$.}
	\end{equation}
\end{theorem}

We start with a simpler condition to obtain comparability of measures whose proof is in Appendix \ref{sec.whitney}.
\begin{lem} \label{l:ac}
	Let $(X,\mathsf{d})$ be a doubling metric space.
	Let $\nu_1,\nu_2$ be two finite Borel measures on $X$ satisfying the following property:
	There exist $C_1 \in (0,\infty), A_1 \in (1,\infty)$ such that for all $x \in X, r > 0$, we have
	\[
	\nu_1(B_{\mathsf{d}}(x,r)) \le C_1 \nu_2(B_{\mathsf{d}}(x,A_1r)).
	\]
	Then there exists $C_2>0$ such that
	\begin{equation} \label{e:qac}
		\nu_1(B) \le C_2\, \nu_2(B)
	\end{equation}
	for all Borel set $B \subset X$.
\end{lem}

Next we compare energy measures on balls for the spaces $N^{1,p}(K,\Lmetric,\Lmeasure)$ and $\sF_p(K,\metric,\measure)$.
\begin{lem} \label{l:mcomp}
	 Suppose Assumption \ref{a:attain}.
	 Then the following are true:
	 \begin{enumerate}[\rm(i)]
	 	\item\label{it:mcomp1} We have $\sF_p(K, \metric, \measure) \cap \contfunc(K) \subseteq N^{1,p}(K,\Lmetric,\Lmeasure) \cap \contfunc(K)$.
	 		Moreover, there exist $C>0,A >1$ such that for all $f \in \sF_p(K, \metric, \measure) \cap \contfunc(K), x \in K, r>0$, we have
	 		\begin{equation} \label{e:ecomp1}
	  			\int_{B_{\Lmetric}(x,r)} g_f^p \,d\Lmeasure  \le C	\Gamma_p \langle f \rangle(B_{\Lmetric}(x,Ar)).
	 		\end{equation}
	 	\item\label{it:mcomp2}  We have $N^{1,p}(K,\Lmetric,\Lmeasure) \cap \contfunc(K) \subseteq \sF_p(K, \metric, \measure) \cap \contfunc(K)$.
	 		Moreover, there exist $C>0,A >1$ such that for all $f \in \sF_p(K, \metric, \measure) \cap \contfunc(K), x \in K, r>0$, we have
	 		\begin{equation} \label{e:ecomp2}
	 			\Gamma_p \langle f \rangle(B_{\metric}(x,r)) \le C \int_{B_{\metric}(x,Ar)} g_f^p \,d\Lmeasure.
	 		\end{equation}
	 \end{enumerate}
\end{lem}
\begin{proof}
	\ref{it:mcomp1} We will start with the proof of \eqref{e:ecomp1}.
	To this end, let $f \in \sF_p(K, \metric, \measure) \cap \contfunc(K), x \in K, r>0$ be arbitrary.
	For $0<s<r$, consider a maximal $s$-separated subset $N$ of $B_{\Lmetric}(x,r)$ in $(K,\Lmetric)$, so that $B_{\Lmetric}(x,r) \subseteq \cup_{y \in N}B_{\Lmetric}(y,s) \subseteq B_{\Lmetric}(x,r+s)$.
	Therefore
	\begin{equation} \label{e:ec1}
		\one_{B_{\Lmetric}(x,r)}(y)	\one_{B_{\Lmetric}(y,s)}(z) \le \sum_{n \in N} \one_{B_{\Lmetric}(n,2s)}(y)   \one_{B_{\Lmetric}(n,2s)}(z).
	\end{equation}
	By the doubling property and \cite[Lemma 4.1.12]{HKST}, for any $\lambda > 1$, there exists $C_\lambda$ depending only on $\lambda$ and the doubling constant of $(K,\Lmetric)$ such that
	\begin{equation} \label{e:ec2}
		\sum_{n \in N} \one_{B_{\Lmetric}(n,\lambda s)} \le C_\lambda \one_{B_{\Lmetric}(x,r + \lambda s)}.
	\end{equation}

	We will use Proposition \ref{p:ks} to show estimate the norm of the upper gradient.
	By \eqref{e:pile} in Proposition \ref{p:2wt}, there exist $C_1,A_1 \in (1,\infty)$ such that for all $f \in \sF_p(K, \metric, \measure) \cap \contfunc(K)$, we have
	\begin{align} \label{e:ec3}
		 \MoveEqLeft{ \int_{B_{\Lmetric}(x,r)} s^{-p} \fint_{B_{\Lmetric}(y,s)} \abs{f(y)-f(z)}^p\,\Lmeasure(dy)\,\Lmeasure(dz)} \nonumber \\
	 	&\lesssim  s^{-2p} \int_{K} \int_{K} \abs{f(y)-f(z)}^p \one_{B_{\Lmetric}(x,r)}(y)\one_{B_{\Lmetric}(y,s)}(z)\,\Lmeasure(dy)\,\Lmeasure(dz) \nonumber \\ 
	 	&\lesssim s^{-2p} \sum_{n \in N}\int_{B_{\Lmetric}(n,2s)} \int_{B_{\Lmetric}(n,2s)} \abs{f(y)-f(z)}^p \,\Lmeasure(dy)\,\Lmeasure(dz) \quad \mbox{(by \eqref{e:ec1})} \nonumber \\
	 	&\lesssim s^{-2p} \sum_{n \in N}\int_{B_{\Lmetric}(n,2s)} \int_{B_{\Lmetric}(n,2s)} \bigl( \abs{f(y) - f_{B_{\Lmetric}(n,2s),\Lmeasure}}^p + \abs{f(z) - f_{B_{\Lmetric}(n,2s),\Lmeasure}}^p \bigr)\,\Lmeasure(dy)\,\Lmeasure(dz) \nonumber \\
	 	&\lesssim  s^{-p} \sum_{n \in N} \inf_{\alpha \in \bR} \int_{B_{\Lmetric}(n,2s)}  \abs{f(y)-\alpha}^p\,\Lmeasure(dy)  \quad \mbox{(by Lemma \ref{lem.p-var} and \hyperref[cond.AR]{AR($p$)} for $\Lmeasure$)} \nonumber \\
	 	&\lesssim \sum_{n \in N} \Gamma_p \langle f \rangle (B_{\Lmetric}(n, A_1 s)) \quad \mbox{(by \eqref{e:pile})} \nonumber \\
	 	& \le C_1 \Gamma_p \langle f \rangle (B_{\Lmetric}(x, r+A_1 s)) \quad \mbox{(by \eqref{e:ec2}).}
	\end{align}
	By letting $r \to \infty$ in \eqref{e:ec3} and using  Proposition \ref{p:ks}, we conclude that
	\[
	\sF_p(K, \metric, \measure) \cap \contfunc(K) \subseteq N^{1,p}(K,\Lmetric,\Lmeasure) \cap \contfunc(K).
	\]
	By \eqref{e:ec3} and \eqref{e:ksn} in Proposition \ref{p:ks}, we obtain \eqref{e:ecomp1}.

	\ref{it:mcomp2} This is similar to part \ref{it:mcomp1}, except that we use Proposition \ref{prop.KS-local.2}, \ref{cond.AR} for $\measure$ and \eqref{e:piel} in place of Proposition \ref{p:ks}, \hyperref[cond.AR]{AR($p$)} for $\Lmeasure$ and \eqref{e:pile} respectively.
\end{proof}

\begin{proof}[Proof of Theorem \ref{t:compenergy}]
	The estimate \eqref{e:emeascomp} follows from Lemma \ref{l:mcomp} along with Lemma \ref{l:ac}.

	It remains to show \eqref{e:e1comp}.
	By normalizing the measures if necessary, we assume that $\measure$ and $\Lmeasure$ are probability measures.
	For $f \in \contfunc(K)$, let  $f_\measure= \int_{K} f \,d\measure$ and $f_\Lmeasure = \int_{K} f\,d\Lmeasure$ denote the averages of $f$ with respect to $\measure$ and $\Lmeasure$ respectively.
	The proof of \eqref{e:pile} with $r = 2\diam(K,\Lmetric)$ yields
	\begin{equation} \label{e:gp1}
		\int_{K}\abs{f-f_{\measure}}^p\,d\Lmeasure \lesssim \norm{f}_{\sF_p}^p \quad \mbox{for all $f \in \sF_p(K, \metric, \measure) \cap \contfunc(K)$.}
	\end{equation}
	Note that for any $f \in \sF_p(K, \metric, \measure) \cap \contfunc(K)$, we have
	\begin{align} \label{e:gp2}
		\int_{K}\abs{f}^p\,d\Lmeasure
		&\le 2^{p-1}\left( \abs{f_\measure}^p  + \int_{K}\abs{f-f_{\measure}}^p\,d\Lmeasure \right) \nonumber \\
		&\lesssim \int_{K}\abs{f}^p \,d\measure + \norm{f}_{\sF_p}^p \quad \mbox{(by \eqref{e:gp1} and Jensen's inequality).}
	\end{align}
	Therefore the first estimate in \eqref{e:e1comp} follows from \eqref{e:emeascomp} and \eqref{e:ecomp}. The proof of the second estimate in \eqref{e:e1comp} is similar.
\end{proof}

We observe two important consequences of Theorem \ref{t:compenergy}.
The first one states that Loewner measures must be minimal energy dominant measures for the self-similar energy $(\sE_p,\sF_p)$.
\begin{theorem} \label{t:medm}
	Suppose Assumption \ref{a:attain}.
	Then $\Lmeasure$ is a minimal energy dominant measure for $(\sE_p,\sF_p)$.
	Furthermore, there exist $C\in (0,\infty)$ and $u \in \mathscr{C}_{p}$, we have
	\begin{equation} \label{e:lmeas}
		C^{-1} \Gamma_p\langle u \rangle (B) \le  \Lmeasure(B) \le 	C  \Gamma_p\langle u \rangle (B) \quad \mbox{for all Borel subset $B \subset K$.}
	\end{equation}
\end{theorem}
\begin{proof}
	By Theorem \ref{t:compenergy}, $\Gamma_{p} \langle f \rangle \ll \mu$ for all $f \in \mathscr{C}_{p}$.
	Combining this with the density of $\contfunc(K) \cap \mathcal{F}_{p}(K, \metric, \measure)$ (Theorem \ref{t:Fp}\ref{t:Fp-regFp}) and Lemma \ref{lem.MED-conti}, we obtain the domination property: $\Gamma_{p}\langle f' \rangle \ll \Lmeasure$ for all $f' \in \mathcal{F}_{p}(K, \metric, \measure)$.

	By \cite[Corollary 8.3.16]{HKST} and a biLipschitz change of metric if necessary, we can assume that $\Lmetric$ is a geodesic metric.
	Consider the function $u(\cdot)= \rho(x_0,\cdot)$ for some $x_0 \in K$.
	Since $u$ is Lipschitz in $(K,\Lmetric)$ and $K$ is compact, we have $u \in N^{1,p}(K,\Lmetric,\Lmeasure)$ by \cite[Lemma 6.2.6]{HKST}.
	Furthermore, by considering geodesics in $(K, \Lmetric)$, we can show that $\on{lip}u \equiv 1$.
	By \cite[Theorem 13.5.1]{HKST}, we have that the minimal $p$-weak upper gradient $g_u$ of $u$ satisfies $g_u = 1$ $\Lmeasure$-almost everywhere.
	By \eqref{e:emeascomp} in Theorem \ref{t:compenergy}, we have that $\Lmeasure \ll \Gamma_p\langle u \rangle$ and hence $\Lmeasure$ is a minimal energy dominant measure and satisfies \eqref{e:lmeas}.
\end{proof}

The second one is the identification of the two different Sobolev spaces $\mathcal{F}_{p}(K, \metric, \measure)$ and $N^{1,p}(K, \Lmetric, \Lmeasure)$.
\begin{thm}\label{t:identify}
	Suppose Assumption \ref{a:attain}.
	Then there exists a bounded, linear bijection $\iota \colon \mathcal{F}_{p}(K, \metric, \measure) \to N^{1,p}(K, \Lmetric, \Lmeasure)$ satisfying
	\begin{equation}\label{e:ecompe1-all}
		C_{1}^{-1}\norm{f}_{\mathcal{F}_{p}} \le \norm{\iota(f)}_{N^{1,p}(\Lmeasure)} \le C_{1}\norm{f}_{\mathcal{F}_{p}} \quad \text{for all $f \in \mathcal{F}_{p}(K, \metric, \measure)$,}
	\end{equation}
	where $C_{1} \ge 1$ is the constant in \eqref{e:e1comp}. Furthermore if $f \in \contfunc(K) \cap \sF_p(K,d,m)$, then $\iota$ maps the equivalence class containing $f$ in $\sF_p(K,d,m)$ to the equivalence class containing $f$ in $N^{1,p}(K,\theta,\mu)$.
\end{thm}
\begin{proof}
	We first note that $\mathscr{C}_{p}$ is a dense linear subspace of both $\mathcal{F}_{p}(K, \metric, \measure)$ and $N^{1,p}(K, \Lmetric, \Lmeasure)$ by Theorem \ref{t:Fp} and \cite[Theorem 8.2.1]{HKST}.
	Let $\iota_{0} \colon (\mathscr{C}_{p}, \norm{\,\cdot\,}_{\mathcal{F}_{p}}) \to N^{1,p}(K, \Lmetric, \Lmeasure)$ be the inclusion map, i.e., $\iota_{0}(f) = [f]_{N^{1,p}}$ for $f \in \mathscr{C}_{p}$, where $[f]_{N^{1,p}}$ is the equivalence class defined in Definition \ref{d:Newton}.
	By \eqref{e:e1comp} in Theorem \ref{t:compenergy}, we have $C_{1}^{-1}\norm{f}_{\mathcal{F}_{p}} \le \norm{\iota_{0}(f)}_{N^{1,p}} \le C_{1}\norm{f}_{\mathcal{F}_{p}}$ for all $f \in \mathscr{C}_{p}$.
	Hence, by \cite[Proposition 1.4.14]{Meg}, $\iota_{0}$ is an isomorphism.
	By \cite[Theorem 1.9.1]{Meg} and the density of $\mathscr{C}_{p}$, there is a unique extension $\iota \colon \mathcal{F}_{p}(K, \metric, \measure) \to N^{1,p}(K, \Lmetric, \Lmeasure)$ of $\iota_{0}$, which is also an isomorphism satisfying $C_{1}^{-1}\norm{f}_{\mathcal{F}_{p}} \le \norm{\iota(f)}_{N^{1,p}} \le C_{1}\norm{f}_{\mathcal{F}_{p}}$ for all $f \in \mathcal{F}_{p}(K, \metric, \measure)$.
\end{proof}

We conclude this subsection by extending the comparability result of energy measures to all functions in Sobolev spaces through the above isomorphism.
\begin{cor}\label{c:comp-full}
	Suppose Assumption \ref{a:attain} and let $\iota \colon \mathcal{F}_{p}(K, \metric, \measure) \to N^{1,p}(K, \Lmetric, \Lmeasure)$ be the identification map in Theorem \ref{t:identify}.
	Then there exists a constant $C \ge 1$ such that the following hold: For any $f \in \mathcal{F}_{p}(K, \metric, \measure)$ and any Borel set $B \in \mathcal{B}(K)$,
	\begin{equation} \label{e:emeascomp-all}
		C^{-1} \Gamma_{p} \langle f \rangle(B) \le \int_{B} g_{\iota(f)}^{p} \,d\Lmeasure \le C \Gamma_{p} \langle f \rangle(B).
	\end{equation}
	In particular,
	\begin{equation} \label{e:ecomp-all}
		C^{-1} \mathcal{E}_{p}(f) \le  \int_{X} g_{\iota(f)}^{p} \,d\Lmeasure \le C \mathcal{E}_{p}(f) \quad \mbox{for all $f \in \mathcal{F}_{p}(K, \metric, \measure)$.}
	\end{equation}
\end{cor}
\begin{proof}
	By \cite[(6.3.18)]{HKST}, for any $u, v \in N^{1,p}(K, \Lmetric, \Lmeasure)$ and $B \in \mathcal{B}(K)$, we have
	\[
	\left(\int_{B}g_{u + v}^{p}\,d\Lmeasure\right)^{1/p} \le \left(\int_{B}g_{u}^{p}\,d\Lmeasure\right)^{1/p} + \left(\int_{B}g_{v}^{p}\,d\Lmeasure\right)^{1/p}.
	\]
	In particular, $\lim_{n \to \infty}\int_{B}g_{u_{n}}^{p}\,d\Lmeasure = \int_{B}g_{u}^{p}\,d\Lmeasure$ whenever $\lim_{n \to \infty}\norm{u - u_{n}}_{N^{1,p}} = 0$.
	Let $f \in \mathcal{F}_{p}(K, \metric, \measure)$ and pick a sequence $\{ f_{n} \}_{n} \subseteq \mathscr{C}_{p}$ such that $\lim_{n \to \infty}\norm{f - f_{n}}_{\mathcal{F}_{p}} = 0$.
	By \eqref{e:ecompe1-all}, we then have $\lim_{n \to \infty}\norm{\iota(f) - \iota(f_{n})}_{N^{1,p}} = 0$.
	Therefore, letting $n \to \infty$ in \eqref{e:emeascomp} for $f_{n}$ yields \eqref{e:emeascomp-all}.
\end{proof}
We are now ready to prove Theorem \ref{t:attain}.

\begin{proof}[Proof of  Theorem \ref{t:attain}]
The first assertion follows from Theorems   \ref{t:compenergy}, \ref{t:identify} and Corollary \ref{c:comp-full}.
The second assertion follows from Theorem \ref{t:medm}.
\end{proof}

\section{Conjectures and open problems} \label{sec.conj}
We conclude this paper by mentioning some related open problems and conjectures.

To construct a H\"older continuous cutoff function with low energy and to obtain Poincar\'e inequality, the condition $\hdim - \beta < 1$ (or equivalently $\zeta<1$) was crucial.  This is because  the conclusion of Theorem \ref{thm.pGCL-gamma} fails without the condition $\zeta<1$. However, it is conceivable that capacity bounds imply Poincar\'e inequality without this restriction but  such a result would require a very different approach.
\begin{problem}\label{prob.relax}
	Relax the conditions $\hdim - \beta < 1$ in Theorem \ref{thm.global.cut-off} and $\zeta<1$ in Theorem \ref{thm.PI-discrete}.
\end{problem}

\noindent
Problem \ref{prob.relax} is similar in spirit to the \emph{resistance conjecture} for the case $p=2$ and hence it appears very challenging \cite[\textsection 6.3]{Mur23+}.

In this paper, we confine ourselves to the planar standard Sierpi\'{n}ski carpet for simplicity.
As mentioned in Remark \hyperref[it:other.super]{\ref{rmk.other}}, the planar generalized Sierpi\'{n}ski carpets should be similar, but we do not know other cases.
\begin{problem}
	Construct Sobolev spaces, $p$-energies, energy measures for other examples such as Sierpi\'{n}ski cross \cite{Kig09}, subsystems of (hyper)cubic tiling \cite{Kig23}, unconstrained Sierpi\'{n}ski carpets \cite{CQ21+, CQ+}, boundaries of hyperbolic groups, Julia sets of conformal dynamical systems \cite{Bon,Kle}. 
	Added in revision: Recently, Anttila and Eriksson-Bique \cite{AE.super} introduced a new framework of fractal spaces arising from \emph{iterated graph systems} and established the combinatorial Loewner property for these fractals. It would be interesting to construct/investigate Sobolev spaces, $p$-energies, energy measures in this framework. 
\end{problem}

Our study also provides a partial result on the uniqueness of $p$-energies on the Sierpi\'{n}ski carpet.
It is natural to expect that such the uniqueness is true for all $p$.
\begin{con}\label{con.uniq}
	For any $p \in (1,\infty)$, self-similar $p$-energy (see Assumption \ref{a:main}) is unique up to multiplications of constants.
	We expect that the uniqueness is true for a wide class of Sierpi\'{n}ski carpets (e.g. generalized Sierpi\'{n}ski carpets).
\end{con}

\noindent
We expect that Conjecture \ref{con.uniq} follows from a converse estimate of Lemma \ref{l:oneside}.
\begin{con}\label{con.uniq2}
	For any $p \in (1,\infty)$, there exists a constant $C_{\ast} > 0$ depending only on $p$ and the geometric data of PSC such that
	\begin{equation}\label{preunique}
		\sup\biggl\{ \frac{\chi(\mathscr{E}_{p})}{\sigma(\mathscr{E}_{p})} \biggm| (\mathscr{E}_{p}, \mathscr{F}_{p}) \in \mathfrak{E}_{p} \biggr\} \le C_{\ast} < \infty.
	\end{equation}
	Furthermore, \eqref{preunique} implies the affirmative answer for Conjecture \ref{con.uniq}.
\end{con}



Compared to our $(1,p)$-Sobolev space $\mathcal{F}_{p}$, the definition of energy measures on a self-similar set heavily depends on the self-similarity.
This is a difference from the case $p = 2$ (Dirichlet form theory) and is an obstacle to develop general theory. This motivates the following question.
\begin{problem}\label{prb.constr-pem}
	Define $p$-energy measures $\Gamma_{p}\langle \,\cdot\, \rangle$ without using the self-similarity and establish their basic properties (cf. Theorem \ref{t:main-em}\ref{em-tri},\ref{em-lip} and \ref{em-chain}).
\end{problem} 

\noindent 
Added in revision: There has been recent progress on Problem \ref{prb.constr-pem}. In \cite{KS.lim}, by using the Riesz--Markov--Kakutani representation theorem, $p$-energy measures are constructed for specific $p$-energy forms called the \emph{Korevaar--Schoen $p$-energy forms}. See also \cite[Section 4]{ARB24+} for a related result on Cheeger spaces. 

It is also natural to expect that $p$-energy measures on typical fractals  are mutually singular with the underlying self-similar measures (cf. \cite{Hin05,KM20} for the case $p = 2$).
\begin{problem}
	For a self-similar set $(K,d)$ satisfying Assumption \ref{a:reg} with $\beta > p$, show that $\Gamma_{p}\langle f \rangle \perp \measure$ for any $f \in \mathcal{F}_{p}$, where $\measure$ is the self-similar measure.
\end{problem}
\noindent
The next two problems are motivated by a desire to understand the dependence of the Sobolev space $\sF_p$ and energy measures on the exponent $p$.
\begin{problem}\label{prb:distinct}
	Let $p, q \in (1, \infty)$ be distinct.
	Let $\nu_{p}, \nu_{q}$ be minimal energy-dominant measures of $(\mathcal{E}_{p}, \mathcal{F}_{p}), (\mathcal{E}_{q}, \mathcal{F}_{q})$ respectively.
	Are $\nu_{p}$ and $\nu_{q}$ mutually singular or absolutely continuous?
\end{problem}


\noindent
We also do not know if there are inclusion relations among $\{ \mathcal{F}_{p} \}_{p > 1}$.
\begin{problem}
	Let $p, q \in (1, \infty)$ be distinct.
	Determine the intersection $\mathcal{F}_{p} \cap \mathcal{F}_{q}$. In particular, does $\mathcal{F}_{p} \cap \mathcal{F}_{q}$ contain any non-constant function?
\end{problem}

Towards the attainment problem of the Ahlfors regular conformal dimension, we expect that the following variant of Theorem \ref{t:attain}\ref{attain-ecomp} to be useful. This conjecture is an analogue of \cite[Theorem 6.54]{KM23}.
\begin{con} \label{con:harmonic}
	Let $(K, \metric, \measure)$ be the Sierpi\'{n}ski carpet.
	Suppose  that $d_{\textup{ARC}}(K,d)$  is attained.
	There exists $h$ which is $d_{\textup{ARC}}$-harmonic (with respect to the self-similar $d_{\textup{ARC}}$-energy $\mathcal{E}_{d_{\textup{ARC}}}$) on $K \setminus \mathcal{V}_{0}$ such that $\Gamma_{d_{\textup{ARC}}}\langle h \rangle$ is also an optimal measure.
\end{con}

\appendix
\section{A collection of useful elementary facts}\label{sec.useful}
The following lemma corresponds to a $5B$-covering lemma for graphs.
\begin{lem}\label{lem.3B-appendix}
    Let $G = (V, E)$ be a graph, and let $\mathscr{B} = \{ B(x_i, r_i) \mid i \in I \}$ be a family of balls such that $r_i > 0$ for all $i \in I$ and $R \coloneqq \sup_{i \in I}r_i < +\infty$.
    (Here, $B(x, r) \coloneqq \{ y \in V \mid d(x, y) < r \}$, where $d$ denotes the graph distance of $G$.)
    Then there exists $J \subseteq I$ such that
    \begin{equation*}
        B(x_{j}, r_{j}) \cap B(x_{k}, r_{k}) = \emptyset \quad \text{for all $j, k \in J$ with $j \neq k$,}
    \end{equation*}
    and
    \begin{equation*}
        \bigcup_{i \in I}\closure{B}(x_i, r_i) \subseteq \bigcup_{j \in J}\closure{B}(x_j, 3r_j).
    \end{equation*}
    Moreover, for any $i \in I$ there exists $j \in J$ such that $\closure{B}(x_i, r_i) \subseteq \closure{B}(x_j, 3r_j)$.
\end{lem}
\begin{proof}
    For each $r > 0$, let $\lfloor r \rfloor \in \mathbb{Z}_{\ge 0}$ denotes the unique non-negative integer such that
    \[
    \lfloor r \rfloor \le r < \lfloor r \rfloor + 1.
    \]
    For any $x \in V$ and $r > 0$, we have $B(x, r) \subseteq B(x, \lfloor r \rfloor + 10^{-1})$.
    Moreover,
    \[
    \closure{B}(x, r) \coloneqq \{ y \in V \mid d(x, y) \le r \} = B(x, \lfloor r \rfloor + 10^{-1}).
    \]
    We write $B_i$ for $B(x_i, \lfloor r_i \rfloor + 10^{-1})$ for simplicity.
    For each $r \in [0, R] \cap \mathbb{Z}$, define
    \[
    I_{r} \coloneqq \{ i \in I \mid \lfloor r_i \rfloor = r \}.
    \]

    Let $I_{R}'$ be a maximal subset of $I_{R}$ such that $\{ B_{i} \mid i \in I_{R}' \}$ are disjoint.
    Inductively, we define $\{ I_{R - m}' \}_{m = 0}^{R}$ as follows: given $I_{R}', \dots, I_{R - m + 1}'$, let $I_{R - m}'$ be a maximal subset of $I_{R - m}$ such that
    \begin{equation}\label{3B-inductive1}
        \text{$\{ B_{i} \mid i \in I_{R - m}' \}$ are disjoint,}
    \end{equation}
    and
    \begin{equation}\label{3B-inductive2}
        \text{$\{ B_{i} \mid i \in I_{R - m}' \}$ are also disjoint from $\displaystyle\Biggl\{ B_{i} \Biggm| i \in \bigcup_{j = R - m + 1}^{R}I_{j}' \Biggr\}$.}
    \end{equation}
    Now set $J \coloneqq \bigcup_{j = 0}^{R}I_{j}'$.
    This construction yields that $\{ B_{j} \mid j \in J \}$ are disjoint.

    We will show that $\bigl\{ \closure{B}(x_j, 3r_{j}) \mid j \in J \bigr\}$ covers $\bigcup_{i \in I}B_{i}$.
    Let $i \in I$.
    If $i \in J$, then it is immediate that $B_i \subseteq \bigcup_{j \in J}\closure{B}(x_j, 3r_j)$ since
    \[
    B_{i} = B(x_i, \lfloor r_i \rfloor + 10^{-1}) = \closure{B}(x_i, r_i) \subseteq \closure{B}(x_i, 3r_i).
    \]
    If not, then there exists $k \in J$ with $\lfloor r_k \rfloor \ge \lfloor r_i \rfloor$ such that $B_i \cap B_k \neq \emptyset$.
    (If such $k$ does not exists, then $I_{\lfloor r_i \rfloor}' \cup \{ i \}$ satisfies \eqref{3B-inductive1} and \eqref{3B-inductive2}. This does not happen due to the maximality of $I_{\lfloor r_i \rfloor}'$.)
    Let $z \in B_{i} \cap B_{k}$.
    Then for any $y \in B_{i}$,
    \begin{align*}
        d(x_k, y) \le d(x_k, z) + d(z, x_i) + d(x_i, y) < \lfloor r_k \rfloor + \lfloor r_i \rfloor + \lfloor r_i \rfloor + 3 \cdot 10^{-1} \le 3\lfloor r_k \rfloor + 3 \cdot 10^{-1}.
    \end{align*}
    Hence we have
    \begin{equation}\label{3B-integer}
        B_{i} \subseteq B(x_{k}, 3\lfloor r_k \rfloor + 3 \cdot 10^{-1}) \subseteq \closure{B}(x_{k}, 3r_{k}),
    \end{equation}
    proving the lemma.
\end{proof}

We heavily use the following version of Mazur's lemma in this paper.
\begin{lem}[{\cite[page 19]{HKST}}]\label{lem.Mazur} 
	Let $(v_{i})_{i \in \mathbb{N}}$ be a sequence in a normed space $V$ converging weakly to some element $v \in V$.
	Then there exist a strictly increasing sequence $\{ l_{n} \}_{n \ge 1}$ of positive integers with $l_{n} \ge n$, and, for each $n \ge 1$, $(\lambda_{i, n})_{i = n}^{l_{n}} \in [0, 1]^{l_{n} - n + 1}$ with $\sum_{i = n}^{l_{n}}\lambda_{i, n} = 1$ such that $\sum_{i = n}^{l_{n}}\lambda_{i, n}v_{i}$ converges strongly to $v$ as $n \to \infty$.
\end{lem}

The next lemma is very useful in some arguments about Poincar\'{e} type inequalities.
The proof can be found in \cite[Lemma 4.17]{BB} for example.
\begin{lem}\label{lem.p-var}
	Let $(X, \mathscr{A}, \mathfrak{m})$ be a measure space and let $E \in \mathscr{A}$ with $\mathfrak{m}(E) > 0$.
	If $u \in L^{1}_{\text{loc}}(X, \mathfrak{m})$, $1 \le p < \infty$, then
	\[
	\left(\fint_{E}\abs{u - u_{E}}^{p}\,d\mathfrak{m}\right)^{1/p} \le 2\inf_{c \in \mathbb{R}}\left(\fint_{E}\abs{u - c}^{p}\,d\mathfrak{m}\right)^{1/p}.
	\]
\end{lem}

The following result state a kind of stability of discrete energies.
A more general version written in terms of \emph{rough isometry} is well-known, but the next simple version is enough for our purpose.
\begin{prop}\label{prop.URI}
	Let $p > 0$.
	Let $G_{i} = (V, E_{i}) \, (i = 1, 2)$ be connected graphs such that $E_{2} \subseteq E_{1}$.
	Let $d_{i}$ be the graph distance of $G_{i}$.
	Suppose that $L_{\ast} \coloneqq \deg(G_{1}) < \infty$ and that there exists $D_{\ast} \ge 1$ such that for any $\{ x, y \} \in E_{1} \setminus E_{2}$, we have
	\[
	d_{2}(x, y) \le D_{\ast}.
	\]
	Then for all $f \in \mathbb{R}^{V}$,
	\[
	\mathcal{E}_{p}^{G_{2}}(f) \le \mathcal{E}_{p}^{G_{1}}(f) \le C_{p, D_{\ast}, L_{\ast}}\mathcal{E}_{p}^{G_{2}}(f),
	\]
	where $C_{p, D_{\ast}, L_{\ast}} = 1 + L_{\ast}^{2D_{\ast}}\bigl(D_{\ast}^{p - 1} \vee 1\bigr)$.
\end{prop}
\begin{proof}
	Since $E_{2} \subseteq E_{1}$, it is immediate that $\mathcal{E}_{p}^{G_{2}}(f) \le \mathcal{E}_{p}^{G_{1}}(f)$.
	To prove the remaining inequality, for each $\{ x, y \} \in E_{1} \setminus E_{2}$, we fix a path $[z_{xy}(0), z_{xy}(1), \dots, z_{xy}(D_{\ast})]$ in $G_{2}$ such that $z_{xy}(0) = x$, $z_{xy}(D_{\ast}) = y$ and
	\[
	\bigl\{ z_{xy}(i - 1), z_{xy}(i) \bigr\} \in E_{2} \cup \bigl\{ \{ x, x \} \bigm| x \in V \bigr\} \quad \text{for each $i = 1, \dots, D_{\ast}$}.
	\]
	Noting that
	\[
	\sup_{\{ x', y' \} \in E_{2}}\#\bigl\{ \{ x, y \} \in E_{1} \bigm| \bigl\{ z_{xy}(i - 1), z_{xy}(i) \bigr\} = \{ x', y' \} \text{ for some $i$} \} \le L_{\ast}^{2D_{\ast}},
	\]
	we have
	\begin{align*}
		\mathcal{E}_{p}^{G_{1}}(f)
		&= \mathcal{E}_{p}^{G_{2}}(f)  + \sum_{\{ x, y \} \in E_{1} \setminus E_{2}}\abs{f(x) - f(y)}^{p} \\
		&\le \mathcal{E}_{p}^{G_{2}}(f) + \bigl(D_{\ast}^{p - 1} \vee 1\bigr)\sum_{\{ x, y \} \in E_{1} \setminus E_{2}}\sum_{i = 1}^{D_{\ast}}\abs{f\bigl(z_{xy}(i - 1)\bigr) - f\bigl(z_{xy}(i)\bigr)}^{p} \\
		&\le \mathcal{E}_{p}^{G_{2}}(f) + L_{\ast}^{2D_{\ast}}\bigl(D_{\ast}^{p - 1} \vee 1\bigr)\mathcal{E}_{p}^{G_{2}}(f),
	\end{align*}
	which finishes the proof.
\end{proof}

\section{Whitney cover and its applications}\label{sec.whitney}
This section aims to prove Lemma \ref{l:ac}.
We will use the following version of a Whitney cover.
\begin{defn}[{\cite[Definition 2.3]{Mur23+}}]\label{d:whitney}
	Let $(X,\mathsf{d})$ be a metric space and $\varepsilon \in (0, 1/2)$.
	Let $U$ be a non-empty proper subset of $X$ such that $U \neq X$.
	A collection of balls $\mathfrak{R} = \{ B(x_i, r_i) \mid x_i \in U, r_i > 0, i \in I \}$ is said to be an \emph{$\varepsilon$-Whitney cover} of $U$ if it satisfies the following conditions:
	\begin{enumerate}[\rm(1)]
		\item The balls in $\mathfrak{R}$ are pairwise disjoint.
		\item The radius $r_i$ satisfies
		\begin{equation}\label{whitney-radi}
			r_{i} = \frac{\varepsilon}{1 + \varepsilon}\dist(x_{i}, X \setminus U), \quad \text{for each $i \in I$.}
		\end{equation}
		\item It holds that $\bigcup_{i \in I}B(x_i, 2(1 + \varepsilon)r_{i}) = U$.
	\end{enumerate}
\end{defn}
\begin{rmk}\label{rmk.whitney-radi}
	From \eqref{whitney-radi}, we observe that $B\bigl(x_{i}, \varepsilon^{-1}(1 + \varepsilon)r_{i}\bigr) \subseteq U$ for all $i \in I$.
\end{rmk}

The existence of such an $\varepsilon$-Whitney cover of any non-empty open subset $U$ of a given metric space $(X, \mathsf{d})$ for all $\varepsilon \in (0, 1/2)$ is ensured by \cite[Proposition 3.2 (a)]{Mur23+}.
The following proposition states a basic overlapping property of Whitney covers on a doubling metric space.
\begin{prop}[{\cite[Proposition 3.2 (d)]{Mur23+}}]\label{p:overlap}
	Let $(X,\mathsf{d})$ be a metric space and let $U$ be a non-empty proper subset of $X$ such that $U \neq X$.
	If $(X, \mathsf{d})$ is metric doubling, then for any $\varepsilon \in (0, 1/2)$ there exists $C > 0$ (depending only on $\varepsilon$ and the doubling constant of $(X, d)$) such that the following hold: for any $\varepsilon$-Whitney cover $\mathfrak{R} = \{ B(x_i, r_i) \mid x_i \in U, r_i > 0, i \in I \}$ of $U$, we have
	\[
	\sum_{i \in I}\indicator{B(x_i, \varepsilon^{-1}r_{i})} \le C.
	\]
\end{prop}

Now we can prove the desired lemma:
\begin{proof}[Proof of Lemma \ref{l:ac}]
	By the outer regularity of measures $\nu_1$ and $\nu_2$ \cite[Proposition 3.3.37]{HKST}, it suffices to verify \eqref{e:qac} for all open sets.

	To this end, let $U$ be an arbitrary non-empty open subset of $X$.
	Let us fix small enough $\varepsilon$ so that $0 < \varepsilon < (3A_1)^{-1}$ and choose a $\varepsilon$-Whitney cover $\mathfrak{R} = \{ B(x_i, r_i) \mid x_i \in U, r_i > 0, i \in I \}$ of $U$.
	Then we note that $B(x_i, 3A_{1}r_{i}) \subseteq U$ for all $i \in I$.
 	By the bounded overlap property Proposition \ref{p:overlap}, there exists $C_2$ depending only on $C_1,A_1$ and the constant associated to the doubling property of $(X, \mathsf{d})$ such that
 	\begin{equation}
 		\nu_1(U) \le \sum_{B(x_i,r_i) \in \mathfrak{R}} \nu_1(B(x_i,3r_i)) \le  \sum_{B(x_i,r_i) \in \mathfrak{R}} C_1 \nu_2(B(x_i,3A_1r_i)) \le C_2 \nu_2(U),
 	\end{equation}
 	which concludes the proof.
\end{proof}

\section{On the conductive homogeneity}
In this section, we discuss relations between our framework (Assumption \ref{a:reg}) and a notion of the $p$-conductive homogeneity introduced in \cite{Kig23}.
More precisely, we will show that a $p$-conductive homogeneous compact metric space with some additinal conditions (see Assumption \ref{a:Kig-MS} for the detail) satisfies Assumption \ref{a:reg}.
The converse direction is rather delicate in a general setting.
We only show that the planar Sierpi\'{n}ski carpet is $p$-conductive homogeneous for any $p \in (1,\infty)$.

\subsection{Partition parametrized by a tree and basic framework}
Let us start with the definition of partition parametrized by trees (see \cite[Definitions 2.1, 2.2 and 2.3]{Kig23}).
\begin{defn}[rooted tree]\label{defn.tree}
	Let $T$ be an (non-directed) locally finite, infinite graph without self-loops whose edge set is given by $\{ v \sim w \}$, i.e. $T$ is countable set and
	\[
	v \sim w \iff w \sim v, \quad \#\{ v \in T \mid v \sim w \} < \infty, \quad \text{and} \quad w \not\sim w \quad \text{for all $v, w \in T$.}
	\]
	A graph $T$ is called a \emph{tree} if and only if there exists a unique simple path between $v$ and $w$ for any $v, w \in T$ with $v \neq w$.
	Such the unique path between $v$ and $w$ is denoted by $\overline{vw}$.
	We write $z \in \overline{vw}$ if $\overline{vw} = [w_{0}, \dots, w_{n}]$ and $w(i) = z$ for some $i = 0, \dots, n$.
	Let $\phi \in T$.
	The tuple $(T, \phi)$ is called a \emph{rooted tree} with a root $\phi$.
	In order to clarify the edge structure, we also use $(T, \sim)$ and $(T, \sim, \phi)$ to denote $T$ and $(T, \phi)$ respectively.
\end{defn}

The following gives fundamental notations on rooted trees.
\begin{defn}\label{defn.tree-notation}
	Let $(T, \phi)$ be a rooted tree.
	\begin{enumerate}[(1)]
		\item For $w \in T$, define $\pi \colon T \to T$ by
		\begin{equation}\label{defn.pi}
			\pi(w) =
			\begin{cases}
				w_{n - 1} \quad &\text{if $w \neq \phi$ and $\overline{\phi w} = [w_{0}, \dots, w_{n}]$,} \\
				\phi \quad &\text{if $w = \phi$.}
			\end{cases}
		\end{equation}
		Set
		\begin{equation}\label{defn.children}
			S(w) = \{ v \in T \mid \pi(v) = w \} \setminus \{ w \},
		\end{equation}
		and
		\begin{equation}\label{defn.Nstar}
			N_{\ast} \coloneqq \sup_{w \in T}\#S(w).
		\end{equation}
		Moreover, for $k \ge 1$, we define $S^{k}(w)$ inductively as
		\[
		S^{k + 1}(w) = \bigcup_{v \in S(w)}S^{k}(v).
		\]
		For $A \subseteq T$, define $S^{k}(A) \coloneqq \bigcup_{w \in A}S^{k}(A)$.
		\item For $w \in T$ and $m \ge 0$, define
		\begin{equation}\label{defn.length-word}
			\abs{w}_{T} = \min\{ n \ge 0 \mid \pi^{n}(w) = \phi \}
		\end{equation}
		and $T_{m} = \{ w \in T \mid \abs{w}_{T} = m \}$.
		We also use $\abs{w}$ to denote $\abs{w}_{T}$ if no confusion may occur.
		\item For $w \in T$, define
		\begin{equation}\label{defn.descendant}
			T(w) = \{ v \in T \mid \text{there exists $n \ge 0$ such that $\pi^{n}(v) = w$} \}.
		\end{equation}
		For $A \subseteq T$, define $T(A) \coloneqq \bigcup_{w \in A}T(w)$.
		\item Define
		\begin{equation}\label{defn.shift-sp}
			\Sigma(T) = \{ (\omega_{i})_{i \ge 0} \mid \text{$\omega_{i} \in T_{i}$ and $\omega_{i} = \pi(\omega_{i + 1})$ for all $i \ge 0$} \}.
		\end{equation}
		For $\omega = (\omega_{i})_{i \ge 0} \in \Sigma(T)$, we write $[\omega]_{m}$ for $\omega_{m} \in T_{m}$.
		For $w \in T$, define
		\begin{equation}\label{defn.subshift-sp}
			\Sigma_{w}(T) = \{ (\omega_{i})_{i \ge 0} \in \Sigma \mid \text{$\omega_{\abs{w}} = w$} \}.
		\end{equation}
		For $A \subseteq T$, define $\Sigma_{A}(T) \coloneqq \bigcup_{w \in A}\Sigma_{w}(T)$.
		We also use $\Sigma$, $\Sigma_{w}$, $\Sigma_{A}$ to denote $\Sigma(T)$, $\Sigma_{w}(T)$ and $\Sigma_{A}(T)$ respectively when no confusion may occur.
	\end{enumerate}
\end{defn}
\begin{rmk}\label{rmk.tree}
	Strictly speaking, we should clarify the underlying rooted tree $(T, \phi)$ in the notations like $\pi$ or $S(\,\cdot\,)$.
	We are going to use $\pi\bigl(\,\cdot\,; (T, \phi)\bigr)$ or $S\bigl(\,\cdot\,; (T, \phi)\bigr)$ if we need such explicit notations.
\end{rmk}

Hereafter in this paper, $(T, \phi)$ is a locally finite rooted tree satisfying $\#\{ v \in T \mid v \sim w \} \ge 2$ for any $w \in T$.

\begin{defn}[partition]\label{defn.partition}
	Let $(K, \mathcal{O})$ be a compact metrizable topological space without isolated points, where $\mathcal{O}$ is the collection of open sets.
	A family of non-empty compact subsets $\{ K_{w} \}_{w \in T}$ is called a \emph{partition of $K$ parametrized by $(T, \phi)$} if and only if it satisfies the following conditions:
	\begin{itemize}
		\item [(P1)]\label{it:P1} $K_{\phi} = K$ and for any $w \in T$, $\#K_{w} \ge 2$ and
		\[
		K_{w} = \bigcup_{v \in S(w)}K_{v}.
		\]
		\item [(P2)]\label{it:P2} For any $w \in \Sigma$, $\bigcap_{m \ge 0}K_{[\omega]_{m}}$ is a single point.
	\end{itemize}
\end{defn}
\begin{rmk}
	In the original definition of partition in \cite[Definition 2.2.1]{Kig20}, the following condition (P$^{\ast}$) is also assumed:
	\begin{itemize}
		\item [(P$^{\ast}$)]\label{it:Past} For any $w \in T$, $K_{w}$ has no isolated points.
	\end{itemize}
	Recently, \cite[Lemma 3.6]{Sas23} shows that \hyperref[it:Past]{(P$^{\ast}$)} is automatically implied by a combination of \hyperref[it:P1]{(P1)} and \hyperref[it:P2]{(P2)}.
	So, we can drop \hyperref[it:Past]{(P$^{\ast}$)} in the definition of partition parametrized by a rooted tree.
\end{rmk}

The following definition is a collection of basic notations used in \cite{Kig20,Kig23}.
\begin{defn}
	Let $\{ K_{w} \}_{w \in T}$ be a partition of $K$ parametrized by $(T, \phi)$.
	\begin{enumerate}[(1)]
		\item For $w \in T$, define
		\begin{equation}\label{defn.Ow}
			O_{w} \coloneqq K_{w} \setminus \bigcup_{v \in T_{\abs{w}} \setminus \{ w \}}K_{v}
		\end{equation}
		and
		\begin{equation}\label{defn.Bw}
			B_{w} \coloneqq K_{w} \cap \bigcup_{v \in T_{\abs{w}} \setminus \{ w \}}K_{v}.
		\end{equation}
		The partition $\{ K_{w} \}_{w \in T}$ is called \emph{minimal} if $O_{w} \neq \emptyset$ for any $w \in T$.
		\item For $n \in \mathbb{Z}_{\ge 0}$, define
		\begin{equation}\label{defn.h-edge}
			E_{n}^{\ast} \coloneqq \bigl\{ \{ v, w \} \bigm| v, w \in T_{n}, v \neq w, K_{v} \cap K_{w} \neq \emptyset \bigr\}.
		\end{equation}
		Let us denote the graph distance of $(T_{k}, E_{k}^{\ast})$ by $d_{k}$.
		For $w \in T_{n}, n \ge 0$ and $M \ge 0$, define
		\begin{equation}\label{defn.h-nbd}
			\Gamma_{M}(w) \coloneqq \{ v \in T_{n} \mid d_{n}(v, w) \le M \},
		\end{equation}
		and for $x \in K$,
		\begin{equation}\label{defn.q-ball}
			U_{M}(x; n) \coloneqq \bigcup_{w \in T_{n}; x \in K_{w}}\bigcup_{v \in \Gamma_{M}(w)}K_{v}.
		\end{equation}
		For $A \subseteq T_{n}$, let $d_{n, A}$ be the graph distance of the subgraph $\bigl(A, E_{n}^{\ast}(A)\bigr)$, where $E_{n}^{\ast}(A) = \bigl\{ \{ v, w \} \in E_{n}^{\ast} \bigm| v, w \in A \bigr\}$, and define
		\begin{equation}\label{defn.h-nbd-rest}
			\Gamma_{M}^{A}(w) \coloneqq \{ v \in A \mid d_{n, A}(v, w) \le M \}.
		\end{equation}
		Also, define $\Gamma_{M}(A) \coloneqq \bigcup_{w \in A}\Gamma_{M}(w)$.
		\item Define
		\begin{equation}\label{defn.Lstar}
			L_{\ast} \coloneqq \sup_{w \in T}\#\Gamma_{1}(w).
		\end{equation}
		The partition $\{ K_{w} \}_{w \in T}$ is called \emph{uniformly finite} if $L_{\ast} < \infty$.
		\item Let $\chi  \colon \Sigma \to K$ be the map defined by $\bigcap_{n \ge 0}K_{[\omega]_{n}} = \{ \chi(\omega) \}$ for each $\omega \in \Sigma$.
		The partition $\{ K_{w} \}_{w \in T}$ is called \emph{strongly finite} if $\sup_{x \in K}\#\chi^{-1}(\{ x \}) < \infty$.
	\end{enumerate}
\end{defn}
\begin{rmk}
	In \cite[Definition 2.2.11]{Kig20}, the symbol $E_{n}^{h}$ is used to denote $E_{n}^{\ast}$.
	In addition, the edge set $E_{n}^{\ast}$ is considered to be directed in \cite{Kig20} and \cite{Kig23}.
	In this paper, we consider non-directed graphs to simplify some notations (the definition of discrete energies for example).
\end{rmk}

For details on basic topological properties of partitions, see \cite[Chapter 2]{Kig20}.


The following property is a consequence of the minimality, which will be used later.
\begin{lem}\label{lem.minimal}
	Let $\{ K_{w} \}_{w \in T}$ be a minimal partition of $K$ parametrized by $(T, \phi)$.
	Let $A, B$ be subsets of $T_{n}$ for some $n \in \mathbb{Z}_{\ge 0}$.
	Then $K_{A} \subseteq K_{B}$ if and only if $A \subseteq B$.
\end{lem}
\begin{proof}
	It is clear that $K_{A} \subseteq K_{B}$ if $A \subseteq B$.
	To prove the converse, suppose that $K_{A} \subseteq K_{B}$.
	Let $w \in A$.
	Then we clearly have $\emptyset \neq O_{w} \subseteq \bigcup_{v \in B}K_{v}$.
	For any $v,v' \in T$ with $\Sigma_{v} \cap \Sigma_{v'} = \emptyset$, we have $K_{v} \cap O_{v'} = \emptyset$ \cite[Lemma 2.2.2(2)]{Kig20}.
	This implies $w \in B$ and hence $A \subseteq B$.
\end{proof}

Now we recall the standing assumption \cite[Assumption 2.15]{Kig23}.
\begin{assumption}\label{assum.BF}
	Let $(K, \mathcal{O})$ be a connected compact metrizable space and let $\{ K_{w} \}_{w \in T}$ be a partition parametrized by the rooted tree $(T, \phi)$.
	Let $\metric$ metrize the topology $(K, \mathcal{O})$ with $\diam(K, \metric) = 1$ and let $\measure$ be a Borel regular probability measure on $K$.
	There exist $M_{\ast} \in \mathbb{N}$ and $r_{\ast} \in (0, 1)$ such that the following conditions \ref{it:BF1}-\ref{it:BF5} hold.
	\begin{enumerate}[\rm(1)]
		\item\label{it:BF1} $K_{w}$ is connected for any $w \in T$, $\{ K_{w} \}_{w \in T}$ is minimal and uniformly finite, and $\inf_{m \ge 0}\min_{w \in T_{m}}\#S(w) \ge 2$.
		\item\label{it:BF2} There exist $c_{i} > 0$, $i = 1,\dots,5$, such that the following conditions \hyperref[it:BF.2A]{(2A)}-\hyperref[it:BF.2C]{(2C)} are true.
		\begin{itemize}
			\item [(2A)]\label{it:BF.2A} For any $w \in T$,
			\begin{equation}\label{BF.bi-Lip}
				c_{1}r_{\ast}^{\abs{w}} \le \diam(K_{w}, \metric) \le c_{2}r_{\ast}^{\abs{w}}.
			\end{equation}
			\item [(2B)]\label{it:BF.2B} For any $n \in \mathbb{N}$ and $x \in K$,
			\begin{equation}\label{BF.adapted}
				B_{d}(x, c_{3}r_{\ast}^{n}) \subseteq U_{M_{\ast}}(x; n) \subseteq B_{d}(x, c_{4}r_{\ast}^{n}).
			\end{equation}
			(In \cite{Kig20}, the metric $d$ is called \emph{$M_{\ast}$-adapted} if the condition \eqref{BF.adapted} holds.)
			\item [(2C)]\label{it:BF.2C} For any $n \in \mathbb{N}$ and $w \in T_{n}$, there exists $x 	\in K_{w}$ satisfying
			\begin{equation}\label{BF.thick}
				K_{w} \supseteq B_{d}(x, c_{5}r_{\ast}^{n}).
			\end{equation}
		\end{itemize}
		\item\label{it:BF3} There exist $m_{1} \in \mathbb{N}$, $\gamma_{1} \in (0, 1)$ and $\gamma \in (0, 1)$ such that
		\begin{equation}\label{BF.super-exp}
			\measure(K_{w}) \ge \gamma\measure(K_{\pi(w)}) \quad \text{for any $w \in T$,}
		\end{equation}
		and
		\begin{equation}\label{BF.sub-exp}
			\measure(K_{v}) \le \gamma_{1}\measure(K_{w}) \quad \text{for any $w \in T$ and $v \in S^{m_{1}}(w)$.}
		\end{equation}
		Furthermore, $\measure$ is volume doubling with respect to $\metric$ and
		\begin{equation}\label{BF-nooverlap}
			\measure(K_{w}) = \sum_{v \in S(w)}\measure(K_{v}) \quad \text{for any $w \in T$.}
		\end{equation}
		\item\label{it:BF4} There exists $M_{0} \ge M_{\ast}$ such that for any $w \in T$, $k \ge 1$ and $v \in S^{k}(w)$,
		\[
		\Gamma_{M_{\ast}}(v) \cap S^{k}(w) \subseteq \Gamma_{M_{0}}^{S^{k}(w)}(v).
		\]
		\item\label{it:BF5} For any $w \in T$, $\pi(\Gamma_{M_{\ast} + 1}(w)) \subseteq \Gamma_{M_{\ast}}(\pi(w))$.
	\end{enumerate}
\end{assumption}
\begin{rmk}
	A partition satisfying the conditions above except for the connectedness of $K_{w}$ and Assumption \ref{assum.BF}\ref{it:BF3} exists if the compact metric space $(K, d)$ is uniformly perfect and metric doubling \cite[Proposition 3.11]{Sas23}.
	We can construct partitions (and a measure) satisfying all conditions in Assumption \ref{assum.BF} for many concrete examples.
\end{rmk}

If a given partition $\{ K_{w} \}_{w \in T}$ satisfies Assumption \ref{assum.BF} with metric $d$ and measure $\measure$, then we also say that $(K, d, \measure, \{ K_{w} \}_{w \in T})$ satisfies Assumption \ref{assum.BF} to denote metric $d$ and measure $\measure$ explicitly.

The following is a collection of consequences of our framework: Assumption \ref{assum.BF}.
\begin{prop}\label{prop.BF-result}
	Suppose that $(K, d, \measure, \{ K_{w} \}_{w \in T})$ satisfies Assumption \ref{assum.BF}.
	\begin{enumerate}[\rm(i)]
		\item\label{it:BF.nobdry} Define
		\begin{equation}\label{defn.boundary}
			S^{k}(w)^{\partial} = \{ v \in S^{m}(w) \mid K_{v} \cap B_{w} \neq \emptyset \}.
		\end{equation}
		Then there exists $m_{0} \ge 1$ such that $S^{k}(w) \setminus S^{k}(w)^{\partial} \neq \emptyset$ for any $w \in T$ and $k \ge m_{0}$.
		\item\label{it:BF.nomeas} The measure $\measure$ satisfies the following properties.
		There exists $\kappa > 0$ such that if $v, w \in T$ satisfy $\abs{v} = \abs{w}$ and $(v, w) \in E_{\abs{v}}^{\ast}$, then
		\begin{equation}\label{measure-gentle}
			\measure(K_{v}) \le \kappa\measure(K_{w}).
		\end{equation}
		For any $v, w \in T$ with $v \neq w$ and $\abs{v} = \abs{w}$,
		\begin{equation}\label{no-meas}
			\measure(K_{v} \cap K_{w}) = 0.
		\end{equation}
		In particular, $\measure(B_{w}) = 0$.
		Moreover, for any $w \in T$, $M \ge 1$ and $k \ge Mm_{0}$ ($m_{0}$ is the same as in \textup{(1)}), $B_{M, k}(w) \coloneqq \{ v \in S^{k}(w) \mid \Gamma_{M - 1}(v) \cap S^{k}(w)^{\partial} \neq \emptyset \}$ satisfies $S^{k}(w) \setminus B_{M, k}(w) \neq \emptyset$ and
		\begin{equation}\label{measure-inner}
			\measure\left(\bigcup_{v \in S^{n}\bigl(S^{k}(w) \setminus B_{M, k}(w)\bigr)}K_{v}\right) \ge \gamma^{m_{0}M}\measure(K_{w}).
		\end{equation}
		\item\label{it:BF.uniffinite} It holds that $N_{\ast} < +\infty$.
		\item\label{it:BF.thick} There exists a constant $c > 0$ (depending only on $r_{\ast}, c_{i}$ in Assumption \ref{assum.BF}) such that the following hold: for any $w \in T$ there exists $x_{w} \in O_{w}$ such that
		\[
		O_{w} \supseteq B_{d}\bigl(x_{w}, cr_{\ast}^{\abs{w}}\bigr).
		\]
	\end{enumerate}
\end{prop}
\begin{rmk}
	In \cite{Kig23}, the symbol $\partial S^{k}(w)$ is used instead of $S^{k}(w)^{\partial}$.
	We employ this notation to avoid conflict with notations used in graph theory.
\end{rmk}
\begin{proof}
	The statement \ref{it:BF.nobdry} is proved in \cite[Proposition 2.16]{Kig23} and \ref{it:BF.uniffinite} is shown in \cite[Lemma 2.13]{Kig23}.
	The statements in \ref{it:BF.nomeas} except for \eqref{no-meas} are proved in \cite[Proposition 2.16 and Lemma 2.14]{Kig23}.
	So, the rest is proving \eqref{no-meas} and \ref{it:BF.thick}.

	Let $v, w \in T$ such that $v \neq w$ and $\abs{v} = \abs{w} = n$ for some $n \ge 0$.
	Enumerate $T_{n}$ as $\{ z(1), z(2), \dots, z(l_{n}) \}$ so that $z(1) = v$ and $z(2) = w$, where $l_{n} = \#T_{n}$.
	Inductively, we define $\widetilde{K}_{z(j)}$ by $\widetilde{K}_{z(1)} \coloneqq K_{z(1)}$ and $\widetilde{K}_{z(j + 1)} \coloneqq K_{z(j + 1)} \setminus \left(\bigcup_{i = 1}^{k}\widetilde{K}_{z(i)}\right)$.
	Then $\bigl\{ \widetilde{K}_{z(j)} \bigr\}_{j = 1}^{l_{n}}$ is a disjoint family of sets and $\bigcup_{j = 1}^{l_{n}}\widetilde{K}_{z(j)} = K$.
	Therefore,
	\[
	1 = \measure(K) = \sum_{j = 1}^{l_{n}}\measure\Bigl(\widetilde{K}_{z(j)}\Bigr).
	\]
	On the other hand, Assumption \ref{assum.BF}\ref{it:BF3} implies that
	\[
	1 = \measure(K_{\phi}) = \sum_{j = 1}^{l_{n}}\measure\bigl(K_{z(j)}\bigr).
	\]
	Therefore, we conclude that $\measure\bigl(K_{z(j)} \setminus \widetilde{K}_{z(j)}\bigr) = 0$ for all $j \in \{1, \dots, l_{n}\}$.
	In particular,
	\[
	0 = \measure\Bigl(K_{z(2)} \setminus \widetilde{K}_{z(2)}\Bigr) = \measure\Bigl(K_{w} \setminus \bigl(K_{w} \setminus (K_{v} \cap K_{w})\bigr)\Bigr) = \measure(K_{v} \cap K_{w}),
	\]
	which proves \eqref{no-meas}.

	As mentioned in the remark after \cite[Assumption 2.15]{Kig23}, by Assumption \ref{assum.BF}\ref{it:BF2}, $d$ is \emph{thick} in the sense of \cite[Definition 3.1.19]{Kig20}.
	Since $\{ K_{w} \}_{w \in T}$ is assumed to be minimal, \ref{it:BF.thick} follows from \cite[Proposition 3.2.2]{Kig20}.
\end{proof}

Let $L \in \mathbb{N}$.
For $x, y \in K$, define
\begin{equation}\label{defn.nxy}
    n_{L}(x, y) \coloneqq \max\left\{ k \in \mathbb{Z}_{\ge 0} \;\middle|\;
    \begin{array}{c}
    \text{there exist $v, w \in T_{k}$ with $v \in \Gamma_{L}(w)$} \\
    \text{such that $x \in K_{v}$ and $y \in K_{w}$} \\
    \end{array}
    \right\}.
\end{equation}
Note that $n_{L}(x, y) \le n_{L'}(x, y)$ whenever $L \le L'$.
The following proposition is a useful characterization of \eqref{BF.adapted} in terms of $n_{L}(x, y)$.
\begin{prop}\label{prop.adapted}
	Suppose that $(K, \metric, \measure, \{ K_{w} \}_{w \in T})$ satisfies Assumption \ref{assum.BF}.
	Then there exists $C \ge 1$ (depending only on $r_{\ast}, M_{\ast}, c_{i}$ in Assumption \ref{assum.BF}) such that
	\begin{equation}\label{adapted.nxy}
		C^{-1}r_{\ast}^{n_{M_{\ast}}(x, y)} \le d(x, y) \le Cr_{\ast}^{n_{M_{\ast}}(x, y)} \quad \text{for any $x, y \in K$.}
	\end{equation}
\end{prop}
\begin{proof}
	This follows from \cite[(2.4.1)]{Kig20}.
	(As mentioned in \cite[page 30; after Definition 6.7]{Kig23}, we have $\delta_{M_{\ast}}^{g}(x, y) = r_{\ast}^{n_{M_{\ast}}(x, y)}$ in this setting, where $\delta_{M_{\ast}}^{g}$ is defined in \cite[Definition 2.3.8]{Kig20}).
\end{proof}

\begin{cor}\label{cor.sep}
	Suppose that $(K, \metric, \measure, \{ K_{w} \}_{w \in T})$ satisfies Assumption \ref{assum.BF}.
	Then there exists $c > 0$ (depending only on $r_{\ast}, M_{\ast}, c_{i}$ in Assumption \ref{assum.BF}) such that
	\begin{equation}\label{separated}
    	\inf\left\{ r_{\ast}^{-n}d(x, y) \;\middle|\;
    	\begin{array}{c}
    	\text{$n \in \mathbb{Z}_{\ge 0}$, $v, w \in T_{n}$, $x \neq y \in K$} \\
    	\text{such that $x \in K_{v}$, $y \in K_{w}$ and $v \not\in \Gamma_{M_{\ast}}(w)$} \\
    	\end{array}
    	\right\} \ge c.
	\end{equation}
\end{cor}
\begin{proof}
	Let $n \in \mathbb{Z}_{\ge 0}$ and $x \neq y \in K$.
	Assume that there exist $v, w \in T_{n}$ with $v \not\in \Gamma_{M_{\ast}}(w)$ such that $x \in K_{v}$ and $y \in K_{w}$.
	Then we have $n > n_{M_{\ast}}(x, y)$.
	Combining with Proposition \ref{prop.adapted}, we see that $r_{\ast}^{-n}d(x, y) \ge C^{-1}$, where $C \ge 1$ is the constant in \eqref{adapted.nxy}.
\end{proof}

Since Assumption \ref{a:reg} includes the following chain condition of the underlying compact metric space, we will assume this condition in addition to Assumption \ref{assum.BF}.
\begin{defn}
	Let $(X, d)$ be a metric space.
	For $\varepsilon > 0$ and $x, y \in X$, a sequence $\{ x_{i} \}_{i = 0}^{N}$ of points in $X$ is said to be a \emph{$\varepsilon$-chain between $x$ and $y$} if
	\[
	N \in \mathbb{N}, \quad x_{0} = x, \quad x_{N} = y \quad \text{and} \quad \max_{i \in \{ 0, \dots, N - 1 \}}d(x_{i}, x_{i + 1}) < \varepsilon.
	\]
	We also define
	\[
	d_{\varepsilon}(x, y) \coloneqq \biggl\{ \sum_{i = 0}^{N - 1}d(x_{i}, x_{i + 1}) \biggm| \text{$\{ x_{i} \}_{i = 0}^{N - 1}$ is an $\varepsilon$-chain between $x$ and $y$} \biggr\}.
	\]
	We say that the metric space $(X, d)$ satisfies the \emph{chain condition} if there exists $C \ge 1$ such that
	\begin{equation}\label{chain}
		d_{\varepsilon}(x, y) \le Cd(x, y) \quad \text{for all $\varepsilon > 0$ and $x, y \in X$.}
	\end{equation}
	The metric space $(X, d)$ is called \emph{geodesic} if for all $x, y \in X$ there exists a continuous map $\gamma \colon [0, 1] \to X$ satisfying
	\[
	\gamma(0) = x, \quad \gamma(1) = y \quad \text{and} \quad d(\gamma(s), \gamma(t)) = \abs{s - t}d(x, y) \text{ for all $s, t \in [0, 1]$,}
	\]
\end{defn}

\begin{prop}[{\cite[Proposition A.1]{KM20}}]\label{prop.chain}
	Let $(X, d)$ be a metric space such that $B_{d}(x, r)$ is relatively compact for any $x \in X$ and $r > 0$.
	Then the following are equivalent:
	\begin{enumerate}[\rm(1)]
		\item $(X, d)$ satisfies the chain condition.
		\item There exists a geodesic metric $\rho$ on $X$ which is bi-Lipschitz equivalent to $d$, i.e. there exists a constant $C \ge 1$ such that
		\begin{equation}\label{bL-geod}
			C^{-1}\rho(x, y) \le d(x, y) \le C\rho(x, y) \quad \text{for all $x, y \in X$.}
		\end{equation}
	\end{enumerate}
\end{prop}
\begin{rmk}\label{rmk.chain}
	The proof in \cite[Proposition A.1]{KM20} proves the following stronger results:
	\begin{enumerate}[label=$\bullet$,align=left,leftmargin=*,topsep=2pt,parsep=0pt,itemsep=2pt]
		\item If $(X, d)$ satisfies the chain condition, then $\rho(x, y) \coloneqq \lim_{\varepsilon \downarrow 0}d_{\varepsilon}(x, y)$ is a geodesic metric and $d \le \rho \le Cd$, where $C \ge 1$ is the same as in \eqref{chain}.
		\item If the condition (2) in the above proposition holds, then $d$ satisfies the chain condition with $d_{\varepsilon} \le C^{2}d$, where $C \ge 1$ is the same as in \eqref{bL-geod}.
	\end{enumerate}
\end{rmk}

The following lemma is a consequences of the chain condition in terms of partitions.
\begin{lem}\label{lem.BF-chain}
	Suppose that $(K, \metric, \measure, \{ K_{w} \}_{w \in T})$ Assumption \ref{assum.BF} and that $(K,\metric)$ satisfies the chain condition.
	\begin{enumerate}[\rm(i)]
		\item\label{it:chain1} There exists a constant $c > 0$ (depending only on $r_{\ast},M_{\ast},c_{i}$ in Assumption \ref{assum.BF} and $C \ge 1$ in \eqref{chain}) such that
			\begin{equation}\label{BF-chain}
    			\inf\left\{ \bigl(kr_{\ast}^{n}\bigr)^{-1}d(x, y) \;\middle|\;
    			\begin{array}{c}
    			\text{$n \in \mathbb{Z}_{\ge 0}$, $k \in \mathbb{N}$, $v, w \in T_{n}$, $x \neq y \in K$} \\
    			\text{such that $x \in K_{v}$, $y \in K_{w}$ and $v \not\in \Gamma_{k(M_{\ast} + 1) - 1}(w)$} \\
    			\end{array}
    			\right\} \ge c.
			\end{equation}
		\item\label{it:chain2} There exists a constant $C \ge 1$ such that for any $w \in T$ and $n \in \mathbb{Z}_{\ge 0}$,
			\[
			\diam(S^{n}(w), d_{n + \abs{w}}) \le C\,r_{\ast}^{-n}.
			\]
	\end{enumerate}
\end{lem}
\begin{proof}
	\ref{it:chain1} By Proposition \ref{prop.chain} (and Remark \ref{rmk.chain}), there exist a geodesic metric $\rho$ on $X$ and a constant $C \ge 1$ (depending only on the constant in \eqref{chain}) such that $C^{-1}\rho \le d \le C\rho$.
	Let $k \in \mathbb{N}$ and $v, w \in T$ with $\abs{v} = \abs{w} \eqqcolon n$ and $v \not\in \Gamma_{k(M_{\ast} + 1) - 1}(w)$.
	For $x \in K_{v}$ and $y \in K_{w}$, let $\gamma_{xy} \colon [0, 1] \to X$ be a geodesic from $x$ to $y$ with respect to $\rho$.
	Since $\gamma_{xy}$ is continuous, for each $j = 1, \dots, k - 1$, there exist $z_{j} \in \Gamma_{j(M_{\ast} + 1)}(v) \setminus \Gamma_{j(M_{\ast} + 1) - 1}$ and $t_{j} \in [0, 1]$ such that $\gamma_{xy}(t_{j}) \in K_{z_{j}}$.
	Then Corollary \ref{cor.sep} yields
	\[
	d(\gamma(t_{j}), \gamma(t_{j + 1})) \ge c\,r_{\ast}^{n} \quad \text{for any $j = 0, \dots, k - 1$,}
	\]
	where $c > 0$ is the same as in \eqref{separated}, $t_{0} = 0$ and $t_{k} = 1$.
	Since $\gamma_{xy}$ is a geodesic, we easily see that
	\begin{align*}
		\rho(x, y) = \sum_{j = 0}^{k - 1}\rho(\gamma(t_{j}), \gamma(t_{j + 1}))
		\ge C^{-1}\sum_{j = 0}^{k - 1}d(\gamma(t_{j}), \gamma(t_{j + 1}))
		\ge cC^{-1}kr_{\ast}^{n},
	\end{align*}
	which implies $\dist_{d}(K_{v}, K_{w}) \ge C'kr_{\ast}^{n}$ if we put $C' \coloneqq cC^{-2}$.

	\ref{it:chain2} Let $n \in \mathbb{Z}_{\ge 0}$ and $w \in T$.
	Choose $v, v' \in S^{n}(w)$ so that $d_{n + \abs{w}}(v, v') = \diam(S^{n}(w), d_{n + \abs{w}})$.
	We can assume that $\diam(S^{n}(w), d_{n + \abs{w}}) \ge M_{\ast}$.
	Let $k \in \mathbb{N}$ be the largest integer satisfying $k(M_{\ast} + 1) - 1 \le \diam(S^{n}(w), d_{n + m})$, i.e.
	\[
	k = \bigl\lfloor (\diam(S^{n}(w), d_{n + \abs{w}}) + 1)/(M_{\ast} + 1) \bigr\rfloor.
	\]
	By Lemma \ref{lem.BF-chain}, for any $x \in K_{v}$ and $x' \in K_{v'}$,
	\begin{align}\label{d-cell.1}
		d(x, x')
		\ge ckr_{\ast}^{n + \abs{w}}
		&\ge \frac{c}{2(M_{\ast} + 1)}(\diam(S^{n}(w), d_{n + \abs{w}}) + 1)\,r_{\ast}^{n + \abs{w}} \nonumber \\
		&\ge \frac{c}{2(M_{\ast} + 1)}\diam(S^{n}(w), d_{n + \abs{w}})\,r_{\ast}^{n + \abs{w}},
	\end{align}
	where $c > 0$ is the same as in \eqref{BF-chain} and we used $k \ge 2^{-1}(\diam(S^{n}(w), d_{n + \abs{w}}) + 1)/(M_{\ast} + 1)$ (since $\diam(S^{n}(w), d_{n + \abs{w}}) \ge M_{\ast}$) in the second inequality.

	On the other hand, we have $d(x, x') \le \diam(K_{w}, d) \le c_{2}\,r_{\ast}^{\abs{w}}$.
	Combining with \eqref{d-cell.1}, we get
	\[
	\diam(S^{n}(w), d_{n + \abs{w}}) \le 2c^{-1}(M_{\ast} + 1)c_{2}\,r_{\ast}^{-n}.
	\]
	We complete the proof.
\end{proof}

\subsection{Conductance and neighbor disparity constants}\label{sec.const}
Next we recall the definitions of conductance constants, neighbor disparity constants, the notion of $p$-conductive homogeneity and the function space $\mathcal{W}^{p}$ by following \cite{Kig23}.
Throughout this subsection, we fix $p \in [1,\infty)$, a compact metrizable space $K$, a partition $\{ K_{w} \}_{w \in T}$ and a Borel regular probability measure $\measure$ on $K$.

\begin{defn}[{\cite[Definitions 2.17 and 3.4]{Kig23}}]\label{defn.con-const}
	Let $n \in \mathbb{Z}_{\ge 0}$, $A \subseteq T_{n}$ and $A_{1},A_{2} \subseteq A$.
    Define
    \[
    \mathcal{E}_{p,k}(A_{1},A_{2},A) \coloneqq \mathrm{cap}_{p}^{(T_{n + k}, E_{n + k}^{\ast})}\bigl(S^{k}(A_{1}), S^{k}(A_{2}); S^{k}(A)\bigr).
    \]
    For $w \in A$ and $M \in \mathbb{N}$, define
	\begin{equation}\label{defn.conductance-local}
		\mathcal{E}_{M, p, k}(w, A) \coloneqq \mathcal{E}_{p,k}(\{w\},A \setminus \Gamma_{M}^{A}(w),A),
	\end{equation}
    which is called the \emph{$p$-conductance constant} of $w$ in $A$ at level $k$.
	We also define
	\begin{equation}\label{defn.conductance}
		\mathcal{E}_{M, p, k} \coloneqq \sup_{w \in T}\mathcal{E}_{M, p, k}(w, T_{\abs{w}}).
	\end{equation}
\end{defn}

\begin{defn}[{\cite[Definitions 2.26 and 2.29]{Kig23}}]
    Let $n \in \mathbb{N}$ and $A \subseteq T_{n}$.
    \begin{enumerate}[\rm(1)]
        \item For $k \in \mathbb{Z}_{\ge 0}$ and $f \colon T_{n + k} \to \mathbb{R}$, define $P_{n,k}f \colon T_{n} \to \mathbb{R}$ by
        	\[
        	(P_{n,k}f)(w) \coloneqq \frac{1}{\sum_{v \in S^{k}(w)}\measure(K_{v})}\sum_{v \in S^{k}(w)}f(v)\measure(K_{v}), \quad w \in T_{n}.
        	\]
        (Note that $P_{n,k}f$ depends on the measure $\measure$.)
        \item For $k \in \mathbb{Z}_{\ge 0}$, define
        \[
        \sigma_{p,k}(A) \coloneqq \sup_{f \colon S^{k}(A) \to \mathbb{R}}\frac{\mathcal{E}_{p,A}^{n}(P_{n,k}f)}{\mathcal{E}_{p,S^{k}(A)}^{n + k}(f)},
        \]
        which is called the \emph{$p$-neighbor disparity constant} of $A$ at level $k$.
        \item Let $\{ A_{i} \}_{i = 1}^{k}$ be a collection of subsets of $T_{n}$ and let $N_{T},N_{E} \in \mathbb{N}$. The family $\{ A_{i} \}_{i = 1}^{k}$ is called a \emph{covering} of $(A, E_{n}^{\ast}(A))$ with \emph{covering numbers} $(N_{T}, N_{E})$ if
        \[
        A = \bigcup_{i = 1}^{k}A_{i}, \quad \max_{x \in A}\#\{ i \mid x \in A_{i} \} \le N_{T},
        \]
        and for any $(u,v) \in E_{n}^{\ast}(A)$, there exist $l \le N_{E}$ and $\{ w(1), \dots, w(l + 1) \} \subseteq A$ such that $w(1) = u$, $w(l + 1) = v$ and $(w(i),w(i + 1)) \in \bigcup_{j = 1}^{k}E_{n}^{\ast}(A_{j})$ for any $i \in \{ 1, \dots, l \}$.
        \item Let $\mathscr{J} \subseteq \bigcup_{n \ge 0}\{ A \mid A \subseteq T_{n}\}$ and $N_{T},N_{E} \in \mathbb{N}$. The collection $\mathscr{J}$ is called a \emph{covering system} with covering numbers $(N_{T},N_{E})$ if the following conditions are satisfied.
        \begin{enumerate}[\rm(i)]
            \item $\sup_{A \in \mathscr{J}}\#A < \infty$.
            \item For any $w \in T$ and $k \in \mathbb{N}$, there exists a finite subset $\mathscr{N} \subseteq \mathscr{J}$ such that $\mathscr{N}$ is a covering of $(S^{k}(w), E_{n + k}^{\ast}(S^{k}(w)))$ with covering numbers $(N_{T}, N_{E})$.
            \item For any $A \in \mathscr{J}$ and $k \in \mathbb{Z}_{\ge 0}$ with $A \subseteq T_{n}$, there exists a finite subset $\mathscr{N} \subseteq \mathscr{J}$ such that $\mathscr{N}$ is a covering of $(S^{k}(A), E_{n + k}^{\ast}(S^{k}(A)))$ with covering numbers $(N_{T}, N_{E})$.
        \end{enumerate}
        The collection $\mathscr{J}$ is simply said to be a \emph{covering system} if there exist $N_{T},N_{E} \in \mathbb{N}$ such that $\mathscr{J}$ is a \emph{covering system} with covering numbers $(N_{T},N_{E})$.
        \item Let $\mathscr{J} \subseteq \bigcup_{n \ge 0}\{ A \mid A \subseteq T_{n}\}$ be a covering system. Define
        \[
        \sigma_{p,k,n}^{\mathscr{J}} \coloneqq \max\{ \sigma_{p,k}(A) \mid A \in \mathscr{J}, A \subseteq T_{n} \} \quad \text{and} \quad \sigma_{p,k}^{\mathscr{J}} \coloneqq \sup_{n \in \mathbb{Z}_{\ge 0}}\sigma_{p,k,n}^{\mathscr{J}}.
        \]
    \end{enumerate}
\end{defn}

For basic properties on conductance constants and neighbor disparity constants, see \cite[Section 2.2-2.4]{Kig23}.

Now we can introduce the notion of $p$-conductive homogeneity and recall its characterization.
\begin{defn}[{\cite[Definition 3.4]{Kig23}}]
    A compact metric space $K$ (with a partition $\{ K_{w} \}_{w \in T}$ and a measure $\measure$) is said to be \emph{$p$-conductively homogeneous} if there exists a covering system $\mathscr{J}$ such that
    \begin{equation}\label{d:pCH}
        \sup_{k \in \mathbb{Z}_{\ge 0}}\sigma_{p,k}^{\mathscr{J}}\mathcal{E}_{M_{\ast},p,k} < \infty.
    \end{equation}
\end{defn}

\begin{thm}[{\cite[Theorem 3.30]{Kig23}}]\label{t:pCH}
    A compact metric space $K$ is $p$-conductively homogeneous if and only if there exist $c_{1},c_{2} > 0$ and $\sigma(p) > 0$ such that
    \begin{equation}\label{pCH.1}
        c_{1}\sigma(p)^{-k} \le \mathcal{E}_{M_{\ast},p,k}(v, T_{n}) \le c_{2}\sigma(p)^{-k} \quad \text{and} \quad c_{1}\sigma(p)^{k} \le \sigma_{p,k,n} \le c_{2}\sigma(p)^{k}
    \end{equation}
    for any $k \in \mathbb{Z}_{\ge 0}$, $n \in \mathbb{N}$ and $v \in T_{n}$.
\end{thm}

We also recall the ``Sobolev'' space $\mathcal{W}^{p}$ due to Kigami.
\begin{defn}[{\cite[Lemma 3.13]{Kig23}}]
	Define
	\[
	\mathcal{W}^{p} \coloneqq \biggl\{ f \in L^{p}(K,\measure) \biggm| \sup_{n \in \mathbb{N}}\sigma_{p,n-1,1}^{\mathscr{J}}\mathcal{E}_{p}^{(T_{n},E_{n}^{\ast})}(P_{n}f) < \infty \biggr\},
	\]
	where $P_{n}f(w) \coloneqq \fint_{K_{w}}f\,d\measure$, $w \in T_{n}$.
\end{defn}

\begin{rmk}
    \begin{enumerate}[\rm(1)]
        \item The limits $\lim_{k \to \infty}\bigl(\mathcal{E}_{M_{\ast},p,k}\bigr)^{-1/k}$ and $\lim_{k \to \infty}\bigl(\sigma_{p,k}^{\mathscr{J}}\bigr)^{1/k}$ always exist by \cite[Corollary 2,24 and Lemma 2.34]{Kig23}.
        If $K$ is $p$-conductively homogeneous, then, by \eqref{pCH.1}, these limits must be equal to the constant $\sigma(p)$ in Theorem \ref{t:pCH}.
        \item Suppose that $(K,\metric,\measure,\{K_{w}\}_{w \in T})$ satisfies Assumption \ref{assum.BF}. Then, by \cite[Theorem 4.6.9]{Kig20}
        \begin{equation}\label{dARC-char}
        	\dim_{\mathrm{ARC}}(K,\metric) = \inf\Bigl\{ p \Bigm| \varlimsup_{k \to \infty}\mathcal{E}_{M_{\ast},p,k} = 0 \Bigr\}.
        \end{equation}
		If $K$ is $p$-conductively homogeneous, then \eqref{dARC-char} tells us that $p > \dim_{\mathrm{ARC}}(K,\metric)$ if and only if $\sigma(p) > 1$.
		However, there is a possibility that $\sigma_{p,k}^{\mathscr{J}} \ge 1$ for any $p \ge 1$ \cite[Proposition 2.31]{Kig23}.
		We need to avoid such a covering system $\mathscr{J}$ in the case of $p \le \dim_{\mathrm{ARC}}(K,d)$.
		For details, see \cite[p. 31]{Kig23}.
		\item If $\sigma_{p,k,1} \lesssim \sigma(p)^{k}$ for any $k \ge 0$, then
		\[
		\mathcal{W}^{p} = \biggl\{ f \in L^{p}(K,\measure) \biggm| \sup_{n \in \mathbb{N}}\sigma(p)^{n}\mathcal{E}_{p}^{(T_{n},E_{n}^{\ast})}(P_{n}f) < \infty \biggr\}.
		\]
    \end{enumerate}
\end{rmk}

\subsection{From Kigami's framework}
We now describe how to interpret partitions parametrized by a tree into the framework introduced in Section \ref{sec.unif}.
First, we fix our framework.
Suppose that $(K,\metric,\measure,\{ K_{w} \}_{w \in T})$ satisfies Assumption \ref{assum.BF} and let $p \in (1,\infty)$.
In addition, suppose that $\measure$ is $\hdim$-Ahlfors regular with respect to $d$ for some $\hdim \ge 1$ and that $(K,\metric)$ is $p$-conductively homogeneous.
Let $r_{\ast} \in (0,1)$ be the constant in Assumption \ref{a:Kig-MS}\ref{it:CHassum.1} and let $\sigma(p) > 0$ be the constant in Theorem \ref{t:pCH}.
Set $R_{\ast} \coloneqq r_{\ast}^{-1}$ and
\begin{equation}\label{d:pwalk-Kig}
	\pwalk \coloneqq \hdim + \frac{\log{\sigma(p)}}{\log{R_{\ast}}}.
\end{equation}
We will work under the following assumption.

\begin{assum}\label{a:Kig-MS}
	Let $p \in (1,\infty)$.
	Let $(K,\metric)$ be a compact metric space with $\diam(K,\metric) = 1$, let $\measure$ be a Borel regular probability measure on $K$, and let $\{ K_{w} \}_{w \in T}$ be a partition parametrized by a rooted tree $(T,\phi)$.
	We suppose the following conditions.
	\begin{enumerate}[label=\textup{(\arabic*)},align=left,leftmargin=*,topsep=2pt,parsep=0pt,itemsep=2pt]
		\item\label{it:CHassum.1} $(K,\metric,\measure,\{ K_{w} \}_{w \in T})$ satisfies Assumption \ref{assum.BF}.
		\item\label{it:CHassum.2} $(K,d)$ satisfies the chain condition.
		\item\label{it:CHassum.3} $\measure$ is $\hdim$-Ahlfors regular with respect to $d$.
		\item\label{it:CHassum.4} $(K,d)$ is $p$-conductively homogeneous.
		\item\label{it:CHassum.5} $\hdim - \pwalk < 1$.
	\end{enumerate}
\end{assum}

Hereafter, we fix $(K,\metric,\measure,\{ K_{w} \}_{w \in T})$ satisfying Assumption \ref{a:Kig-MS}.
We consider a sequence of finite connected graphs $\mathbb{G}_{n} \coloneqq (T_n, E_{n}^{\ast})$, $n \in \mathbb{N}$.
For $n > k \ge 1$, define $\pi_{n, k} \colon T_{n} \to T_{k}$ by
\[
\pi_{n, k}(w_{1}w_{2}\dots w_{n}) = w_{1}w_{2}\dots w_{k} \quad \text{for $w = w_{1}w_{2}\dots w_{n} \in T_{n}$.}
\]
Equivalently, $\pi_{n, k} = \pi^{n - k}|_{T_{n}}$, where $\pi$ is the map in \eqref{defn.pi}.
Then it is clear that $\{ \pi_{n, k} ; 1 \le k < n \}$ is a projective family (see Definition \ref{defn.pf}).
Furthermore, we easily see that $\pi_{n + k, n}^{-1}(w) = S^{k}(w)$ for any $n \in \mathbb{N}$, $k \ge 0$  and $w \in T_{n}$.
Define probability measures $\measure_n$ on $T_n$ by setting $\measure_{n}(w) = \measure(K_w)$ for $w \in T_{n}$.
Then $(\measure_n)_{n \ge 0}$ is consistent (with respect to $\{ \pi_{n, k} \}$) by \eqref{BF-nooverlap} and $\pi_{n + k, n}^{-1}(w) = S^{k}(w)$.

The next theorem is the main result of this section.
\begin{thm}\label{t:Kig-MS}
	Suppose that $(K,\metric,\measure,\{ K_{w} \}_{w \in T})$ satisfies Assumption \ref{a:Kig-MS}.
	Let $R_{\ast} \in (0,1), \hdim \ge 1, \pwalk > 0,\{ \mathbb{G}_{n} \}_{n \in \mathbb{N}}, \{ \pi_{n,k} \mid 1 \le k < n \}$ be given as above.
	Then $\{ \mathbb{G}_{n} \}$ along with $\{ \pi_{n,k} \}$ satisfies Assumption \ref{a:reg}.
\end{thm}
\begin{proof}
	We first show that $\{ \mathbb{G}_{n} \}$ along with $\{ \pi_{n,k} \}$ is $R_{\ast}$-scaled and $R_{\ast}$-compatible with $(K,\metric)$.
	To this end, we introduce a new family $\{ \widetilde{K}_{w} \}_{w \in T}$ as follows.
	Set $\widetilde{K}_{\phi} \coloneqq K$ and enumerate $T_{n}$ so that $T_{n} = \{ w(1;n), \dots, w(l_{n};n) \}$, $n \in \mathbb{N}$.
	Inductively, we define $\bigl\{ \{ \widetilde{K}_{w} \}_{w \in T_{n}} \bigm| n \in \mathbb{Z}_{\ge 0} \bigr\}$ by
	\[
	\widetilde{K}_{w(1;n)} \coloneqq K_{w(1;n)} \cap \widetilde{K}_{\pi(w(1;n))} \quad \text{and} \quad \widetilde{K}_{w(j;n)} \coloneqq \left(K_{w(j;n)} \setminus \bigcup_{i = 1}^{j - 1} \widetilde{K}_{w(i;n)}\right) \cap \widetilde{K}_{\pi(w(1;n))}.
	\]
	Then it is clear that $\{ \widetilde{K}_{w} \}_{w \in T_{n}}$ are disjoint family of Borel sets and $\widetilde{K}_{w} = \bigcup_{v \in S(w)}\widetilde{K}_{v}$ for any $w \in T$.

	Note that $\diam(\pi_{n+k,n}^{-1}(w),d_{n+k}) \le CR_{\ast}^{k}$ for any $w \in T_{n}$ and $k \in \mathbb{Z}_{\ge 0}$, where $C \ge 1$ is the same as in Lemma \ref{lem.BF-chain}\ref{it:chain2}.
	For each $w \in T_{n}$, choose $p_{n}(w) \in O_{w}$ so that $B_{\metric}(p_{n}(w),cR_{\ast}^{-n}) \subseteq O_{w} \subseteq \widetilde{K}_{w}$, where $c > 0$ is the constant in Proposition \ref{prop.BF-result}\ref{it:BF.thick}.
	Let $c_{k}(w) \in T_{k + \abs{w}}$ be the element such that $p_{n}(w) \in \widetilde{K}_{c_{k}(w)}$ for each $k \ge 0$.
	Then we immediately have $d_{k+n}(c_{k}(v),c_{k}(w)) \le 2CR_{\ast}^{k}$ for any $\{ v,w \} \in E_{n}^{\ast}$, i.e., \eqref{e:sc2} holds.
	Let $A_{1} \ge 1$ and set $B_{k}(w) \coloneqq B_{d_{k + \abs{w}}}(c_{k}(w),A_{1}^{-1}R_{\ast}^{k})$ for $k \in \mathbb{N}$ and $w \in T$.
	If $A_{1}$ is large enough so that $2c_{2}A_{1}^{-1} \le c$, where $c_{2} > 0$ is the constant in \eqref{BF.bi-Lip}, then
	\[
	K_{B_{k}(w)} \subseteq B_{\metric}(p_{n}(w),2A_{1}^{-1}R_{\ast}^{k} \times c_{2}r_{\ast}^{n + k}) \subseteq B_{\metric}(p_{n}(w), cR_{\ast}^{-n}) \subseteq  K_{w} = K_{\pi_{n+k,n}^{-1}(w)},
	\]
	which together with Lemma \ref{lem.minimal} implies $B_{k}(w) \subseteq \pi_{n+k,n}^{-1}(w)$.
	Hence \eqref{e:sc1} holds by putting $A_{1} \coloneqq C \vee (2c^{-1}c_{2})$.
	Therefore $\{ \mathbb{G}_{n} \}$ along with $\{ \pi_{n,k} \}$ is $R_{\ast}$-scaled.

	Next we show that $\{ \mathbb{G}_{n} \}$ along with $\{ \pi_{n,k} \}$ is $R_{\ast}$-compatible with $(K,\metric)$.
	It is immediate from \eqref{BF.bi-Lip} that $\metric(p_{n}(v), p_{n}(w)) \le 2d_{n}(v,w) \times c_{2}R_{\ast}^{-n}$ for any $v,w \in T_{n}$, which gives the upper estimate of \eqref{e:holder}.
	The converse estimate $\metric(p_{n}(v), p_{n}(w)) \gtrsim d_{n}(v,w)R_{\ast}^{-n}$ follows from Lemma \ref{lem.BF-chain}\ref{it:chain1}, and hence Definition \ref{d:compatible}\ref{it:compat.comp} holds.
	The other properties \ref{it:compat.parti}-\ref{it:compat.round} in Definition \ref{d:compatible} are obvious, so $\{ \mathbb{G}_{n} \}$ along with $\{ \pi_{n,k} \}$ is $R_{\ast}$-compatible.

	Lastly, we show \ref{cond.UPI} and \hyperref[cond.UCF]{\textup{U-CF$_{p}(\vartheta, \beta)$}} (for some $\vartheta \in (0,1]$.)
	By virtue of Propositions \ref{prop.UPI} and \ref{prop.UCF}, it is enough to show that $\{ \mathbb{G}_{n} \}$ satisfies \ref{cond.UAR}, \hyperref[cond.UBCL]{\textup{U-BCL$_{p}^{\textup{low}}(\hdim - \pwalk)$}} and \hyperref[cond.Ucap]{\textup{U-cap$_{p,\le}(\pwalk)$}}.
	Note that $\measure_{n}(w) = \measure(K_{w}) = \measure(\widetilde{K}_{w})$ by \eqref{BF-nooverlap} and hence \ref{cond.UAR} is immediate from Lemma \ref{l:bp}.
	Combining \eqref{pCH.1} and \eqref{e:sc1}, we easily obtain \hyperref[cond.Ucap]{\textup{U-cap$_{p,\le}(\pwalk)$}}.
	The rest of this proof will be devoted to \hyperref[cond.UBCL]{\textup{U-BCL$_{p}^{\textup{low}}(\hdim - \pwalk)$}}.
	(The argument is very similar to the proof of Proposition \hyperref[it:PSCanalysis1]{\ref{prop.PSC-analysis}}\ref{it:PSCanalysis3}.)

	Let $\kappa > 0$, $n \in \mathbb{N}$ and $1 \le R \le \diam(\mathbb{G}_{n})$.
	Let $B_{i} = B_{d_{n}}(x_{i},R)$, $x_{i} \in T_{n}$, $i = 1,2$, such that $\dist_{d_{n}}(B_{1}, B_{2}) \le \kappa R$.
	Recall that $C \ge 1$ is a constant such that $\diam(\pi_{n+k,n}^{-1}(w),d_{n+k}) \le CR_{\ast}^{k}$.
	Choose $n(R) \in \mathbb{Z}$ so that
	\[
	2CR_{\ast}^{n(R)} < R \le 2CR_{\ast}^{n(R) + 1}.
	\]
	By $R \le \diam(\mathbb{G}_{n})$ and $\diam(\mathbb{G}_{n}) \le 2Ca_{\ast}^{n}$, we then have $n \ge n(R)$.

	First, we consider the case of $R > 2C$.
	Then $n(R) \ge 0$.
	It is a simple observation that there exist $w(1), w(2) \in T_{n - n(R)}$ such that
	\[
	\text{$S^{n(R)}\bigl(w(i)\bigr) \subseteq B_{i}$} \quad \text{and} \quad \text{$x_{i} \in S^{n(R)}(w(i))$} \quad \text{for each $i = 1, 2$.}
	\]
	Then, we have
	\[
	\dist_{d_{n}}\Bigl(S^{n(R)}\bigl(w(1)\bigr), S^{n(R)}\bigl(w(2)\bigr)\Bigr) \le R + \kappa R + R \le 2(2 + \kappa)R_{\ast} \cdot R_{\ast}^{n(R)}.
	\]
	This together with Lemma \ref{lem.BF-chain}\ref{it:chain1} implies that there exist $w \in T_{n - n(R)}$ and $M(\kappa) \in \mathbb{N}$ (depending only on $\kappa,R_{\ast},M_{\ast}$ and the constants $c_{i}$ in Assumption \ref{assum.BF}) such that $w(i) \in \Gamma_{M(\kappa)}(w)$.
	Set $L(\kappa) \coloneqq (2M(\kappa) + 1)A_{1}/(2C)$, where $A_{1} \ge 1$ is the constant in \eqref{e:sc1}.
	Using Lemma \ref{lem.basic-pMod}\ref{it:mod.mono} and following a similar argument to \eqref{KM-mod}, we can show that
	\begin{align*}
		&\MOD_{p}^{\mathbb{G}_{n}}(\{ \theta \in \PATH(B_1, B_2; \mathbb{G}_{n}) \mid \diam(\theta, d_{n}) \le L(\kappa)R \}) \\
		&\ge \MOD_{p}^{\mathbb{G}_{n}}\Bigl(\bigl\{ \theta \in \PATH\bigl(S^{n(R)}\bigl(w(1)\bigr), S^{n(R)}\bigl(w(2)\bigr); \mathbb{G}_{n}\bigr) \bigm| \diam(\theta, d_{n}) \le 2CL(\kappa)R_{\ast}^{n(R)} \bigr\}\Bigr) \\
		&\ge \MOD_{p}^{\mathbb{G}_{n}}\bigl(S^{n(R)}\bigl(w(1)\bigr), S^{n(R)}\bigl(w(2)\bigr); S^{n(R)}\bigl(\Gamma_{M(\kappa)}(w)\bigr)\bigr) \\
		&\gtrsim \mathcal{E}_{p,n(R)}(w(1),w(2),\Gamma_{M(\kappa)}(w)) \quad \text{(by Lemma \ref{lem.mod/cap})}.
	\end{align*}
	By \cite[(2.16)]{Kig23}, \eqref{pCH.1} and a similar argument as the proof of \cite[Lemma 3.32]{Kig23}, we have $\sigma(p)^{n(R)}\mathcal{E}_{p,n(R)}(w(1),w(2),\Gamma_{M(\kappa)}(w)) \gtrsim 1$.
	Since $\sigma(p)^{-n(R)} \asymp R^{\hdim - \pwalk}$, there exists a constant $c(\kappa) > 0$ (depending only on $p,\kappa,R_{\ast},M_{\ast}$ and the constants $c_{i}$ in Assumption \ref{assum.BF}) such that
	\[
	\MOD_{p}^{\mathbb{G}_{n}}(\{ \theta \in \PATH(B_1, B_2; \mathbb{G}_{n}) \mid \diam(\theta, d_{n}) \le L(\kappa)R \}) \ge c(\kappa)R^{\hdim - \pwalk}.
	\]

	Let us consider the case $1 \le R \le 2C$ to complete the proof.
	By \eqref{smallscale} in Lemma \ref{lem.shortPath},
	\begin{align*}
		\MOD_{p}^{\mathbb{G}_{n}}(\{ \theta \in \PATH(B_1, B_2; \mathbb{G}_{n}) \mid \diam(\theta, d_{n}) \le L(\kappa)R \})
		&\ge \bigl(L(\kappa)R\bigr)^{1 - p} \\
		&\ge (2C)^{-p}L(\kappa)^{1 - p}R^{\hdim - \pwalk},
	\end{align*}
	where we used $\hdim - \pwalk < 1$ (Proposition \hyperref[it:PSCanalysis1]{\ref{prop.PSC-analysis}}\ref{it:PSCanalysis1}) and $R \le 2C$ in the last inequality.
\end{proof}

\begin{cor}\label{cor.LB-Kig}
	Suppose that $(K,\metric,\measure,\{ K_{w} \}_{w \in T})$ satisfies Assumption \ref{a:Kig-MS}.
	\begin{enumerate}[\rm(i)]
		\item\label{it:LB-Kig.WM} We have $\mathcal{W}^{p} = \mathcal{F}_{p} = B_{p,\infty}^{\pwalk/p}$, with comparability of norms.
		Moreover, there exist a constant $C \ge 1$ such that for any $f \in L^{p}(K,\measure)$,
		\[
		\sup_{r > 0}\int_{K}\fint_{B_{\metric}(x,r)}\frac{\abs{f(x) - f(y)}^{p}}{r^{\pwalk}}\,\measure(dy)\measure(dx)
		\le C\varliminf_{r \downarrow 0}\int_{K}\fint_{B_{\metric}(x,r)}\frac{\abs{f(x) - f(y)}^{p}}{r^{\pwalk}}\,\measure(dy)\measure(dx).
		\]
		\item\label{it:LB-Kig.PI} There exist constants $C \ge 1$ and $A \ge 1$ such that for any $f \in L^{p}(K,\measure)$, $z \in K$ and $R > 0$,
		\[
		\int_{B_{\metric}(z,R)}\abs{f - f_{B_{\metric}(z,R)}}^{p}\,d\measure \le CR^{\pwalk}\varliminf_{s \downarrow 0}\int_{B_{\metric}(z,AR)}\fint_{B_{\metric}(x,s)}\frac{\abs{f(x) - f(y)}^{p}}{s^{\pwalk}}\,\measure(dy)\measure(dz).
		\]
	\end{enumerate}
\end{cor}
\begin{proof}
	\ref{it:LB-Kig.WM} Recall the definition of the normalized energy $\widetilde{\mathcal{E}}_{p}^{(n)}$ in \eqref{e:normalizedE}.
	The identity $\mathcal{W}^{p} = \mathcal{F}_{p}$ immediately follows from $\widetilde{\mathcal{E}}_{p}^{(n)}(f) = \sigma(p)^{n}\mathcal{E}_{p}^{\mathbb{G}_{n}}(M_{n}f)$.
	Hence Theorem \ref{thm.LB} yields the desired statements.

	\ref{it:LB-Kig.PI} This follows from a combination of Lemmas \ref{lem.PI-like}, \ref{lem.lower} and \ref{lem.upper}.
\end{proof}

\subsection{Conductive homogeneity of the Sierpi\'{n}ski carpet}
In this subsection, we prove the $p$-conductive homogeneity of the planar Sierpi\'{n}ski carpet.
Hereafter, let $p \in (1,\infty)$, let $(K,S,\{ F_{i} \}_{i \in S})$ be the planar Sierpi\'{n}ski carpet, let $\{ \mathbb{G}_{n} \}_{n \in \mathbb{N}}$ be the sequence of finite graphs as in Section \ref{sec.PSC} and let $\measure$ be the self-similar measure on $K$ with the weight $(1/8,\dots,1/8)$.
Recall that $a_{\ast} = 3$, $\hdim = \log{8}/\log{3}$, $\pwalk = \log{(8\rho(p))}/\log{3}$ and $P_{n}f(w) = M_{n}f(w) = \fint_{K_w}f\,d\measure$ for $n \in \mathbb{Z}_{\ge 0}$, $f \in L^{p}(K,\measure)$.

The following main theorem in this subsection follows from a combination of \hyperref[cond.UPI]{\textup{U-PI$_{p}(\pwalk)$}} and the self-similarity.
\begin{thm}\label{t:pch}
	The Sierpi\'{n}ski carpet equipped with the self-similar measure with the equal weight is $p$-conductively homogeneous for any $p \in (1,\infty)$.
	In particular, $\sigma(p) = \rho(p)$ and $\mathcal{F}_{p} = \mathcal{W}^{p}$.
\end{thm}
\begin{proof}
	First, we fix a choice of covering systems.
	Define $\mathscr{J}_{\ell}$ (\cite[(4.15)]{Kig23}) by
	\[
	\mathscr{J}_{\ell} = \{ \{v,w\} \mid \text{$\{ v,w \} \in E_{n}^{\ast}$ for some $n \in \mathbb{Z}_{\ge 0}$, $\#(K_{v} \cap K_{w}) \ge 2$} \}.
	\]

	By Theorem \ref{thm.assum-PSC}\ref{it:PSCassum1}, we can choose a constant $\lambda \ge 1$ so that the following statement holds: For any $k,l \in \mathbb{N}$ and $\{ v,w \} \in \mathscr{J}_{\ell} \cap E_{l}^{\ast}$, there exists $c_{k}(v,w) \in S^{k}(\{v,w\})$ such that $S^{k}(\{v,w\}) \subseteq B_{d_{k + l}}(c_{k}(v,w),\lambda a_{\ast}^{k})$.
	Fix a large enough $k_{\ast} \in \mathbb{N}$ so that $\lambda A_{\mathrm{PI}} a_{\ast}^{-k_{\ast}} < 1$, where $A_{\mathrm{PI}}$ is the constant in \hyperref[cond.UPI]{\textup{U-PI$_{p}(\pwalk)$}} (Theorem \ref{thm.assum-PSC}\ref{it:PSCassum2}).
	We note that, by choosing $R = 2\diam(\mathbb{G}_{n})$ in \hyperref[cond.UPI]{\textup{U-PI$_{p}(\pwalk)$}}, there exists $\widetilde{C}_{\mathrm{PI}} > 0$ such that
	\begin{equation}\label{globalUPI}
		\sum_{y \in W_{n}}\abs{f(y) - f_{W_{n}}}^{p} \le \widetilde{C}_{\mathrm{PI}}a_{\ast}^{n\pwalk}\mathcal{E}_{p}^{\mathbb{G}_{n}}(f) \quad \text{for any $n \in \mathbb{N}$ and $f \in \mathbb{R}^{W_{n}}$}.
	\end{equation}
	To prove the $p$-conductive homogeneity \eqref{d:pCH} with $\mathscr{J} = \mathscr{J}_{\ell}$, it is enough to show that $\sigma_{p,n}^{\mathscr{J}_{\ell}} \lesssim \rho(p)^{n}$ for any $n \in \mathbb{N}$ by \eqref{cc-comparable}.
	Let us fix $l \in \mathbb{N}$ and $\{ v,w \} \in \mathscr{J}_{\ell} \cap E_{l}^{\ast}$.
	It is easy to find $(v', w') \in S^{k_{\ast}}(v) \times S^{k_{\ast}}(w)$ satisfying $\{ v',w' \} \in \mathscr{J}_{\ell}$ and $B_{d_{n + l + k_{\ast}}}(c_{n}(v',w'), \lambda A_{\mathrm{PI}}a_{\ast}^{n}) \subseteq S^{n}(\{v,w\})$ since `there are $a_{\ast}^{k_{\ast} + l}$ copies of a $n$-cell along the intersection $K_{v} \cap K_{w}$'.
	For simplicity, set
	\[
	B_{v',w'}^{n} \coloneqq B_{d_{n + l + k_{\ast}}}(c_{n}(v',w'), \lambda a_{\ast}^{n}) \quad \text{and} \quad A_{\mathrm{PI}}B_{v',w'}^{n} \coloneqq B_{d_{n + l + k_{\ast}}}(c_{n}(v',w'), \lambda A_{\mathrm{PI}} a_{\ast}^{n}).
	\]
	Let $z_{1},z_{2} \in W_{k_{\ast}}$ such that $v' = vz_{1}$ and $w' = wz_{2}$.
	Similar to \eqref{e:wm4}, we have from \hyperref[cond.UPI]{\textup{U-PI$_{p}(\pwalk)$}} that for any $f \in L^{p}(K,\measure)$ and $n \in \mathbb{N}$,
	\begin{align*}
		\abs{f_{K_{v'}} - f_{K_{w'}}}^{p}
		\le C_{1}a_{\ast}^{(n + k_{\ast})(\pwalk-\hdim)}\mathcal{E}_{p,A_{\mathrm{PI}}B_{v',w'}^{n}}^{\mathbb{G}_{n + l + k_{\ast}}}(M_{n + l + k_{\ast}}f)
		\le C_{1}\rho(p)^{n + k_{\ast}}\mathcal{E}_{p,S^{n + k_{\ast}}(\{v,w\})}^{\mathbb{G}_{n + l + k_{\ast}}}(M_{n + l + k_{\ast}}f),
	\end{align*}
	where $C_{1} > 0$ is independent of $f,n,v,w$.
	In addition,
	\begin{align}\label{pch.diff.2}
		\abs{f_{K_{v'}} - f_{K_{v}}}^{p}
		&= \abs{(f \circ F_{v})_{K} - (f \circ F_{v})_{K_{z_1}}}^{p}
		= \abs{\bigl(M_{n + k_{\ast}}(f \circ F_{v})\bigr)_{W_{n + k_{\ast}}} - \bigl(M_{n + k_{\ast}}(f \circ F_{v})\bigr)_{S^{n}(z_{1})}}^{p} \nonumber \\
		&\le (\#S^{n}(z_1))^{-1}\sum_{x \in W_{n + k_{\ast}}}\abs{\bigl(M_{n + k_{\ast}}(f \circ F_{v})\bigr)(x) - \bigl(M_{n + k_{\ast}}(f \circ F_{v})\bigr)_{W_{n + k_{\ast}}}}^{p} \nonumber \\
		&\le (\#S^{n}(z_1))^{-1}\widetilde{C}_{\mathrm{PI}}a_{\ast}^{(n + k_{\ast})\pwalk}\mathcal{E}_{p}^{\mathbb{G}_{n + k_{\ast}}}\bigl(M_{n + k_{\ast}}(f \circ F_{v})\bigr) \quad \text{(by \eqref{globalUPI})} \nonumber \\
		&\le C_{2}a_{\ast}^{(n + k_{\ast})(\pwalk - \hdim)}\mathcal{E}_{p,S^{n + k_{\ast}}(\{ v,w \})}^{\mathbb{G}_{n + l + k_{\ast}}}(M_{n + l + k_{\ast}}f),
	\end{align}
	where $C_{2} > 0$ is also independent of $f,n,v,w$.
	Similar to \eqref{pch.diff.2}, we have $\abs{f_{K_{w'}} - f_{K_{w}}}^{p} \le C_{2}a_{\ast}^{n(\pwalk - \hdim)}\mathcal{E}_{p,S^{n + k_{\ast}}(\{ v,w \})}^{\mathbb{G}_{n + l + k_{\ast}}}(M_{n + l + k_{\ast}}f)$.
	Combining these estimates, we show that
	\begin{align*}
		&\abs{(M_{n + l + k_{\ast}}f)_{S^{n + k_{\ast}}(v)} - (M_{n + l + k_{\ast}}f)_{S^{n + k_{\ast}}(w)}}^{p}
		= \abs{f_{K_v} - f_{K_w}}^{p} \\
		&\quad\le 3^{p - 1}\Bigl(\abs{f_{K_v} - f_{K_{v'}}}^{p} + \abs{f_{K_{v'}} - f_{K_{w'}}}^{p}+ \abs{f_{K_{w}} - f_{K_{w'}}}^{p}\Bigr) \\
		&\quad\le 3^{p}Ca_{\ast}^{(n + k_{\ast})(\pwalk - \hdim)}\mathcal{E}_{p,S^{n + k_{\ast}}(\{ v,w \})}^{\mathbb{G}_{n + l + k_{\ast}}}(M_{n + l + k_{\ast}}f)
		= 3^{p}C\rho(p)^{n + k_{\ast}}\mathcal{E}_{p,S^{n + k_{\ast}}(\{ v,w \})}^{\mathbb{G}_{n + l + k_{\ast}}}(M_{n + l + k_{\ast}}f),
	\end{align*}
	where $C \coloneqq C_1 \vee C_2$.
	This estimate implies $\sigma_{p,n + k_{\ast}}^{\mathscr{J}_{\ell}} \le 3^{p}C\rho(p)^{n + k_{\ast}}$ for any $n \in \mathbb{N}$ and proves \eqref{d:pCH}.
\end{proof}

\noindent
\textbf{Acknowledgments.} We are grateful to Mario Bonk, Jun Kigami, Naotaka Kajino, Bruce Kleiner for some illuminating conversations related to this work, and the anonymous referees for their careful reading of this manuscript and helpful suggestions.
In particular, we thank Naotaka Kajino for his valuable comments on earlier versions of this paper and his suggestions on Assumption \ref{assum.ss} and Lemma \ref{lem.em-wlocal}, and Bruce Kleiner for the proof of Proposition \ref{prop.singular}. 


\noindent \textbf{Mathav Murugan} \\
Department of Mathematics, The University of British Columbia,
Vancouver, BC V6T 1Z2, Canada. \\
E-mail: \texttt{mathav@math.ubc.ca}\smallskip\\

\noindent \textbf{Ryosuke Shimizu} (JSPS Research Fellow PD) \\
Waseda Research Institute for Science and Engineering, Waseda University, 3-4-1 Okubo, Shinjuku-ku, Tokyo 169-8555, Japan. \\
E-mail: \texttt{r-shimizu@aoni.waseda.jp}


\begin{thebibliography}{}
\setlength{\itemsep}{-1.0pt}

	\bibitem[ABPPW]{ABPPW} N.~Albin, M.~Brunner, R.~Perez, P.~Poggi-Corradini and N.~Wiens,
	Modulus on graphs as a generalization of standard graph theoretic quantities,
	\emph{Conform. Geom. Dyn.} \textbf{19} (2015), 298--317. \mr{3430866}


	\bibitem[ARB23+]{AB23+} P.~Alonso Ruiz and F.~Baudoin,
	Dirichlet forms on metric measure spaces as Mosco limits of Korevaar--Schoen energies, 
	\emph{Ann. Sc. Norm. Super. Pisa Cl. Sci. (5)} (to appear). 
	{\tt \href{https://doi.org/10.2422/2036-2145.202302_013}{DOI:10.2422/2036-2145.202302\_013}}
	
	\bibitem[ARB25]{ARB24+} P.~Alonso Ruiz and F.~Baudoin,
    Korevaar-Schoen $p$-energies and their $\Gamma$-limits on Cheeger spaces,
    \emph{Nonlinear Anal.} (in press), (2025). \doi{10.1016/j.na.2025.113779} 

	\bibitem[AHM23]{AHM23} R.~Alvarado, P.~Haj\l asz and L.~Mal\'{y},
	A simple proof of reflexivity and separability of $N^{1, p}$ Sobolev spaces,
	\emph{Ann. Fenn. Math.} \textbf{48} (2023), no. 1, 255--275. \mr{4559212}
	
	\bibitem[AGS13]{AGS13} L.~Ambrosio, N.~Gigli and G.~Savar\'e, 
	Density of Lipschitz functions and equivalence of weak gradients in metric measure spaces,
	\emph{Rev. Mat. Iberoam.}, \textbf{29} (2013) no. 3, 969--996. \mr{3090143} 
	
	\bibitem[AEB24+a]{AE.kleiner} R.~Anttila and S.~Eriksson-Bique,
	On constructions of fractal spaces using replacement and the combinatorial Loewner property, 
	preprint (2024). 
	\arxiv{2406.08062}
	
	\bibitem[AEB24+b]{AE.super} R.~Anttila and S.~Eriksson-Bique,
	Iterated graph systems and the combinatorial Loewner property, 
	preprint (2024). 
	\arxiv{2408.15692}

	\bibitem[BCLS]{BCLS} D.~Bakry, T.~Coulhon, M.~Ledoux and L.~Saloff-Coste,
	Sobolev inequalities in disguise,
	\emph{Indiana Univ. Math. J.} \textbf{44} (1995), no. 4, 1033–1074. \mr{1386760}

	\bibitem[Bar]{Bar.RW} M.~T.~Barlow,
    Random walks and heat kernels on graphs,
    London Mathematical Society Lecture Note Series, {\bf 438}.
    {\em Cambridge University Press}, Cambridge (2017). \mr{3616731}

	\bibitem[BB89]{BB89} M.~T.~Barlow and R.~F.~Bass,
	The construction of Brownian motion on the Sierpi\'{n}ski carpet,
	\emph{Ann. Inst. H. Poincar\'{e} Probab. Statist.}\ \textbf{25} (1989), no.\ 3, 225--257. \mr{1023950}

	\bibitem[BB90]{BB90} M.~T.~Barlow and R.~F.~Bass,
	On the resistance of the Sierpi\'{n}ski carpet,
	\emph{Proc. Roy. Soc. London Ser. A}\ \textbf{431} (1990), no.\ 1882, 345--360. \mr{1080496} 

	\bibitem[BB99]{BB99} M.~T.~Barlow and R.~F.~Bass,
	Brownian motion and harmonic analysis on Sierpinski carpets,
	\emph{Canad. J. Math.}\ \textbf{51} (1999), no.\ 4, 673--744. \mr{1701339}

	\bibitem[BBKT]{BBKT} M.~T.~Barlow, R.~F.~Bass, T.~Kumagai and A.~Teplyaev,
	Uniqueness of Brownian motion on Sierpi\'{n}ski carpets,
	\emph{J. Eur. Math. Soc. (JEMS)}\ \textbf{12} (2010), no.\ 3, 655--701. \mr{2639315}
	
	\bibitem[BCK05]{BCK05} M.~T.~Barlow, T.~Coulhon and T.~Kumagai,
	Characterization of sub-Gaussian heat kernel estimates on strongly recurrent graphs, 
	\emph{Comm. Pure Appl. Math.} \textbf{58} (2005), no.\ 12, 1642--1677. \mr{2177164}

	\bibitem[Bau24]{Bau24} F.~Baudoin,
	Korevaar-Schoen-Sobolev spaces and critical exponents in metric measure spaces, 
    \emph{Ann. Fenn. Math.}, \textbf{49} (2024) no. 2, 487--527. \mr{4791991}

	\bibitem[BBR24]{BBR24} C.~Beznea, L.~Beznea and M.~R\"{o}ckner,
    Nonlinear Dirichlet forms associated with quasiregular mappings, 
    \emph{Potential Anal.} (2024). \doi{10.1007/s11118-024-10145-5} 

	\bibitem[BV05]{BV05} M.~Biroli and P.~G.~Vernole,
	Strongly local nonlinear Dirichlet functionals and forms,
	\emph{Adv. Math. Sci. Appl.}\ \textbf{15} (2005), no.\ 2, 655--682. \mr{2198582}

	\bibitem[BB]{BB}
	A.~Bj\"{o}rn and J.~Bj\"{o}rn,
	\emph{Nonlinear potential theory on metric spaces},
	EMS Tracts Math., 17 \emph{European Mathematical Society (EMS), Z\"{u}rich}, 2011. \mr{2867756}

	\bibitem[BBS22]{BBS22} A.~Bj\"{o}rn, J.~Bj\"{o}rn and N.~Shanmugalingam,
	Extension and trace results for doubling metric measure spaces and their hyperbolic fillings,
	\emph{J. Math. Pures Appl. (9)} \textbf{159} (2022), 196--249. \mr{4377995}

	\bibitem[Bon]{Bon}
	M. Bonk.
	Quasiconformal geometry of fractals.
	{\em International Congress of Mathematicians.} Vol. {\bf II}, 1349--1373, Eur. Math. Soc., Z\"urich, 2006. \mr{2275649}

	\bibitem[BK02]{BK02} M.~Bonk and B.~Kleiner,
	Quasisymmetric parametrizations of two-dimensional metric spheres,
	\emph{Invent. Math.} \textbf{150} (2002), no. 1, 127--183. \mr{1930885}

	\bibitem[BK05]{BK05} M.~Bonk and B.~Kleiner,
	Conformal dimension and Gromov hyperbolic groups with $2$-sphere boundary,
	\emph{Geom. Topol.} Volume \textbf{9}, Number 1 (2005), 219--246. \mr{2116315}

	\bibitem[BM22]{BM22}
	M.~Bonk and D.~Meyer,
	Uniformly branching trees,
	\emph{Trans. Amer. Math.} \textbf{375} (2022), no. 6, 3841--3897. \mr{4419049}

	\bibitem[BS18]{BS18} M.~Bonk and E.~Saksman,
	Sobolev spaces and hyperbolic fillings,
	\emph{J. Reine Angew. Math.} \textbf{737} (2018), 161--187. \mr{3781334}

	\bibitem[BK13]{BK13} M.~Bourdon, B.~Kleiner,
	Combinatorial modulus, the combinatorial Loewner property, and Coxeter groups,
	\emph{Groups Geom. Dyn.} \textbf{7} (2013), 39--107. \mr{3019076}

	\bibitem[BP03]{BP03} M.~Bourdon, H.~Pajot,
	Cohomologie $\ell_{p}$ et espaces de Besov,
	\emph{J. Reine Angew. Math.} \textbf{558} (2003), 85--108. \mr{1979183}

	\bibitem[BBI]{BBI} D. Burago, Y. Burago and S. Ivanov.
	A course in Metric Geometry,
	{\em Graduate Studies in Mathematics}, {\bf 33}. American Mathematical Society, Providence, RI, 2001. \mr{1835418}

	\bibitem[BS]{BS}
	S.~Buyalo and V.~Schroeder,
	{\em Elements of asymptotic geometry},
	EMS Monographs in Mathematics., {\em European Mathematical Society (EMS), Z\"{u}rich}, 2007. \mr{2327160}

	\bibitem[CCK24]{CCK24} S.~Cao, Z.-Q.~Chen and T.~Kumagai,
	On Kigami's conjecture of the embedding $\mathcal{W}^{p} \subset C(K)$,
	\emph{Proc. Amer. Math. Soc.} {\bf 152} (2024), no. 8, 3393--3402. \mr{4767270}


	\bibitem[CQ24]{CQ21+} S.~Cao and H.~Qiu,
	Dirichlet forms on unconstrained Sierpinski carpets,
	\emph{Probab.\ Theory Related Fields} \textbf{189} (2024), no.\ 1-2, 613--657. 
    \mr{4750565}

	\bibitem[CQ23+]{CQ+} S.~Cao and H.~Qiu,
	Dirichlet forms on unconstrained Sierpinski carpets in $\mathbb{R}^{3}$, 
	preprint (2023). 
	\arxiv{2305.09292}

	\bibitem[CP13]{CP13} M.~Carrasco Piaggio,
	On the conformal gauge of a compact metric space,
	\emph{Ann. Sci. \'{E}c. Norm. Sup\'{e}r.} (4) \textbf{46} (2013), no. 3, 495--548. \mr{3099984}

	\bibitem[Che99]{Che99} J.~Cheeger,
	Differentiability of Lipschitz functions on metric measure spaces,
	{\em Geom. Funct. Anal.} {\bf 9} (1999), no. 3, 428--517. \mr{1708448}

	\bibitem[CE]{CE} J.~Cheeger and S.~Eriksson-Bique,
	Thin Loewner carpets and their quasisymmetric embeddings in $S^2$,
	{\em Comm. Pure Appl. Math.} {\bf 76} (2023), no. 2, 225--304. \mr{4544798}

	\bibitem[CF]{CF}
	Z.-Q.~Chen and M.~Fukushima,
	{\em Symmetric Markov processes, time change, and boundary theory},
	London Mathematical Society Monographs Series, \textbf{35},
	Princeton University Press, Princeton, NJ, 2012. \mr{2849840}
	
	\bibitem[Chi90]{Chi90} M.~Christ,
	A $T(b)$ theorem with remarks on analytic capacity and the Cauchy integral,
	{\em Colloq. Math.} {\bf 60/61} (1990), no.\ 2, 601--628. \mr{1096400}


	\bibitem[Cla36]{Cla36} J.~A.~Clarkson,
	Uniformly convex spaces,
	\emph{Trans. Amer. Math. Soc.}\ \textbf{40} (1936),
	no.\ 3, 396--414. \mr{1501880}

	\bibitem[CW]{CW} R.~R.~Coifman and G.~Weiss,
	Analyse Harmonique Non-commutative sur Certains Espaces Homog\'{e}nes,
	\emph{Lecture Notes in Mathematics} vol. 242, Springer-Verlag, Berlin-New York (1971) \mr{4535921}

	\bibitem[Dal]{DalMaso}
	G.~Dal~Maso,
	{\em An introduction to $\Gamma$-convergence},
	Progr. Nonlinear Differential Equations Appl., 8 {\em Birkh\"{a}user Boston, Inc., Boston, MA}, 1993. \mr{1201152}
	
	\bibitem[Dav88]{Dav88} G.~David,
	Morceaux de graphes lipschitziens et int\'egrales singuli\'eres sur une surface, 
	\emph{Rev. Mat. Iberoamericana} \textbf{4} (1988), no.\ 1, 73--114. \mr{1009120}

	\bibitem[DS90]{DS90} G.~David and S.~Semmes,
	Strong $A^{\infty}$ weights, Sobolev inequalities and quasiconformal mappings, In:
	\emph{Analysis and partial differential equations}, 101--111.
	Lecture Notes in Pure and Appl. Math., 122, \emph{Marcel Dekker, Inc., New York}, 1990. \mr{1044784}

	\bibitem[Dud]{Dud}
	R.~M.~Dudley,
	{\em Real analysis and probability}, Revised reprint of the 1989 original,
	Cambridge Stud. Adv. Math., 74
	\emph{Cambridge University Press, Cambridge}, 2002. \mr{1932358}
	
	\bibitem[Ele97]{Ele97} G.~Elek, 
	The $\ell_p$-cohomology and the conformal dimension of hyperbolic cones, 
	\emph{Geom. Dedicata} \textbf{68} (1997), no.\ 3, 263--279. \mr{1486435} 


	\bibitem[FOT]{FOT} M.~Fukushima, Y.~Oshima, and M.~Takeda,
	\emph{Dirichlet Forms and Symmetric Markov Processes},
	Second revised and extended edition, de Gruyter Studies in Mathematics, vol.\ 19,
	Walter de Gruyter \& Co., Berlin, 2011. \mr{2778606}

	\bibitem[GHL03]{GHL03}
	A.~Grigor'yan, J.~Hu and K.-S.~Lau,
	Heat kernels on metric measure spaces and an application to semilinear elliptic equations,
	{\em Trans. Amer. Math. Soc.} {\bf 355} (2003), no. 5, 2065--2095. \mr{1953538}

	\bibitem[Ha\"i09]{Hai09}
	P.~Ha\"{i}ssinsky,
	Empilements de cercles et modules combinatoires,
	{\em Ann. Inst. Fourier (Grenoble)} {\bf 59} (2009), no. 6, 2175--2222. \mr{2640918}

	\bibitem[Haj96]{Haj96}
	P.~Haj\l asz,
	Sobolev spaces on an arbitrary metric space,
	{\em Potential Anal.} {\bf 5} (1996), no. 4, 403--415. \mr{1401074}

	\bibitem[HK95]{HK95}
	P.~Haj\l asz and P.~Koskela,
	Sobolev meets Poincar\'e,
	{\em C. R. Acad. Sci. Paris S\'{e}r. I Math.} {\bf 320} (1995), no. 10, 1211--1215. \mr{1336257}

	\bibitem[HK00]{HK00}
	P.~Haj\l asz and P.~Koskela,
	Sobolev met Poincar\'e,
	{\em Mem. Amer. Math. Soc.} {\bf 145} (2000), no. 688. \mr{1683160}

	\bibitem[HS95]{HS95}
	Z.-X.~He and O.~Schramm,
	Hyperbolic and parabolic packings,
	{\em Discrete Comput. Geom.} {\bf 14} (1995), no. 2, 123--149. \mr{1331923}

	\bibitem[Hei]{Hei} J.~Heinonen,
	Lectures on Analysis on Metric Spaces,
	{\em Universitext. Springer-Verlag}, New York, 2001. \mr{1800917}

 	\bibitem[HK98]{HK98}
	J.~Heinonen and P.~Koskela,
	Quasiconformal maps in metric spaces with controlled geometry,
	{\em Acta Math.} {\bf 181} (1998), no. 1, 1--61. \mr{1654771}

	\bibitem[HKST]{HKST}
	J.~Heinonen, P.~Koskela, N.~Shanmugalingam and J.~T.~Tyson.
	\emph{Sobolev spaces on metric measure spaces. An approach based on upper gradients.} 
	New Math. Monogr., 27 \emph{Cambridge, University Press, Cambridge}, 2015. \mr{3363168}

	\bibitem[Hin05]{Hin05} M.~Hino,
	On singularity of energy measures on self-similar sets,
	{\em Probab. Theory Related Fields} {\bf 231} (2005), no. 2, 265--290. \mr{2199293}

	\bibitem[Hin10]{Hin10} M.~Hino,
	Energy measures and indices of Dirichlet forms, with applications to derivatives on some fractals,
	{\em Proc. Lond. Math. Soc.} (3) {\bf 100} (2010), no. 1, 269--302. \mr{2578475}

	\bibitem[Hin13]{Hin13} M.~Hino,
	Upper estimate of martingale dimension for self-similar fractals,
	\emph{Probab. Theory Related Fields}\ \textbf{156} (2013), no.\ 3-4, 739--793. \mr{3078285}
	
	\bibitem[Hol03]{Hol03} I.~Holopainen, 
	Quasiregular mappings and the $p$-Laplace operator; in \emph{Heat kernels and analysis on manifolds, graphs, and metric spaces (Paris, 2002)}, 219--239. 
	Contemp. Math., \textbf{338} 
	\emph{American Mathematical Society, Providence, RI}, 2003 \mr{2039956}

	\bibitem[HS97]{HS97} I.~Holopainen and P.~M.~Soardi,
	$p$-harmonic functions on graphs and manifolds,
	{\em Manuscripta Math.} {\bf 94} (1997), no. 1, 95--110. \mr{1468937}

	\bibitem[Hut81]{Hut81} J.~E.~Hutchinson,
	Fractals and self-similarity,
	\emph{Indiana Univ. Math. J.}\ \textbf{30} (1981), no.\ 5, 713--747. \mr{0625600}

	\bibitem[HK12]{HK12} T.~Hyt\"{o}nen and A.~Kairema,
	Systems of dyadic cubes in a doubling metric space,
	\emph{Colloq. Math.} \textbf{126} (2012), no.\ 1, 1--33. \mr{2901199}

	\bibitem[Jon96]{Jon96} A.~Jonsson,
	Brownian motion on fractals and function spaces,
	\emph{Math. Z.}\ \textbf{222} (1996), no.\ 3, 495--504. \mr{1400205}
	
	\bibitem[Jost]{Jost}
	J.~Jost,
	{\em Postmodern analysis}, Third edition, Universitext,
	{\em Springer-Verlag, Berlin}, 2005. \mr{2166001}
	
	\bibitem[KRS12]{KRS12} A.~K\"{a}enm\"{a}ki, R.~Rajala and V.~Suomala,
	Existence of doubling measures via generalised nested cubes,
	\emph{Proc. Amer. Math. Soc.} \textbf{140} (2012), no. 9, 3275--3281. \mr{2917099}

	\bibitem[Kaj10]{Kaj10} N.~Kajino,
	Remarks on non-diagonality conditions for Sierpinski carpets, In: {\em Probabilistic approach to geometry}, 231--241,
	Adv. Stud. Pure Math., {\bf 57}, {\em Math. Soc. Japan, Tokyo}, 2010. \mr{2648262}

	\bibitem[KM20]{KM20} N.~Kajino and M.~Murugan,
	On singularity of energy measures for symmetric diffusions with full off-diagonal heat kernel estimates,
	{\em Ann. Probab.} {\bf 48} (2020), no. 6, 2920--2951. \mr{4164457}

	\bibitem[KM23]{KM23} N.~Kajino and M.~Murugan,
	On the conformal walk dimension: quasisymmetric uniformization for symmetric diffusions,
	{\em Invent. Math.} {\bf 231} (2023), no. 1, 263--405. \mr{4526824} 
	
	
	\bibitem[KS24+a]{KS.lim} N.~Kajino and R.~Shimizu,
	Korevaar--Schoen $p$-energy forms and associated energy measures on fractals,
	\emph{Springer Tohoku Series in Mathematics} (to appear). \arxiv{2404.13435}
	
	\bibitem[KS24+b]{KS24+} N.~Kajino and R.~Shimizu,
	Contraction properties and differentiability of $p$-energy forms with applications to nonlinear potential theory on self-similar sets, preprint (2024).   
	\arxiv{2404.13668v2} 



	\bibitem[KL04]{KL04} S.~Keith and T.~Laakso,
	Conformal Assouad dimension and modulus,
	\emph{Geom. Funct. Anal.}\ \textbf{14} (2004), no. 6, 1278--1321. \mr{2135168}

	\bibitem[Kig00]{Kig00} J.~Kigami,
	Markov property of Kusuoka-Zhou's Dirichlet forms on self-similar sets,
	\emph{J. Math. Sci. Univ. Tokyo}\ \textbf{7} (2000), no.\ 1, 27--33. \mr{1749979}

	\bibitem[Kig01]{Kig01} J.~Kigami,
	\emph{Analysis on Fractals}, Cambridge Tracts in Math., vol. 143,
	Cambridge University Press, Cambridge, 2001. \mr{1840042}

	\bibitem[Kig09]{Kig09} J.~Kigami,
	Volume doubling measures and heat kernel estimates on self-similar sets,
	\emph{Mem.\ Amer.\ Math.\ Soc.}\ \textbf{199}, no.\ 932 (2009). \mr{2512802}
	

	\bibitem[Kig20]{Kig20} J.~Kigami,
	Geometry and analysis of metric spaces via weighted partitions,
	\emph{Lecture Notes in Mathematics}\ \textbf{2265},
	Springer, Cham, 2020. \mr{4175733}

	\bibitem[Kig23]{Kig23} J.~Kigami,
	Conductive homogeneity of compact metric spaces and construction of $p$-energy,
	\emph{Mem. Eur. Math. Soc.}\ Vol.\ 5 (2023). \mr{4615714} 

	\bibitem[Kle]{Kle}
	B. Kleiner.
	The asymptotic geometry of negatively curved spaces: uniformization, geometrization and rigidity,
	{\em International Congress of Mathematicians.} Vol. {\bf II}, 743--768, Eur. Math. Soc., Z\"urich, 2006. \mr{2275621}

	\bibitem[Kle+]{Kle+} B.~Kleiner,
	Private communication (March 2023).

	\bibitem[KoSc]{KS} N.~J.~Korevaar, R.~M.~Schoen,  Sobolev spaces and harmonic maps for metric space targets. {\em Comm. Anal. Geom.} {\bf 1} (1993), no. 3-4, 561--659.

	\bibitem[KoMa]{KoMa} P.~Koskela and P.~MacManus,
	Quasiconformal mappings and Sobolev spaces.
	{\em Studia Math.} {\bf 131} (1998), no. 1, 1--17. \mr{1628655}

	\bibitem[Kum00]{Kum00} T.~Kumagai,
	Brownian motion penetrating fractals: an application of the trace theorem of Besov spaces,
	\emph{J. Funct. Anal.} \textbf{170} (2000), no.\ 1, 69--92. \mr{1736196}
	
	\bibitem[Kus89]{Kus89} S.~Kusuoka,
	Dirichlet forms on fractals and products of random matrices, 
	\emph{Publ. Res. Inst. Math. Sci.} \textbf{25} (1989), no. 4, 659--680. \mr{1025071}

	\bibitem[KZ92]{KZ92} S.~Kusuoka and X.~Y.~Zhou,
	Dirichlet forms on fractals: Poincar\'{e} constant and resistance,
	\emph{Probab.\ Theory Related Fields} \textbf{93} (1992), no.\ 2, 169--196. \mr{1176724}
	
	\bibitem[Kuw24]{Kuw24} K.~Kuwae,
    $(1,p)$-Sobolev spaces based on strongly local Dirichlet forms, 
    \emph{Math. Nachr.}, \textbf{297} (2024), no. 10, 3723–3740. \doi{10.1002/mana.202400025}
    
    \bibitem[Laa00]{Laa00} T.~Laakso,
    Ahlfors $Q$-regular spaces with arbitrary $Q>1$ admitting weak Poincar\'{e} inequality, 
	\emph{Geom. Funct. Anal.} \textbf{10} (2000), no. 1, 111--123. \mr{1748917}
    
    \bibitem[Laa02]{Laa02} T.~Laakso,
    Plane with $A_{\infty}$-weighted metric not bi-Lipschitz embeddable to $\mathbb{R}^{N}$, 
    \emph{Bull. London Math. Soc.} \textbf{34} (2002), no. 6, 667--676. \mr{1924353}

	\bibitem[LP04]{LP04} G.~Lupo-Krebs and H.~Pajot,
	Dimensions conformes, espaces Gromov-hyperboliques et ensembles autosimilaires, In: \emph{S\'{e}minaire de Th\'{e}orie Spectrale et G\'{e}om\'{e}trie.} Vol. 22. Ann\'{e}e 2003-2004, 153--182. S\'{e}min. Th\'{e}or. Spectr. G\'{e}om., 22, \emph{Universit\'{e} de Grenoble I, Institut Fourier, Saint-Martin-d'H\'{e}res}, 2004. \mr{2136141}

	\bibitem[MR]{MR} Z.~M.~Ma and M. R\"{o}ckner,
	Introduction to the theory of (nonsymmetric) Dirichlet forms,
	{\em Universitext. Springer-Verlag}, Berlin, 1992. \mr{1214375}
	
	\bibitem[Mc02]{Mc02} I.~McGillivray,
	Resistance in higher-dimensional Sierpi\'{n}ski carpets,
	\emph{Potential Anal.} \textbf{16} (2002), no.\ 3, 289--303. 
	\mr{1885765}

	\bibitem[MT]{MT} J.~M.~Mackay and J.~T.~Tyson,
	{\em Conformal dimension. Theory and application.},
	University Lecture Series, 54. American Mathematical Society, Providence, RI, 2010. \mr{2662522}

	\bibitem[Maz]{Maz} V.~Maz'ya,
	Sobolev spaces with applications to elliptic partial differential equations, Second, revised and augmented edition,
	Grundlehren der mathematischen Wissenschaften [Fundamental Principles of Mathematical Sciences], 342. {\em Springer, Heidelberg}, 2011. \mr{2777530}

	\bibitem[Meg]{Meg} R.~E.~Megginson,
	\emph{An introduction to Banach space theory},
	Grad. Texts in Math., vol. 183, 
	\emph{Springer-Verlag, New York}, 1998. \mr{1650235}

	\bibitem[Mor46]{Mor46} P.~A.~P.~Moran,
	Additive functions of intervals and Hausdorff measure,
	\emph{Proc. Cambridge Philos. Soc.}\ \textbf{142} (1946), 15--23. \mr{0014397}

	\bibitem[MY92]{MY92} A.~Murakami and M.~Yamasaki,
	Nonlinear potentials on an infinite network,
	\emph{Mem. Fac. Sci. Shimane Univ.} \textbf{26} (1992), 15--28. \mr{1212945}

	\bibitem[Mur20]{Mur20} M.~Murugan,
	On the length of chains in a metric space,
	\emph{J. Funct. Anal.}\ \textbf{279} (2020), no.\ 6, 108627, 18 pp. \mr{4099475}

	\bibitem[Mur23a]{Mur18+} M.~Murugan,
	A note on heat kernel estimates, resistance bounds and Poincar\'{e} inequality,
	\emph{Asian J. Math.} \textbf{27} (2023), no.\ 6, 853--866. \mr{4787144}
    
    \bibitem[Mur23b]{Mur23b} M.~Murugan,
	Conformal Assouad dimension as the critical exponent for combinatorial modulus. 
	\emph{Ann. Fenn. Math.} {\bf 48} (2023), no.2, 453--491. \mr{4611381}
	
	\bibitem[Mur24]{Mur23+} M.~Murugan,
	Heat kernel for reflected diffusion and extension property on uniform domains,
    \emph{Probab.\ Theory Related Fields} (2024). \mr{4797375} 

	\bibitem[Nak85]{Nak85} S.~Nakao,
	Stochastic calculus for continuous additive functionals of zero energy,
	\emph{Z. Wahrsch. Verw. Gebiete}\ \textbf{69} (1985), no.\ 4, 557--578. \mr{0772199}

	\bibitem[Pan]{Pan} P. Pansu. Dimension conforme et sph\`ere \`a l'infini des vari\'et\'es \`a courbure n\'egativ. \emph{Ann. Acad. Sci. Fenn. Ser. A I Math.} {\bf 14} (1989), no. 2, 177--212.

	\bibitem[PP99]{PP99} K.~Pietruska-Pa\l uba,
	Some function spaces related to the Brownian motion on simple nested fractals,
	{\em Stochastics Stochastics Rep.}\ \textbf{67} (1999), no.\ 3-4, 267--285. \mr{1729479}

	\bibitem[Sal02]{Sal02} L.~Saloff-Coste,
	\emph{Aspects of Sobolev-type inequalities},
	London Math. Soc. Lecture Note Ser., 289 
	\emph{Cambridge University Press, Cambridge}, 2002. \mr{1872526}

	\bibitem[Sas23]{Sas23} K.~Sasaya,
	Systems of dyadic cubes of complete, doubling, uniformly perfect metric spaces without detours,
	{\em Colloq. Math.} \textbf{172} (2023), no. 1, 49--64. \mr{4565993}  

	\bibitem[Sha00]{Sha00} N.~Shanmugalingam,
	Newtonian spaces: an extension of Sobolev spaces to metric measure spaces,
	{\em Rev. Mat. Iberoamericana} \textbf{16} (2000), no. 2, 243--279. \mr{1809341}
	
	\bibitem[Sha23]{Sha23} N.~Shanmugalingam, 
	On Carrasco Piaggio's theorem characterizing quasisymmetric maps from compact doubling spaces to Ahlfors regular spaces. \emph{Potentials and partial differential equations—the legacy of David R. Adams}, 23--48.
		\emph{Adv. Anal. Geom.}, {\bf 8}
		De Gruyter, Berlin, [2023]. \mr{4654511}

	\bibitem[Shi21]{Shi21} R.~Shimizu,
	Parabolic index of an infinite graph and Ahlfors regular conformal dimension of a self-similar set, in: \emph{Analysis and partial differential equations on manifolds, fractals and graphs}, 201--274, dv. Anal. Geom., 3 
	\emph{De Gruyter, Berlin} (2021). \mr{4320091}

	\bibitem[Shi24]{Shi24} R.~Shimizu,
	Construction of $p$-energy and associated energy measures on Sierpi\'{n}ski carpets,
    \emph{Trans. Amer. Math. Soc.} \textbf{377} (2024), no.\ 2, 951--1032. \mr{4688541}


	\bibitem[Tys00]{Tys00} J.~T.~Tyson,
	Sets of minimal Hausdorff dimension for quasiconformal maps,
	{\em Proc. Amer. Math. Soc.} \textbf{128} (2000), no. 11, 3361--3367. \mr{1676353}

	\bibitem[Yan23+]{Yan+} M.~Yang,
	Korevaar--Schoen spaces on Sierpi\'{n}ski carpets, 
	preprint (2023). 
	\arxiv{2306.09900} 

\end{thebibliography}
\end{document}